\documentclass{amsart}
\usepackage{amssymb}
\usepackage{amsfonts}
\usepackage{amsmath}
\usepackage{amsthm, amsmath, amssymb, enumerate}
\usepackage{amssymb, enumitem}
\usepackage{bbm}
\usepackage{amscd}
\usepackage[all]{xy}
\usepackage[hidelinks]{hyperref}
\usepackage{amsmath}
\usepackage{amssymb}
\usepackage{wasysym}
\usepackage{amsthm}
\usepackage[left=2cm,right=2cm,bottom=3cm,top=3cm]{geometry}
\usepackage[sort&compress,numbers]{natbib}
\usepackage{tikz}
\usepackage{tikz-cd}
\usepackage{wrapfig}

\makeatletter
\@namedef{subjclassname@2020}{\textup{2020} Mathematics Subject Classification}
\makeatother

\newtheorem{theorem}{Theorem}[section]
\newtheorem{proposition}[theorem]{Proposition}
\theoremstyle{definition}
\newtheorem{lemma}[theorem]{Lemma}
\newtheorem*{claim*}{Claim}
\newtheorem{claim}[theorem]{Claim}
\newtheorem{definition}[theorem]{Definition}
\newtheorem{corollary}[theorem]{Corollary}

\newtheorem{remark}[theorem]{Remark}

\newtheorem{question}[theorem]{Question}
\newtheorem{task}[theorem]{Task}

\newtheorem{example}[theorem]{Example}

\AtBeginDocument{%
   \def\MR#1{}
}

\begin{document}
\title[The definable content of homological invariants II]{The definable content of homological invariants  II:\\ 
\v{C}ech cohomology and homotopy classification}
\author{Jeffrey Bergfalk}
\address{Departament de Matem\`{a}tiques i Inform\`{a}tica \\
Universitat de Barcelona \\
Gran Via de les Corts Catalanes, 585 \\ 08007 Barcelona, Catalonia}
\email{bergfalk@ub.edu}
\author{Martino Lupini}
\address{Dipartimento di Matematica\\
Universit\`{a} di Bologna\\
Piazza di Porta S. Donato, 5\\
40126 Bologna BO\\
Italy}
\email{lupini@tutanota.com}
\urladdr{http://www.lupini.org/}
\author{Aristotelis Panagiotopoulos}
\address{Department of Mathematical Sciences\\ Carnegie Mellon University\\
 Pittsburgh, PA 15213, USA}
\email{aristotelis.panagiotopoulos@gmail.com}
\thanks{Part of this work was done during visits of J.B. and M.L. to the
California Institute of Technology, and of A.P. to Victoria University of
Wellington. The authors gratefully acknowledge the hospitality and the
financial support of these institutions. J.B. was partially supported by Austrian Science Foundation (FWF) Grant Number
Y1012-N35 and a University of Barcelona Mar\'{i}a Zambrano Grant. 
M.L. was partially supported by the
NSF Grant DMS-1600186, by a Research Establishment Grant from Victoria
University of Wellington, by a Research Establishment Grant from Victoria
University of Wellington, by the Marsden Fund Fast-Start
Grant VUW1816, by the Rutherford Discovery Fellowship VUW2002 from the Royal
Society of New Zealand, and by the Starting Grant 101077154 ``Definable
Algebraic Topology'' from the European Research Council. A.P.  was partially supported by the NSF Grant DMS-2154258.}

\begin{abstract}
This is the second installment in a series of papers applying descriptive set theoretic techniques to both analyze and enrich classical functors from homological algebra and algebraic topology.
In it, we show that the \v{C}ech cohomology functors $\check{\mathrm{H}}^n$ on the category of locally compact separable metric spaces each factor into (i) what we term their \emph{definable version}, a functor $\check{\mathrm{H}}^n_{\mathrm{def}}$ taking values in the category $\mathsf{GPC}$ of \emph{groups with a Polish cover} (a category first introduced in this work's predecessor), followed by (ii) a forgetful functor from $\mathsf{GPC}$ to the category of groups.
These definable cohomology functors powerfully refine their classical counterparts: we show that they are complete invariants, for example, of the homotopy types of mapping telescopes of $d$-spheres or $d$-tori for any $d\geq 1$, and, in contrast, that there exist uncountable families of pairwise homotopy inequivalent mapping telescopes of either sort on which the classical cohomology functors are constant.
We then apply the functors $\check{\mathrm{H}}^n_{\mathrm{def}}$ to 
show that a seminal problem in the development of algebraic topology, namely Borsuk and Eilenberg's 1936 problem of classifying, up to homotopy, the maps from a solenoid complement $S^3\backslash\Sigma$ to the $2$-sphere, is essentially hyperfinite but not smooth.

Fundamental to our analysis is the fact that the \v{C}ech cohomology functors $X\mapsto\check{\mathrm{H}}^n(X;G)$ admit two main formulations: a more combinatorial one, and a more homotopical formulation as the group $[X,P]$ of homotopy classes of maps from $X$ to a polyhedral $K(G,n)$ space $P$. We describe the Borel-definable content of each of these formulations and prove a definable version of Huber's theorem reconciling the two. In the course of this work, we record definable versions of Urysohn's Lemma and the simplicial approximation and homotopy extension theorems, along with a definable Milnor-type short exact sequence decomposition of the space $\mathrm{Map}(X,P)$ in terms of its subset of \emph{phantom maps}; relatedly, we provide a topological characterization of this set for any locally compact Polish space $X$ and polyhedron $P$.
In aggregate, this work may be more broadly construed as laying foundations for the descriptive set theoretic study of the homotopy relation on such spaces $\mathrm{Map}(X,P)$, a relation which, together with the more combinatorial incarnation of $\check{\mathrm{H}}^n$, embodies a substantial variety of classification problems arising throughout mathematics. We show in particular that if $P$ is a polyhedral $H$-group then this relation is both Borel and idealistic. In consequence, $[X,P]$ falls in the category of \emph{definable groups}, an extension of the category $\mathsf{GPC}$ introduced herein for its regularity properties, which facilitate several of the aforementioned computations.
\end{abstract}

\maketitle
\tableofcontents

\section{Introduction}

This work is the second in a series of papers enacting a simple but far-reaching recognition, which is the following: many of the classical functors $F:\mathcal{C}\to\mathcal{D}$ of homological algebra and algebraic topology factor through what we might loosely term a \emph{definable version} of the category $\mathcal{D}$, as in the diagram below.
\begin{equation}
\label{diagram:definable_lift}
\begin{tikzcd}
\mathcal{C}  \arrow[rr, dashed,"\text{definable }F"] \arrow[rrd,"F"] && \text{definable }\mathcal{D} \arrow[d,"E\text{, a forgetful functor}"] \\
&& \mathcal{D}
\end{tikzcd}
\end{equation}
Two fundamental examples are the functors $\mathrm{Ext}$ and $\mathrm{lim}^1$, each of which maps to the category $\mathcal{D}=\mathsf{Grp}$ of groups.
As this work's predecessor  \citep{BLPI} showed, the restrictions of these functors to pairs and towers of countable abelian groups, respectively, each admit a canonical lift to the category $\mathsf{GPC}$ of \emph{groups with a Polish cover}, a category first identified in \cite{BLPI} as playing, in these contexts, exactly this role of a \emph{definable $\mathcal{D}$}.
The objects of $\mathsf{GPC}$ are pairs $(G,N)$ consisting of a Polish group $G$ together with a Polishable normal subgroup $N\leq G$; its morphisms $(G,N)\to (H,M)$ are those group homomorphisms $G/N\to H/M$ which lift to a Borel function $G\to H$. More conceptually, its morphisms are those homomorphisms $G/N\to H/M$ which are definable from the Polish topologies on $G$ and $H$ via a computational or expressive power consisting in ``logic gates" of countably infinite length.
As shown in \cite{BLPI}, $\mathsf{GPC}$ is a significantly finer or more rigid category than $\mathsf{Grp}$; in particular, $\mathrm{Hom}_{\mathsf{GPC}}((G,N),(H,M))$ is often a proper --- and sometimes quite sparse --- subset of $\mathrm{Hom}_{\mathsf{Grp}}((G/N),(H/M))$, and from this it follows that the aforementioned \emph{definable lifts} of the functors $\mathrm{Ext}$ and $\mathrm{lim}^1$ provide strictly stronger invariants of pairs and towers of countable abelian groups than their classical counterparts.

The present work extends this analysis to the \v{C}ech cohomology functors from the category $\mathcal{C}=\mathsf{LC}$ of locally compact Polish spaces to the category $\mathcal{D}$ of groups.
This entails a coordinated study of the Borel content of both the combinatorial and homotopical presentations of these groups.
For our work on the latter, and on the homotopy bracket $[-,-]$ more generally, we introduce several further definable categories, $\mathsf{DSet}$ and $\mathsf{DGrp}$, each of which extends the subcategory $\mathsf{GPC}$ while retaining many of its regularity properties.
By way of this analysis, we show that, like $\mathrm{Ext}$ and $\mathrm{lim}^1$, the functors $\check{\mathrm{H}}^n:\mathsf{LC}\to\mathsf{Grp}$ each factor into a \emph{definable cohomology functor} to $\mathsf{GPC}$ followed by a forgetful functor from $\mathsf{GPC}$ to $\mathsf{Grp}$.
In consequence, much as before, definable cohomology is a significantly stronger invariant of topological spaces than its classical counterpart. This we concretely show via a comparison of the classical and definable cohomology groups of mapping telescopes of spheres and tori, objects of central importance to the field of algebraic topology (to the study of localizations, or the construction of Eilenberg-MacLane spaces, for example) \cite{milnor_axiomatic_1962, sullivan_geometric_1970, bousfield_homotopy_1972, may_more_2012, bott_differential_1982}.

We precede a more detailed discussion of this paper's contents with a few further words of context.
We would stress from the outset, for example, that none of the definable functors under discussion provide particularly esoteric, unmanageable, or hard to compute invariants for topological spaces or groups; on the contrary, \emph{these invariants amount to little more than a retention of the topologies arising naturally in classical computations}. As such, they realize a historically persistent impulse within the fields of homological algebra and algebraic topology; to take just one example, Moore's 1976 work \cite{moore_group_1976} records the following striking premonition of the category $\mathsf{GPC}$:
\smallskip
\begin{quote}
There is one complication in this theory which, as we shall see, simply cannot be avoided and this comes about as follows: We have polonais [i.e., Polish] groups $A$ and $B$ (say abelian) together with a continuous homomorphism from $A$ into $B$.
The quotient group $B/j(A)$ will be of interest and significance even though in many cases $j(A)$ is not closed in $B$.
This quotient group would be polonais if $j(A)$ were closed but in general it is some non-Hausdorff topological group, arising in some sense as the ``quotient'' of two polonais groups.
It will be somewhat useful to talk about such objects which with some trepidation one might call pseudo-polonais groups. We would define such objects as triples $C=(A,B,j)$ where $A$ and $B$ are polonais and $j$ is a continuous homomorphism of $A$ into $B$, subject to an appropriate equivalence relation which we shall not pursue at this moment.
\end{quote}
\smallskip
See also Brown's 1975 remarks on \emph{almost polonais groups} in \cite{brown_operator_1975}, and his references therein and in his 1977 article \cite{brown_extensions_1977} to projected (but never subsequently published) works on this theme.
It seems likely that what these recognitions' further development awaited was some framework for the efficient manipulation of Moore's triples $(A,B,j)$ (which are evidently equivalent in content to the objects $(G,N)$ of $\mathsf{GPC}$), and it is in the derivation of just such a framework from the apparatus of invariant descriptive set theory that our work's main contribution and novelty arguably consist.

Along these lines, affinities of the present work with another major contemporary research orientation should be noted. This is the program of ``doing algebra with topology'' motivating, for example, the \emph{condensed mathematics} and \emph{pyknotic mathematics} frameworks of Clausen and Scholze, and Barwick and Haine, respectively (\cite{scholze_condensed_2019,barwick_pyknotic_2019}; see also \cite{hoffmann_homological_2007}). Though our methods are rather different, much of the underlying impetus is the same: it is the issue of ``bad quotients'' like those which Moore describes above; more formally, it is the failure of settings like the collection of Polish abelian groups to form an abelian category.
In contrast, rather remarkably, as this work's second author has recently shown, the category $\mathsf{APC}$ of groups with an abelian Polish cover \emph{is} an abelian category, one which may moreover be regarded as the canonical abelian extension of the category of Polish abelian groups, in the precise sense that it forms that category's \emph{left heart} \cite{lupini_looking_22}.
This helps to explain why such large portions of homological algebra and algebraic topology lift to $\mathsf{APC}$ and related settings; establishing the main theorems of \cite{BLPI} and the present work, for example, entailed the development of definable versions of such a range of core results like the homotopy extension theorem, the simplicial approximation theorem, Milnor exact sequences, Urysohn's Lemma, the Snake Lemma, and so on, that it grows natural to speak of an emergent field of \emph{definable homological algebra}.
And it is in part in these terms also that our project should be understood.

At the same time, the present work extends the framework of diagram \ref{diagram:definable_lift} to \emph{homotopical} and non-abelian settings, through our study of the representation of $\check{\mathrm{H}}^n(X;G)$ as the group $[X,K(G,n)]$ of homotopy classes of maps from $X$ to an Eilenberg-MacLane space $K(G,n)$.
As noted, and as we'll describe in greater detail just below, this analysis necessitated the introduction of a larger category of \emph{definable groups} $\mathsf{DGrp}$, one which features $\mathsf{GPC}$ as a full subcategory. This is the natural category for the definable analysis of the homotopy bracket $[-,P]$ of maps to a polyhedron $P$, and hence (by Brown Representability) of generalized homology and cohomology theories more broadly, as we discuss in our conclusion. The development of definable homotopy groups is a naturally ensuing prospect as well (see Remark \ref{remark:homotopy_groups}).

This brings us to one last preliminary point. It should be clear from our account so far that much of what we herein term \emph{definable} connotes what might be more precisely rendered \emph{Borel definable}, but it can also, as in the case of definable sets and groups, mean a little more (see Section \ref{Ss:DefGrp} below).
Put differently, our uses of the term herein have been guided by considerations both of concision and of the interrelations between our operative categories, but should be everywhere read to signal a concern for structures and functions embodying only countable amounts of data, in senses such as we noted above.

We turn now to a more detailed description of our results.


\subsection{Definable \v{C}ech cohomology} As noted, this work builds on its predecessor \cite{BLPI}, which introduced definable variants of the functors $\mathrm{Ext}$, $\mathrm{Pext}$, and $\mathrm{lim}^{1}$, each taking values in the category \textsf{GPC} of groups with a Polish cover. What warrants a view of $\mathrm{Ext}(B,F)$, $\mathrm{Pext}(B,F)$, and $\mathrm{lim}^{1}(\boldsymbol{A})$ as groups with a Polish cover is the observation in \cite{BLPI} that each is naturally construed as a cohomology group  $\mathrm{H}^n(C^{\bullet})=\mathrm{ker}(\delta^{n})/\mathrm{im}(\delta^{n-1})$ of an appropriate \emph{Polish cochain complex} $C^{\bullet}$. The latter  is simply a cochain complex
\[C^{\bullet}:=( \quad   \quad \quad \quad \quad \quad \quad \quad \quad\cdots\longrightarrow \,C^{n+1}\overset{\delta^{n-1}}{\longrightarrow} C^n\overset{\delta^{n}}{\longrightarrow} C^{n+1} {\longrightarrow} \cdots  \quad \quad  \quad \quad \quad \quad \quad \quad) \quad \quad\]
in which each $\delta^n\colon C^n\to C^{n-1}$ is a \emph{continuous} homomorphism between \emph{Polish} abelian groups; see Section \ref{S:Definable cohomo}. 

The present work takes as its focus the definable enrichment of the \v{C}ech cohomology groups $\Check{\mathrm{H}}^n(X;G)$ of locally compact separable metric spaces $X$ with coefficients in a countable abelian group $G$.  
As abstract groups, these \v{C}ech cohomology groups admit several formulations.
The \emph{combinatorial} (and classical) approach to \v{C}ech cohomology associates to each open cover $\mathcal{U}\in\mathrm{Cov}(X)$ of $X$ the simplicial cohomology group $\mathrm{H}^n(\mathrm{Nv}(\mathcal{U});G)$ of its nerve, then defines $\Check{\mathrm{H}}^n(X;G)$ as the colimit, over the refinement-ordering of $\mathrm{Cov}(X)$, of these groups. In its \emph{homotopical} incarnation, on the other hand, $\Check{\mathrm{H}}^n(X;G)$ is the set $[X,K(G,n)]$ of homotopy classes of maps from $X$ to an Eilenberg-MacLane space $K(G,n)$; since $K(G,n)$ is an abelian $H$-group --- that is, a space equipped with multiplication and inverse operations which satisfy the abelian group axioms up to homotopy --- $[X,K(G,n)]$ has the structure of an abelian group. Mediating between these two approaches, and between homological and homotopical perspectives more generally, is Huber's 1961 theorem \cite{huber_homotopical_1961} which states that the groups $[X,K(G,n)]$ and $\Check{\mathrm{H}}^n(X;G)$ are naturally isomorphic. 

In order to definably enrich these \v{C}ech cohomology groups, we isolate and endow the combinatorially-defined $\Check{\mathrm{H}}^n(X;G)$ and $[X,K(G,n)]$ each with a natural Borel structure. At a first pass, the associated \emph{definable cohomology groups} $\Check{\mathrm{H}}^n_{\mathrm{def}}(X;G)$ and  $[X,K(G,n)]_{\mathrm{def}}$ differ in several interesting and complementary ways.
For example, much like the aforementioned definable invariants introduced in \cite{BLPI}, $\Check{\mathrm{H}}^n_{\mathrm{def}}(X;G)$  admits realization as a group with a Polish cover $\Check{\mathrm{H}}^n_{\mathrm{def}}(X;G):=\mathrm{Z}^n(\boldsymbol{\mathcal{U}};G)/\mathrm{B}^n(\boldsymbol{\mathcal{U}};G)$, where $\mathrm{Z}^n(\boldsymbol{\mathcal{U}};G)$ and $\mathrm{B}^n(\boldsymbol{\mathcal{U}};G)$ are $n$-dimensional cocycle and coboundary groups deriving from a Polish cochain complex associated to $X$ and $G$. A drawback of this realization is its reliance on choices of \emph{covering systems} $\boldsymbol{\mathcal{U}}$ for each $X$, rendering the more global coordination or, more precisely, \emph{functoriality} of the associated assignments $X\mapsto \Check{\mathrm{H}}^n_{\mathrm{def}}(X;G)$ somewhat obscure.   In contrast, the group $[X,K(G,n)]_{\mathrm{def}}$ does not in general manifest as a group with a Polish cover; some care, in fact, is required in handling its Borel structure. Within the category of \emph{definable groups}, however --- a category which, as noted, shares many regularity properties with the category of groups with a Polish cover --- the functoriality of the assignment $X\mapsto [X,K(G,n)]_{\mathrm{def}}$ is clear.

The following \emph{definable version} of Huber's theorem says that, up to a natural definable isomorphism, these two approaches are equivalent. In particular, the assignment $X\mapsto \Check{\mathrm{H}}^n_{\mathrm{def}}(X;G)$ is functorial and, in particular, independent of our choices of covering systems $\boldsymbol{\mathcal{U}}$. Furthermore, the definable group $[X,K(G,n)]_{\mathrm{def}}$ is essentially a group with a Polish cover.
   
\begin{theorem}\label{T1:intro}
The functors determined by the assignments $X\mapsto \Check{\mathrm{H}}^n_{\mathrm{def}}(X;G)$ and $X\mapsto [X,K(G,n)]_{\mathrm{def}}$ are naturally  isomorphic in the category of definable groups.
\end{theorem}   
   
The proof of Theorem \ref{T1:intro}, which occupies much Section \ref{S:Huber}, involves several subsidiary results of independent interest. For example, to show that Huber's abstract isomorphism $[X,K(G,n)]\to \Check{\mathrm{H}}^n(X;G)$ admits a Borel lift, we prove a definable version of the simplicial approximation theorem. The fact that the  inverse map $ \Check{\mathrm{H}}^n(X;G)\to [X,K(G,n)]$ also admits a Borel lift, and hence that it induces an isomorphism between $\Check{\mathrm{H}}^n_{\mathrm{def}}(X;G)$ and $[X,K(G,n)]_{\mathrm{def}}$ in the category of definable groups, follows from the general theory of definable groups that we develop in Section \ref{S:DSets}.

\subsection{Definable groups}
\label{Ss:DefGrp}
At first glance, $[X,K(G,n)]_{\mathrm{def}}$ is just the quotient $\mathrm{Map}(X,K(G,n))/\!\simeq$ of a Polish space by an analytic equivalence relation with the property that the operations of multiplication and inversion on the quotient level lift to Borel maps at the level of $\mathrm{Map}(X,K(G,n))$. While Borel-definable homomorphisms between such ``quotient groups" --- which we term \emph{semidefinable groups} below ---  determine a category which strictly extends $\mathsf{GPC}$, the category of semidefinable groups lacks the regularity properties which make $\mathsf{GPC}$ robust and convenient to work with. For example, a salient feature of the category of groups with a Polish cover, and one instrumental in arguing results like Theorem \ref{T1:intro}, is the fact that an isomorphism on the level of quotient groups admits a Borel lift if and only if its inverse does. Unfortunately, such symmetries do not extend to the generality of the category of semidefinable sets. One may recover them, however, by moving to the intermediate category of \emph{definable groups}, whose objects, loosely speaking, are Borel and idealistic equivalence relations $E$ on Polish spaces $Y$ whose quotients by $E$ carry natural group structures.

\begin{figure}[h!]
\begin{tikzpicture}[x=0.75pt,y=0.75pt,yscale=-1,xscale=1]

\draw   (189.5,109.28) .. controls (189.5,98.24) and (225.32,89.28) .. (269.5,89.28) .. controls (313.68,89.28) and (349.5,98.24) .. (349.5,109.28) .. controls (349.5,120.33) and (313.68,129.28) .. (269.5,129.28) .. controls (225.32,129.28) and (189.5,120.33) .. (189.5,109.28) -- cycle ;
\draw   (146.86,95.14) .. controls (146.86,76.29) and (201.77,61) .. (269.5,61) .. controls (337.23,61) and (392.14,76.29) .. (392.14,95.14) .. controls (392.14,114) and (337.23,129.28) .. (269.5,129.28) .. controls (201.77,129.28) and (146.86,114) .. (146.86,95.14) -- cycle ;
\draw   (90,82.14) .. controls (90,56.11) and (170.36,35) .. (269.5,35) .. controls (368.64,35) and (449,56.11) .. (449,82.14) .. controls (449,108.18) and (368.64,129.28) .. (269.5,129.28) .. controls (170.36,129.28) and (90,108.18) .. (90,82.14) -- cycle ;
\draw[dotted]    (275,110) .. controls (392,90) .. (510,105) ;
\draw[dotted]     (275,80) .. controls (392,60) ..  (510,75) ;
\draw[dotted]     (275,50) .. controls (392,30) ..  (510,45)  ;
\node (p1) at  (510,50)  {\text{semidefinable groups}}; 
\node (p1) at  (510,80)  {\text{definable groups}};
\node (p1) at  (510,110)  {\text{groups with a Polish cover}};
\end{tikzpicture}
    \caption{Definable groups: an extension of $\mathsf{GPC}$ sharing many of its regularity properties.}
    \label{fig1}
\end{figure}
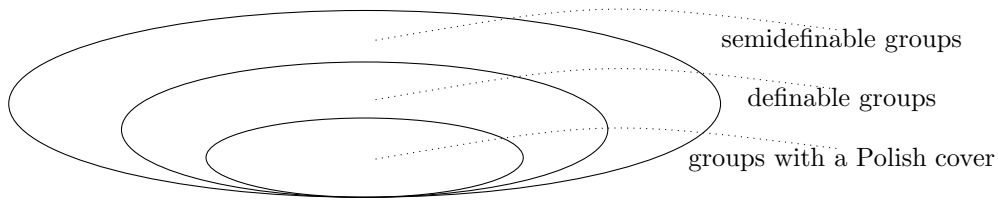

The fact that $[X,K(G,n)]_{\mathrm{def}}$ is a definable group --- a fact which,  among others, plays a role in the proof of Theorem \ref{T1:intro} --- is a consequence of the following more general theorem.


\begin{theorem}\label{T2:intro}
Let $X$ be a locally compact separable metric space and let $P$ be a polyhedral $H$-group. The relation of \emph{being homotopic} defines a Borel and idealistic equivalence relation  on the Polish space $\mathrm{Map}(X,P)$ of all continuous functions from $X$ to $P$.
\end{theorem}

We note that Theorem \ref{T2:intro} is very far from being true if we reverse the roles of $X$ and $P$. Indeed, every analytic equivalence relation is Borel bireducible to the path connectedness relation between points of an appropriately chosen compact metrizable space $X$; see 
\cite[Theorem 4.1]{beckerPath}. Hence, in general, the relation of being homotopic is neither Borel nor idealistic on $\mathrm{Map}(P,X)$, even when $P=\{*\}$ consists of a single point. 

As this observation might suggest, Theorem \ref{T2:intro} is a rather subtle result; the \emph{idealistic} portion of the theorem (holding in fact for any $P$ which is the geometric realization of a locally finite, countable simplicial complex) forms the main labor of Section \ref{S:Hidealistic}. The \emph{Borel} portion of the theorem derives from close analysis of the class of \emph{phantom maps} --- that is, those maps whose restriction to any compact subset of $X$ is nullhomotopic --- from $X$ to $P$.

\subsection{Phantom maps and the descriptive set theory of homotopy relations} The  analysis of the phantom maps $(X,A)\to (P,*)$ from a locally compact Polish pair  to a pointed polyhedron forms the focus of Section \ref{S:HomotopyClassification} and is of some interest in its own right, for the following reasons:
\begin{itemize}
\item We provide, for any locally compact Polish pair $(X,A)$ and polyhedral $H$-group $(P,*)$, a definable short exact sequence decomposition of the definable group $[(X,A),(P,\ast)]$ whose kernel is the group of homotopy classes of phantom maps from $(X,A)$ to $(P,*)$; see Theorem \ref{Theorem:phantom-H-group}. This kernel also takes the form of a definable $\mathrm{lim}^1$ term within a definable Milnor-type short exact sequence (Proposition \ref{Proposition:asymptotic-cohomology}), a connection with the analyses of \cite{BLPI} which we exploit in Section \ref{Section:telescopes}.
\item More generally, we provide a \emph{topological} characterization of the set of phantom maps from $(X,A)$ to a pointed polyhedron $(P,\ast)$: it is the closure in $\mathrm{Map}((X,A),(P,\ast))$ of the set of nullhomotopic maps; see Proposition \ref{prop:topological_char_of_phantoms}, and also Proposition \ref{Proposition:subfunctor}. This is a representative benefit of approaches which, as we put it above, amount essentially to ``a retention of the topologies arising naturally in classical computations.''
\item Most broadly, the analyses of Sections \ref{S:Hidealistic} through \ref{Section:telescopes} should be understood as foundational work in the descriptive set theoretic study of the homotopy relation on maps from a locally compact Polish space $X$ to a polyhedral $P$, a framework encompassing a substantial variety of classification problems in mathematics; in addition to Theorem \ref{T1:intro} and the aforementioned phantom decompositions, Lemma \ref{Lemma:countable-homotopy} (characterizing the homotopy relation when $X$ is compact), Lemma \ref{Lemma:definable-Urysohn} (a definable version of Urysohn's Lemma), Theorem \ref{Theorem:definable-homotopy-extension} (a definable version of Borsuk's homotopy extension theorem), and Theorem \ref{theorem:definable_action_of_pi1} (definably mediating between based and unbased homotopy classes of maps) may all be viewed in these terms as well.
\end{itemize}

This last point brings us to the two main sorts of consequences of our work. The first concerns the study of the Borel complexity of classification problems. 

\subsection{Classification by (co)homological invariants} 
One of the most central programs in descriptive set theory measures the intrinsic complexity of various classification problems in mathematics by locating them within the Borel reduction hierarchy. Two of the most prominent benchmarks within this hierarchy are the \emph{smooth} and the \emph{classifiable by countable structures} classes of classification problems, largely for the reason that each of these benchmarks captures classification schemes which occur frequently in mathematical practice. The \emph{smooth} classification problems, for example, are precisely those which may be ``definably classified" by real number invariants, such as the problem of classifying all Bernoulli shifts of a given amenable group up to isomorphism; by the celebrated results of Ornstein \cite{ornstein_bernoulli_1970} and Ornstein-Weiss \cite{ornstein_entropy_1987}, these dynamical systems are completely classified by their entropy. Similarly, the \emph{classifiable by countable structures} classification problems are precisely those admitting ``definable classification'' by the isomorphism types of countable structures. A well-known example of dynamical systems which are classified up to isomorphism by such invariants is the class of all ergodic measure-preserving transformations of discrete spectrum; by an equally celebrated but much older result of Halmos and von Neumann, these are completely classified by the isomorphism type of a countable structure coding their spectrum. See \cite{NeumannZurOI}.

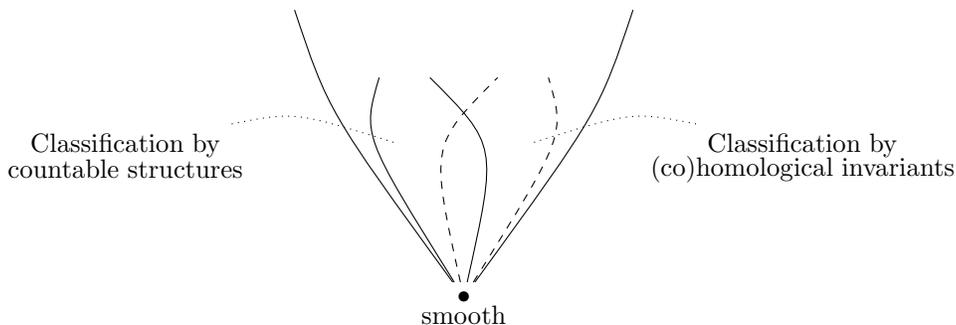
\begin{figure}[!htbp]
    \centering
\begin{tikzpicture}[scale=0.45]

\node (e) at (1.8, 6) {};
\node (e') at (10, 6) {Classification by};
\node (e'') at (10, 5.3) {(co)homological invariants};
\draw[dotted] (e) .. controls (5,7) .. (e');

\node (d) at (-1.8, 6) {};
\node (d') at (-10, 6) {Classification by};
\node (d'') at (-10, 5.3) {countable structures};
\draw[dotted] (d) .. controls (-5,7) .. (d');

\node (b) at (0, 1.5) {$\bullet$};
\node (b') at (0, 1) {smooth};

\draw (b) .. controls (-3,6.3) .. (-2.5, 8);
\draw (b) .. controls (1,6) .. (-1, 8);
\draw[dashed] (b) .. controls (3,6.3) .. (2.5, 8);
\draw[dashed] (b) .. controls (-1,6) .. (1, 8);
\draw (b) .. controls (-4,7) .. (-5,10);
\draw (b) .. controls (4,7) .. (5,10);
\end{tikzpicture}
    \caption{Classification by (co)homological invariants within the Borel reduction hierarchy.}
    \label{fig2}
\end{figure}

An upshot of our analysis here and in \cite{BLPI} is that \emph{classifiability by (co)homological invariants} forms an equally robust complexity class within the Borel reduction hierarchy, one containing a wide array of classification problems from throughout mathematics.  
This class contains all classification problems which can be ``definably classified" using as invariants the elements of a (co)homology group of some Polish (co)chain complex. More formally, let $E$ be an analytic equivalence relation on a Polish space $X$. We say that $(X,E)$ is \emph{classifiable by (co)homological invariants} if there exists a Polish abelian group $G$,  a Polishable subgroup $N$ of $G$, and a Borel map $f\colon X\to G$ so that for all $x,y\in X$ we have  $x\,E\,y$ if and only if $f(x)+N=f(y)+N$, i.e.,  if $E$ is 
Borel reducible to the coset equivalence relation $\mathcal{R}(G/N)$.

It follows from the recognition that many classical invariants from homological algebra and algebraic topology arise as the (co)homology groups of Polish (co)chain complexes that 
the aforementioned complexity class is a rich one. By 
   \citep{BLPI}, for example, it contains the problem of classifying the extensions $0\to F \to E \to B \to 0$ of any fixed pair $(B,F)$ of countable abelian groups. Similarly, it will follow from Section \ref{S:Definable cohomo} below that any classification problem admitting formulation in terms of the \v{C}ech cohomology of a locally compact Polish space falls in this class as well.
In particular, by Theorem \ref{T1:intro}, for any countable abelian group $G$ the problem of classifying maps $X\to Y$ from a locally compact Polish space $X$ to a $K(G,n)$ space $Y$ up to homotopy falls within this class; such $Y$ include, of course, infinite-dimensional real and complex projective space, any knot complement or, more generally, aspherical space, any finite wedge or product or mapping telescope of circles $S^1$, and so on.
These examples may easily be multiplied, within areas of study as varied as that of tiling spaces, group cohomology, or gerbes or fiber bundles, for example \cite{sadun_topology_2008, brown_cohomology_1992, brylinski_loop_2008}; we close this subsection with two representative instances. If $B$ is a locally compact Polish space, a careful inspection of how the assignments $p\mapsto c(p)$ and $\mathcal{A}\mapsto c(\mathcal{A})$ are defined below shows that they are induced by Borel maps at the level of cocycles. As a consequence, the problem of \emph{classifying  Hermitian line bundles over $B$ up to isomorphism} and of \emph{classifying continuous-trace separable $C^{*}$-algebras with spectrum $B$ up to Morita equivalence} are classifiable by (co)homological invariants in the above formal sense.   
   \begin{enumerate}
\item There is an assignment $p\mapsto c(p)$, from Hermitian line bundles  $p\colon E\to B$ over $B$,  to $\Check{\mathrm{H}}^2(B;\mathbb{Z})$, so that $p$ and $p'$ are isomorphic  iff $c(p)=c(p')$; see   \cite[Proposition 4.53]{Morita}.
\item There is an assignment $\mathcal{A}\mapsto c(\mathcal{A})$, 
from continuous-trace separable $C^{*}$-algebras $\mathcal{A}$  with   spectrum $B$,  to  $\Check{\mathrm{H}}^3(B;\mathbb{Z})$, so that $\mathcal{A}$ and $\mathcal{A}'$ are Morita equivalent  iff $c(\mathcal{A})=c(\mathcal{A}')$; see \cite[Proposition 5.24, Theorem 5.56]{Morita} or \cite{Blackadar}, and also \cite{BLPIII}. 
\end{enumerate} 

\subsection{Definable cohomology as a strong invariant}
A second consequence of our work is potentially of even wider significance than the first; this is the existence of definable cohomological functors $\check{\mathrm{H}}^n_{\mathrm{def}}:\mathsf{LC}\to\mathsf{GPC}$ which strictly refine their classical counterparts. In Section \ref{Section:telescopes} we record three sample applications of this technology. We have mentioned the first of these already; it is the following (see Sections \ref{subsection:mappingtelescopes} and \ref{ss:hocolims_of_tori} for precise definitions):
\begin{theorem} The definable \v{C}ech cohomology groups completely classify homotopy colimits, or, equivalently, mapping telescopes, of nontrivial inductive sequences of $d$-spheres up to homotopy equivalence, for all $d\geq 1$.
In contrast, there exist uncountable families of pairwise homotopy inequivalent mapping telescopes of sequences of spheres whose \emph{classical} \v{C}ech cohomology groups, viewed as graded abelian groups, are all isomorphic.
\end{theorem}
We also prove the variant of this theorem resulting from replacing every instance of \emph{spheres} therein with \emph{tori}. Each of these theorems draws on Ulam stability results first established in \cite{BLPI}, and this, indeed, is much of their interest: each applies fundamentally \emph{descriptive set theoretic rigidity results} to algebraic and topological settings. We retain hopes for deeper applications along these lines, a point we return to in our conclusion.

Our second application of the functors $\check{\mathrm{H}}^n_{\mathrm{def}}$ is to a generalization of a problem dating to Borsuk and Eilenberg in 1936 \cite{borsuk_uber_1936}. Writing $\Sigma_p$ for the standard unknotted realization of the $p$-adic solenoid in $S^3$, this is the problem of classifying the maps $S^3\backslash\Sigma_p\to S^2$ up to homotopy. We review the problem's pivotal role in the development of algebraic topology in Section \ref{ss:Borsuk_Eilenberg}, where we also show the following, together with its higher-dimensional generalizations; the theorem's statement is an amalgamation of Theorems \ref{Theorem:BE} and \ref{thm:smooth!}.
\begin{theorem}
\label{intro:T3}
For any prime number $p$, $[S^{3}\backslash\Sigma_{p},S^{2}]$ is a definable set, and there is a basepoint-preserving
definable bijection between it and the definable group $\mathrm{Ext}(\mathbb{Z}[1/p],\mathbb{Z})$.
In consequence, the Borsuk-Eilenberg problem of classifying maps $S^3\backslash\Sigma_p\to S^2$ up to homotopy is essentially hyperfinite but not smooth.
\end{theorem}
The result connects to $\check{\mathrm{H}}^n_{\mathrm{def}}$ via Hopf's Theorem (see Section \ref{ss:Hopf}), and connects to homotopy colimits via the recognition that $S^{3}\backslash\Sigma_{p}$ is homotopy equivalent to a mapping telescope of $1$-spheres. Note that this is a stronger and subtler result than the recognition that $[S^{3}\backslash\Sigma_{p},S^{2}]$ is uncountable, in precisely that sense in which \emph{Borel cardinality} exhibits a richer and subtler degree structure than classical cardinality does; see \cite{hjorth_effective_notes} for a brief introduction to these matters. This brings us to the third sample application of our technology, in which we consider \emph{equivariant} versions of the aforementioned generalized Borsuk-Eilenberg problems. It follows readily from results in both \cite{BLPI} and the present work that these classification problems realize both an infinite antichain and an infinite chain of degrees within the Borel reduction heirarchy; a description of these degrees appears as Corollary \ref{Corollary:complexity}.

\subsection{Concluding preliminaries}
A few final words are in order about the organization of the paper. To begin with, it is, in conception, largely self-contained: although we will repeatedly invoke results from its predecessor \cite{BLPI}, we do not presume any close familiarity with that work. Not unrelatedly, we have drafted it with readers with a wide variety of backgrounds in mind. Its table of contents, together with our preceding remarks and section introductions below, should convey its overall plan.

A standard setting for much of the topological material we'll be considering below is the category \textsf{HCW} of spaces homotopy equivalent to a (countable) CW complex; we should therefore say a bit about our decision to mainly work with two other categories of spaces, that of locally compact separable metric spaces (\textsf{LC}) and that of spaces homotopy equivalent to the geometric realization of a countable, locally finite simplicial complex. As it happens, the latter category is identical to \textsf{HCW} \cite{milnor_spaces_1959}; a polyhedral emphasis merely facilitates several of our arguments. As for the category $\mathsf{LC}$, its virtues for our purposes are multiple:
\begin{itemize}
\item for any two objects $X,Y$ in $\mathsf{LC}$, the compact-open topology renders the set $\mathrm{Map}(X,Y)$ a Polish space;
\item Huber's theorem applies in the generality of $\mathsf{LC}$, and \v{C}ech and sheaf cohomology coincide therein as well;
\item any space in $\mathsf{LC}$ is compactly generated \cite{steenrod_convenient_1967}; relatedly, the fundamental adjunctions of algebraic topology all hold in $\mathsf{LC}$ \cite[chs. 5-6]{bradley_topology_2020};
\item any object of $\mathsf{LC}$ is a countable increasing union of compact subspaces, or, more briefly, is $\sigma$-compact.
\end{itemize}
We have focused exclusively on the group structure of \v{C}ech cohomology simply for reasons of space. We count $0$ among the natural numbers $\mathbb{N}$. We turn, after the following acknowledgement, to the main body of our paper.
\subsubsection*{Acknowledgments}
We would like to thank Alexander Kechris for many valuable discussions, particularly of idealistic equivalence relations, in the course of this paper's composition, as well as the anonymous referee for a persistently alert and valuable reading.
\section{Definable cohomology: the combinatorial approach}\label{S:Definable cohomo}

A \emph{Polish space} $X$ is a second countable topological space whose topology is induced by a complete metric. A subset $Z$ of $X$ is \emph{analytic} if it is the continuous image of a Polish space, and \emph{Borel} if it belongs to the $\sigma$-algebra generated by the open subsets of $X$. A function $f:X\to Y$ between Polish spaces is Borel if its graph is a Borel subset of $X\times Y$ or, equivalently, if $f^{-1}(U)$ is Borel for every open $U\subseteq Y$. A subset $Z$ of $X$ is \emph{meager} if it is a countable union of nowhere dense subsets of $X$, and \emph{comeager} if its complement in $X$ is meager. A \emph{Polish group} is a topological group whose topology is Polish; a \emph{Polishable} topological group is one whose Borel structure is generated by a (possibly finer) Polish group topology. Descriptive set theoretic results as well-known as the fact, for example, that a closed subgroup of a Polish group is Polish will typically be invoked below without citation; readers are referred to \cite{kechris_classical_1995} or \cite{gao_invariant_2009} for complete introductions to this material.

A {\em Polish cochain complex} $C^{\bullet}$ is an $\mathbb{Z}$-indexed sequence of continuous homomorphisms of Polish abelian groups
\[C^{\bullet}:=( \quad   \quad \quad \quad \quad \quad \quad \quad \quad\cdots\longrightarrow \,C^{n-1}\overset{\delta^{n-1}}{\longrightarrow} C^n\overset{\delta^{n}}{\longrightarrow} C^{n+1} {\longrightarrow} \cdots  \quad \quad  \quad \quad \quad \quad \quad \quad) \quad \quad\]
in which $\delta^n\delta^{n-1}=0$ for all $n\in\mathbb{Z}$. Observe that every group $\mathrm{Z}^n:=\mathrm{ker}(\delta^n)$ of {\em $n$-cocycles} arising in $C^{\bullet}$ is a closed and therefore Polish subgroup of $C^n$, and that the {\em $n$-coboundary} groups $\mathrm{B}^n:=\mathrm{im}(\delta^{n-1})$ arising in $C^{\bullet}$, being continuous homomorphic images of Polish groups, are all Polishable. We denote by $\mathrm{H}^n:=\mathrm{Z}^n/\mathrm{B}^n$ the (abstract) cohomology group of degree $n$ associated to $C^{\bullet}$.
The {\em definable cohomology group} $\mathrm{H}^n_{\mathrm{def}}$ of degree $n$ associated to $C^{\bullet}$ is $\mathrm{Z}^n/\mathrm{B}^n$ endowed with the structure of a {\em group with a Polish cover}. As above, recall from \cite{BLPI} that groups with a Polish cover $G/N$ form a category $\mathsf{GPC}$ whose objects are pairs $(G,N)$ in which $N$ is a Polishable normal subgroup of a Polish group $G$, and whose morphisms, also known as  {\em definable homomorphisms}, are those group homomorphisms $f\colon G/N\to G'/N'$ which lift to a Borel function $\hat{f}\colon G\to G'$ satisfying $f(gN)=\hat{f}(g)N'$. Within the definable setting of $\mathsf{GPC}$, the functor $\mathrm{H}^n_{\mathrm{def}}$ conserves, in general, considerably more of the data of $C^{\bullet}$ than its classical counterpart $\mathrm{H}^n$; this renders it a significantly stronger invariant than the latter.

As shown in \cite{BLPI}, several of the most prominent invariants of homological algebra arise as the cohomology groups of cochain complexes carrying natural Polish topologies, including suitable restrictions of the functors $\mathrm{Ext}$ and $\mathrm{lim}^1$. In this section we show that the \v{C}ech cohomology groups $\Check{\mathrm{H}}^n(X;G)$ of a locally compact metrizable space $X$ with coefficients in any countable abelian group $G$ may be similarly construed as groups with a Polish cover, giving rise to the \emph{definable cohomology groups} $\Check{\mathrm{H}}^n_{\mathrm{def}}(X;G)$ of $X$.

This work involves technical challenges which were mercifully absent in the cases of $\mathrm{Ext}$ or $\mathrm{lim}^1$. To better describe them, recall that most standard definitions of \v{C}ech cohomology group $\Check{\mathrm{H}}^n(X;G)$ are some variation on the following:
\begin{equation}
\Check{\mathrm{H}}^n(X;G):= \mathrm{colim}_{\mathcal{U}\in\mathrm{Cov}(\mathcal{U})}\, \mathrm{H}^n(\mathrm{Nv}(\mathcal{U});G),
\end{equation}
where $\mathcal{U}$ ranges over the collection $\mathrm{Cov}(X)$ of all locally finite open covers of $X$, ordered by refinement, and $\mathrm{H}^n(\mathrm{Nv}(\mathcal{U});G)$ denotes the $n^{\mathrm{th}}$ simplicial cohomology group of the \emph{nerve} of $\mathcal{U}$; see Section \ref{SS:DefCohSimCom} below. Implicit in this expression is the fact that the refinement relation  $\mathcal{U}\preceq\mathcal{V}$ (i.e., \emph{$\mathcal{V}$ refines $\mathcal{U}$}) induces a canonical homomorphism $\mathrm{H}^n(\mathrm{Nv}(\mathcal{U});G)\to \mathrm{H}^n(\mathrm{Nv}(\mathcal{V});G)$.
Two main issues complicate the impulse to regard this object as a group with a Polish cover. The first difficulty is that, as defined above, $\Check{\mathrm{H}}^n(X;G)$ is not the cohomology group of any single explicit cochain complex. The second  issue is that $\mathrm{Cov}(X)$ does not in general contain a countable cofinal subset; in consequence, it is less than immediately clear how to extract the information contained in $\Check{\mathrm{H}}^n(X;G)$ from any separable space of data.

We address these difficulties by introducing the notion of a \emph{covering system} $\boldsymbol{\mathcal{U}}$ for $X$; in Section \ref{SS:coveringsystems} we show that every locally compact metrizable space $X$ admits such a system. Covering systems are families of open covers which are cofinal in $\mathrm{Cov}(X)$ and \emph{continuously parametrized by $\mathbb{N}^\mathbb{N}$}.  In Section \ref{SS:DefCoho} we associate to each covering system  $\boldsymbol{\mathcal{U}}$ a Polish cochain complex $C^{\bullet}(\boldsymbol{\mathcal{U}};G)$ and we introduce the \emph{definable cohomology groups} of $X$ as the groups with a Polish cover
\[\Check{\mathrm{H}}^n_{\mathrm{def}}(X;G):= \mathrm{Z}(\boldsymbol{\mathcal{U}};G)/\mathrm{B}(\boldsymbol{\mathcal{U}};G)\]
associated to  $C^{\bullet}(\boldsymbol{\mathcal{U}};G)$. As abstract groups, $\Check{\mathrm{H}}^n_{\mathrm{def}}(X;G)$ coincide with the classical \v{C}ech cohomology groups. As the notation suggests, 
$\Check{\mathrm{H}}^n_{\mathrm{def}}(X;G)$ does not depend on the choice of $\boldsymbol{\mathcal{U}}$ up to definable isomorphism. Moreover, the assignment $X\mapsto \Check{\mathrm{H}}^n_{\mathrm{def}}(X;G)$ is functorial and invariant under homotopy equivalences. However, we defer the proofs of these claims to Section \ref{S:Huber}, since they will apply the 
definable theory of homotopy equivalence which we first develop in Sections \ref{S:DSets} and \ref{S:Hidealistic}. This brings us to a few last points: first, in the following section we recall only what we need of nerves and simplicial complexes for the combinatorial development of definable cohomology; we extend our treatment of these matters in Sections \ref{SS:polyhedra} and \ref{SS:compact pairs}. Second, similarly, for simplicity's sake we undertake this development primarily in the context of single spaces $X$, only recording the modifications needed for the cohomology of \emph{pairs} of spaces $(X,A)$ at its conclusion, in Section \ref{SS:combinatorial_cohomology_of_pairs}; there we also show that definable cohomology satisfies the definable version of the Exactness Axiom. Lastly, readers seeking, as background to this section, a more complete classical combinatorial treatment of \v{C}ech cohomology are referred to \cite[Chapter IX]{eilenberg_foundations_1952}.


\subsection{Definable cohomology for simplicial complexes}\label{SS:DefCohSimCom}

A {\em simplicial complex} $K$ is a family of finite sets that is closed
downwards, i.e., $\sigma \subseteq \tau \in K\implies \sigma \in K$. A \emph{simplex} or {\em 
face of $K$} is any element $\sigma \in K$. A {\em vertex of $K$} is any
element $v$ of $\mathrm{dom}(K):=\bigcup K$.  Let $K,L$ be two simplicial
complexes. A {\em simplicial map} $f\colon K\rightarrow L$ is then any function $
f\colon \mathrm{dom}(K)\rightarrow \mathrm{dom}(L)$ such that $
\{f(v_{0}),\ldots ,f(v_{n})\}\in L$ for all $\{v_{0},\ldots ,v_{n}\}\in K$.
The {\em dimension} $\mathrm{dim}(\sigma )$ of a face $\sigma $ of $K$ is simply the number $|\sigma|-1$. For example, $\mathrm{dim}(\emptyset )=(-1)$ and $\mathrm{dim}(\{v\})=0$
for every $v\in \mathrm{dom}(K)$. The {\em dimension} $\mathrm{dim}(K)$ of $
K$ is the supremum over $\{\mathrm{dim}(\sigma )\mid \sigma \in K\}$. A
simplicial complex is {\em finite} if it has finitely many vertices, and 
{\em countable} if it has countably many vertices. It is {\em locally
finite} if each vertex belongs to finitely many faces.
For each $n\in \mathbb{N}$, the {\em singular $n$-faces}\footnote{See \cite{camarena_turning_notes} for some justification of this terminology.} 
of $K$ comprise the set
\[K^{(n)}:=\{(v_0,\ldots,v_{n})\in \mathrm{dom}(K)^{n+1}\colon \{v_0,\ldots,v_{n}\}\in K \}.\]
\noindent We fix  an abelian Polish group $G$ and consider, for every $n\in\mathbb{N}$, the collection
\[C^{n}(K;G):= C(K^{(n)},G)\]
of all maps from the countable set $K^{(n)}$ to $G$. Endowed with  the group operation $(\zeta,\eta)\mapsto(\zeta+\eta)$ of pointwise addition $(\zeta+\eta)(\bar{v})=\zeta(\bar{v})+\eta(\bar{v})$, and the product topology of countably many copies of $G$, the collection $C^{n}(K;G)$ forms the abelian Polish group of all {\em  $G$-valued cochains of $K$}. For every $n\geq 0$, we have the
 {\em coboundary map}:
\begin{align}\label{Def:SimplicialCoboundary}
\begin{split}
\delta^{n}:C^{n}(K;G) &\rightarrow C^{n+1}(K;G), \text{ with}\\
\left( \delta^{n}( \zeta) \right) \big((v_{0},\ldots ,v_{n+1})\big)
&=\sum_{i=0}^{n}\left( -1\right) ^{i}\zeta (v_{0},\ldots ,\hat{v}_{i},\ldots
,v_{n+1}),
\end{split}
\end{align}
where $\hat{v}_{i}$ denotes the omission of $v_i$ from $\bar{v}$. Notice that $\delta^{n}$ is a continuous homomorphism. This gives rise to the {\em Polish cochain complex $C^{\bullet}(K;G)$ of  $G$-valued cochains of $K$}:
\[0\longrightarrow C^{0}(K;G)\overset{\delta^0} \longrightarrow C^{1}(K;G)\longrightarrow\cdots\longrightarrow C^{n}(K;G) \overset{\delta^{n+1}}\longrightarrow C^{n}(K;G) \longrightarrow \cdots\]

\begin{definition}\label{Def:cohomoForSimplexes}
Let $K$ be a countable simplicial complex and let $G$ be an abelian Polish group. For every $n\in\mathbb{N}$, the {\em $n$-dimensional definable cohomology group $\mathrm{H}^{n}_{\mathrm{def}}(K;G)$ of $K$ with coefficients in $G$} is  the $n$-dimensional cohomology group $\mathrm{H}^{n}(K;G)$ of the Polish cochain complex $C^{\bullet}(K;G)$, viewed as the group with a Polish cover:
\[
 0\longrightarrow \mathrm{B}^{n}(K;G)\longrightarrow \mathrm{Z}^{n}(K;G)  \longrightarrow \mathrm{Z}^{n}(K;G)/
\mathrm{B}^{n}(K;G)\longrightarrow 0\]
 where  $\mathrm{Z}^{n}(K;G)= \mathrm{ker}(\delta^{n})$ is the Polish group of  {\em $n$-dimensional  $G$-valued cocycles of $K$}  and   $\mathrm{B}^{n}(K;G)=\mathrm{im}(\delta^{n-1})$ is the Polishable group of   {\em $n$-dimensional  $G$-valued coboundaries of $K$}.
 \end{definition}
Of course, if we forget about the quotient Borel structure on   $\mathrm{H}^{n}_{\mathrm{def}}(K;G)$ and instead treat it as an abstract group then we recover the classical $n$-dimensional simplicial cohomology group of $K$ with coefficients in $G$.

\begin{remark}
Several alternative definitions of the  simplicial cohomology groups $\mathrm{H}^{n}(K;G)$ may be found in the literature. Each involves variants of the cochain complex $C^{\bullet}(K;G)$ described above. For example, one may work with the chain complex $C^{\bullet}_{\mathrm{alt}}(K;G)$ of \emph{alternating cochains} instead. An {\em alternating $n$-cochain} is any $n$-cochain $\zeta\in C^{\bullet}(K;G)$ with the property that 
\[\zeta(v_0,\ldots,v_n)=\mathrm{sgn}(\pi)\zeta(v_{\pi(0)},\ldots,v_{\pi(n)}),\]
for any permutation $\pi$ of the set $\{0,\ldots,n\}$
. The inclusion $C^{\bullet}_{\mathrm{alt}}(K;G)\hookrightarrow  C^{\bullet}(K;G)$ yields an isomorphism between $\mathrm{H}^{n}(K;G)$  and  $\mathrm{H}^{n}_{\mathrm{alt}}(K;G)$, as is well-known (see, e.g., \cite{serre_faisceaux_1955}). It is clear that  $C^{\bullet}_{\mathrm{alt}}(K;G)$ inherits a Polish structure from $C^{\bullet}(K;G)$ and that 
$C^{\bullet}_{\mathrm{alt}}(K;G)\hookrightarrow  C^{\bullet}(K;G)$ is continuous. Hence   $\mathrm{H}^{n}(K;G)$  and  $\mathrm{H}^{n}_{\mathrm{alt}}(K;G)$ are definably isomorphic as groups with a Polish cover. As this isomorphism commutes with direct limits, similar remarks will apply to the \v{C}ech cohomology groups of locally compact Polish spaces which we define below.
\end{remark}

\subsection{An indexing poset} Below we identify $n$ with the set $\{0,\ldots,n-1\}$. Let $\mathbb{N}^n$ be the space of all functions from $n$ to $\mathbb{N}$ and set $\mathbb{N}^{<\mathbb{N}}=\bigcup_{n\in\mathbb{N}}\mathbb{N}^n$. The {\em Baire space} $\mathcal{N}:=\mathbb{N}^{\mathbb{N}}$ is the space of all functions from $\mathbb{N}$ to $\mathbb{N}$, which we view as a Polish space  equipped with the product topology on discrete copies of $\mathbb{N}$. For every $\alpha\in\mathcal{N}$ and all $n\in\mathbb{N}$ we let $\alpha|n\in\mathbb{N}^n$ be the finite sequence $(\alpha(0),\ldots,\alpha(n-1))$. For every  $s=(s(0),\ldots,s(n-1))\in\mathbb{N}^{<\mathbb{N}}$, we denote by $\mathcal{N}_s$ the clopen subset $\{\alpha\in \mathcal{N}\colon \alpha|n=s \}$ of $\mathcal{N}$. The collection  $\{\mathcal{N}_s\colon s\in \mathbb{N}^{<\mathbb{N}}\}$ of all such sets forms a basis  for the topology on $\mathcal{N}$.
Let $(X_s\colon s\in \mathbb{N}^{<\mathbb{N}})$ be a family of sets parametrized by $\mathbb{N}^{<\mathbb{N}}$. To each such family we may apply {\em Suslin's $\mathcal{A}$-operation}, producing the set
\[\mathcal{A}(X_{s}):=\bigcup_{\alpha\in \mathcal{N}}\bigcap_{n\in \mathbb{N}}X_{\alpha|n}.\]
Notice that while the family $(X_s)$ is countable, the operation $\mathcal{A}$ involves an uncountable union. Historically, this operation was used to explicitly define subsets of the real line which are analytic but not Borel. Indeed, every Borel subset of a Polish space $X$ is the result of the $\mathcal{A}$-operation applied to some system $(X_s)$ of closed  subsets of $X$ --- but not all sets   derived from the $\mathcal{A}$-operation applied to a system of closed  subsets is Borel; see \cite[Theorem 25.7]{kechris_classical_1995}.
 
Here we will be interested in the closed subset $\mathcal{N}^*$ of $\mathcal{N}$ consisting of all {\em non-decreasing}  such functions, i.e., all $\alpha=(\alpha(0),\alpha(1),\ldots)\in \mathcal{N}$ such  that  $\alpha(n)\leq \alpha(n+1)$ for all $n\in\mathbb{N}$. It is easy to see that  $\mathcal{N}^*$ and $\mathcal{N}$ are homeomorphic. Similarly, we denote by $(\mathbb{N}^n)^*$ and $(\mathbb{N}^{<\mathbb{N}})^*$ the sets of all {\em non-decreasing functions} in $\mathbb{N}^n$ and  $\mathbb{N}^{<\mathbb{N}}$, respectively. We set $\mathcal{N}^*_s:=\mathcal{N}_s\cap\mathcal{N}^*$ for all $s\in(\mathbb{N}^{<\mathbb{N}})^*$.
We endow $\mathcal{N}^*$ with the {\em pointwise  partial ordering} $\leq$    and the associated {\em meet} operation $(\alpha,\beta)\mapsto\alpha\wedge \beta$  and  {\em join}  operation $(\alpha,\beta)\mapsto\alpha\vee \beta$, where
\[\alpha\leq \beta \iff  \forall n\in \mathbb{N} \; (\alpha(n)\leq \beta(n)),\]
\[(\alpha\wedge \beta)(n):= \min\{\alpha(n),\beta(n)\} \text{ and  }  (\alpha\vee \beta)(n):= \max\{\alpha(n),\beta(n)\}.\]
Occasionally we will also employ the {\em lexicographic linear ordering} $\leq_{\mathrm{lex}}$ of $\mathcal{N}^*$ with
\[\alpha\leq_{\mathrm{lex}}\beta \iff  \big( \alpha=\beta \text{ or } \exists n\in \mathbb{N}  \; (\alpha|n=\beta|n \text{ and }  \alpha(n)<\beta(n)\big). \]
Lastly, observe that to every family of sets $\big{(}X_s\colon s\in (\mathbb{N}^{<\mathbb{N}})^*\big{)}$, parametrized by elements of  $(\mathbb{N}^{<\mathbb{N}})^*$, we may still apply Suslin's $\mathcal{A}$-operation by first setting $X_s:=\emptyset$ for all $s\in\mathbb{N}^{<\mathbb{N}}\setminus(\mathbb{N}^{<\mathbb{N}})^*$.
 
\subsection{Covering systems}\label{SS:coveringsystems}
Let $X$ be a locally compact Polish space. A \emph{cofiltration} or {\em exhaustion of $X$} is an increasing  sequence  $X_0\subseteq X_1\subseteq \cdots$ of compact subsets of $X$, with $X_0=\emptyset$ and $\bigcup_{n\in\mathbb{N}}X_n=X$.
An {\em open cover} of $X$ is any family $\mathcal{U}$ of open sets so that $\bigcup\mathcal{U}=X$. An open cover $\mathcal{U}$ is {\em locally finite} if for every $x\in X$ there is an open neighborhood $O\subseteq X$ of $x$ which intersects only finitely many elements of $\mathcal{U}$. If $K\subseteq X$ then we let $\mathcal{U}\upharpoonright K:=\{U\in \mathcal{U}\colon U\cap K \neq \emptyset\}$. Notice that if $\mathcal{U}$ is locally finite and $K$ is compact, then $\mathcal{U}\upharpoonright K$ is finite. Let $\mathcal{U},\mathcal{V}$ be open covers of $X$. We write $\mathcal{U}\preceq\mathcal{V}$ if $\mathcal{V}$ {\em refines} $\mathcal{U}$, i.e., if there exists a function $r\colon \mathcal{V}\to \mathcal{U}$ so that $V\subseteq  r(V)$ for all $V\in\mathcal{V}$. We term such an $r$ a {\em refinement map}.

\begin{definition}\label{DefCovers2}
A {\em covering system for $X$} is a triple $\boldsymbol{\mathcal{U}}=\big{(}(X_n),(\mathcal{U}_{\alpha}),(r^{\beta}_{\alpha})\big{)}$ such that
\begin{itemize}
\item $(X_n\colon n\in\mathbb{N})$ is an exhaustion of $X$,
\item $(\mathcal{U}_{\alpha}\colon \alpha\in\mathcal{N}^*)$ is a family of locally finite open covers of $X$, and  
\item $r^{\beta}_{\alpha}\colon \mathcal{U}_{\beta}\to\mathcal{U}_{\alpha}$ are refinement maps indexed by the pairs $\alpha\leq \beta$ in $\mathcal{N}^*$ and satisfying $r^{\alpha}_{\alpha}=\mathrm{id}$
  and $r^{\gamma}_{\alpha}= r^{\beta}_{\alpha} \circ r^{\gamma}_{\beta}$  for all $\alpha\leq \beta\leq \gamma$,
\end{itemize}  
   which moreover satisfy the following  {\em locality} and {\em  extensionality} axioms for all $n\in\mathbb{N}$:
\begin{enumerate}
 \item[(L1)] if  $\alpha|n=\beta|n$, then   $\mathcal{U}_{\alpha}\upharpoonright X_n=\mathcal{U}_{\beta}\upharpoonright X_n$;
 \item[(L2)] if $\alpha\leq \beta$ and $\gamma\leq \delta$,  with  $\alpha|n=\gamma|n$ and $\beta|n=\delta|n$,  then   $r^{\beta}_{\alpha}\upharpoonright (\mathcal{U}_{\beta}\upharpoonright X_n)=r^{\delta}_{\gamma}\upharpoonright (\mathcal{U}_{\delta}\upharpoonright X_n)$;
 \item[(L3)] if $\alpha\leq \beta$ then $r^{\beta}_{\alpha}\upharpoonright (\mathcal{U}_{\beta}\upharpoonright X_n)$ is surjective on $\mathcal{U}_{\alpha}\upharpoonright X_n$; 
 \item[(E1)]  for every open cover $\mathcal{U}$ of $X$ and every $m\in\mathbb{N}$ with  $n<m$, if $\mathcal{U}\upharpoonright X_n\preceq \mathcal{U}_{\alpha}\upharpoonright X_n$, then there exists $\beta \in \mathcal{N}^*_{\alpha|n}$ so that  $\mathcal{U}\upharpoonright X_m\preceq \mathcal{U}_{\beta}\upharpoonright X_m$. 
\end{enumerate}
\end{definition}

Notice that if $\boldsymbol{\mathcal{U}}$ is a covering system for $X$, then for every open cover $\mathcal{U}$ of $X$ and every $m\in\mathbb{N}$ there exists $\beta\in\mathcal{N}^{*}$  so that
\[\mathcal{U}\upharpoonright X_m\preceq \mathcal{U}_{\beta}\upharpoonright X_m.\] This follows simply from (E1), since $\{X\}=\mathcal{U}\upharpoonright X_0\preceq \mathcal{U}_{\alpha}\upharpoonright X_0$ for all $\alpha\in\mathcal{N}^{*}$.

While the definition of a covering system  $\boldsymbol{\mathcal{U}}=\big{(}(X_n),(\mathcal{U}_{\alpha}),(r^{\beta}_{\alpha})\big{)}$ for $X$ involves an  uncountable family of open covers and  an uncountable family of  refinement maps, $\boldsymbol{\mathcal{U}}$ can still be fully recovered from a certain countable family of ``finitary approximations," via a procedure which resembles Suslin's operation $\mathcal{A}$. The next definition and the remark  following it make this  precise. Notice that by (L1) and (L2), the  definition of   $\mathcal{U}_s$ and $r^{t}_{s}$ below does not depend on the choice of  $\alpha$ and $\beta$.

\begin{definition}\label{Def:suslin} Let $\boldsymbol{\mathcal{U}}=\big{(}(X_n),(\mathcal{U}_{\alpha}),(r^{\beta}_{\alpha})\big{)}$  be a covering system for $X$. For every $n\in\mathbb{N}$ and every $s,t\in (\mathcal{N}^n)^*$, with $s\leq t$, we choose some $\alpha\in\mathcal{N}^*_s$, $\beta\in\mathcal{N}^*_t$  and we define:
\begin{itemize}
\item  $\mathcal{U}_s:= \mathcal{U}_{\alpha}\upharpoonright X_n$;
\item  $r^{t}_{s}:= r^{\beta}_{\alpha} \upharpoonright (\mathcal{U}_{\beta}\upharpoonright X_n)$.
\end{itemize}
We collect this data into a pair  $\boldsymbol{\mathcal{U}}_{\mathrm{fin}}=\big{(} (\mathcal{U}_{s}),(r^{t}_{s})\big{)}$,  where $s,t$ range over $(\mathbb{N}^{<\mathbb{N}})^*$ with $|s|=|t|$ and $s\leq t$. We say that 
$\boldsymbol{\mathcal{U}}_{\mathrm{fin}}$ is  an {\em approximation of $\boldsymbol{\mathcal{U}}$} and we write   $\boldsymbol{\mathcal{U}}=\mathcal{A}(\boldsymbol{\mathcal{U}}_{\mathrm{fin}})$.
\end{definition}

\begin{remark}\label{Rema1}
Notice that we can fully recover the covering system $\boldsymbol{\mathcal{U}}$ from its approximation $\boldsymbol{\mathcal{U}}_{\mathrm{fin}}$. Indeed, for each $n\in\mathbb{N}$  and  any  sequence $(s_k)$ in  $(\mathbb{N}^n)^*$ with $s_0(i)<s_1(i)<\cdots$ for all $i<n$,  we have that:
 \[X_n=\bigcap_{k\in\mathbb{N}} \bigcup_{U\in \mathcal{U}_{s_k}}\mathrm{cl}(U).\]
The above equality is easily established using (1) and (3) of Lemma \ref{L:combinatorialapproximations} below.
If for $s\in  (\mathbb{N}^n)^*$ we momentarily denote by $\widehat{\mathcal{U}}_{s}$ the set of all open covers $\mathcal{U}$ of $X$ with $\mathcal{U}\upharpoonright X_{|s|}= \mathcal{U}_s$ then 
\[\mathcal{A}\big(\widehat{\mathcal{U}}_{s}\big)=\bigcup_{\alpha\in\mathcal{N}^*}\bigcap_{n\in\mathbb{N}}\widehat{\mathcal{U}}_{\alpha|n} =\{\mathcal{U}_{\alpha}\colon \alpha\in\mathcal{N}^*\},\]
and similarly for the collection $(r^{\beta}_\alpha)_{\alpha\leq \beta}$. This motivates the notation $\boldsymbol{\mathcal{U}}=\mathcal{A}(\boldsymbol{\mathcal{U}}_{\mathrm{fin}})$ in Definition \ref{Def:suslin} for indicating the relationship between a covering system $\boldsymbol{\mathcal{U}}$  and its approximation  $\boldsymbol{\mathcal{U}}_{\mathrm{fin}}$.
\end{remark}

\begin{lemma}\label{L:combinatorialapproximations}
Let $\boldsymbol{\mathcal{U}}=\big{(}(X_n),(\mathcal{U}_{\alpha}),(r^{\beta}_{\alpha})\big{)}$ be a covering system for $X$. Then: 
\begin{enumerate}
\item
for every compact $K\subseteq X$ and every $\alpha\in \mathcal{N}$,  $\mathcal{U}_{\alpha}\upharpoonright K$ is a finite cover of $K$;
\item $\bigcup_{\alpha\in\mathcal{N}^*} \mathcal{U}_{\alpha}$ is countable;
\item for every open cover $\mathcal{U}$ of $X$, there exists an $\alpha\in\mathcal{N}$ so that $\mathcal{U}\preceq\mathcal{U}_{\alpha}$.
\item there exists an $\alpha\in\mathcal{N}$ so that for every $\beta\geq \alpha$, if $U\in \mathcal{U}_{\beta}$, then $\mathrm{cl}(U)$ is compact. 
\end{enumerate}
\end{lemma}
\begin{proof}
(1) follows from the local finiteness of  $\mathcal{U}_{\alpha}$ and (4) is a direct consequence of (3), since $X$ is locally compact.   

For (2), we may assume that $\emptyset\not\in\bigcup_{\alpha\in\mathcal{N}} \mathcal{U}_{\alpha}$. The rest follows from (1), since  by (L1) we have
\[\bigcup_{\alpha\in\mathcal{N}^*} \mathcal{U}_{\alpha}=\bigcup_{\alpha\in\mathcal{N}^*}\bigcup_{n\in \mathbb{N}} (\mathcal{U}_{\alpha}\upharpoonright X_n)=\bigcup_{n\in \mathbb{N}} \bigcup_{s\in\mathbb{N}^{n}} \mathcal{U}_s,\]

For (3), we will find $\alpha$ as the limit of a sequence $(\alpha_n)$ in $\mathcal{N}$. We inductively define $\alpha_n$ as follows: let $\alpha_0$ be any element of $\mathcal{N}$; assuming that $\alpha_{n-1}$ has 
been defined, use  (E1) to get  $\alpha_{n}\in\mathcal{N}$ with $\alpha_{n}|(n-1)=\alpha_{n-1}|(n-1)$, 
so that $\mathcal{U}\upharpoonright X_{n}\preceq \mathcal{U}_{\alpha|n}\upharpoonright X_n$. Notice that $(\alpha_n)$ 
converges to the unique $\alpha\in\mathcal{N}$ with $ \alpha|n=\alpha_n|n$. By (L1), we have that  $\mathcal{U}\upharpoonright X_n\preceq \mathcal{U}_{\alpha}\upharpoonright X_n$ for all $n\in\mathbb{N}$. Since $\bigcup_n X_n=X$, it follows that $\mathcal{U}\preceq \mathcal{U}_{\alpha}$.
\end{proof}

\begin{proposition} Every locally compact Polish space $X$ admits a covering system  $\boldsymbol{\mathcal{U}}$.
\end{proposition}
\begin{proof}
Let $d$ be a metric on $X$ that is compatible with the topology and let $(X_n)$ be any exhaustion of $X$ with $X_{n}\subseteq \mathrm{int}(X_{n+1})$. We will attain the  desired covering system $\boldsymbol{\mathcal{U}}=\big{(}(X_n),(\mathcal{U}_{\alpha}),(r^{\beta}_{\alpha})\big{)}$ in the form of $\mathcal{A}(\boldsymbol{\mathcal{U}}_{\mathrm{fin}})$ where $\boldsymbol{\mathcal{U}}_{\mathrm{fin}}=((\mathcal{U}_s),(r^{t}_s))$ will be defined by a double induction.
\medskip{}

\noindent\underline{Induction on the length of $s,t$}. 

Set $U_{\emptyset}=\emptyset$ and assume that for some $n\in\mathbb{N}$ we have defined for all $s,t,q\in(\mathbb{N}^n)^*$ with $s\leq t\leq q$:
\begin{enumerate}
\item   a finite family $\mathcal{U}_s$ of open subsets of $X$  covering $X_n$, so that $\mathcal{U}_s\upharpoonright X_m= \mathcal{U}_{s|m}$ if $m<n$, and

\[\mathrm{diam}(U)<\frac{1}{\min\{s(i)\colon i<n\}+1}, \text{ for all } U\in\mathcal{U}_s.\] 
\item a surjective refinement map $r^t_s \colon \mathcal{U}_t\to \mathcal{U}_s$ with $r^{t}_t=\mathrm{id}$ and $r^t_s\circ r^q_t=r^q_s$.
\end{enumerate}
We will extend this system to a system indexed by $(\mathbb{N}^{n+1})^*$, which satisfies  properties (1), (2) above with $n+1$ in place of $n$. For that, set  $\boldsymbol{\bar{\ell}}:=(\ell,\ldots,\ell)\in \mathbb{N}^n$ and  $O_{\ell}:=\bigcup\mathcal{U}_{\boldsymbol{\bar{\ell}}}$ for every $\ell\in\mathbb{N}$. Notice that by our inductive assumption, we have a decreasing sequence of open sets 
\[O_{0}\supseteq O_{1}\supseteq\cdots\supseteq O_{\ell}\supseteq \cdots\supseteq X_n.\]

\noindent\underline{Induction on $\ell \in\mathbb{N}$}.
 
By induction on $\ell\in\mathbb{N}$ we choose a sequence $\mathcal{V}_{0}, \mathcal{V}_{1},\ldots, \mathcal{V}_{\ell},\ldots$ of finite families of open subsets of $X$ so that  $\mathcal{V}_{\ell}$ covers $X_{n+1}\setminus O_{\ell}$ with sets of diameter less than $1/(\ell+1)$, which do not intersect $X_n$. We can also make sure that  $ \mathcal{V}_{\ell}\cup \mathcal{U_{\boldsymbol{\bar{\ell}}}}\preceq \mathcal{V}_{\ell+1}$. We may now set
\[\mathcal{U}_{s^{\frown}\ell}:=\mathcal{U}_s\cup \mathcal{V}_{\ell},\]
for every   $s^{\frown}\ell\in (\mathbb{N}^{n+1})^*$ and observe that the analogue of point (1) above is satisfied by the new system. For example, notice that if  $s^{\frown}\ell\in (\mathbb{N}^{n+1})^*$, then $s\leq \boldsymbol{\bar{\ell}}$, and hence $\mathcal{U}_s\preceq \mathcal{U}_{\boldsymbol{\bar{\ell}}}$. As a consequence, $O_{\ell}\subseteq\bigcup\mathcal{U}_s$ and therefore,    $\mathcal{U}_{s^{\frown}\ell}$ covers $X_{n+1}$.

We now turn to the definition of the refining maps. Fix $\ell\in\mathbb{N}$ and assume inductively that for every pair  $s^{\frown}k, t^{\frown} m\in(\mathbb{N}^{n+1})^*$ with $s^{\frown}k\leq  t^{\frown} m$ and $m\leq \ell$ we have defined a surjective refinement  map
\[r^{ t^{\frown} m}_{s^{\frown}k}\colon \mathcal{U}_{t^{\frown} m} \to \mathcal{U}_{s^{\frown}k}, \]
which extends  $r^{ t}_{s}$, and assume that these maps altogether cohere with respect to composition as in point (2) above. 
Fix now any refinement map: 
 \[p^{\ell+1}_{\boldsymbol{\bar{\ell}}^{\frown}\ell}: \mathcal{V}_{\ell+1}\to \mathcal{U}_{\boldsymbol{\bar{\ell}}^{\frown}\ell} \]
 Notice that in the definition of $\mathcal{V}_{\ell}$ we could have arranged, by removing superfluous elements, that $\mathcal{V}_{\ell}$  is ``minimal," i.e., if we remove any $V$ from $\mathcal{V}_{\ell}$, then  the resulting family is not going to be a cover of $X_{n+1}\setminus O_{\ell}$.  By imposing this  minimality assumption on   $\mathcal{V}_{\ell}$, we have that $p^{\ell+1}_{\boldsymbol{\bar{\ell}}^{\frown}\ell}$ is surjective on $\mathcal{V}_{\ell}$.  For every $s^\frown k \in (\mathbb{N}^{n+1})^*$ with $k\leq \ell+1$ we define a refinement map
 \[p^{\ell+1}_{s^{\frown} k }: \mathcal{V}_{\ell+1}\to \mathcal{U}_{s^{\frown} k } \]
 by setting $p^{\ell+1}_{s^{\frown} k }=\mathrm{id}$, if $k=\ell+1$; if  $k<\ell+1$, then notice that  $s^{\frown} k\leq \boldsymbol{\bar{\ell}}^{\frown}\ell$, and we may define: 
  \[p^{\ell+1}_{s^{\frown} k}:= r^{\boldsymbol{\bar{\ell}}^{\frown}\ell}_{s^{\frown} k}\circ p^{\ell+1}_{\boldsymbol{\bar{\ell}}^{\frown}\ell},\]
 where the map $r^{\boldsymbol{\bar{\ell}}^{\frown}\ell}_{s^{\frown} k}$ is given by inductive assumption. Finally, for all $s^{\frown} k \leq t^{\frown} (\ell+1)$ we set
 \[r^{ t^{\frown} (\ell+1)}_{s^{\frown} k}:=\big(r^t_s\cup p^{\ell+1}_{s^{\frown} k  }\big).\]
The fact that our new system of maps which  is indexed by  pairs  $s^{\frown}k, t^{\frown} m\in(\mathbb{N}^{n+1})^*$ with $s^{\frown}k\leq  t^{\frown} m$ and $m\leq \ell+1$ coheres in the sense of point (2) above, follows from our inductive assumptions and the fact that all new maps which are not equal to $\mathrm{id}$ on $\mathcal{V}_{\ell+1}$ factor through   $\mathcal{U}_{\boldsymbol{\bar{\ell}}^{\frown}\ell}$.

This ends the induction on $\ell$ and consequently the induction on the length of $s,t$. Let now $$\boldsymbol{\mathcal{U}}=\big{(}(X_n),(\mathcal{U}_{\alpha}),(r^{\beta}_{\alpha})\big{)}:=\mathcal{A}(\boldsymbol{\mathcal{U}}_{\mathrm{fin}})$$ where $\boldsymbol{\mathcal{U}}_{\mathrm{fin}}=((\mathcal{U}_s),(r^{t}_s))$. Using points (1) and (2) above it is easy to see that $\boldsymbol{\mathcal{U}}$ is a covering system for $X$. For example,  (E1) follows from the shrinking diameters in point (1) above and Lebesgue's covering lemma; the fact that each $\mathcal{U}_{\alpha}$ is locally finite follows from the assumption $X_n\subseteq \mathrm{int}(X_{n+1})$ and the fact that each $\mathcal{U}_s$ is finite.
\end{proof}

\subsection{Definable cohomology with discrete coefficients}\label{SS:DefCoho}
Let $G$ be a countable discrete abelian group and let $X$ be a locally compact Polish space. We will fix a covering system $\boldsymbol{\mathcal{U}}=\big{(}(X_n),(\mathcal{U}_{\alpha}),(r^{\beta}_{\alpha})\big{)}$ for $X$ and use it to define a Polish cochain complex $C^{\bullet}(\boldsymbol{\mathcal{U}};G)$. 
The associated cohomology groups  $\Check{\mathrm{H}}_{\mathrm{def}}^{n}(X;G)$, viewed as groups with Polish cover, form the {\em definable cohomology of $X$ with coefficients in $G$}. If we forget the definable content from $\Check{\mathrm{H}}_{\mathrm{def}}^{n}(X;G)$ we recover the classical \v{C}ech cohomology groups for $X$.
In Section \ref{S:Huber} we will show that, up to definable isomorphism,  $\Check{\mathrm{H}}_{\mathrm{def}}^{n}(X;G)$ does not depend on the choice of $\boldsymbol{\mathcal{U}}$. One may extend these ideas to develop definable cohomology  groups with coefficients in an arbitrary Polish abelian group $G$. That project, however, falls beyond the scope of this paper; see Remark \ref{RemGeneralPolish}.

For every open cover $\mathcal{U}$ of $X$, the {\em nerve} of $\mathcal{U}$ is the simplicial complex  $\mathrm{Nv}(\mathcal{U})$ with $\sigma\in \mathrm{Nv}(\mathcal{U})$ if and only if $\sigma$ is a finite subset of $\mathcal{U}$ with $\cap_{U\in \sigma}U\neq \emptyset.$ Notice that $\mathrm{dom}(\mathrm{Nv}(\mathcal{U}))=\mathcal{U}\setminus \{\emptyset\}$. Set
\[N_{\alpha}:= \mathrm{Nv}(\mathcal{U_{\alpha}})\]
for every $\alpha\in \mathcal{N}^*$ and notice that every refinement map $r^{\beta}_{\alpha}$ induces a simplicial map
\[r^{\beta}_{\alpha}\colon N_{\beta}\to N_{\alpha}.\]
In this way $\boldsymbol{\mathcal{U}}$ determines an inverse system $((N_{\alpha}),(r^{\beta}_{\alpha}))$ of simplicial maps. We similarly define the  finite complexes $N_s:= \mathrm{Nv}(\mathcal{U}_{s})$ and consider the simplicial maps $r^{t}_s\colon N_{t}\to N_s$, for every $s\leq t$ in $(\mathbb{N}^{<\mathbb{N}})^*$.

For every $n\in\mathbb{N}$ we will define the Polish group $C^{n}(\boldsymbol{\mathcal{U}};G)$ of $n$-dimensional cocycles of $\boldsymbol{\mathcal{U}}$ as the quotient of a Polish semigroup $C_{\mathrm{sem}}^{n}(\boldsymbol{\mathcal{U}};G)$ by a certain closed semigroup congruence. Let 
\[C_{\mathrm{sem}}^{n}(\boldsymbol{\mathcal{U}};G):=\bigcup_{\alpha\in\mathcal{N}^*}C^{n}(N_{\alpha};G)=\bigcup_{\alpha\in\mathcal{N}^*}C(N^{(n)}_{\alpha};G)\]
be the collection of all functions from the set of singular $n$-faces of some $N_{\alpha}$ to $G$. 
 We endow  $C_{\mathrm{sem}}^{n}(\boldsymbol{\mathcal{U}};G)$ with a semigroup structure by setting for all  $\zeta\colon N^{(n)}_{\alpha}\to G$ and $\eta\colon N^{(n)}_{\beta}\to G$
\begin{equation}
(\zeta+\eta)\colon N^{(n)}_{\alpha\vee \beta}\to G,  \text{ with } (\zeta+\eta)(\bar{U})=\zeta(r^{\alpha\vee \beta}_{\alpha}(\bar{U}))+\eta(r^{\alpha\vee \beta}_{\beta}(\bar{U}))
\end{equation}  
We also endow  $C_{\mathrm{sem}}^{n}(\boldsymbol{\mathcal{U}};G)$  with the topology generated by basic open sets of the form $V_{s,p}$, with $s\in\mathbb{N}^{<\mathbb{N}}$ and $p\colon N^{(n)}_s\to G$,  which is given by all  $\zeta\colon N^{(n)}_{\alpha}\to G$ with  $\alpha\in \mathcal{N}_s$ and  $\zeta\upharpoonright N^{(n)}_s= p$.
\begin{proposition}\label{Prop:semi}
$C_{\mathrm{sem}}^{n}(\boldsymbol{\mathcal{U}};G)$ is an  abelian Polish semigroup.
\end{proposition}
\begin{proof}

 Since $\alpha\vee \beta=\beta\vee\alpha$ and $G$ is abelian, so is $C_{\mathrm{sem}}^{n}(\boldsymbol{\mathcal{U}};G)$. It is also a topological semigroup since by (L2) the operation $(\zeta,\eta)\mapsto(\zeta+\eta)$ is continuous.
  The topology is clearly second countable since it has, by definition, a basis consisting of countably many open sets. Finally it is easy to check that $\rho_{\mathrm{sem}}$ below is a complete metric on  $C_{\mathrm{sem}}^{(n)}(\boldsymbol{\mathcal{U}};G)$ that is compatible with the topology:  
 for every $\zeta\in  C^{n}(N_{\alpha};G)$  and $\eta\in C^{n}(N_{\beta};G)$, set $\delta_k:=0$ if $\zeta \upharpoonright   N^{(n)}_{\alpha|k}=\eta \upharpoonright   N^{(n)}_{\beta|k}$, and set $\delta_k:=1$ otherwise. Let 
 \begin{equation}\label{Eq:rhosem}
\rho_{\mathrm{sem}}(\zeta,\eta):= \sum_{k\in\mathbb{N}}\frac{\delta_k}{2^{k}}.
\end{equation}
\end{proof}

We define a {\em congruence} $\sim$ on $C_{\mathrm{sem}}^{(n)}(\boldsymbol{\mathcal{U}};G)$,  by setting for all  $\zeta\colon N^{(n)}_{\alpha}\to G$ and $\eta\colon N^{(n)}_{\beta}\to G$  
\[ \zeta \sim \eta \iff \zeta\circ r^{\alpha\vee \beta}_{\alpha}=\eta\circ r^{\alpha\vee \beta}_{\beta}. \]
\begin{lemma}
The relation $\sim$ is a closed semigroup congruence on $C_{\mathrm{sem}}^{n}(\boldsymbol{\mathcal{U}};G)$
\end{lemma}
\begin{proof}
First we prove that $\sim$ is an equivalence relation on  $C_{\mathrm{sem}}^{n}(\boldsymbol{\mathcal{U}};G)$. Symmetry follows from $\alpha\vee \beta =\beta\vee\alpha$ and reflexivity from $\alpha\vee\alpha=\alpha$. For
transitivity, let   $\zeta\colon N^{(n)}_{\alpha}\to G$, $\eta\colon N^{(n)}_{\beta}\to G$, and $\theta\colon N^{(n)}_{\gamma}\to G$, with $\zeta\sim\eta$ and $\eta\sim \gamma$. In particular, we have that 
\[\zeta\circ r^{\alpha\vee \beta}_{\alpha}\circ r^{\alpha\vee \beta\vee \gamma}_{\alpha\vee \beta}=\eta\circ r^{\alpha\vee \beta}_{\beta}\circ r^{\alpha\vee \beta\vee \gamma}_{\alpha\vee \beta}\; \text{ and } \; \eta\circ r^{\beta \vee \gamma}_{\beta}\circ r^{\alpha\vee \beta\vee \gamma}_{\beta \vee \gamma}=\theta\circ r^{\beta \vee \gamma}_{\gamma}\circ r^{\alpha\vee \beta\vee \gamma}_{\beta \vee \gamma}.\]
But since $\eta\circ r^{\alpha\vee \beta}_{\beta}\circ r^{\alpha\vee \beta\vee \gamma}_{\alpha\vee \beta}=\eta\circ r^{\alpha\vee \beta\vee \gamma}_{\beta}=\eta\circ r^{\beta \vee \gamma}_{\beta}\circ r^{\alpha\vee \beta\vee \gamma}_{\beta \vee \gamma}$, we have that  $\zeta\circ r_{\alpha}^{\alpha\vee \beta\vee \gamma}=\theta\circ r_{\gamma}^{\alpha\vee \beta\vee \gamma}$, and therefore that  $\zeta\circ  r_{\alpha}^{\alpha\vee \gamma}\circ  r_{\alpha\vee\gamma}^{\alpha\vee \beta\vee \gamma}=\theta\circ r_{\gamma}^{\alpha\vee\gamma}\circ r_{\alpha\vee\gamma}^{\alpha\vee \beta\vee \gamma}$. Since $r_{\alpha\vee\gamma}^{\alpha\vee \beta\vee \gamma}$ is surjective, we have that 
$\zeta\circ  r_{\alpha}^{\alpha\vee \gamma}=\theta\circ r_{\gamma}^{\alpha\vee\gamma}$; hence $\zeta\sim \theta$.

The relation $\sim$ is closed as a subset of $C_{\mathrm{sem}}^{n}(\boldsymbol{\mathcal{U}};G)\times C_{\mathrm{sem}}^{n}(\boldsymbol{\mathcal{U}};G)$: indeed if  $\zeta\circ r^{\alpha\vee \beta}_{\alpha}=\eta\circ r^{\alpha\vee \beta}_{\beta}$ fails, then there is an
$s=(\alpha\vee \beta)|k$, for some $k\in\mathbb{N}$, so that the values of some $\overline{U}\in N^{(n)}_{s}$  under $\zeta\circ r^{\alpha\vee \beta}_{\alpha}$ and $\eta\circ r^{\alpha\vee \beta}_{\beta}$ differ. This failure is witnessed by some open condition.

Finally, to see that $\sim$ is a semigroup congruence, let  $\zeta\colon N^{(n)}_{\alpha}\to G$ and $\eta\colon N^{(n)}_{\beta}\to G$ and  $\zeta'\colon N^{(n)}_{\alpha'}\to G$ and $\eta'\colon N_{\beta'}\to G$, with  $\zeta\sim \eta$ and $\zeta'\sim \eta'$. We will show that $\zeta+\zeta'\sim \eta+\eta'$. Indeed, since $\zeta\circ r^{\alpha\vee \beta}_{\alpha}=\eta\circ r^{\alpha\vee \beta}_{\beta}$ and $\zeta'\circ r^{\alpha'\vee \beta'}_{\alpha'}=\eta'\circ r^{\alpha'\vee \beta'}_{\beta'}$, we have:
\begin{align*}
\zeta\circ r^{\alpha\vee \beta}_{\alpha} \circ r_{\alpha\vee \beta}^{\alpha\vee \alpha'\vee\beta\vee\beta'} + \zeta'\circ r^{\alpha'\vee \beta'}_{\alpha'}\circ r_{\alpha'\vee \beta'}^{\alpha\vee \alpha'\vee\beta\vee\beta'}   &=\eta\circ r^{\alpha\vee \beta}_{\beta} \circ r_{\alpha\vee \beta}^{\alpha\vee \alpha'\vee\beta\vee\beta'}  +\eta'\circ r^{\alpha'\vee \beta'}_{\beta'}\circ r_{\alpha'\vee \beta'}^{\alpha\vee \alpha'\vee\beta\vee\beta'}\\
\zeta\circ r_{\alpha}^{\alpha\vee \alpha'\vee\beta\vee\beta'} + \zeta'\circ r_{\alpha'}^{\alpha\vee \alpha'\vee\beta\vee\beta'}   &=\eta\circ r_{\beta}^{\alpha\vee \alpha'\vee\beta\vee\beta'}  +\eta'\circ r_{\beta'}^{\alpha\vee \alpha'\vee\beta\vee\beta'} \\
(\zeta+\zeta')\circ r_{\alpha\vee\alpha'}^{\alpha\vee \alpha'\vee\beta\vee\beta'}   &= (\eta+ \eta' )\circ r_{\beta\vee\beta'}^{\alpha\vee \alpha'\vee\beta\vee\beta'} \\
(\zeta+\zeta') &\sim (\eta+\eta').\\
\end{align*}
\end{proof}

We now define  $C^{n}(\boldsymbol{\mathcal{U}};G):=C_{\mathrm{sem}}^{n}(\boldsymbol{\mathcal{U}};G)/\sim $ to be the collection of all congruence classes $[\zeta]$ of elements $\zeta$ of $C_{\mathrm{sem}}^{n}(\boldsymbol{\mathcal{U}};G)$. A priori,     $C^{n}(\boldsymbol{\mathcal{U}};G)$ is merely an abelian semigroup which is additionally endowed with the quotient topology: $V\subseteq C^{n}(\boldsymbol{\mathcal{U}};G)$ is open in $C^{n}(\boldsymbol{\mathcal{U}};G)$ if and only if its union $\bigcup V=\{\zeta \colon [\zeta]\in V\}$ is open in $C_{\mathrm{sem}}^{n}(\boldsymbol{\mathcal{U}};G)$. It turns out that  $C^{n}(\boldsymbol{\mathcal{U}};G)$ is a much richer structure.

 \begin{proposition}
 $C^{n}(\boldsymbol{\mathcal{U}};G)$ is a non-archimedean abelian Polish group.
 \end{proposition}
 \begin{proof}
 First notice that any two maps $0_{\alpha}\colon N^{(n)}_{\alpha}\to G$ and $0_{\beta}\colon N^{(n)}_{\beta}\to G$ which are constantly equal to $0_{G}$ are congruent to each other.  Moreover, it is clear that the associated congruence class $0_{\boldsymbol{\mathcal{U}}}:=[0_{\alpha}]=[0_{\beta}]$ is the identity element of  $C^{n}(\boldsymbol{\mathcal{U}};G)$ and that for every $\zeta\colon N^{(n)}_{\alpha}\to G$ we have that $[\zeta]+[-\zeta]=[0_{\alpha}]$. It follows that   $C^{n}(\boldsymbol{\mathcal{U}};G)$ is an abelian group. Moreover, notice that if $\zeta'\sim \zeta$ for some $\zeta\colon N^{(n)}_{\alpha}\to G$ and $\zeta'\colon N^{(n)}_{\beta}\to G$ with $\zeta\upharpoonright N^{(n)}_{\alpha|k}=0_{\alpha}\upharpoonright N^{(n)}_{\alpha|k}$, then $\eta\upharpoonright N^{(n)}_{\beta|k}=0_{\beta}\upharpoonright N^{(n)}_{\beta|k}$. Hence the identity $0_{\boldsymbol{\mathcal{U}}}$ admits a neighborhood basis in $C^{n}(\boldsymbol{\mathcal{U}};G)$ consisting of the following open subgroups:
 \[V_k:=\{[\zeta]\in C^{n}(\boldsymbol{\mathcal{U}};G) \colon \exists \alpha \in \mathcal{N}^* \text{ so that } \zeta\colon N^{(n)}\to G \text{ with }  \zeta \upharpoonright N^{(n)}_{\alpha|k}=0_{\alpha}\upharpoonright N^{(n)}_{\alpha|k}\}.\]

The topology of a quotient of a separable space is separable. We now define a complete metric on $C^{n}(\boldsymbol{\mathcal{U}};G)$, that is compatible with its topology. Let  $\rho_{\mathrm{sem}}$ be the complete metric on $C_{\mathrm{sem}}^{n}(\boldsymbol{\mathcal{U}};G)$ given by (\ref{Eq:rhosem}) above. For every  $[\zeta],[\eta]\in C^{n}(\boldsymbol{\mathcal{U}};G)$ we let
\begin{equation}\label{Eq:rho}
\rho([\zeta],[\eta]):= \inf \{\rho_{\mathrm{sem}}(\zeta',\eta') : \zeta'\in[\zeta], \eta'\in [\eta] \}.
\end{equation}
It is easy to see that one can directly compute the value of $\rho([\zeta],[\eta])$ by picking any representatives $\zeta'\colon N^{(n)}_{\alpha}\to G$ and $\eta'\colon N^{(n)}_{\beta}\to G$ of $[\zeta]$ and $[\eta]$ and alternatively  setting
\begin{equation}\label{Eq:rho'}
\rho([\zeta],[\eta])= \rho_{\mathrm{sem}}(\zeta'\circ r^{\alpha\vee \beta}_{\alpha},\eta'\circ r^{\alpha\vee \beta}_{\beta}).
\end{equation}
The fact that the latter quantity does not depend on the choice of $\eta'$ and $\zeta'$ follows by the same argument we used to prove transitivity of $\sim$, using local surjectivity (L3) in place of surjectivity.  
 By (\ref{Eq:rho}) it is clear that $\rho$ is a metric, given that by (\ref{Eq:rho'}) we have  $\rho([\zeta],[\eta])>0$ if $[\zeta]\neq [\eta]$.  To see that $\rho$ is compatible with the topology of $C^{n}(\boldsymbol{\mathcal{U}};G)$, notice that since the topology is first countable, it suffices to show that a subset $F$ of  $C^{n}(\boldsymbol{\mathcal{U}};G)$ is closed if and only if for every sequence $([\zeta_n])_{n}$ in $F$ with $[\zeta_n]\to_{\rho} [\zeta]$ we have that $[\zeta]\in F$. 
 
If  $[\zeta_n]\to_{\rho} [\zeta]$ and $F$ is closed, then using (\ref{Eq:rho'}) we may choose the representatives $\zeta_n$ and $\zeta$ so that $\zeta_n\to \zeta$ in $C^{n}_{\mathrm{sem}}(\boldsymbol{\mathcal{U}};G)$. Since the quotient map $\pi\colon  C^{n}_{\mathrm{sem}}(\boldsymbol{\mathcal{U}};G)\to C^{n}(\boldsymbol{\mathcal{U}};G)$ is continuous we have that $\pi^{-1}(F)$ is closed. Hence, $\zeta\in \pi^{-1}(F)$, i.e., $[\zeta]\in F$.
 Conversely, assume that whenever  $[\zeta_n]\to_{\rho} [\zeta]$, we have $[\zeta]\in F$. We will show that $\pi^{-1}(F)$ is closed. Let $(\zeta_n),\zeta$ be elements of $C^{n}_{\mathrm{sem}}(\boldsymbol{\mathcal{U}};G)$ such that $\zeta_n\to \zeta$. Then  by (\ref{Eq:rho'}) we have that  $[\zeta_n]\to_{\rho} [\zeta]$. Hence, $[\zeta]\in F$ and therefore $\zeta\in \pi^{-1}(F)$.
Finally, the fact that $\rho$ is complete follows from completeness of $\rho_{\mathrm{sem}}$ and    (\ref{Eq:rho'}).
\end{proof}

Finally, notice that for every $\zeta\colon N^{(n)}_{\alpha}\to G$  and $\zeta'\colon N^{(n)}_{\alpha'}\to G$  with $\zeta\sim \zeta'$, we have 
$\eta\sim\eta'$, where
\[\eta \big(v_{0},\ldots ,v_{n}\big)=\sum_{i=0}^{n}\left( -1\right) ^{i}\zeta (v_{0},\ldots ,\hat{v}_{i},\ldots
,v_{n}) \text{ and } \eta' \big(v_{0},\ldots ,v_{n}\big)=\sum_{i=0}^{n}\left( -1\right) ^{i}\zeta' (v_{0},\ldots ,\hat{v}_{i},\ldots
,v_{n}).\]
As a consequence, for  every $n>0$,   we have a continuous 
 {\em coboundary map}:  
\begin{align}
\begin{split}
\delta^{n}:C^{n-1}(\boldsymbol{\mathcal{U}};G)&\rightarrow C^{n}(\boldsymbol{\mathcal{U}};G), \text{ were } \\
\delta^{n}\left( [\zeta] \right)= [\eta], \text{ with }
\eta (v_{0},\ldots ,v_{n})&=\sum_{i=0}^{n}\left( -1\right) ^{i}\zeta (v_{0},\ldots ,\hat{v}_{i},\ldots
,v_{n}).
\end{split}
\end{align} 
 
This defines the {\em Polish cochain complex $C^{\bullet}(\boldsymbol{\mathcal{U}};G)$ of $G$-valued cochains of $\boldsymbol{\mathcal{U}}$}.

\begin{definition}\label{Def:maincoh}
Let $X$ be a locally compact Polish space and let $G$ be a countable abelian Polish group. Fix a covering system $\boldsymbol{\mathcal{U}}$ for $X$ and for every $n\in\mathbb{N}$  define the {\em $n$-dimensional definable cohomology group $\Check{\mathrm{H}}_{\mathrm{def}}^{n}(X;G)$  of $X$ with coefficients in $G$} to be  the $n$-dimensional cohomology group  of the Polish cochain complex $C^{\bullet}(\boldsymbol{\mathcal{U}};G)$, viewed as the group with a Polish cover:
\[
0\longrightarrow \mathrm{B}^{n}(\boldsymbol{\mathcal{U}};G)\longrightarrow \mathrm{Z}^{n}(\boldsymbol{\mathcal{U}};G)  \longrightarrow \mathrm{Z}^{n}(\boldsymbol{\mathcal{U}};G)/
\mathrm{B}^{n}(\boldsymbol{\mathcal{U}};G)\longrightarrow 0\]
 where  $\mathrm{Z}^{n}(\boldsymbol{\mathcal{U}};G)= \mathrm{ker}(\delta^{n})$ is the Polish group of {\em $n$-dimensional $G$-valued cocycles of  $\boldsymbol{\mathcal{U}}$}  and   $\mathrm{B}^{n}(\boldsymbol{\mathcal{U}};G)=\mathrm{im}(\delta^{n-1})$ is the Polishable group of {\em $n$-dimensional $G$-valued coboundaries of $\boldsymbol{\mathcal{U}}$}.
 \end{definition}
 
Notice that while the definition of $\mathrm{H}^{n}(X;G)$ involves a choice of covering system $\boldsymbol{\mathcal{U}}$ for $X$, the latter is not explicit in the notation $\Check{\mathrm{H}}_{\mathrm{def}}^{n}(X;G)$.  This is because, as we will show in Corollary \ref{CorDefCohIndependentU},  a different choice of $\boldsymbol{\mathcal{U}}$ will induce the exact same definable cohomology groups up to  definable isomorphism.
 
Since the Polish groups in $C^{\bullet}(\boldsymbol{\mathcal{U}};G)$ are  non-archimedean, the rigidity results of \cite{BLPI} provide a powerful resource for analyzing the definable cohomology groups of any locally compact Polish space. Note that as a sequence of abstract groups, $\mathrm{H}^{\bullet}(X;G)$ coincides with the classical \v{C}ech cohomology groups of $X$ with coefficients in $G$.

\begin{theorem}\label{T:Cech=Def}
As an abstract group,  $\Check{\mathrm{H}}_{\mathrm{def}}^{n}(X;G)$ is isomorphic to the $n$-dimensional \v{C}ech cohomology group $\Check{\mathrm{H}}^{n}(X;G)$ of $X$ with coefficients in $G$.  
\end{theorem}

\begin{proof}
Throughout this proof every group is viewed as an abstract group. The proof of Theorem \ref{T:Cech=Def}  follows by combining the two usual arguments which show: (1) without loss of generality, the colimit  in the definition of $\Check{\mathrm{H}}^{n}(X;G)$  can be  taken over any $\preceq$--cofinal collection of open covers; (2) the group  $\Check{\mathrm{H}}^{n}(X;G)$ is isomorphic to the sheaf of germs of  constant $G$-valued functions. We sketch the argument here, as the explicit abstract group isomorphism $\psi\colon \Check{\mathrm{H}}^{n}(X;G)\to \Check{\mathrm{H}}_{\mathrm{def}}^{n}(X;G)$
generated by this argument will be needed in the proof of Theorem \ref{T:DefinableHub}.

Recall that in order to define the  $n$-dimensional \v{C}ech cohomology group $\Check{\mathrm{H}}^{n}(X;G)$  of $X$ one first considers the poset $\mathrm{Cov}(X)$ of all open covers $\mathcal{U}$ of $X$ ordered by the relation of refinement $\mathcal{U}\preceq \mathcal{V}$ and observes that any two refinement maps $r^{\mathcal{V}}_{\mathcal{U}},s^{\mathcal{V}}_{\mathcal{U}}\colon \mathcal{V}\to \mathcal{U}$ induce contiguous (see  Section \ref{SS:polyhedra}) simplicial maps $r^{\mathcal{V}}_{\mathcal{U}},s^{\mathcal{V}}_{\mathcal{U}}\colon \mathrm{Nv}(\mathcal{V})\to \mathrm{Nv}(\mathcal{U})$. Hence, while the associated chain maps $(r^{\mathcal{V}}_{\mathcal{U}})^{\bullet},(s^{\mathcal{V}}_{\mathcal{U}})^{\bullet}\colon C^{\bullet}(\mathrm{Nv}(\mathcal{U}),G)\to C^{\bullet}(\mathrm{Nv}(\mathcal{V}),G)$ with $(r^{\mathcal{V}}_{\mathcal{U}})^{\bullet}(\zeta):=\zeta\circ r^{\mathcal{V}}_{\mathcal{U}}$ and $(s^{\mathcal{V}}_{\mathcal{U}})^{\bullet}(\zeta):=\zeta\circ s^{\mathcal{V}}_{\mathcal{U}}$ may differ, they induce the same homomorphism $(r^{\mathcal{V}}_{\mathcal{U}})^*=(s^{\mathcal{V}}_{\mathcal{U}})^*$ on the level of simplicial cohomology $\mathrm{H}^{n}(\mathrm{Nv}(\mathcal{U}),G)\to \mathrm{H}^{n}(\mathrm{Nv}(\mathcal{V}),G)$ for every $n\in\mathbb{N}$; see \cite[Corollary IX.2.14]{eilenberg_foundations_1952}, for example.

As a consequence we have a well-defined direct system  of group homomorphisms $\mathrm{H}^{n}(\mathrm{Nv}(\mathcal{U}),G)\to \mathrm{H}^{n}(\mathrm{Nv}(\mathcal{V}),G)$ indexed by the poset $(\mathrm{Cov}(X),\preceq)$.  By definition, see \cite[Definition IX.3.1]{eilenberg_foundations_1952}, we have that:
\begin{equation}\label{EQ:Cech}
\Check{\mathrm{H}}^{n}(X;G):=\mathrm{colim}_{\;\mathcal{U}\in\mathrm{Cov}(X)} \;\big(\mathrm{H}^n(\mathrm{Nv}(\mathcal{U});G), (r^{\mathcal{V}}_{\mathcal{U}})^* \big),
\end{equation}
where $(r^{\mathcal{V}}_{\mathcal{U}}\colon\mathcal{V}\to \mathcal{U})_{\mathcal{U}\preceq \mathcal{V}}$ is any fixed pre-chosen collection of refinement maps.

 Let now $(U_{\alpha}\colon \alpha\in \mathcal{N}^*)$ and  $(r^{\beta}_{\alpha}\colon  \alpha\leq \beta)$ be the  open covers and refinement maps appearing in  the covering system  $\boldsymbol{\mathcal{U}}$  used in the definition of  $\Check{\mathrm{H}}_{\mathrm{def}}^{n}(X;G)$, and set $N_{\alpha}:=\mathrm{Nv}(\mathcal{U}_{\alpha})$.   Since the colimit in (\ref{EQ:Cech})  is isomorphic to the colimit taken over any cofinal collection of $\mathcal{U}\in \mathrm{Cov}(X)$, by Lemma \ref{L:combinatorialapproximations}(3), we have that:
 \begin{equation} \label{EQ:Cech1}
 \Check{\mathrm{H}}^{n}(X;G) \cong \mathrm{colim}_{\;\alpha \in\mathcal{N}^*} \;\big(\mathrm{H}^n(N_{\alpha};G), (r^{\mathcal{\beta}}_{\mathcal{\alpha}})^* \big)= \mathrm{colim}_{\;\alpha \in\mathcal{N}^*} \;\big(\mathrm{Z}^{n}(N_{\alpha};G)/ \mathrm{B}^{n}(N_{\alpha};G), (r^{\mathcal{\beta}}_{\mathcal{\alpha}})^* \big).
 \end{equation}
However, the system $(r^{\beta}_{\alpha}\colon  \alpha\leq \beta)$ coheres, i.e., for all $\alpha\leq \beta\leq \gamma$ we have $r^{\gamma}_{\alpha}=r^{\beta}_{\alpha}\circ r^{\gamma}_{\beta}$; see Definition \ref{DefCovers2}. Moreover, if $\delta^{\bullet}_{\alpha}$ and $\delta^{\bullet}_{\beta}$ are the coboundary maps of the cochain complexes $\mathrm{C}^{\bullet}(N_{\alpha};G)$ and $\mathrm{C}^{\bullet}(N_{\beta};G)$ with $\alpha\leq \beta$, then for all $n\in\mathbb{N}$ we have that $(r^{\beta}_{\alpha})^{n+1}\circ\delta^{n}_{\alpha}=\delta^{n}_{\beta}\circ (r^{\beta}_{\alpha})^{n}$. 
As a consequence, we have a direct system of chain complexes $\big(\mathrm{C}^{\bullet}(N_{\alpha};G) ,(r^{\beta}_{\alpha})^{\bullet}\big)_{\alpha\in\mathcal{N}^*}$. But then, by the definition of a  colimit of a directed system of abelian groups (see e.g. \cite[Definition IX.4.1]{eilenberg_foundations_1952}),  for every $n\in\mathbb{N}$ we have that:
\begin{equation}
\mathrm{colim}_{\;\alpha \in\mathcal{N}^*} \;\big( \mathrm{C}^{n}(N_{\alpha};G) ,(r^{\beta}_{\alpha})^{n} \big)\quad  :=\quad \bigg(\bigsqcup_{\alpha \in\mathcal{N}^*} C^{n}(N_{\alpha};G) \; \bigg)\bigg/\sim \quad = \quad C^n_{\mathrm{sem}}(\boldsymbol{\mathcal{U}};G)/\sim \quad = \quad C^{n}(\boldsymbol{\mathcal{U}};G),
\end{equation}
where $\sim$ is the congruence we defined earlier in this section, and similarly have that  $\delta^n= \mathrm{colim}_{\alpha}\delta^n_{\alpha}$. Consequently:
\begin{equation}\label{EQ:colimits}
\mathrm{colim}_{\;\alpha \in\mathcal{N}^*} \;\big( \mathrm{Z}^{n}(N_{\alpha};G) ,(r^{\beta}_{\alpha})^{n} \big)  = \mathrm{Z}^{n}(\boldsymbol{\mathcal{U}};G) \quad \text{ and } \quad  \mathrm{colim}_{\;\alpha \in\mathcal{N}^*} \;\big( \mathrm{B}^{n}(N_{\alpha};G) ,(r^{\beta}_{\alpha})^{n} \big)  = \mathrm{B}^{n}(\boldsymbol{\mathcal{U}};G).
\end{equation}
But for every direct system of pairs of abelian groups $(H_i\leq G_{i},i\in I)$, colimits commute with quotients, i.e.,  $\mathrm{colim}_i(G_i)/\mathrm{colim}_i(
H_i) \cong \mathrm{colim}_i(G_i/H_i)$; see \cite[Theorem IX.6.3]{eilenberg_foundations_1952}. Hence, by combining (\ref{EQ:Cech1}) and (\ref{EQ:colimits}) we have:
\begin{equation} \label{EQ:Cech2}
\Check{\mathrm{H}}^{n}(X;G)\cong \mathrm{colim}_{\;\alpha \in\mathcal{N}^*} \;\big(\mathrm{Z}^{n}(N_{\alpha};G)/ \mathrm{B}^{n}(N_{\alpha};G), (r^{\mathcal{\beta}}_{\mathcal{\alpha}})^* \big) \cong \mathrm{Z}^{n}(\boldsymbol{\mathcal{U}};G)/\mathrm{B}^{n}(\boldsymbol{\mathcal{U}};G) = \Check{\mathrm{H}}_{\mathrm{def}}^{n}(X;G).
\end{equation}
An explicit formula for the composition of these isomorphisms, $\psi\colon \Check{\mathrm{H}}^{n}(X;G)\to \Check{\mathrm{H}}_{\mathrm{def}}^{n}(X;G)$, can be given as follows. 
First, let 
$[\zeta+  \mathrm{B}^{n}(\mathrm{Nv}(\mathcal{U});G) ]_{\Check{\mathrm{H}}}$ be any element of $\Check{\mathrm{H}}^{n}(X;G)$,  where $\zeta\in  \mathrm{Z}^{n}(\mathrm{Nv}(\mathcal{U});G)$ for some $\mathcal{U}\in\mathrm{Cov}(X)$, and $[a]_{\Check{\mathrm{H}}}$ is the image of $a\in \mathrm{H}^{n}(\mathrm{Nv}(\mathcal{U});G)$ under the inclusion $\mathrm{H}^{n}(\mathrm{Nv}(\mathcal{U});G)\hookrightarrow \Check{\mathrm{H}}^{n}(X;G)$. Then, 
\begin{equation} \label{EQ:Cech_psi}
\psi\big( [\zeta +  \mathrm{B}^{n}(\mathrm{Nv}(\mathcal{U});G) ]_{\Check{\mathrm{H}}}\big)=[ \zeta\circ r^{\mathcal{U}_{\alpha}}_{\mathcal{U}} ]+\mathrm{B}^{n}(\boldsymbol{\mathcal{U}};G) 
\end{equation}
where $\alpha$ can be taken to be any element of $\mathcal{N}^{*}$ with $\mathcal{U}\preceq\mathcal{U}_{\alpha}.$ Such $\alpha$ exists by Lemma \ref{L:combinatorialapproximations}(3).\end{proof}

\begin{remark}\label{RemGeneralPolish}
In this section we developed definable cohomology groups with coefficients in a countable discrete abelian group $G$. Implicit in our definitions was the passage from  classical definition of \v{C}ech cohomology to the sheaf-theoretic one. As outlined in  the proof of Theorem \ref{T:Cech=Def},  this passage was effected essentially by alternating the order between two processes, namely: taking the colimit of abelian groups, and  computing the cohomology groups of a cochain complex.  One can more generally develop definable cohomology groups $\mathrm{H}_{\mathrm{def}}^{n}(X;G')$ for $X$ with coefficients in an arbitrary Polish abelian group $G'$ by endowing the sheaf of germs of all continuous $G'$-valued functions with an appropriate Polish structure.  However, this would require significant amounts of bookkeeping. Perhaps one way of simplifying the process of endowing the aforementioned sheaf with a Polish structure would be to  first pass to a ``completion" of the inverse system $((\mathcal{U}_{\alpha}),(r^{\beta}_{\alpha}))$ by  considering locally profinite topological simplicial complexes similar to the  profinite topological simplicial complexes used in projective Fra\"iss\'e theory; see \cite{SP}. This, however, is  beyond the scope of the present paper. 
\end{remark}

\subsection{Definable cohomology of pairs of spaces}

\label{SS:combinatorial_cohomology_of_pairs}

The process of enriching the \v{C}ech cohomology groups of locally compact spaces with definable content readily extends to the \v{C}ech cohomology groups of locally compact pairs, i.e, pairs $(X,A)$ in which $X$ is a locally compact Polish space and $A$ is closed in $X$; see Section \ref{SS:homotopy}. In this section we sketch the details.
We begin by generalizing the theory from Section \ref{SS:DefCohSimCom} to pairs $(K,L)$ of simplicial complexes.

Let $K$ be a simplicial complex. A \emph{subcomplex} $L$ of $K$ is any simplicial complex with $\mathrm{dom}(L)\subseteq \mathrm{dom}(K)$ for which $\sigma\in L$ implies $\sigma \in K$. By a \emph{simplicial pair} we mean a pair $(K,L)$ of simplicial complexes in which $L$ is a subcomplex of $K$. A \emph{simplicial map (of pairs)} $f\colon (K,L)\to (K',L')$ is any simplicial map $f\colon K\to K'$ such that $\{f(v)\colon v\in \sigma\}\in L'$ for all $\sigma\in L$. For any abelian Polish group $G$ and $n\in\mathbb{N}$, let  
\[C^{n}(K,L;G):=\{\zeta \in C^{n}(K;G) \colon   \zeta(\bar{v})=0_G  \text{ for all } \bar{v}\in L^{(n)} \}.\]
$C^{n}(K,L;G)$ is a closed subgroup of $C^{n}(K;G)$ with respect to the product topology of the $K^{(n)}$-fold product of copies of $G$, hence it is a Polish group. The coboundary map $ \delta^{n}:C^{n}(K;G) \rightarrow C^{n+1}(K;G)$ clearly restricts to a map $C^{n}(K,L;G) \rightarrow C^{n+1}(K,L;G)$ which we will also denote by $\delta^{n}$.  This induces a Polish cochain complex:
\[C^{\bullet}(K,L;G):=(  \quad \quad \quad \quad \quad \quad\cdots\longrightarrow \,C^{n-1}(K,L;G)\overset{\delta^{n-1}}{\longrightarrow} C^n(K,L;G)\overset{\delta^{n}}{\longrightarrow} C^{n+1}(K,L;G) {\longrightarrow} \cdots  \quad \quad  \quad \quad \quad ). \quad \quad\]
The \emph{definable cohomology groups of $(K,L)$ with coefficients in $G$} are the groups with a Polish cover
\[\mathrm{H}_{\mathrm{def}}^n(K,L;G):=  \mathrm{Z}^{n}(K,L;G)/ \mathrm{B}^{n}(K,L;G) := \mathrm{ker}(\delta^{n})/\mathrm{im}(\delta^{n-1}). \]
Note that the natural homomorphisms determining the short exact sequences
\begin{equation}\label{EQ:SES}
0\longrightarrow   C^n(K,L;G)  \overset{i^n}{\longrightarrow} C^n(K;G)  \overset{r^n}{\longrightarrow} C^n(L;G)  \longrightarrow  0
\end{equation}
are continuous and aggregate into a short exact sequence of Polish cochain complexes
\begin{equation}\label{EQ:SEScc}
0\longrightarrow   C^{\bullet}(K,L;G)  \longrightarrow C^{\bullet}(K;G)  \longrightarrow C^{\bullet}(L;G)  \longrightarrow  0.
\end{equation}

For the rest of this section let $G$ be a countable abelian group 
and $(X,A)$ be a locally compact pair. Fix any covering system $\boldsymbol{\mathcal{U}}=\big{(}(X_n),(\mathcal{U}_{\alpha}),(r^{\beta}_{\alpha})\big{)}$ for $X$ and notice that the triple $\boldsymbol{\mathcal{V}}=\big{(}(A_n),(\mathcal{V}_{\alpha}),(s^{\beta}_{\alpha})\big{)}$, where $A_n:=A\cap X_n$,  $\mathcal{V}_{\alpha}:=\mathcal{U}_{\alpha}\upharpoonright A\subseteq \mathcal{U}_{\alpha}$ and  $s^{\beta}_{\alpha}$    denotes the restriction of $r^{\beta}_{\alpha}$  to $\mathcal{V}_{\beta}$, is a covering system for $A$. Moreover, the inclusion $\mathcal{V}_{\alpha}\subseteq \mathcal{U}_{\alpha}$ induces a simplicial pair  $(N_{\alpha},M_{\alpha}):=\big{(}\mathrm{Nv}(\mathcal{U}_{\alpha}),\mathrm{Nv}(\mathcal{V}_{\alpha})\big{)}$.
For every  $n\in\mathbb{N}$, we set
\[C^{n}( \boldsymbol{\mathcal{U}}, \boldsymbol{\mathcal{V}};G):=\{[\zeta] \in C^{n}( \boldsymbol{\mathcal{U}},;G)   \colon \;   \;\forall \alpha\in \mathcal{N}^*  \;  \;  \forall \bar{v}\in M_{\alpha}^{(n)} \; \;\forall \zeta' \in [\zeta] \cap C^n(N_{\alpha};G),   \text{ we have that }   \zeta' (\bar{v})=0_G   \}.\]
Then  $C^{n}( \boldsymbol{\mathcal{U}}, \boldsymbol{\mathcal{V}};G)$ is a Polish group, as it is a closed subgroup of $C^{n}( \boldsymbol{\mathcal{U}};G)$. Moreover, the coboundary map $ \delta^{n}:C^{n}( \boldsymbol{\mathcal{U}};G) \rightarrow C^{n+1}( \boldsymbol{\mathcal{U}};G)$ restricts to  $ \delta^{n}:C^{n}( \boldsymbol{\mathcal{U}},\boldsymbol{\mathcal{V}};G) \rightarrow C^{n+1}( \boldsymbol{\mathcal{U}},\boldsymbol{\mathcal{V}};G)$,  inducing a Polish cochain complex:
\[C^{\bullet}( \boldsymbol{\mathcal{U}}, \boldsymbol{\mathcal{V}};G):=(  \quad \quad \quad \quad \quad \quad\cdots\longrightarrow \,C^{n-1}( \boldsymbol{\mathcal{U}}, \boldsymbol{\mathcal{V}};G)\overset{\delta^{n-1}}{\longrightarrow} C^n( \boldsymbol{\mathcal{U}}, \boldsymbol{\mathcal{V}};G)\overset{\delta^{n}}{\longrightarrow} C^{n+1}( \boldsymbol{\mathcal{U}}, \boldsymbol{\mathcal{V}};G) {\longrightarrow} \cdots  \quad \quad  \quad \quad \quad ). \quad \quad\]
The \emph{definable cohomology groups of $(X,A)$ with coefficients in $G$} are the  groups with a Polish cover:
\[\Check{\mathrm{H}}_{\mathrm{def}}^n(X,A;G):=  \mathrm{Z}^{n}( \boldsymbol{\mathcal{U}}, \boldsymbol{\mathcal{V}};G)/ \mathrm{B}^{n}( \boldsymbol{\mathcal{U}}, \boldsymbol{\mathcal{V}};G) := \mathrm{ker}(\delta^{n})/\mathrm{im}(\delta^{n-1}). \]
Arguing exactly as for Theorem \ref{T:Cech=Def} shows that, as abstract groups, these are precisely the \v{C}ech cohomology groups $\Check{\mathrm{H}}^n(X,A;G)$ of the pair $(X,A)$; see \cite[Chapter IX]{eilenberg_foundations_1952} for a definition of the latter.
Moreover, these $\check{\mathrm{H}}^n_{\mathrm{def}}$ groups array into a definable long exact sequence of the usual form; put differently, by the following theorem, they satisfy the Exactness Axiom as maps to the category of \emph{groups with a Polish cover}.
\begin{theorem}\label{T:definable_LES}
For any countable discrete group $G$ and locally compact pair $(X,A)$ there exists a long exact sequence of groups with a Polish cover and definable group homomorphisms
\[
\dots\longrightarrow \Check{\mathrm{H}}_{\mathrm{def}}^{n-1}(A;G) \overset{\partial^{n-1}}{\longrightarrow} \Check{\mathrm{H}}_{\mathrm{def}}^n(X,A;G)
\overset{\iota^{n}}{\longrightarrow} \Check{\mathrm{H}}_{\mathrm{def}}^n(X;G)  \overset{\rho^{n}}{\longrightarrow} \Check{\mathrm{H}}_{\mathrm{def}}^n(A;G) \overset{\partial^{n}}{\longrightarrow} \Check{\mathrm{H}}_{\mathrm{def}}^{n+1}(X,A;G) \longrightarrow\cdots 
\]
\end{theorem}

Theorem \ref{T:definable_LES}  is an instance of Theorem \ref{Th:LongExactGeneral} below which,  in turn, is
 a direct consequence of the following definable version of the Snake Lemma:

\begin{lemma}
\label{Lemma:SnakeLemma} Consider the following diagram of continuous homomorphisms of  abelian Polish groups:
\begin{center}
\begin{tikzcd}
& A \arrow[r, "e"] \arrow[d, "a"] & B \arrow[d, "b"]\arrow[r, "f"]  & C \arrow[r] \arrow[d, "c"] & 0  \\
0 \arrow[r] & A' \arrow[r, "g"] & B' \arrow[r, "h"]  & C' &   
\end{tikzcd}
\end{center}
If the rows are exact, then there exists a Borel function $\varepsilon \colon \mathrm{ker}(c)\rightarrow A^{\prime }$
so that the map $x\mapsto \big(\varepsilon(x)+\mathrm{im}(A)\big)$ links the middle terms in the following
exact sequence of groups with a Polish cover and definable group homomorphisms: 
\begin{equation*}
\mathrm{ker}(a)\overset{e}{\longrightarrow }\mathrm{ker}(b)\overset{f}{%
\longrightarrow }\mathrm{ker}(c)\longrightarrow \mathrm{coker}(a)\overset{%
g}{\longrightarrow }\mathrm{coker}(b)\overset{h}{%
\longrightarrow }\mathrm{coker}(c).
\end{equation*}
Observe that we have notationally conflated all other maps with those which they induce on subgroups and quotients.
\end{lemma}

\begin{proof}
By \cite[Theorem 12.15]{kechris_classical_1995} there exists a Borel selector $\gamma$ for the $\mathrm{ker}(f)$-coset relation on $B$; put differently, there exists a
Borel function $\gamma \colon C\rightarrow B$ so that $f\circ \gamma =\mathrm{id}%
_{C}.$ Being a continuous injection, $g$ is a Borel
isomorphism between $A^{\prime }$ and $\mathrm{im}(g)$ and hence possesses a Borel inverse $\varphi \colon \mathrm{im}(g)\rightarrow
A^{\prime }$. Set $\varepsilon =\varphi \circ
b\circ (\gamma \upharpoonright \mathrm{ker}(c))$. Since the rightmost square
commutes, for every $x\in \mathrm{ker}(c)$ we have that $b\circ \gamma
(x)\in \mathrm{ker}(h)=\mathrm{im}(g)$ so $\varepsilon$ is
well-defined and Borel. The rest of the argument is essentially algebraic and standard; see \cite[Lemma 9.1]{lang_algebra_2002}.
\end{proof}

\begin{theorem}\label{Th:LongExactGeneral}
Let $0\longrightarrow   A^{\bullet}  \overset{i^{\bullet}}{\longrightarrow}  B^{\bullet} \overset{r^{\bullet}}{\longrightarrow}  C^{\bullet} \longrightarrow  0$
be a short exact sequence of Polish cochain complexes where the group homomorphisms $i^n,r^n$ are continuous for all $n\in\mathbb{Z}$. Then there is a long exact sequence of definable homomorphisms between the associated groups with a Polish cover:
\begin{equation}\label{EQ:longExact}
\dots\longrightarrow   \mathrm{H}^{n-1}(C^{\bullet})\overset{\partial^{n-1}}{\longrightarrow} \mathrm{H}^{n}(A^{\bullet})
\overset{\widehat{\iota}^{n}}{\longrightarrow} \mathrm{H}^{n}(B^{\bullet})  \overset{\widehat{\rho}^{n}}{\longrightarrow} \mathrm{H}^{n}(C^{\bullet})\overset{\partial^{n}}{\longrightarrow} \mathrm{H}^{n+1}(A^{\bullet})\longrightarrow\cdots 
\end{equation}  
\end{theorem}
\begin{proof}
On the level of abstract cochain complexes and abstract abelian groups the long exact sequence  (\ref{EQ:longExact}) can be formed by taking  $\widehat{\iota}^{n}$ and $\widehat{\rho}^{n}$ to be the maps induced by the homomorphisms  $\iota_n\upharpoonright \mathrm{Z}^{n}(A^{\bullet})$ and  $\rho_n\upharpoonright \mathrm{Z}^{n}(B^{\bullet})$, and setting $\partial^{n}$ to be the map induced by an application of the classical Snake Lemma on the diagram:
\begin{center}
\begin{tikzcd}
& A^n \arrow[r, "\iota^n"] \arrow[d, "\delta^n_A"] & B^n \arrow[d, "\delta^n_B"]\arrow[r, "\rho^n"]  & C^n \arrow[r] \arrow[d, "\delta^n_C"] & 0  \\
0 \arrow[r] & A^{n+1} \arrow[r, "\iota^{n+1}"] & B^{n+1} \arrow[r, "\rho^{n+1}"]  & C^{n+1} &   
\end{tikzcd}
\end{center}
But then, $\widehat{\iota}^{n}$ and $\widehat{\rho}^{n}$  are  both definable since they admit $\iota_n\upharpoonright \mathrm{Z}^{n}(A^{\bullet})$ and  $\rho_n\upharpoonright \mathrm{Z}^{n}(B^{\bullet})$  as continuous lifts. The homomorphism $\partial^n$ is also definable as a consequence of Lemma \ref{Lemma:SnakeLemma}.
\end{proof}

Theorem \ref{T:definable_LES} now follows by an application of Theorem \ref{Th:LongExactGeneral} to the short exact sequence (\ref{EQ:SEScc}).

\section{Definable sets and groups}\label{S:DSets}

At work in the main results of \cite{BLPI} is the fact that the category $\mathsf{GPC}$ of groups with a Polish cover is well-behaved in a number of ways.
For example, if a group isomorphism $G/H\to G'/H'$ lifts to a Borel function then so too does its inverse; see \cite[Remark 3.3]{BLPI}.
Put differently, in $\mathsf{GPC}$ every bijective morphism is an isomorphism; this is just as we would hope, of course, and sharply contrasts with with the behavior of categories of topological groups, for example.
Deeper regularities of this construction manifest as the second author's more recent result that the \emph{category of groups with an abelian Polish cover} $\mathsf{APC}$ forms an abelian category \cite{lupini_looking_22}.

Algebraic topology, however, abounds in spaces possessing a group structure \emph{only modulo a homotopy relation}; such \emph{$H$-groups}, moreover, are fundamental to the homotopical representation of cohomology theories, and of \v{C}ech cohomology in particular.
We were led in this way to consideration of a more general category of what we term \emph{definable groups}, one which contains $\mathsf{GPC}$ as a full subcategory, as well as the homotopical examples we have in mind, and we were gratified to discover how many of $\mathsf{GPC}$'s regularity properties persist in this wider setting.
In a nutshell, a definable group  $(X/E,*,\mu)$ is the quotient $X/E$ of a Polish space by an idealistic Borel equivalence relation $E$, together with a group structure given by a multiplication $\mu\colon X/E \times X/E \to X/E$ so that both $\mu$ and the group-inverse map $X/E\to X/E$ lift to Borel maps at the level of $X$. Just as in the category of groups with a Polish cover, maps in the category of definable groups are definable homomorphisms, i.e., group homomorphisms $X/E\to Y/F$ which lift to a Borel map $X\to Y$.

We will be working with a slightly more general notion of an  idealistic equivalence relation than is standard. This is introduced in Section \ref{Section:er}, where we also record some technical preliminaries for our later work. Definable groups are then introduced as group objects in the category $\mathsf{DSet}$ of \emph{definable sets}; see Sections \ref{Section:dset} and \ref{Section:dgroup}.
The prototypical example of a definable group is the group $[X, K(G,n)]$ of homotopy classes of maps from a locally compact Polish space $X$ to an Eilenberg--MacLane space $K(G,n)$; see Section \ref{S:Hidealistic}.  In Section \ref{S:Huber} we will  establish that    $[X, K(G,n)]$ is \emph{essentially a group with a Polish cover}, i.e., is definably isomorphic to a group with a Polish cover. It remains an open question whether every definable group is essentially a group with a Polish cover, as we note in our conclusion.

\subsection{Idealistic equivalence relations\label{Section:er}}

We say that an equivalence relation $E$ on a Polish space $X$ is \emph{open}, \emph{closed}, or \emph{Borel}, respectively, if it
is an open, closed, or Borel subset of $X\times X$ endowed with the product
topology.  We will denote by $[x]_E$, or simply by $[x]$ if $E$ is understood, the $E$-equivalence class of $x$ in $X$.

The following is a slight generalization of the classical notion of an \emph{idealistic}
equivalence relation; it coincides with  
\cite[Definition 5.4.9]{gao_invariant_2009} if we require that the map $\zeta :X\rightarrow X$ in Definition \ref{Definition:idealistic} is the identity map. And while our definition's allowance for a potentially wider range of functions $\zeta$ appears indispensible to the proof of Lemma \ref{Lemma:invariance} below, for example, it remains unclear whether \cite[Definition 5.4.9]{gao_invariant_2009} is in fact equivalent to Definition \ref{Definition:idealistic}.
 Recall that a \emph{$%
\sigma $-filter} of subsets of a set $C$ is a nonempty collection $\mathcal{F}
$ of nonempty subsets of $C$ which is closed under supersets and
countable intersections.

\begin{definition}
\label{Definition:idealistic}An equivalence relation $E$ on a Polish space $X$ is \emph{idealistic} if there
exist an assignment $[x]\mapsto \mathcal{F}_{[x]}$ mapping each $E$-class $[x]$ to
a $\sigma$-filter $\mathcal{F}_{[x]}$ of subsets of $[x]$ and a Borel function $\zeta :X\rightarrow X$ satisfying
\begin{itemize}
\item $x\,E\, \zeta( x) $ for
every $x\in X$, and
\item for each Borel subset $A$ of $X\times X$ the
set $A_{\mathcal{F}}\subseteq X$ defined by%
\begin{equation*}
x\in A_{\mathcal{F}}\Leftrightarrow \left\{ x^{\prime }\in \left[ x\right] :( \zeta( x) ,x^{\prime }) \in A\right\} \in 
\mathcal{F}_{\left[ x\right]}
\end{equation*}%
is Borel.
\end{itemize}
\end{definition}

A common convention facilitates statements like the second bulleted point above: if $%
\mathcal{F}$ is a filter of subsets of $C$ and $P(x) $ is a property which elements of $C$ may or may not have, write \textquotedblleft $\mathcal{F}x\:\:P( x) $\textquotedblright{}  for the assertion ``$\{x\in C:P( x) $ holds$\}\in \mathcal{F}$.'' We may then rephrase the second item of Definition \ref%
{Definition:idealistic} as the requirement that for every Borel $A\subseteq X\times X$ the set $\{x\in X:\mathcal{F}_{\left[ x\right] }x^{\prime }\:\:( \zeta
( x) ,x^{\prime }) \in A\}$ is Borel. 

\begin{example} Let $G/N$ be a group with a Polish cover and let $\mathcal{R}(G/N)$ be the coset equivalence relation on $G$, so that $x \, \mathcal{R}(G/N) \, x' \iff xN= x'N$. Then $\mathcal{R}(G/N)$ is idealistic, as witnessed by  $\zeta:=\mathrm{id}_{G}$ and the assignment $[x]\mapsto \mathcal{F}_{[x]}$, with $D\in  \mathcal{F}_{[x]}\iff$ ``$x^{-1}D$ is a comeager subset of the Polish group $N$"; see 
\cite[Proposition 5.4.10]{gao_invariant_2009}.
\end{example}

The collection of idealistic equivalence relations enjoys several desirable closure properties.
It is clear, for example, that if $E$ is an idealistic equivalence relation on $X$ and $%
X_{0}\subseteq X$ is an $E$-invariant Borel subset then $E|_{X_0}$ is
idealistic.

\begin{lemma}
\label{Lemma:product}Suppose that $E,F$ are idealistic equivalence relations
on Polish spaces $X,Y$, respectively. Define $E\times F$ to be the
equivalence relation on $X\times Y$ defined by $( x,y) (
E\times F) ( x^{\prime },y^{\prime }) $ if and only if $%
x\,E\,x^{\prime }$ and $y\,F\,y^{\prime }$. Then $E\times F$ is idealistic.
\end{lemma}

\begin{proof}
Let the assignments $[x]\mapsto \mathcal{F}_{[x]}^{X}$ and $\zeta
^{X}:X\rightarrow X$ witness that $E$ is idealistic, and let $[y]\mapsto \mathcal{F
}_{[y]}^{Y}$ and $\zeta ^{Y}:Y\rightarrow Y$ witness that $F$ is idealistic.
Observe that an $( E\times F) $-class $[(x,y)]$ is of the form $[x]\times [y]$
where $[x]$ is an $E$-class and $[y]$ is an $F$-class. Define $\mathcal{F}_{[(x,y)]}$ by setting $S\in \mathcal{F}_{[(x,y)]}$
if and only if $[\mathcal{F}_{[x]}^{X}x'\:\:\mathcal{F}_{[y]}^{Y}y'\:\:(
x',y') \in S]$, and observe that $\mathcal{F}_{[(x,y)]}$ is indeed a $\sigma $-filter. Define
also the function $\zeta :X\times Y\rightarrow X\times Y$, $(
x,y) \mapsto ( \zeta ^{X}( x) ,\zeta ^{Y}(
y) ) $.

For any $( x,y) \in X\times Y$ and Borel $A\subseteq ( X\times Y) \times ( X\times Y) $ we have that $[\mathcal{F}_{%
\left[ ( x,y) \right] }( x^{\prime },y^{\prime }) \:\:
( \zeta ( x,y) ,( x^{\prime },y^{\prime })
) \in A]$ if and only if $[\mathcal{F}_{\left[ x\right] }^{X}x^{\prime }\:\:\mathcal{F}_{\left[ y\right] }^{Y}y^{\prime }\:\:( ( \zeta
^{X}( x^{\prime }) ,\zeta ^{Y}( y^{\prime }) )
,( x^{\prime },y^{\prime }) ) \in A]$. The set%
\begin{equation*}
A_{\mathcal{F}}=\{( x,y) \in X\times Y:\mathcal{F}_{\left[ x%
\right] }^{X}x^{\prime }\:\:\mathcal{F}_{\left[ y\right] }^{Y}y^{\prime
}\:\:( ( \zeta^{X}( x) ,\zeta^{Y}(
y) ) ,( x^{\prime },y^{\prime }) ) \in A\}
\end{equation*}%
is Borel by the assumption that the assignments $[x]\mapsto \mathcal{F}_{[x]}^{X}$
and $\zeta ^{X}$ and $%
[y]\mapsto \mathcal{F}_{[y]}^{Y}$ and $\zeta ^{Y}$, respectively, witness that $E$ and $
F$ are idealistic. Hence $[(x,y)]\mapsto \mathcal{F}_{[(x,y)]}$ and $
\zeta :X\times Y\rightarrow X\times Y$ witness that $E\times F$ is
idealistic.
\end{proof}

\begin{lemma}
\label{Lemma:cokernel}Suppose that $F$ is an idealistic equivalence relation
on a Polish space $Y$ and $Z$ is an $F$-invariant Borel subset of $Y$.
Define $F_Z$ to be the equivalence relation on $Y$ given by $%
x\,F_Z\,x^{\prime }$ if and only if $x\,F\,x^{\prime }$ or $x,x^{\prime }\in Z$.
Then $F_Z$ is idealistic.
\end{lemma}
\begin{proof}
Suppose that the assignments $[x]\mapsto \mathcal{F}_{[x]}$ and $\zeta
:Y\rightarrow Y$ witness that $F$ is idealistic. Fix a Polish topology $\tau _Z$ on $Z$ that is compatible with the Borel structure inherited from $Y$.
 Let $C$ be an $F_Z$-equivalence class. Set  $\mathcal{F}_{C}^{\prime }:=
\mathcal{F}_{C}$ if $C=[x]$ with $x\not\in Z$ and let $\mathcal{F}_Z^{\prime }=\{S\subseteq
Z:S$ contains a $\tau _Z$-comeager Borel subset of $Z\}$.
Notice that for a Borel subset $A$ of $Y\times Y$ and $x\in Y$ we have $[\mathcal{F}_{\left[ x\right]
}^{\prime }x^{\prime }\:\:( \zeta ( x) ,x^{\prime })
\in A]$ if and only if either $x\in Z$ and $A_{\zeta(x)}=\left\{ x^{\prime }\in
Y:( \zeta ( x) ,x^{\prime }) \in A\right\} $ is
comeager in $Z$ or $x\notin Z$ and $[\mathcal{F}_{\left[ x\right]
}x^{\prime }\:\:( \zeta ( x) ,x^{\prime }) \in A]$. It then follows from the reasoning, e.g., of \cite[Theorem 16.1]{kechris_classical_1995} that%
\begin{equation*}
A_{\mathcal{F}^{\prime }}=\{x\in Y:\mathcal{F}_{\left[ x\right] }^{\prime
}x^{\prime }\:\:( \zeta ( x) ,x^{\prime }) \in A\}
\end{equation*}%
is Borel. Thus the assignments $C\mapsto \mathcal{F}_{C}^{\prime }$ and $%
\zeta $ witness that $F_Z$ is idealistic.
\end{proof}

We now recall several notions of reduction between equivalence relations; see 
\cite[Definition 5.1.1]{gao_invariant_2009} and \cite[Definitions 2.1 and 2.2%
]{motto_ros_complexity_2012}. Suppose that $E$ and $F$ are 
equivalence relations on Polish spaces $X$ and $Y$, respectively, and let $%
X/ E$ and $Y/
F$ be the set of $E$ and $F$-equivalence classes on $X$ and $Y$, respectively. If $f:X/ E
\rightarrow Y/ F$ is a function then a \emph{lift }of $f$ is a
function $\hat{f}:X\rightarrow Y$ such that $f( [x]_{E}) =[\hat{f}%
( x) ]_{F}$ for every $x\in X$; if some such $\hat{f}$ is, furthermore, Borel then we say that $f$ \emph{lifts to a Borel function} $%
X\rightarrow Y$.

\begin{definition} \label{Definition:reduction} Suppose that  $E$ and $F$ are equivalence relations on Polish spaces $X$ and $Y$.

A \emph{Borel homomorphism from $E$ to $F$} is a function $X/E\rightarrow Y/ F$ which lifts to a Borel function $%
X\rightarrow Y$.

A \emph{Borel reduction from $E$ to $F$} is an injective function $%
X/ E\rightarrow Y/ F$ which lifts to a Borel
function $X\rightarrow Y$.

A \emph{classwise Borel isomorphism from $E$ to $F$} is a bijective
function $X/ E\rightarrow Y/ F$ which lifts to a
Borel function $X\rightarrow Y$ and whose inverse also lifts to a Borel function 
$Y\rightarrow X$;

A \emph{classwise Borel embedding from $E$ to $F$} is a classwise
Borel isomorphism from $E$ to $F|_{Z}$ for some $F$-invariant Borel
subset $Z$ of $Y$.
\end{definition}

By the following lemma, Borel and idealistic equivalence
relations are downwards closed with
respect to Borel reducibility and classwise Borel embedding relations, respectively.

\begin{lemma}
\label{Lemma:invariance}Suppose that $E$ and $F$ are equivalence relations on
Polish spaces $X$ and $Y$.
If $E$ classwise Borel embeds into $F$ and $F$ is idealistic then $E$ is idealistic.
\end{lemma}

\begin{proof}
Suppose that $F$ is idealistic and $E$ classwise Borel embeds into $F$. Without loss of generality we may assume $E$ and $F$ to be
classwise Borel isomorphic. Thus there exists a bijection $f:X/ E
\rightarrow Y/ F$ such that $f$ lifts to a Borel function $\hat{f}:X\rightarrow Y$ and $f^{-1}$ lifts to a Borel function $\hat{f}^{\ast
}:X\rightarrow Y$.

Suppose that the assignment $[y]\mapsto \mathcal{F}_{[y]}$ and the map $\tau:Y\rightarrow Y$ witness that $F$ is idealistic. We define an assignment $[x]\mapsto \mathcal{E}_{[x]}$ from $E$-classes to $\sigma $-filters as follows:
for each $E$-class $[x]$ and subset $S$ of $[x]$ let $S$ be in $\mathcal{E}_{[x]}$ if and
only if $\hat{f}^{\ast -1}( S) \in \mathcal{F}_{f([x])}$, i.e., $%
\mathcal{F}_{[f(x)]}\:\:\hat{f}^{\ast }( y) \in S$. Notice that for
a Borel subset $A$ of $X\times X$ and $x\in X$ we have that $[\mathcal{E}%
_{\left[ x\right] }x^{\prime }\:\:( x,x^{\prime }) \in X]$ if and
only if $[\mathcal{F}_{f( \left[ x\right] ) }y^{\prime }\:\:(x,%
\hat{f}^{\ast }( y^{\prime }) )\in A]$. We also define the Borel
map $\zeta =\hat{f}^{\ast }\circ \tau \circ \hat{f}:X\rightarrow X$. Notice
that $\zeta(x)\,E\,x$ for every $x\in X$.

Let $A$ be a Borel subset of $X\times X$. Consider the Borel
subset $B$ of $Y\times Y$ defined by $( y,y^{\prime }) \in B$ if
and only if $(\hat{f}^{\ast }( y) ,\hat{f}^{\ast }(
y^{\prime }) )\in A$. Then by assumption we have that $\{y\in Y:%
\mathcal{F}_{\left[ y\right] }y^{\prime }\:\:( \tau ( y)
,y^{\prime }) \in B\}$ is Borel. Hence $\hat{f}^{-1}( \{y\in
Y:\mathcal{F}_{\left[ y\right] }y^{\prime }\:\:( \tau ( y)
,y^{\prime }) \in B\}) \subseteq X$ is Borel. We then have 
\begin{equation*}
x\in \hat{f}^{-1}( \{y\in Y:\mathcal{F}_{\left[ y\right] }y^{\prime
}\:\:( \tau ( y) ,y^{\prime }) \in B\})
\end{equation*}%
if and only if 
\begin{equation*}
\mathcal{F}_{f( \left[ x\right] ) }y^{\prime }\:\:((\tau \circ \hat{f%
})( x) ,y^{\prime })\in B
\end{equation*}%
if and only if 
\begin{equation*}
\mathcal{F}_{f( \left[ x\right] ) }y^{\prime }\:\:((\hat{f}%
^{\ast }\circ \tau \circ \hat{f})( x) ,\hat{f}^{\ast }(
y^{\prime }) )\in A
\end{equation*}%
if and only if 
\begin{equation*}
\mathcal{E}_{\left[ x\right] }x^{\prime }\:\:(\zeta ( x) ,x^{\prime
})\in A.
\end{equation*}%
Hence $\mathcal{E}$ and $\zeta $ witness that $E$ is idealistic.
\end{proof}

We now recall some results from \cite{gao_invariant_2009} and \cite%
{kechris_borel_2016}. We will always assume $E$ and $F$ to be analytic
equivalence relations on Polish spaces $X,Y$ respectively. The following
result is essentially \cite[Lemmas 3.7, 3.8]{kechris_borel_2016}, although a
slightly more restrictive notion of idealistic equivalence relation is
considered there; its proof readily adapts to our more generous notion.

\begin{proposition}[Kechris--Macdonald]
\label{Proposition:reduce-then-embeds}Suppose that $E$ is idealistic and $F$
is Borel. Then a Borel reduction from $E$ to $F$ is a classwise Borel
embedding.
\end{proposition}

\begin{proof}
Suppose that $[x]\mapsto \mathcal{F}_{[x]}$ and $\zeta :X\rightarrow X$ witness
that $E$ is idealistic. Suppose that $f:X/E \rightarrow Y/F $ is a Borel reduction from $E$ to $F$, and let $\hat{f}%
:X\rightarrow Y$ be a Borel lift of $f$. We will show that the $F$-saturation 
\begin{eqnarray*}
\lbrack \hat{f}( X) ]_{F} &=&\{y\in Y:\exists x\in X\:\,y\,F\hat{f}%
( x) \} \\
&=&\{y\in Y:\exists x\in X\:\,y\,F\,(\hat{f}\circ \zeta )( x) \}
\end{eqnarray*}%
of $\hat{f}( X) $ is a Borel subset $Z$ of $Y$, and that there exists a Borel lift $\hat{f}^{*}$ of $f^{-1}:Z/F\to X/E$.

Define the subset $P$ of $Y\times X$ by setting $( y,x) \in P$ if
and only if $y\,F\,(\hat{f}\circ \zeta )( x) $. Then $[\hat{f}(
X) ]_{F}$ is the projection of $P$ onto the first coordinate. \ By the
\textquotedblleft large sections\textquotedblright\ uniformization theorem
figuring in the proof of \cite[Lemma 3.7]{kechris_borel_2016}, it will suffice to
show that there is a map $y\mapsto \mathcal{G}_{y}$ that assigns to each $y\in \lbrack \hat{f}( X) ]_{F}$ a $\sigma $-filter $\mathcal{G}%
_{y}$ of subsets of $f^{-1}( [ y]) $ such that
for each Borel subset $R$ of $Y\times X$ there exist a $\boldsymbol{\Sigma }
_{1}^{1}$ set $S\subseteq Y$ and a $\boldsymbol{\Pi }_{1}^{1}$ set $%
T\subseteq Y$ such that, setting%
\begin{equation*}
R_{y}:=\left\{ x\in X:( y,x) \in R\right\} \text{,}
\end{equation*}%
one has 
\begin{equation*}
R_{y}\in \mathcal{G}_{y}\Leftrightarrow y\in S\Leftrightarrow y\in T
\end{equation*}%
for every $y\in \lbrack \hat{f}( X) ]_{F}$. The uniformization given by the aforementioned theorem will then be such an $\hat{f}^{*}$ as we had desired.

 To this end, define $\mathcal{G}_{y}=\mathcal{F}_{f^{-1}([y])}$ for $y\in \lbrack 
\hat{f}( X) ]_{F}$. To see that the assignment $y\mapsto 
\mathcal{F}_{f^{-1}([y])}$ satisfies the above requirement, suppose $R\subseteq Y\times X$ is Borel. Then for $y\in
\lbrack \hat{f}( X) ]_{F}$ we have that $R_{y}\in \mathcal{F}%
_{f^{-1}([y])}$ if and only if $[\mathcal{F}_{f^{-1}([y])}x^{\prime }\:\:( y,x^{\prime }) \in R]$, if and only if 
\begin{equation*}
\exists x\in X\:\big[(( f\circ \zeta ) ( x)\, F\,y)\wedge
( \mathcal{F}_{f^{-1}([y])}x^{\prime }\:\:( ( f\circ \zeta
) ( x) ,x^{\prime }) \in R)\big],
\end{equation*}%
if and only if 
\begin{equation*}
\forall x\in X\:\big[(( f\circ \zeta ) ( x)
\,F\,y\Rightarrow \mathcal{F}_{f^{-1}([y])}x^{\prime }\:\:( (
f\circ \zeta ) ( x) ,x^{\prime }) \in R)\big]\text{.}
\end{equation*}%
This concludes the proof.
\end{proof}

\begin{corollary}
\label{Corollary:reduce-then-embeds}Suppose that $E$ is idealistic and $F$
is Borel. Then a \emph{surjective} Borel reduction from $E$ to $F$ is a
classwise Borel isomorphism.
\end{corollary}

The following related result is \cite[Proposition 2.3]{motto_ros_complexity_2012}.

\begin{proposition}[Motto Ros]
\label{Proposition:Cantor-Bernstein}If there is a classwise Borel embedding
from $E$ to $F$ and there is a classwise Borel embedding from $F$ to $E$
then there is a classwise Borel isomorphism from $E$ to $F$.
\end{proposition}

Recall that a \emph{Borel selector} for a Borel equivalence relation $E$ on
a Polish space $X$ is a function $s:X\rightarrow X$ such that $s(x)\,E\,x$ and $%
x\,E\,y\Leftrightarrow s( x) =s( y) $ for every $x,y\in X$. The principal filters that selectors determine on the classes of $E$ witness that $E$ is idealistic.

\begin{lemma}
\label{Lemma:selector-idealistic}Let $E$ be a Borel equivalence
relation on a Polish space $X$. If $E$ admits a Borel selector, then $E$ is
idealistic.
\end{lemma}

\begin{proof}
Let $s$ be a Borel selector for $E$. Then $s$ determines a function $f:X/
E \rightarrow X$, $\left[ x\right] \mapsto s( x) $.
For each $E$-class $[x]$ let $ \mathcal{F}_{[x]}=\{S\subseteq [x]\mid f([x])\in S\}$. For every Borel subset $A$ of $%
X\times X$ and $x\in X$, then, [$\mathcal{F}_{\left[ x\right] }x^{\prime }\:\:
(x,x^{\prime }) \in A]$ if and only if $(x,s(x)) \in A$. Hence $A_{\mathcal{F}}=\left\{ x\in X:%
\mathcal{F}_{\left[ x\right] }x^{\prime }\:\:(x,x^{\prime }) \in
A\right\}$  is a Borel subset of $X$, as desired. 
\end{proof}

\subsection{Definable sets}\label{Section:dset}

A \emph{definable set} is a pair $\left( X,E\right) $ such that $X$\emph{\ }%
is a Polish space and $E$ is a \emph{Borel} and \emph{idealistic} equivalence relation on $X$. We will indulge below in the common abuse whereby the notation $X/E$ may stand either for this pair or for its associated quotient set, as context will indicate. A \emph{definable subset} of a definable
set $X/E$ is a definable set of the form $Y/E \subseteq X/E$ where $Y$ is an $E$-invariant Borel subset of $%
X $. We regard a Polish space $X$, in particular, as a definable set $%
X/E$, where $E$ is the relation of equality.

If $X/E$ and $Y/F$ are definable sets, then a 
\emph{definable function }$X/E \rightarrow Y/F$
is a Borel homomorphism from $E$ to $F$, i.e., a function $X/E
\rightarrow Y/F$ that lifts to a Borel function $%
X\rightarrow Y$. We let $\mathsf{DSet}$ be the category that has definable sets as objects and definable functions as morphisms. The identity morphism on $X/E$ is the identity function, and composition of morphisms is given by composition of functions.
We now observe that this category has similar properties to the paradigmatic category $\mathsf{Set}$ of sets.

In the framework of $\mathsf{DSet}$, for example, Corollary \ref%
{Corollary:reduce-then-embeds} takes the following form:

\begin{proposition}[Kechris--Macdonald]
\label{Proposition:bij-iso}Suppose that $X/E$ and $Y/F$ are definable sets, and $f:X/E \rightarrow Y/F$ is a definable function. The following assertions are equivalent:

\begin{enumerate}
\item $f$ is a bijection;

\item $f$ is an isomorphism in $\mathsf{DSet}$.
\end{enumerate}
\end{proposition}

The following also is a consequence of Proposition \ref{Proposition:reduce-then-embeds}.

\begin{proposition}[Kechris--Macdonald]
\label{Proposition:reduce-then-embeds2}Suppose that $X/E$ and $%
Y/F$ are definable sets, and $f:X/E\rightarrow
Y/F$ is an \emph{injective} definable function. Then there is
a definable subset $Z/F$ of $Y/F$ such
that $f:X/E \rightarrow Z/F $ is a
definable bijection.
\end{proposition}

Proposition \ref%
{Proposition:Cantor-Bernstein} above, similarly, may be rephrased as a
Cantor--Bernstein theorem for definable sets.

\begin{proposition}[Motto Ros]
\label{Proposition:Cantor-Bernstein2}Suppose that $X/E$ and $%
Y/F$ are definable sets. If there exist an injective definable
function $X/E \rightarrow Y/F $ and an injective
definable function $Y/F \rightarrow X/E$ then
there exists a bijective definable function $X/E \rightarrow
Y/F $.
\end{proposition}

The category $\mathsf{DSet}$ has a number of further pleasant and $\mathsf{Set}$-like properties. Observe for example that if the definable sets $X/E$ and $Y/F$ are countable sets then there is a \emph{definable} bijection from $%
X/E$ to $Y/F$ if and only if there is a
bijection from $X/E$ to $Y/F$.

Recall that a morphism $f:x\rightarrow y$ in a category $\mathcal{C}$ is:

\begin{itemize}
\item \emph{monic }if for any two morphisms $g_{0},g_{1}:z\rightarrow x$,
if $fg_{0}=fg_{1}$ then $g_{0}=g_{1}$, and

\item \emph{epic }if for any two morphisms $g_{0},g_{1}:y\rightarrow z$, if 
$g_{0}f=g_{1}f$ then $g_{0}=g_{1}$.
\end{itemize}

It is easy to see that a definable function is monic in $\mathsf{DSet}$
if and only if it is injective, and epic in $\mathsf{DSet}$ if and only if it is surjective. The initial object is $\mathsf{DSet}$ is the empty
definable set, and a final object in $\mathsf{DSet}$ is a definable set with
only one element.

If $X/E$ and $Y/F $ are definable sets, then the 
\emph{coproduct} $X/E\amalg Y/F $ is the
definable set obtained as the quotient of the disjoint union $X\amalg Y$ of $%
X$ and $Y$ by the equivalence relation on $X\amalg Y$ that restricts to $E$
on $X$ and to $F$ on $Y$. Coproducts of
sequences $\left( X_{n}/ E_{n}\right) _{n\in \omega }$ of
definable sets are defined similarly.

The \emph{product }of $X/E$ and $Y/F$ is the
definable set obtained as the quotient of the product $X\times Y$ by the
equivalence relation $E\times F$ defined as in Lemma \ref{Lemma:product}.
In our conclusion, we record the question of whether the category $\mathsf{DSet}$ possesses arbitrary countable products.
Given two definable functions $f,g:X/ E \rightarrow Y/
F$ the \emph{equalizer }of $f$ and $g$ is the definable subset of $X/E$ given by 
\begin{equation*}
\left\{ x\in X/ E :f( x) =g( x) \right\} 
\text{.}
\end{equation*}%
The existence of products and equalizers ensures the existence of \emph{pullbacks }of
definable functions $f:X/E\to Y/F$ and $%
f^{\prime }:X^{\prime }/E^{\prime } \rightarrow Y/F$. These are the definable subsets of $X/E \times
X^{\prime }/E^{\prime }$ consisting of pairs $(
x,y) $ such that $f( x) =f^{\prime }( x^{\prime
}) $.

\subsection{Pointed definable sets}

We will also consider \emph{pointed definable sets}, which are just
definable sets $X/E$ with a distinguished point (the \emph{%
basepoint}). A basepoint-preserving definable function between pointed
definable sets is simply a function that maps the basepoint to the
basepoint. This defines a category $\mathsf{DSet}_{\ast }$, which may be
seen as the under category $\left\{ \ast \right\} \downarrow \mathsf{DSet}$
of $\mathsf{DSet}$-morphisms from $\left\{ \ast \right\} $, where $\left\{
\ast \right\} $ is a definable set with just one point.
In $\mathsf{DSet}_{\ast }$ the set $\left\{ \ast \right\} $ is a null object, i.e., it is
both initial and terminal. For any two pointed definable sets there exists a
unique basepoint-preserving constant function or \emph{zero arrow}. Thus we may
define the \emph{cokernel} of a basepoint-preserving \emph{injective} definable
function $f:X/E \rightarrow Y/F$. This is the
definable set obtained as follows. We have that $f(X/E) =Z/F$ for some definable subset $Z/F$ of $Y/F$. Define the equivalence relation $F_{Z}$
on $Y$ in terms of $Z$ as in Lemma \ref{Lemma:cokernel}. The cokernel of $f$ is then $Y/F_{Z}$. Intuitively, $%
Y/F_{Z}$ is the definable set obtained from $Y/F$
by identifying the definable subset $Z/F$ to a single
basepoint.

\subsection{Definable monoids and definable groups\label{Section:dgroup}}

One may interpret \emph{definable monoids} and \emph{definable groups} as monoids and
groups in $\mathsf{DSet}$, as a particular instance of \emph{monoids and groups in
a category} in the sense of \cite[Section III.6]{mac_lane_categories_1998}.
Explicitly, a \emph{definable monoid} is a pointed definable set $X/E$
together with a definable function $\mu :X/E\times X/E \rightarrow X/E$ that is an associative binary
operation with the basepoint as identity element. A \emph{%
definable group }is a definable monoid $%
\left( X/E,\mu \right) $ that is a also a group, and such that
the function $X/E \rightarrow X/E$ mapping each
element to its inverse is a definable function. Evidently, definable monoids
and groups form subcategories of the category of pointed definable sets. A 
\emph{definable group homomorphism }between definable groups is a definable
function that is also a group homomorphism. A \emph{definable subgroup }of $%
X/E$ is a simply a definable subset $Y/E$ that
is also a subgroup. Notice that the kernel of a definable group homomorphism
is a definable subgroup. A definable exact sequence of definable groups is a
sequence 
\begin{equation*}
X^{\prime }/ E^{\prime }\overset{f}{\rightarrow }X/E \overset{g}{\rightarrow }X^{\prime \prime }/ E^{\prime \prime
}
\end{equation*}%
of definable groups and definable group homomorphisms such that $f$ is
injective, $g$ is surjective, and the kernel of $g$ is the image of $f$.
Notice that the category of groups with a Polish cover $G/N$ is a full subcategory of the category of definable groups. We say that a definable group $X/E$ is \emph{essentially a group with a Polish cover} if it definably isomorphic to a group with a Polish cover.

\subsection{Semidefinable sets}\label{Section:semidset}

In the following, we will also consider analytic equivalence relations on
Polish spaces that we do not know or have not yet shown to be Borel or
idealistic. We will call a quotient $X/E$ of a Polish space by
an analytic equivalence relation a \emph{semidefinable set}. Such pairs form the objects of a category $\mathsf{SemiDSet}$ of semidefinable sets in which morphisms are, as above, definable functions $f:X/E\to Y/F$, i.e. functions $f$ which lift to Borel functions $X\to Y$. Notice that an
isomorphism in this category is a \emph{definable bijection whose inverse is
also definable}. By definition, the category of definable sets is a full
subcategory of the category of semidefinable sets. The notions of
semidefinable subset, semidefinable group, and semidefinable subgroup may be defined just as in the definable case.
The following is immediate from Lemma \ref{Lemma:invariance}.

\begin{lemma}
If a semidefinable set $X/E$ is isomorphic in $\mathsf{SemiDSet%
}$ to a definable set $Y/F$, then $X/E$ is a
definable set.
\end{lemma}

As in Section \ref{Section:er}, the existence of a Borel selector may entail that a semidefinable object is, in fact, in $\mathsf{DSet}$.

\begin{lemma}
\label{Lemma:quotient}Suppose that $X/E$ is a semidefinable
group and $Y/E$ is a semidefinable subgroup. Define the
equivalence relation $F$ on $X$ by setting $x\,F\,x^{\prime }$ if and only if $%
\left[ x\right] _{E}\cdot \left[ x^{\prime }\right] _{E}^{-1}\in Y/E$. Suppose that $F$ admits a Borel selector. Then:

\begin{enumerate}
\item if $E|_{Y}$ is Borel, then $E$ is Borel.

\item if $E|_{Y}$ is idealistic, then $E$ is idealistic.
\end{enumerate}
\end{lemma}

\begin{proof}
We let $X\times X\rightarrow X$, $\left( x,y\right) \mapsto x\cdot y$ be a
Borel lift of the group operation on $X/E$. Let $s$ be a
Borel selector for $F$.
Notice that the relation $E$ is finer than $F$, hence the Borel selector $s$ for $F$ induces a definable function $%
\hat{s}:X/E\rightarrow X$, $\left[ x\right] \mapsto s( x) 
$.
Observe that if $T=\left\{ s(
x) :x\in X\right\} $ then $T\subseteq X$ is a Borel \emph{transversal 
}for $F$, i.e.\ $T$ meets each $F$-equivalence class in exactly one point.
Notice that the map%
\begin{eqnarray*}
Y/ E\times T &\rightarrow &X/ E\\
\left( \left[ y\right] _{E},t\right) &\mapsto &[y\cdot t]_{E}
\end{eqnarray*}%
is a definable bijection with definable inverse%
\begin{eqnarray*}
X/E &\rightarrow &Y/E \times T \\
\left[ x\right] _{E} &\mapsto &(\left[ x\right] _{E}\cdot \left[ s(
x) \right] _{E}^{-1},s(x) )\text{.}
\end{eqnarray*}%
This shows that the semidefinable set $X/ E$ is isomorphic in $%
\mathsf{SemiDSet}$ to $Y/E \times T$. If $Y/E$ is
Borel or idealistic, respectively, then $Y/E\times T$ is Borel
or idealistic, respectively; see Lemma \ref{Lemma:product}. The conclusion
thus follows from Lemma \ref{Lemma:invariance}.
\end{proof}

\section{Homotopy is idealistic}\label{S:Hidealistic}

Let $X$ and $Y$ be locally compact metrizable spaces and let $[X,Y]$ be the set of all homotopy classes of maps from $X$ to $Y$. Since the space $\mathrm{Map}(X,Y)$ of all maps from $X$ to $Y$ is Polish when endowed with the compact-open topology and the homotopy relation between elements of $\mathrm{Map}(X,Y)$ is clearly analytic, $[X,Y]$ is naturally viewed as an object in the category $\mathsf{SemiDSet}$ of semidefinable sets. By a theorem of Becker's, any analytic equivalence relation on the Cantor space may be realized as $[*,Y]$ for some compact $Y\subseteq\mathbb{R}^3$ (see 
\cite[Theorem 4.1]{beckerPath}); hence $[X,Y]$ is not in general a definable set. In contrast, for any locally compact Polish space $X$, if $Y=P$ is a \emph{polyhedral $H$-group} then $[X,Y]$ \emph{is} a definable set; this is the result appearing as Theorem \ref{T2:intro} in our introduction.
A main ingredient in this theorem is the more general fact that for any locally compact Polish space $X$ and countable polyhedron $Y$, the homotopy relation on $\mathrm{Map}(X,Y)$ is idealistic.
This we show in Section \ref{SS:homISidealistic} below.
This and several subsequent arguments make use of Borsuk's Homotopy Extension Theorem; we record the classical version in the present section, proving and applying its definable version in Section \ref{S:definable_HET} below.
\subsection{Homotopy}\label{SS:homotopy}

Let $\mathsf{LC}$ denote the category of locally compact Polish spaces and
continuous functions; we will often term the latter \emph{maps}, simply. For any two locally compact spaces $X$ and $Y$ we endow the hom-set $\mathsf{LC}( X,Y)$ with the compact-open
topology. This is the topology with subbasis the collection of all sets of
the form $\{f\in \mathsf{LC}(X,Y):f(K)
\subseteq U\} $ for some compact subset $K\subseteq X$ and open
subset $U\subseteq Y$. With this topology each such $\mathsf{LC}( X,Y) $ is a
Polish space, and what is more, the composition functions $\mathsf{LC}(
Y,Z) \times \mathsf{LC}( X,Y) \rightarrow \mathsf{LC}(X,Z)$ are all continuous. In this way we regard $\mathsf{LC}$ as a \emph{Polish category}, by which we mean a category enriched over the category of Polish spaces \cite{kelly_enriched_2005}. 
We denote by $\mathsf{C}$ the full subcategory of $\mathsf{LC}$
consisting of compact spaces.

Let $\mathcal{C}$ be a category with small (i.e., set-sized) hom-sets. A \emph{congruence} $\equiv$ on $\mathcal{C}$ \cite[Section I.8]{mac_lane_categories_1998} is given by an assignment to each
pair of objects $x,y$ of $\mathcal{C}$ of an equivalence relation $\equiv
_{x,y}$ on $\mathcal{C}( x,y) $ such that for objects $x,y,z$ and morphisms $f_{1},f_{2}:x\rightarrow y$ and $%
g_{1},g_{2}:y\rightarrow z$ in 
$\mathcal{C}$, if $f_{1}\equiv _{x,y}f_{2}$ and $g_{1}\equiv
_{y,z}g_{2}$ then $g_{1}f_{1} \equiv_{x,z}
g_{2}f_{2} $. One may regard the pair $\left( \mathcal{C},\equiv
\right) $ as a (strict) $2$-category by declaring that for objects $%
x,y$ and morphisms $f,g:x\rightarrow y$ in $\mathcal{C}$ there exists a
unique $2$-cell $f\Rightarrow g$ if and only if $f\equiv _{x,y}g$.  The corresponding quotient category $\mathcal{C}/\!\equiv$ is the category whose objects are those of $\mathcal{C}
$ and whose morphisms from $x$
to $y$ are the quotient set $\mathcal{C}(x,y)/\!\equiv
_{x,y}$.

Let $X$ and $Y$ be topological spaces; henceforth, we will let $I$ denote the closed unit interval $[0,1]$. A \emph{homotopy} $h:f\Rightarrow
g:X\rightarrow Y$ from the map $f:X\to Y$ to the map $g:X\to Y$ is a
function $h:X\times I\rightarrow Y$ such that $h\left( -,0\right) =f$ and $%
h\left( -,1\right) =g$. The homotopy relation $\simeq_{\mathsf{LC}}$ on $\mathsf{LC}(X,Y)$ is defined by setting $%
f\simeq_{\mathsf{LC}}g$ if and only if there exists a homotopy $%
h:f\Rightarrow g:X\rightarrow Y$; evidently, $\simeq_{\mathsf{LC}}$ defines a congruence on $\mathsf{LC}$. We let $\mathsf{Ho}( \mathsf{LC}) $ denote the
corresponding quotient category, which may be regarded as a category
enriched over the category $\mathsf{SemiDSet}$; $[X,Y] $ will denote the semidefinable set of homotopy classes of
maps $X\rightarrow Y$ or, equivalently, of the path-components of $\mathsf{LC}(X,Y)$.
A \emph{homotopy equivalence} $f\in\mathsf{LC}(X,Y)$ is a map whose image in the quotient category $\mathsf{Ho}(\mathsf{LC})$ is an isomorphism.

More generally, we consider \emph{locally
compact pairs}; these are pairs of topological spaces $(X,A)$ in which $X$ is locally
compact and Polish and $A$ is a closed subspace of $X$. These form the objects of the Polish
category $\mathsf{LCP}$ of locally compact pairs. Its morphisms $(X,A)\to (Y,B)$ are those maps $f:X\rightarrow Y$
such that $f(A) \subseteq B$. The hom-sets $\mathsf{LCP}((X,A),(Y,B)) $ then form $G_{\delta}$ subsets of $\mathsf{LC}(X,Y) $, and their subspace topologies are consequently Polish. Note also that the identification $X\mapsto (X,\varnothing)$ embeds $\mathsf{LC}$ as a full subcategory of $\mathsf{LCP}$.
We write $f:(X,A) \rightarrow(Y,B) $ to indicate that $f\in \mathsf{LCP}((
X,A),(Y,B))$. We let $\mathsf{CP}$ denote the full subcategory of $\mathsf{LCP}$ consisting of \emph{compact pairs}, i.e., of locally compact pairs $(X,A)$ in which $X$ is compact. Given maps $f,g:( X,A) \rightarrow (Y,B) $ a
homotopy --- sometimes termed \emph{pair homotopy} --- $h:f\Rightarrow g:( X,A) \rightarrow (Y,B)$ from $f$ to $g$
is a homotopy $h:f\Rightarrow
g:X\rightarrow Y$ such that $h( -,t) :(
X,A) \rightarrow (Y,B) $ for every $t\in I$. This defines
a congruence relation $\simeq _{\mathsf{LCP}}$ on $\mathsf{LCP}$. We
let \textsf{Ho}$(\mathsf{LCP}) $ denote the corresponding quotient
category; note that \textsf{Ho}$(\mathsf{LC})$ and \textsf{Ho}$(\mathsf{LCP})$ are the homotopy categories associated to the Hurewicz model category structures on $\mathsf{LC}$ and $\mathsf{LCP}$, respectively (see \cite{dwyer_spalinski_homotopy_1995}). We let $\left[(X,A),(Y,B) \right] $ be the semidefinable set of homotopy classes of maps $(X,A) \rightarrow (Y,B) $. Just as above, a \emph{homotopy equivalence} $f\in\mathsf{LCP}((X,A),(Y,B))$ is a map whose image in the quotient category $\mathsf{Ho}(\mathsf{LCP})$ is an isomorphism.

A \emph{pointed locally compact Polish space }is a locally compact pair $(X,A) $ in which $A$ is a singleton $\{*\}$; we call $*$ the \emph{basepoint} of this space. Such spaces comprise a full subcategory $\mathsf{LC}_{\ast }$ of the category $\mathsf{LCP}$ of locally compact pairs. The morphisms in $\mathsf{LC}_{\ast }$ are maps which are \emph{basepoint-preserving}, i.e., they are exactly those maps which send basepoints to basepoints. Similarly one defines
the notion of basepoint-preserving homotopy, and the corresponding
congruence relation $\simeq _{\mathsf{LC}_{\ast }}$ on $\mathsf{LC}%
_{\ast}$ (which is simply the restriction of $\simeq_{\mathsf{LCP}}$
to the full subcategory $\mathsf{LC}_{\ast}$).

A central concern in homotopy theory is the existence (or nonexistence) of extensions of maps or homotopies from a topological space $A$ to some larger space $X\supseteq A$. In the locally compact setting, a main theorem describing conditions of their existence is Borsuk's Homotopy Extension Theorem (see \cite[Chapter I.3, Theorem 9]{mardesic_shape_1982}):
\begin{theorem}
\label{Theorem:definable-extension}Suppose that $A$ is a closed subspace of a locally compact
Polish space $X$ and that $P$ is a polyhedron.
Then for every map $f:A\times I\rightarrow P$ and $h:X\times \left\{
0\right\} \rightarrow P$ such that $h|_{A\times \left\{ 0\right\}
}=f|_{A\times \left\{ 0\right\} }$, there exists a map $\tilde{h}:X\times
I\rightarrow P$ which simultaneously extends both $h$ and $f$.
\end{theorem}
Beginning in Section \ref{SS:homISidealistic}, we will make repeated use of this theorem; as noted, we will also prove its definable version in Section \ref{SS:defhomotopyextension} below.
\subsection{Polyhedra}\label{SS:polyhedra}
Henceforth, all simplicial complexes will be tacitly understood to be countable and locally finite; let $K$ be such a complex. We may assume
without loss of generality that $\mathrm{dom}(K)
\subseteq \mathbb{N}$ and then associate with $K$ a locally compact Polish
space $\left\vert K\right\vert \subseteq \mathbb{R}^{\mathbb{N}}$, called its 
\emph{topological realization}, as follows. Let $(
e_{v})_{v\in \mathbb{N}}$ denote the canonical basis of the $\mathbb{R}$-vector space $\mathbb{R}^{\mathbb{N}}$. For each simplex $\sigma =\{ i_{0},\ldots ,i_{n}\} $ of $K$ define
\begin{align}\label{equation:realization}
\left\vert \sigma \right\vert =\left\{ t_{0}e_{i_{0}}+\cdots
+t_{n}e_{i_{n}}:t_{0},\ldots ,t_{n}\in \left[ 0,1\right] ,t_{0}+\cdots
+t_{n}=1\right\} \text{.}
\end{align}
The \emph{topological boundary} $\left\vert \dot{\sigma}%
\right\vert $ of $\sigma $ is the set of points $t_{0}e_{i_{0}}+\cdots
+t_{n}e_{i_{n}}$ of $\left\vert \sigma \right\vert $ for which at least one
of the coordinates $t_{0},\ldots ,t_{n}$ is zero. The  \emph{topological interior} of $%
\left\vert \sigma \right\vert $ is $\left\langle \sigma \right\rangle
:=\left\vert \sigma \right\vert \setminus \left\vert \dot{\sigma}\right\vert 
$. We then define $\left\vert K\right\vert $ to be the union of $\left\vert
\sigma \right\vert $ where $\sigma$ ranges among the simplices of $K$. 
Hence, every $x\in |K|$ can be uniquely identified with some sum 
\[x=\sum_{v\in \mathrm{dom}(K)}x_v e_v\]
where the coefficients $(x_v)_{v\in \mathrm{dom}(K)}$ satisfy $\sum_{v\in \mathrm{dom}(K)}x_v=1$. We call $(x_v)$ the \emph{barycentric coordinates} of $x$.

Observe that if $
L $ is a subcomplex of $K$ then $\left\vert L\right\vert $ forms a
closed subspace of $\left\vert K\right\vert $. A
topological space is called a \emph{polyhedron }if it is homeomorphic to the
topological realization of a simplicial complex. We let $\mathsf{P}$ be the full subcategory of the category $\mathsf{LC}$ of locally
compact Polish spaces consisting of polyhedra and let $\mathsf{Ho}(\mathsf{P})$ be its quotient by $\simeq_{\mathsf{LC}}$. Note that a polyhedron is compact
if and only if it is homeomorphic to the realization of a finite simplicial complex.

Call two simplicial maps $f,g:K\rightarrow L$ \emph{contiguous} if for every $\sigma \in K$, $f(\sigma)$ and $g(\sigma)$ are faces of a single simplex in $L$, and let the
relation of \emph{contiguous equivalence }be the transitive closure of the
contiguity relation (see \cite[Section VI.3]{eilenberg_foundations_1952} or \cite[Section 3.5]{spanier_algebraic_1995}).
This defines a congruence relation $\simeq_{\mathsf{S}}$ on $\mathsf{S}$, and an associated quotient category \textsf{Ho}$(\mathsf{S}):=\mathsf{S}/\!\simeq _{\mathsf{S}}$.
A simplicial map $f:K\rightarrow L$ between simplicial complexes induces a continuous function $\left\vert f\right\vert :\left\vert
K\right\vert \rightarrow \left\vert L\right\vert $ defined by setting%
\begin{equation*}
\left\vert f\right\vert \left( t_{0}e_{i_{0}}+\cdots +t_{n}e_{i_{n}}\right)
=t_{0}e_{f\left( i_{0}\right) }+\cdots +t_{n}e_{f\left( i_{n}\right) }
\end{equation*}%
for all $\{ i_{0},\ldots ,i_{n}\} \in K$ and $t_{0},\ldots
,t_{n}\in \left[ 0,1\right] $ such that $t_{0}+\cdots +t_{n}=1$.
Moreover, if $f,g:K\rightarrow L$ are contiguously equivalent then the
corresponding maps $\left\vert f\right\vert ,\left\vert g\right\vert
:\left\vert K\right\vert \rightarrow \left\vert L\right\vert $ are
homotopic (a fact invoked in the proof of Theorem \ref{T:Cech=Def}); put differently, the functor $K\mapsto \left\vert
K\right\vert $ from $\mathsf{S}$ to $\mathsf{P}$ induces a functor
from \textsf{Ho}$(\mathsf{S})$ to \textsf{Ho}$(\mathsf{P})$.

The \emph{barycentric subdivision }$\beta K$ of a simplicial complex $K$ is the simplicial
complex with $\mathrm{dom}( \beta K) $ equal to the set of nonempty simplices of $K$. A simplex of $\beta K$ is a set $\{ \sigma
_{0},\ldots ,\sigma _{n}\} $ of simplices of $K$ which is linearly ordered by inclusion.
A \emph{selection map for $K$} is a function $s:K\setminus\{ \varnothing
\} \rightarrow \mathrm{dom}( K) $ such that $s(\sigma ) \in \sigma $ for every $\sigma \in K\setminus\{
\varnothing \}$; equivalently, $s$ is a simplicial map $\beta K\rightarrow K$. Notice that any two selection maps are contiguous and therefore represent the same morphism in $\mathsf{Ho}(\mathsf{S})$.

Precomposing the realization construction (\ref{equation:realization}) above with a bijection $b:\mathbb{N}^{<\mathbb{N}}\to\mathbb{N}$ determines a topological realization of a subdivided complex $\beta K$, and if $s:\beta K\rightarrow K$ is a selection map then the
corresponding map $\left\vert s\right\vert :\left\vert \beta K\right\vert
\rightarrow \left\vert K\right\vert $ is homotopic to the homeomorphism $%
\left\vert \beta K\right\vert \rightarrow \left\vert K\right\vert $ defined
by the maps%
\begin{equation*}
e_{b(\{ v_{0},\ldots ,v_{n}\}) }\mapsto \frac{1}{n+1}\left(
e_{v_{0}}+\cdots +e_{v_{n}}\right)
\end{equation*}%
for each nonempty simplex $\left\{ v_{0},\ldots ,v_{n}\right\} $ of $K$, together with their linear extension.

For a vertex $v$ of $K$, we let the \emph{open star} \textrm{St}$_{K}(
v) \subseteq \left\vert K\right\vert $ of $v$ be the union of the
interiors of $\left\vert \sigma \right\vert $ where $\sigma $ ranges among
the simplices that contain $v$ as a vertex. Equivalently,%
\begin{equation*}
\mathrm{St}_K( v) =\left\{ \sum_{w\in \mathrm{\mathrm{dom}}\left(
K\right) }a_{w}e_{w}\;:\;a_{v}>0\right\}\cap\left\vert K\right\vert  \text{.}
\end{equation*}
\noindent We will sometimes omit the subscript $K$ when it is contextually clear. In what follows we will often make use of the following important well-known property of the  open cover $\{\mathrm{St}_K(v)\colon v\in\mathrm{dom}(K)\}$ of $|K|$.

\begin{lemma}\label{L:OpenStarHomotopy}
Let $K$ be a simplicial complex and let $D\subseteq |K|\times |K|$ be the union 
of all sets of the form  $\mathrm{St}_K( v)\times\mathrm{St}_K( v)$ where $v$ ranges over $\mathrm{dom}(K)$. Then there exists a continuous map $\lambda\colon D\times [0,1]\to P$  so that:
\begin{enumerate}
\item  for all $(x,y)\in  D$ we have that  $\lambda(x,y,0)=x, \quad  \lambda(x,y,1)=y$;
\item if $x\in |\sigma|$ and $y\in |\tau|$ with $\sigma,\tau\in K$, then for all $t\in[0,1]$ we have that $\lambda(x,y,t)\in |\sigma|\cup|\tau|$.
\end{enumerate}
\end{lemma}
\begin{proof}
Following the proof of Theorem 2 from  \cite{milnor_spaces_1959} we first consider the map $\mu\colon D\to |K|$ which is defined as follows: if $(x_v)$ and $(y_v)$ are the barycentric coordinates of $x,y$ where $(x,y)\in D$ then   $\mu(x,y):=(z_v\colon v\in \mathrm{dom}(K))$, where 
\[z_w:= \min(x_w,y_w) \; \;  \bigg/\sum_{v\in\mathrm{dom}(K)} \min(x_v,y_v)  \]
Notice that if $x\in |\sigma|$ and $y\in |\tau|$ with $\sigma,\tau\in K$, then $\mu(x,y)\in |\sigma|\cap |\tau|$. Then the map $\lambda$ is defined by setting:
\[\lambda(x,y,\frac{1}{2} t):=(1-t)x+t\mu(x,y) \quad\text{ and }\quad \lambda(x,y,\frac{1}{2}+\frac{1}{2} t):=(1-t)\mu(x,y)+ty \]
\end{proof}

The generalization to topological realizations $(|K|,|L|) $ of simplicial pairs $(K,L) $ is straightforward. A \emph{polyhedral pair} is a locally
compact pair that arises in this fashion from a simplicial pair $(K,L)$. The category of polyhedral
pairs is a full subcategory of $\mathsf{LCP}$; we denote it $\mathsf{PP}$, writing $\mathsf{P}_*$ for the full subcategory of pointed polyhedra.

\subsection{Maps from compact pairs to polyhedral pairs}\label{SS:compact pairs}
It will be useful to henceforth adopt slightly strengthened notions of \emph{cover} and \emph{refinement}; portions of the next few paragraphs therefore amount to an updating of definitions first appearing in Section \ref{S:Definable cohomo}.

Let $X$ be a locally compact Polish space. An \emph{(open)} \emph{cover} of $X$ is
a countable family $\mathcal{U}=\left( U_{j}\right) _{j\in J}$ of open
subsets of $X$ with compact closure such that $X$ is the union of $\left\{
U_{j}:j\in J\right\} $. The cover $\mathcal{U}$ is \emph{star-finite} if,
for every $j\in J$, the set $\left\{ i\in J:U_{i}\cap U_{j}\neq \varnothing
\right\} $ is finite. The \emph{nerve} of a star-finite cover $\mathcal{U}%
=\left( U_{j}\right) _{j\in J}$ of $X$ is the (countable, locally finite)
simplicial complex with $\mathrm{dom}(N_{\mathcal{U}})=\left\{ j\in
J:U_{j}\neq \varnothing \right\} $ and $\left\{ j_{0},\ldots ,j_{n}\right\} $
a simplex of $N_{\mathcal{U}}$ if and only if $U_{j_{0}}\cap \cdots \cap
U_{j_{n}}\neq \varnothing $. A \emph{canonical map for $\mathcal{U}$} is a
function $f:X\rightarrow \left\vert N_{\mathcal{U}}\right\vert $ such that $%
f^{-1}\left( \mathrm{St}_{N_{\mathcal{U}}}( j) \right) \subseteq
U_{j}$, where $\mathrm{St}_{N_{\mathcal{U}}}(j) $ is the open star of the vertex $j$ of $N_{\mathcal{U}}$.

A \emph{partition of unity subordinate to a star-finite cover $%
\mathcal{U}=\left( U_{j}\right) _{j\in J}$} is a family $\left( f_{j}\right)
_{j\in J}$ of continuous functions $f_{j}:X\rightarrow I$ such that
the closure of $\mathrm{supp}( f_{j}) $ is contained in $U_{j}$ for every $j\in J$, and $\sum_{j\in J}f_{j}( x) =1$ for
every $x\in X$. A partition of unity 
$( f_{j}) _{j\in J}$ of $\mathcal{U}$ gives rise to a canonical
map $f:X\rightarrow \left\vert N_{\mathcal{U}}\right\vert $ for $\mathcal{U}$
defined by $x\mapsto \sum_{j\in J}f_{j}( x) e_{j}$.

A cover $\mathcal{V}$ of $X$ \emph{refines} a cover $\mathcal{U}$
of $X$ if for every $V\in \mathcal{V}$ there exists $U\in \mathcal{U}$ such
that $\overline{V}\subseteq U$. A \emph{refinement map} from $\mathcal{V}$ to $%
\mathcal{U}$ is a simplicial map $p:N_{\mathcal{V}}\rightarrow N_{\mathcal{U}%
}$ such that $\overline{V}_{j}\subseteq U_{p(j) }$ for every $%
j\in \mathrm{dom}( N_{\mathcal{U}}) $. Note that any two
refinement maps from $\mathcal{V}$ to $\mathcal{U}$ are contiguous.

Every cover of a locally compact Polish space admits a star-finite refinement \cite{kaplan_homology_1947}, and every star-finite cover of such a space admits a subordinate partition of unity. If $( X,A) $ is a locally compact pair, then a cover of $%
( X,A) $ is a cover $\mathcal{U}=( U_{j})
_{j\in J}$ of $X$ such that if $U_{j_{\ell
}}\cap A\neq \varnothing $ for $\ell \in \left\{ 0,1,\ldots
,n\right\} $ and $U_{j_{0}}\cap \cdots \cap U_{j_{n}}\neq \varnothing $,
then $U_{j_{0}}\cap \cdots \cap U_{j_{n}}\cap A\neq \varnothing $.
This condition ensures that $\mathcal{U}^{\prime }=( U_{j}\cap A) _{j\in J}$
is a cover of $A$ such that the identity map $J\rightarrow J$
induces an inclusion $N_{\mathcal{U}^{\prime }}\rightarrow N_{\mathcal{U}}$
as a full subcomplex.

If $( P,Q) $ is a polyhedral pair and $(
K,L) $ is a simplicial pair and $(P,Q) =(|K|,|L|)$ then there exists a canonical star-finite cover of $(P,Q)$, namely  $\mathcal{U}%
_{K}^{P}:=\{\mathrm{St}_{K}(v)\mid v\in \mathrm{dom}(K)\}$, where as above $\mathrm{St}_{K}(v)$ is the open
star of $v$ in $K$. Notice that $\mathrm{St}_{K}(v)\cap |L| \neq \varnothing$ if and only if $v\in L$.

Recall that an equivalence relation $E$ on a Polish space $Y$ is \emph{open} (or \emph{closed}) if it is an open (or closed) subset of $Y\times Y$ endowed with the product topology.
Recall also that by a homotopy of maps in $\mathsf{LCP}$ we mean a \emph{pair} homotopy.

\begin{lemma}\label{lemma:homotopyequivalenceinduceshomotopyequivalence}
For any $(Z,C)$ and homotopy equivalence $f:(X,A)\to (Y,B)$ in $\mathsf{LCP}$,
\begin{enumerate}
\item the map $f^*:\mathsf{LCP}((Y,B),(Z,C))\to \mathsf{LCP}((X,A),(Z,C))$ given by $s\mapsto s\circ f$ is a homotopy equivalence, and
\item the homotopy relation on $\mathsf{LCP}((Y,B),(Z,C))$ is open if and only if it is open on $\mathsf{LCP}((X,A),(Z,C))$.
\end{enumerate}
\end{lemma}
\begin{proof} Item (1) is not particularly difficult to see, and appears as Corollary 2.4.14 of \cite{tom_dieck_algebraic_2008}. 
Note in particular that the map $f^*$ is continuous, and that $s\simeq t$ if and only if $f^*(s)\simeq f^*(t)$.
Hence if the homotopy relation on $\mathsf{LCP}((X,A),(Z,C))$ is open, then so too is its continuous preimage under the map $f^*$, which is exactly the homotopy relation on $\mathsf{LCP}((Y,B),(Z,C))$. Repeating this argument for a homotopy equivalence $g:(Y,B)\to (X,A)$ in $\mathsf{LCP}$ completes the proof of item (2).
\end{proof}

\begin{lemma}
\label{Lemma:countable-homotopy}Suppose that a locally compact pair $(X,A)$ is homotopy equivalent to a compact
pair and that $(P,Q)$ is a polyhedral pair. Then the
relation of homotopy among maps $(X,A) \rightarrow
(P,Q)$ is open (and, in consequence, closed as well).
\end{lemma}
\begin{proof}
By Lemma \ref{lemma:homotopyequivalenceinduceshomotopyequivalence}, we may without loss of generality assume that $(X,A)$
is a compact pair.
Let $(K,L)$ be a locally finite simplicial pair
such that $(P,Q)=(|K|,|L|)$ and 
 let $\mathcal{U}=\{\mathrm{St}_K(v)\colon v\in \mathrm{dom}(K)\}$ denote the canonical star-finite cover of $(P,Q)$. Let $\mathcal{V}$ be any finite cover of $(X,A)$ which refines $\{f^{-1}(U)\colon U\in \mathcal{U}\}$; 
 there must then exist a function $p:\mathcal{V} \rightarrow \mathrm{dom}(K) $ such that $f(\overline{V}) \subseteq \mathrm{St}_K(p(V))$ for every $V\in \mathcal{V}$. 
 Consider the open neighborhood
\begin{equation*}
N(f) =\{g\in\mathsf{LC}((X,A)
,( P,Q)) :\forall \; V \in \mathcal{V} \; g(\overline{V})\subseteq\mathrm{St}(p(V))\}
\end{equation*}
of $f$ in $\mathsf{LC}((X,A),(P,Q))$. We claim that  every element $g$ of $N(f) $ is homotopic to $f$ via a homotopy of pairs.
 
Indeed, consider the open neighborhood   $D:=\bigcup_{U\in\mathcal{U}} U\times U$ of the diagonal of $P\times P$ and let  $\lambda\colon D\times [0,1]\to P$ be the map given by Lemma \ref{L:OpenStarHomotopy}.
Define $F\colon f \implies g$  by setting $F(x,t):=\lambda(f(x),g(x),t)$. By Lemma \ref{L:OpenStarHomotopy}(1),  $F$ is indeed a homotopy from $f$ to $g$.  Since $f(A)\subseteq Q$, this implies that $p(V)\in \mathrm{dom}(L)$ for all $V\in \mathcal{V}$ with $V\cap A\neq \emptyset$. From this observation and Lemma \ref{L:OpenStarHomotopy}(2), we have that $F$ is a homotopy of pairs.
\end{proof}

\begin{corollary}
\label{Corollary:countable-homotopy}Suppose that $(X,A)$ is a compact pair and $(P,Q)$ is a
polyhedral pair. Then $[(X,A),(P,Q)]$ is countable.
\end{corollary}

Letting $A$ and $Q$ denote basepoints of $X=S^n$ and an arbitrary polyhedron $P$, respectively, we recover the following well-known fact:

\begin{corollary}
For any $n\geq 0$ and countable, locally finite polyhedron $P$, the set $\pi_n(P)$ is countable.
\end{corollary}

\begin{remark}
\label{remark:homotopy_groups}
As indicated, the homotopy bracket $[-,-]$ will play an increasingly prominent role in the remainder of our paper, which might even be primarily regarded as a laying of foundations for its descriptive theoretic study.
In part for reasons noted in this section's introduction, within such a program of study, some care is in order concerning the range of spaces permitted to appear in the target position.
Allowing, accordingly, spaces to vary most freely in the \emph{source} position meshes well with our present focus on cohomology, as will grow clearer below. That said, the reverse setup, the descriptive set theoretic study of the homotopy classes of maps from a fixed polyhedron $P$ to a suitable range of spaces $X$, remains of considerable interest, and would include the development of the \emph{definable homotopy groups} as a special case. Such groups do, of course, tacitly figure in the present work, and it is only considerations of space and focus which have prevented us from saying more about them.
\end{remark}

\subsection{Homotopy is idealistic}
\label{SS:homISidealistic}

In this section, we show that the relation of homotopy for maps from a locally compact pair $(X,A) $ to a pointed polyhedron $P$ is \emph{idealistic} in the sense of Definition \ref{Definition:idealistic}.
We precede this result (Theorem \ref{Theorem:homotopy-idealistic}) with a selection principle of general utility (Proposition \ref{Proposition:uniformize1}); we then follow it with a more particular, and closely related, selection principle (Corollary \ref{Corollary:select-homotopy}) which we will apply to argue Theorem \ref{Theorem:phantom1}.

Suppose that $X,Y$ are Polish spaces and $R\subseteq X\times Y$ is a closed
relation such that for every $x\in X$ the section $R_{x}=\{ y\in
Y:( x,y) \in R\} $ is nonempty. Let $\mathcal{V}$ be a
countable basis for the topology of $Y$. For any Borel set $B\subseteq R$
let
\begin{equation*}
B_{\ast }=\left\{ x\in X:B_{x}\text{ is comeager in }R_{x}\right\}
\end{equation*}%
and%
\begin{equation*}
B_{\Delta }=\left\{ x\in X:B_{x}\text{ is nonmeager in }R_{x}\right\} \text{.%
}
\end{equation*}%
Notice that%
\begin{equation*}
x\in B_{\ast }\Leftrightarrow B_{x}\text{ is comeager in }%
R_{x}\Leftrightarrow R_{x}\setminus B_{x}\text{ is meager in }%
R_{x}\Leftrightarrow x\notin (R\setminus B)_{\Delta }\text{.}
\end{equation*}%
Hence $( R\setminus B) _{\Delta }=X\setminus B_{\ast }$.
In particular, $B_{\ast }$ is Borel if and only if $(R\setminus
B)_{\Delta }$ is Borel, and $B_{\Delta }$ is Borel if and only if $(R\setminus B)_{\ast}$ is Borel.
If $\{B_{n}:n\in\mathbb{N}\}$ is
a sequence of Borel subsets of $R$ and $B=\bigcap_{n\in\mathbb{N}}B_{n}$, then%
\begin{equation*}
B_{\ast }=\bigcap_{n\in\mathbb{N}}(B_{n})_{\ast }\text{.}
\end{equation*}%
Hence $B_{\ast }$ is Borel if $(B_{n})_{\ast }$ is Borel for
every $n\in\mathbb{N}$. Similarly, if $C=\bigcup_{n\in\mathbb{N}}B_{n}$, then%
\begin{equation*}
C_{\Delta}=\bigcup_{n\in\mathbb{N}}(B_{n})_{\Delta}\text{.}
\end{equation*}%
Hence $C_{\Delta}$ is Borel if $(B_{n})_{\Delta}$ is Borel
for every $n\in\mathbb{N}$.

\begin{lemma}
\label{Lemma:transform1}Adopt the notations above. Suppose that for every $%
V\in \mathcal{V}$ 
\begin{equation*}
R^{V}:=\left\{ x\in X:V\cap R_{x}\neq \varnothing \right\}
\end{equation*}%
is Borel. Then $B_{\ast }$ and $%
B_{\Delta }$ are Borel subsets of $X$ for each Borel set $B\subseteq R$.
\end{lemma}

\begin{proof}
Notice that if $B\subseteq R$ is Borel and $x\in X$, then%
\begin{eqnarray}
x \in B_{\ast }& \Leftrightarrow & \forall\, V\in \mathcal{V}\,(V\cap R_{x}\neq
\varnothing \Rightarrow B_{x}\cap V=(B\cap (X\times V))_{x}\text{ is
nonmeager in }R_{x})  \notag \\
&\Leftrightarrow &\forall\, V\in \mathcal{V}\,\left( x\in R^{V}\Rightarrow x\in
\left( B\cap (X\times V)\right) _{\Delta }\right) \text{.}
\end{eqnarray}%
Since $R^{V}$ is Borel for all $V\in \mathcal{V}$, it follows that $B_{\ast }$
is Borel whenever $( B\cap (X\times V)) _{\Delta }$ is Borel for
every $V\in \mathcal{V}$. Since $\left( R\setminus B\right) _{\Delta
}=X\setminus B_{\ast }$, this implies that $(R\setminus B)_{\Delta }$ is
Borel whenever $( B\cap (X\times V)) _{\Delta }$ is Borel for
every $V\in \mathcal{V}$.
By these observations, the lemma will follow if we prove by induction on $\alpha <\omega _{1}$ that if $B\in \boldsymbol{%
\Sigma }_{\alpha }^{0}$, then $B_{\Delta }$ is a Borel subset of $X$.

For the base case, suppose that $B=U\times V$ for some open set $U\subseteq
X$ and $V\subseteq Y$ such that $V\in \mathcal{V}$. Then
\begin{equation*}
B_{\Delta}=\left\{ x\in U:V\cap R_{x}\neq \varnothing \right\} =R^{V}\cap U
\end{equation*}%
is Borel by hypothesis. It follows that $B_{\Delta }$ is Borel for every open $B\subseteq R$.

For successor steps, suppose that our induction hypothesis holds for some $\alpha<\omega_1$. Then $(
B\cap (X\times V))_{\Delta }$ is Borel for every $V\in \mathcal{V}$
and $B\in \boldsymbol{\Sigma }_{\alpha }^{0}$. Hence $B_{\Delta }$ is Borel
for every $B\in \boldsymbol{\Pi }_{\alpha }^{0}$. If $C\in \boldsymbol{%
\Sigma }_{\alpha +1}^{0}$ then $C=\bigcup_{n}B_{n}$ for $B_{n}\in 
\boldsymbol{\Sigma }_{\alpha }^{0}$. Hence $C_{\Delta }$ is Borel.

Lastly, if $C\in \boldsymbol{\Sigma }_{\beta}^{0}$ for some limit ordinal $\beta$ below which our inductive hypothesis holds, then since $C=\bigcup_{n}B_{n}$ for some $B_{n}\in \boldsymbol{\Sigma }_{\alpha _{n}}^{0}$
with $\alpha_{n}<\beta$ for all $n\in \omega$, $C_{\Delta}$ is Borel.
This concludes the limit case of our argument, and with it the proof.
\end{proof}

\begin{proposition}
\label{Proposition:uniformize1} There exists a
Borel uniformization of any $R$ as in Lemma \ref{Lemma:transform1}; more precisely, there exists a Borel function $f:X\rightarrow Y$ such that $f(x) \in R_{x}$ for every $x\in X$.
\end{proposition}

\begin{proof}
Consider the map $x\mapsto I_{x}$ assigning to $x\in X$ the $\sigma $-ideal
of meager subsets of $X$. By Lemma \ref{Lemma:transform1}, if $B\subseteq R$
is Borel, then $\{ x\in X:B\cap R_{x}\in I_{x}\} $ is Borel.
Therefore the conclusion follows from the large section uniformization theorem \cite[Theorem 18.6]%
{kechris_classical_1995}, in the form appearing as Theorem 18.6$^*$ in \cite[p. 2]{Kechris_CDST_corrections} and
featuring in \cite[p. 8]{kechris_borel_2016} as well.
\end{proof}

We turn now more directly to the argument that the homotopy relation on $\mathsf{LC}((X,A),(P,*))$ is idealistic.
We will argue this from a series of lemmas in which the closed relation $R$ figuring in the definitions above will be the following set:
\begin{equation*}
R=\left\{ \left( f,\alpha \right) \in \mathsf{LCP}\left( \left( X,A\right) ,\left( P,\ast \right) \right)\times\mathsf{LCP}\left( \left(
X\times I,A\times I\right) ,\left( P,\ast \right) \right) 
:\alpha |_{X\times \left\{ 0\right\} }=f\right\} 
\end{equation*}%
In particular, for the remainder of this subsection, for any Borel $B\subseteq\mathsf{LCP}((X,A),( P,\ast))\times \mathsf{LCP}((X\times I,A\times I),( P,\ast))$, it will be in reference to this $R$ that the sets $B_{\ast}$ and $B_{\Delta}$ are defined.
We begin by proving two lemmas concerning compact pairs $(X,A)$.
\begin{lemma}
\label{Lemma:basis_W}
Suppose that $(X,A) $ is a compact pair and $(P,\ast)$ is a pointed polyhedron. Fix a countable open basis $\mathcal{V}$ for $X$ and let $(K,\ast)$ be a pointed simplicial complex such that $(P,\ast) =(|K|,\ast)$. Then $\mathsf{LCP}((X\times I,A\times
I),(P,\ast))$ has a countable basis consisting of open sets of the form%
\begin{equation}
\label{Equation:basis.c1}
W=\{ \alpha \in \mathsf{LCP}(( X\times I,A\times
I),(P,\ast)) :\forall i,j<\ell \;\alpha( 
\overline{U}_{i}\times L_{j}) \subseteq \mathrm{St}_{\beta
^{s}K}(p( i,j))\}
\end{equation}
where
\begin{itemize}
\item $s,\ell \in \mathbb{N}$,
\item $\beta^{s}K$ is the $s^{\mathrm{th}}$ barycentric subdivision of $K$,
\item $\mathcal{U}=(U_{i})_{i<\ell}$ is a finite
open cover of $X$ consisting of open sets from $\mathcal{V}$,
\item $\mathcal{L}%
=(L_{i}) _{i<\ell}$ is a finite cover of $\left[ 0,1\right] $
consisting of closed intervals $L_{j}=[a_{j},a_{j+1}]$ with rational endpoints
such that $0=a_{0}<a_{1}<\cdots <a_{\ell }=1$, and
\item $p$ is a function $\ell
\times \ell \rightarrow \mathrm{\mathrm{dom}}\left( \beta ^{s}K\right) $.
\end{itemize}  
\end{lemma}
\begin{proof}
Fix a nonempty open subset $V$ of $\mathsf{LCP}((X\times
I,A\times I),(P,\ast))$ and $\alpha
_{0}\in V$. We prove that there exists an open set $W$ as above such that $\alpha_{0}\in W\subseteq V$.
Fix compatible metrics $d_{X}$ on $X$ and $d_{P}$ on $P$ and endow $X\times I$ with
the metric
\begin{equation*}
d((x,t),(x^{\prime },t^{\prime }))
=\max \{d_X( x,x^{\prime }) ,|t-t^{\prime
}| \} \text{.}
\end{equation*}

Since $(X,A)$ is a compact pair, there exists a $\delta>0$ such that $V$ contains%
\begin{equation*}
W'=\{ \alpha \in \mathsf{LCP}((X\times I,A\times
I),(P,\ast)):\forall z\in X\times I\;d_{P}(
\alpha(z) ,\alpha _{0}(z)) <\delta\} \text{,}
\end{equation*}%
and there exists a finite subcomplex $L\subseteq K$ such that $\alpha _{0}(X\times I)
\subseteq |L|$.
And since $L$ is a finite
simplicial complex, there exists an $s\in \omega$ such that every $\mathrm{St}_{\beta^{s}L}(v)$ has diameter less than $\delta$.

For every $z\in X\times I$ there exists an open set $E_{z}\subseteq X\times I$ such that $z\in E_{z}$ and $\alpha_{0}(\overline{E}_z)\subseteq \mathrm{St}_{\beta ^{s}L}(v_{z}) $ for some $v_{z}\in \beta^s(L)$. By
compactness, there exists an $\ell\geq 1$ such that every subset of $X\times I$
of diameter at most $1/\ell$ is contained in $E_z$ for some $z\in X\times
I$. Therefore let $\mathcal{U}=(U_i)_{i<\ell}$ be a finite open
cover of $X$ consisting of open sets from $\mathcal{V}$ of diameter less
than $1/\ell$, and set $a_{i}=i/\ell $ for $0\leq i\leq \ell $.
For each $i,j<\ell $ the set $\overline{U}_{i}\times [a_{j},a_{j+1}] $ has diameter at most $1/\ell $ and is contained in $E_{z(i,j)}$ for some $z(i,j) \in X\times I$, hence if we let $p(i,j)=v_{z(i,j)}$ then
\begin{equation*}
\alpha_{0}(\overline{U}_{i}\times [a_j,a_{j+1}]) \subseteq \alpha_0(\overline{E}_{z(i,j)})\subseteq \mathrm{St}_{\beta
^{s}L}(v_{z(i,j)}) \subseteq\mathrm{St}_{\beta
^{s}K}(p(i,j))\text{.}
\end{equation*}
\color{black}
These parameters define a $W$ as in equation \ref{Equation:basis.c1} with $\alpha _{0}\in W\subseteq W'\subseteq
V$.
\end{proof}

\begin{lemma}
\label{Lemma:transform-homotopy}Suppose that $(X,A)$
is a compact pair and $(P,\ast)$ is a pointed polyhedron and that $W$ is an open subset of $\mathsf{LCP}((X\times I,A\times I),(P,\ast))$.
Then the set
\begin{equation*}
R^W=\{g\in \mathsf{LCP}((X,A),(P,\ast)):\exists \alpha \in W\;\alpha |_{X\times \{
0\} }=g\}
\end{equation*}%
is Borel.
\end{lemma}

\begin{proof}
We retain the notation of Lemma \ref{Lemma:basis_W}. It will suffice to
argue the lemma when $W$ is as in equation \ref{Equation:basis.c1}. 
In this case we claim that $R^W$ equals the open set
\begin{equation*}
S=\{g\in \mathsf{LCP}((X,A),(P,\ast)):\forall i<\ell\; g( \overline{U}_i)
\subseteq \mathrm{St}_{\beta ^{s}K}^{P}(p(i,0))\}
\end{equation*}%
where, of course, the sets $U_i$ and function $p$ are those determining $W$.
The relation $R^W\subseteq S$ is clear, since $g(\overline{U}_i) \subseteq \mathrm{St}_{\beta^{s}K}(p(i,0))$ for every $g\in R_W$ and $i<\ell$, simply by the definition of $W$. For the reverse containment, suppose that $%
g_{0}\in \mathsf{LCP}((X,A),(P,\ast))$ satisfies $g_{0}(\overline{U}_i) \subseteq 
\mathrm{St}_{\beta ^{s}K}( p( i,0))$ for every $i<\ell $. Fix a partition of unity $\left( \rho _{i}\right)_{i\in\omega}$
subordinate to $\mathcal{U}$, and consider the function $g_{1/2}:(X,A) \rightarrow (P,\ast)$ defined by
\begin{equation*}
x\mapsto \sum_{i<\ell}\rho _{i}(x) e_{p(i,0)}\text{.}
\end{equation*}
Notice that $g_{1/2}^{-1}\left( \mathrm{St}_{\beta ^{s}K}(p(i,0)) \right) \subseteq U_{i}$ for every $i< \ell$, hence
\begin{equation*}
g_{0}\left( g_{1/2}^{-1}\left( \mathrm{St}_{\beta ^{s}K}( p(
i,0)) \right) \right) \subseteq g_{0}(U_i)
\subseteq \mathrm{St}_{\beta ^{s}K}(p(i,0))
\end{equation*}%
for every $i<\ell$ and we may define a homotopy $h_0:g_{0}\Rightarrow g_{1/2}:(X,A) \rightarrow (P,\ast)$ by
letting $\lambda$ and $D$ be as in Lemma \ref{L:OpenStarHomotopy} and letting
 \begin{equation*}
h_{0}(x,t) = \lambda( g_0(x), \sum_{i<d}\rho_{i}(x) e_{p(i,0)} ,t).
\end{equation*}%
 
  We now claim that if $\{ i_{0},\ldots
,i_{m}\} $ is a simplex in the nerve $N_{\mathcal{U}}$ of the cover $\mathcal{U}$ of $X$, then 
\begin{equation}\label{equation:beta-simplex}
\bigcup_{0\leq k\leq m}\{ p( i_{k},0) ,p( i_{k},1)\} 
\end{equation}%
is a simplex in $\beta^s K$. To see this, fix an $\alpha_0\in W$ and let $g_1=\alpha_0\vert_{X\times \{a_1\}}$, suppose that $\{ i_{0},\ldots
,i_{m}\} \in N_{\mathcal{U}}$, and hence that there exists an $x\in
U_{i_{0}}\cap \cdots \cap U_{i_{m}}$. We then have
\begin{equation*}
\alpha _{0}( x,a_{1}) \in \alpha _{0}( \overline{U}_{i_{k}}\times L_{0}) \cap \alpha _{0}( \overline{U}
_{i_{k}}\times L_{1}) \subseteq \mathrm{St}_{\beta^{s}K}(
p( i_{k},0)) \cap \mathrm{St}_{\beta ^{s}K}(
p( i_{k},1))\text{.}
\end{equation*}
for $0\leq k\leq m$; our claim follows immediately. 
In consequence, we may define a homotopy $
h_{1}:g_{1/2}\Rightarrow g_{1}:( X,A) \rightarrow
(P,\ast) $ by letting $\lambda$ and $D$ be as in Lemma \ref{L:OpenStarHomotopy} and letting
 \begin{equation*}
h_{1}(x,t) = \lambda( \sum_{i<d}\rho _{i}(x)  e_{p(i,0)}, \sum_{i<d}e_{p(i,1)},t).
\end{equation*}
To see that 
$h_{1}$ is well-defined, observe as before that if $x\in X$ and $\{ i_{0},\ldots ,i_{m}\} =\{i<\ell:\rho
_{i}(x) >0\} $, then $x\in U_{i_{0}}\cap
\cdots \cap U_{i_{m}}$ and hence the set described by expression \ref{equation:beta-simplex} above is again a simplex in $\beta^s K$. 

We now define an $\alpha \in W$ such that $\alpha |_{X\times \left\{
0\right\} }=g_{0}$ by setting%
\begin{equation*}
\alpha \left( x,t\right) =\left\{ 
\begin{array}{ll}
h_{0}(x,\frac{2t}{a_{1}}) & 0\leq t\leq \frac{a_{1}}{2}\text{;} \\ 
h_{1}(x,\frac{2t}{a_{1}}-1)) & \frac{a_{1}}{2}\leq t\leq a_{1}\text{;} \\ 
\alpha _{0}\left( x,t\right)  & a_{1}\leq t\leq 1\text{.}%
\end{array}%
\right. 
\end{equation*}
This concludes the proof.
\end{proof}
We now extend this analysis to the locally compact setting.

\begin{lemma}
\label{Lemma:transform-homotopy-locally-compact}If $(X,A)$ is a locally compact pair and $(P,\ast)$
is a pointed polyhedron and $W$ is an open subset of $\mathsf{LCP}((X\times I,A\times I),(P,\ast))$ then the set%
\begin{equation*}
R^W=\{ g\in \mathsf{LCP}((X,A),(P,\ast)):\exists \alpha \in W\;\alpha |_{X\times \{
0\} }=g\} 
\end{equation*}%
is Borel.
\end{lemma}

\begin{proof}
Note first that by replacing $X$ with $X/A$ (if $A$ is nonempty) or
with the space $X_{+}$ obtained by adding to $X$ an additional basepoint (if 
$A$ is empty), we may assume that $A$ is a singleton $\{\star\}$.

As above, it will suffice to prove the statement for basic open subsets $W$ subset $\mathsf{LCP}((X,A),(P,\ast))$. Therefore we may assume that there exist compact subsets $K_{1},\ldots ,K_{n}\subseteq
X\times I$ and open subsets $U_{1},\ldots ,U_{n}\subseteq P$ such that
\begin{equation*}
W=\{\alpha \in \mathsf{LCP}((X,\star),(P,\ast)):\forall i<n\; \alpha(K_i) \subseteq U_i\} \text{.}
\end{equation*}
Let
\begin{equation*}
K=\{\star\} \cup \mathrm{proj}_X( K_0\cup \cdots \cup K_{n-1})
\subseteq X
\end{equation*}
and let
\begin{equation*}
W[K]=\{\alpha \in \mathsf{LCP}((K\times I,\{\star\}\times I),(P,\ast)):\forall i<n\; \alpha(K_i) \subseteq U_i\}\text{.} 
\end{equation*}
By Lemma \ref{Lemma:transform-homotopy}, the set 
$$R^{W[K]}=\{g\in \mathsf{LCP}((K,\star),(P,\ast)):\exists\alpha\in W[K]\;\alpha
|_{K\times\{0\} }=g\}$$
is Borel.
By the Homotopy Extension Theorem applied to $K\subseteq X$ and $P$,
\begin{equation*}
R^W=\{g\in \mathsf{LCP}((X,\star),(P,\ast)):g|_K\in R^{W[K]}\} 
\end{equation*}%
is a Borel subset of $\mathsf{LCP}((X,\star),(P,\ast))$. This concludes the proof.
\end{proof}

\begin{lemma}
\label{Lemma:transform-homotopy-locally-compact Borel}If $(X,A)$ is a locally compact pair and $(P,\ast)$
is a pointed polyhedron and $B$ is a Borel subset of $\mathsf{LCP}((X\times I,A\times I),(P,\ast))$, then the subsets $B_{\ast}$ and $B_{\Delta}$ of $\mathsf{LCP}((X,A),(P,\ast))$ are Borel as well.
\end{lemma}
\begin{proof}
Apply Lemma \ref{Lemma:transform1} to Lemma \ref{Lemma:transform-homotopy-locally-compact}.
\end{proof}

\begin{lemma}
\label{Lemma:invariant}Suppose that $(X,A) $ is a
locally compact pair and $P$ is a pointed polyhedron. If $S\subseteq X$ and 
\begin{equation*}
T=\left\{ \alpha \in \mathsf{LCP}\left( \left( X\times I,A\times
I\right) ,\left( P,\ast \right) \right) :\alpha |_{X\times \left\{ 1\right\}
}\in S\right\} 
\end{equation*}%
then the set $T_{\ast }\subseteq \mathsf{LCP}((X,A),(P,\ast))$ is homotopy-invariant.
\end{lemma}

\begin{proof}
Suppose that $f\in T_*$ and $f$ is homotopic to $g$. Let
\begin{equation*}
H_{f}=\left\{ \alpha \in \mathsf{LCP}\left( \left( X\times I,A\times I\right) ,\left( P,\ast \right) \right) :\alpha \left( -,0\right)
=f\right\} \text{,}
\end{equation*}%
\begin{equation*}
H_{g}=\left\{ \alpha \in \mathsf{LCP}\left( \left( X\times I,A\times I\right) ,\left( P,\ast \right) \right) :\alpha \left( -,0\right)
=g\right\} \text{, and }
\end{equation*}%
\begin{equation*}
L=\left\{ \alpha \in \mathsf{LC}\left( \left( X\times I,A\times
I\right) ,\left( P,\ast \right) \right) :\alpha \left( -,0\right) =f\text{
and }\alpha \left( -,1\right) =g\right\} .
\end{equation*}%
Notice that $\left\{ \alpha \in
H_{f}:\alpha \left( -,1\right) \in S\right\} $ is comeager in $H_{f}$, since $f\in T_*$.

For $\left( \alpha ,\beta \right) \in L\times H_{g}$ define $\alpha \ast
\beta \in H_{f}$ by setting%
\begin{equation*}
\left( \alpha \ast \beta \right) \left( t\right) =\left\{ 
\begin{array}{ll}
\alpha \left( 2t\right) & 0\leq t\leq 1/2\text{,} \\ 
\beta \left( 2t-1\right) & 1/2\leq t\leq 1\text{.}%
\end{array}%
\right.
\end{equation*}%
This defines a continuous and open function $L\times H_{g}\rightarrow
H_{f}$. Therefore
\begin{equation*}
\left\{ \left( \alpha ,\beta \right) \in L\times H_{g}:\left( \alpha \ast
\beta \right) |_{X\times \left\{ 1\right\} }\in S\right\} =\left\{ \left(
\alpha ,\beta \right) \in L\times H_{g}:\beta |_{X\times \left\{ 1\right\}
}\in S\right\}
\end{equation*}%
is comeager in $L\times H_{g}$. By the Kuratowski--Ulam theorem \cite[%
Theorem 8.41]{kechris_classical_1995}, this implies that%
\begin{equation*}
\left\{ \beta \in H_{g}:\beta \left( -,1\right) \in S\right\}
\end{equation*}%
is comeager in $H_{g}$, and hence that $g\in T_*$.
\end{proof}
We turn now to this section's main result.
\begin{theorem}
\label{Theorem:homotopy-idealistic}If $(X,A)$ is a locally compact pair and $(P,\ast)$ is a pointed
polyhedron then the relation of homotopy for maps $(X,A)\to(P,\ast)$ is idealistic.
\end{theorem}

\begin{proof}
For any map $f:(X,A)\to(P,\ast)$, let $[f] \in [(X,A),(P,\ast)]$ denote its homotopy class.
We define a $\sigma $-filter $\mathcal{F}_{[f]}$ of subsets of $[f]$ by letting $S\in \mathcal{F}_{[f]}$ if and only if 
\begin{equation*}
\{\alpha\in\mathsf{LC}((X\times I,A\times I),(P,\ast)):\alpha |_{X\times \{1\}}\in
S\} 
\end{equation*}
is comeager in 
\begin{equation*}
R_f=\{\alpha\in\mathsf{LC}(X\times
I,A\times I),(P,\ast)):\alpha |_{X\times\{0\}}=f\} .
\end{equation*}
By Lemma \ref{Lemma:invariant}, the definition of $\mathcal{F}_{[f]}$ does not depend on the choice of representative $f$ of the homotopy class $[f]$.

Suppose now that $E\subseteq\mathsf{LC}((X,A),P,\ast)) \times \mathsf{LC}((X,A),(P,\ast))$ is Borel.
In the notation of Definition \ref{Definition:idealistic}, our task is to show that, for $\zeta=\mathrm{id}$ and the filters $\mathcal{F}_{[f]}$ described above, the set $E_{\mathcal{F}}$ is Borel.
This set unpacks as follows:
\begin{eqnarray*}
E_{\mathcal{F}} &=&\{ f\in \mathsf{LC}(X,A),(P,\ast)):\mathcal{F}_{[f]
}g,(f,g) \in E\}  \\
&=&\{ f\in \mathsf{LC}(X,A),(P,\ast)):\{ \alpha \in R_{f}:(f,\alpha |_{X\times\{1\}})\in E\} \text{ is comeager in }R_{f}\}  \\
&=&B_{\ast }
\end{eqnarray*}
where $B\subseteq R$ is the Borel set
\begin{equation*}
\{(f,\alpha) \in R:(f,\alpha |_{X\times\{
1\}})\in E\} \text{.}
\end{equation*}%
By Lemma \ref{Lemma:transform-homotopy-locally-compact Borel}, $E_{\mathcal{F}}$ is Borel, as desired, showing that the assignment $[f]\mapsto \mathcal{F}_{[f]}$ indeed witnesses that the relation of homotopy for maps $(X,A)\to(P,\ast)$ is idealistic.
\end{proof}
We now deduce, from a sequence of minor variations on the lemmas just recorded, a selection principle which we will want in Section \ref{S:HomotopyClassification}.

For any compact pair $(X,A)$ and pointed polyhedron and $(P,*)$ let $\mathsf{LCP}_{0}((X,A),(P,*))$ denote the subspace of $\mathsf{LCP}((X,A),(P,*))$ consisting of maps which are homotopic to the constant map $\ast:(X,A)\to (P,*)$. By Lemma \ref{Lemma:countable-homotopy}, $\mathsf{LCP}_{0}((X,A),(P,*))$ is open in $\mathsf{LCP}((X,A),(P,*))$. Let $\mathcal{Z}((X,A),(P,*))$ denote the space of nullhomotopies; more precisely, let $\mathcal{Z}((X,A),(P,*))$ be the closed subset of $\mathsf{LCP}((X\times I,A\times I),(P,\ast))$ consisting of those $\alpha$ for which $\alpha |_{X\times\{1\}}=\ast$.
As the proof of the following lemma is essentially identical to that of Lemma \ref{Lemma:basis_W}, we omit it.
\begin{lemma}
\label{Lemma:basis_W.2}
Suppose that $(X,A) $ is a compact pair and $(P,\ast)$ is a pointed polyhedron. Fix a countable open basis $\mathcal{V}$ for $X$ and let $(K,\ast)$ be a pointed simplicial complex such that $(P,\ast) =(|K|,\ast)$. Then $\mathcal{Z}((X,A),(P,*))$ has a countable basis consisting of open sets of the form%
\begin{equation}
\label{Equation:basis.c2}
W=\{ \alpha \in\mathcal{Z}((X,A),(P,*)):\forall i,j<\ell \;\alpha( 
\overline{U}_{i}\times L_{j}) \subseteq \mathrm{St}_{\beta
^{s}K}(p( i,j))\}
\end{equation}
where
\begin{itemize}
\item $s,\ell \in \mathbb{N}$,
\item $\beta^{s}K$ is the $s^{\mathrm{th}}$ barycentric subdivision of $K$,
\item $\mathcal{U}=(U_{i})_{i<\ell}$ is a finite
open cover of $X$ consisting of open sets from $\mathcal{V}$,
\item $\mathcal{L}%
=(L_{i}) _{i<\ell}$ is a finite cover of $\left[ 0,1\right] $
consisting of closed intervals $L_{j}=[a_{j},a_{j+1}]$ with rational endpoints
such that $0=a_{0}<a_{1}<\cdots <a_{\ell }=1$, and
\item $p$ is a function $\ell
\times \ell \rightarrow \mathrm{\mathrm{dom}}\left( \beta ^{s}K\right) $.
\end{itemize}  
\end{lemma}
We also have the analogue of Lemma \ref{Lemma:transform-homotopy} (and its proof).
\begin{lemma}
\label{Lemma:transform-homotopy.2}Suppose that $(X,A)$
is a compact pair and $(P,\ast)$ is a pointed polyhedron and that $W$ is an open subset of $\mathcal{Z}((X,A),(P,*))$.
Then the set
\begin{equation*}
R^W=\{g\in \mathsf{LCP}_0((X,A),(P,\ast)):\exists \alpha \in W\;\alpha |_{X\times \{
0\} }=g\}
\end{equation*}%
is Borel.
\end{lemma}
\begin{proposition}
\label{Proposition:select-homotopy2}If $(X,A)$ is a compact pair and $(P,*)$ is a pointed polyhedron then there exists a Borel function $\Phi:\mathsf{LCP}_{0}((X,A),(P,*))\to \mathcal{Z}((X,A),(P,*))$, $g\mapsto \Phi(g)$ such that $\Phi(g)$ is a
homotopy $g\Rightarrow \ast:(X,A)\to(P,\ast)$.
\end{proposition}

\begin{proof}
This is an immediate consequence of Lemma \ref{Lemma:transform-homotopy.2} and
Proposition \ref{Proposition:uniformize1}.
\end{proof}
We apply the following corollary in the proof of Theorem \ref{Theorem:phantom1}.
\begin{corollary}
\label{Corollary:select-homotopy}Let $(Y,B)$ a locally compact pair homotopy equivalent to the compact pair $(X,A)$ and let $(P,\ast)$ be a pointed polyhedron.
Then there exists a Borel function $\Phi _{Y}:\mathsf{LCP}_{0}((Y,B),(P,*))\to \mathcal{Z}((Y,B),(P,*))$, $g\mapsto \Phi_Y(g)$ such that $\Phi_Y(g)$ is a
homotopy $g\Rightarrow \ast:(Y,B)\to(P,\ast)$.
\end{corollary}

\begin{proof}
By Proposition \ref{Proposition:select-homotopy2}, there exists a Borel
function $\Phi_X:\mathsf{LCP}_{0}((X,A),(P,*))\to \mathcal{Z}((X,A),(P,*))$, $g\mapsto \Phi_X(g)$ such that $\Phi_X(g)$ is a
homotopy $g\Rightarrow \ast:(X,A)\to(P,\ast)$.
Fix a homotopy equivalence $h:(X,A)
\rightarrow (Y,B)$, with homotopy inverse $k:(Y,B)\to (X,A)$.
Let also $\alpha :\mathrm{id}_{(Y,B)}\Rightarrow h\circ k:(Y,B) \rightarrow (Y,B)$ be a homotopy.

For any $g\in \mathsf{LCP}_{0}((Y,B),(P,*))$, $\Phi _{X}(
g\circ h)$ is a homotopy $g\circ h\Rightarrow \ast :(X,A) \rightarrow (P,\ast)$; hence $\Phi
_{X}(g\circ h) \circ(k\times \mathrm{id}_{I}) $ is
a homotopy $g\circ h\circ k\Rightarrow \ast :(Y,B)
\rightarrow (P,\ast)$. Define then $\Phi _{Y}(g)$
to be the composition of the homotopy $g\circ \alpha :g\Rightarrow g\circ
h\circ k$ and the homotopy $\Phi _{X}( g\circ h) \circ (
k\times \mathrm{id}_{I}) :g\circ h\circ k\Rightarrow \ast$.
\end{proof}

\section{Definable cohomology: the homotopical approach}\label{S:Huber}
We showed in the previous section that $[X,P]$ is idealistic whenever $P$ is a polyhedron. If $P$ carries, in addition, an $H$-group structure, then $[X,P]$ is naturally regarded as a semidefinable group; if the homotopy relation on $\mathsf{LC}(X,P)$ is, moreover, Borel, then by our previous results we will have succeeded in showing that $[X,P]$ is a \emph{definable} group. We show exactly this in Theorem \ref{Theorem:phantom-H-group} below.

We show in the present section that when $P$ is an Eilenberg-MacLane space of type $(G,n)$, we can do even better. We adopt the standard abuse of denoting such a $P$ (which is only well-defined up to homotopy equivalence) by $K(G,n)$. By a classical theorem of Huber's, $[X,K(G,n)]$ is naturally isomorphic to the \v{C}ech cohomology group $\Check{\mathrm{H}}^n(X;G)$, and in this section we show that this isomorphism is in fact definable. This carries several immediate and pleasant consequences, namely: (1) $[X,K(G,n)]$ is essentially (i.e., is definably isomorphic to) a \emph{group with a Polish cover}; (2) up to definable isomorphism, $\Check{\mathrm{H}}_{\mathrm{def}}^n(X;G)$ is homotopy invariant and does not depend on the choice of covering system $\boldsymbol{\mathcal{U}}$; (3) $\Check{\mathrm{H}}_{\mathrm{def}}^n$ is a \emph{contravariant functor} from $\mathsf{LC}$ (or more generally $\mathsf{LCP}$) to $\mathsf{GPC}$. The main ingredient of the proof of the definable version of Huber's theorem is a definable version of the simplicial approximation theorem; see Lemma \ref{L:DefSimplApprox}.

In the remainder of the paper, we will argue our results in whichever of the settings $\mathsf{LC}$, $\mathsf{LC}_{*}$, or $\mathsf{LCP}$ seems most representative or most encompassing, depending on the context, often only sketching their extension to any of the others. We note, looking ahead, that the hazards of such an approach are largely allayed by our results in Section \ref{SS:definable_relation_based_and_unbased}, which help to definably mediate between these settings.

\subsection{Preliminaries: $H$-groups, $H$-cogroups, and $K(G,n)$ spaces}
\label{SS:H-groups}

Before proceeding, we recall some of the basic materials and operations of homotopy theory, in particular those which bear on the group and degree structures of $[-,-]$ and cohomology functors, respectively. The natural and standard context for these operations is the pointed setting (as in \cite{arkowitz_introduction_2011, spanier_algebraic_1995, may_concise_1999, switzer_algebraic_2002}), any of which may be taken as a reference for this subsection); our initial framework, accordingly, will be $\mathsf{LC}_{*}$, although mild generalizations of these structures and operations will arise as we proceed.

The $\mathsf{LC}_*$-analogue of the well-known suspension operation $SX$ on unbased spaces is the \emph{reduced suspension} operation $$\Sigma:\,\mathsf{LC}_*\to\mathsf{LC}_*:\,(X,\ast)\mapsto \frac{X\times I}{X\times\{0\}\cup X\times\{1\}\cup \{*\}\times I},$$
with the $\Sigma$-image of the basepoint as the basepoint of $\Sigma(X,\ast)$. The \emph{sum} or \emph{wedge} $(X\vee Y,\ast)$ of $(X,\star)$ and $(Y,\hexstar)$ in $\mathsf{LC}_*$ simply identifies the basepoints in the spaces' disjoint union. We will occasionally elide notation of basepoints in $\mathsf{LC}_*$; the \emph{smash product} $X\wedge Y$ of two pointed spaces, for example, is the quotient of $X\times Y$ by the canonical copy of $X\vee Y$ therein. Note that $\Sigma X\cong X\wedge S^1$, where $S^1$ is the basepointed sphere of dimension one. 

An \emph{$H$-group} (also called a \emph{grouplike space}) is a group object in the
category $\mathsf{Ho}(\mathsf{LC}_{\ast })$ (although it will occasionally be convenient to make this definition in $\mathsf{Ho}(\mathsf{Top}_{*})$). 
Thus an $H$-group is a pointed locally compact Polish space $X$ endowed with a map $\mu:( X\wedge X,\ast ) \rightarrow ( X,\ast
) $ (\emph{multiplication}) such that:
\begin{enumerate}
\item the maps $\left( X,\ast \right) \rightarrow \left( X,\ast \right) $, $%
x\mapsto \mu( x,\ast) $ and $x\mapsto ( \ast ,x) $ are
homotopic to the identity map $1_{X}$ of $X$ (a \emph{homotopy identity});

\item the maps $\mu\circ \left( \mu\wedge 1_{X}\right) $ and $\mu\circ \left(
1_{X}\wedge \mu\right) $, each taking  $\left( X\wedge X\wedge
X,\ast \right) $ to $(X,*)$, are homotopic (\emph{homotopy associativity});

\item there exists a map $z:\left( X,\ast \right) \rightarrow \left( X,\ast
\right) $ such that the maps $( X,\ast) \rightarrow (
X,\ast) $, $x\mapsto \mu( x,z( x)) $ and $%
x\mapsto \mu( z( x) ,x) $ are nullhomotopic (a \emph{%
homotopy inverse}).
\end{enumerate}
\begin{example}
\label{ex:h-group}
Any locally compact Polish group is an $H$-group, with its neutral element as basepoint. Less trivial examples of $H$-group structures are given by the spaces $\Omega X:=\mathsf{LC}_{*}((S^1,\ast),(X,\star))$ for $X$ in $\mathsf{LC}_{*}$ (note, however, that $\Omega X$ may itself fail to be locally compact);
the multiplication operation is given by concatenation of maps in a manner subsumed (since $S^1=\Sigma S^0$) by Example \ref{Ex:suspensions} below.
\end{example}
The $H$-group $X$ is \emph{abelian} (or \emph{homotopy commutative}) if the maps $\mu$ and $\mu\circ \sigma $, where $%
\sigma :\left( X\wedge X,\ast \right) \rightarrow \left( X,\ast \right) $, $%
\left( x,y\right) \mapsto \left( y,x\right) $ is the ``flip,'' are homotopic. 
If $(X,\ast,\mu)$ is an $H$-group and $(Y,B)$ is
locally compact pair, then $[(Y,B),(X,\ast)]$ is a semidefinable group
with respect to the operation defined by setting $[f]\cdot [g] =[\mu\circ( f\wedge g)]$; here the
identity element of $[(Y,B),(X,\ast)]$ is represented by the constant map, which we also denote by $*$.

Dually, an \emph{$H$-cogroup} is a cogroup in the category $\mathsf{Ho}(\mathsf{LC}_{\ast })$ \cite[Section III.6]{mac_lane_categories_1998}. Explicitly, an $H$-cogroup is a
pointed locally compact Polish space $X$ endowed with a continuous map $\nu
:X\rightarrow X\vee X$ (\emph{comultiplication}) such that:

\begin{enumerate}
\item the maps $\left( X,\ast \right) \rightarrow \left( X,\ast \right) $
given by $\left( \ast \vee 1_{X}\right) \circ \nu $ and $\left( 1_{X}\vee
\ast \right) \circ \nu $ are homotopic to $1_{X}$;

\item the maps $\left( X,\ast \right) \rightarrow \left( X\vee X\vee X,\ast
\right) $ defined by $\left( 1_{X}\vee \nu \right) \circ \nu $ and $\left(
\nu \vee 1_{X}\right) \circ \nu $ are homotopic,

\item there exists a map $\zeta :\left( X,\ast \right) \rightarrow \left(
X,\ast \right) $ such that the maps $\left( X,\ast \right) \rightarrow
\left( X,\ast \right) $ defined by $\left( 1_{X}\vee \zeta \right) \circ \nu 
$ and $\left( \zeta \vee 1_{X}\right) \circ \nu $ are homotopic to $\ast $.
\end{enumerate}

\begin{example}
\label{Ex:suspensions}
Main examples of $H$-cogroups are given by suspensions of spaces.
Writing $\langle x,t\rangle$ for the image of $(x,t)\in X\times I$ in $\Sigma X$, we have a comultiplication operation
\begin{equation*}
\nu (\langle x,t\rangle) =\left\{ 
\begin{array}{ll}
( \langle x,2t\rangle ,\ast) & 0\leq t\leq 1/2 \\ 
( \ast ,\langle x,2t-1\rangle) & 1/2\leq t\leq 1%
\end{array}%
\right.
\end{equation*}%
(for ease of notation, we identify $X\vee X$ with its canonical copy in $X\times X$). 
The homotopy inverse $\zeta :\Sigma X\rightarrow \Sigma X$ is defined by $\langle x,t\rangle \mapsto
\langle x,1-t\rangle$.
\end{example}

The cogroup $\left( X,\ast ,\nu \right) $ is \emph{homotopy commutative} if the maps $%
X\rightarrow X\vee X$ defined by $\nu $ and $\sigma \circ \nu $ are
homotopic.
As above, for an $H$-cogroup $\left( X,\nu \right)$ and a pointed
space $Y$, the operation $[f]\cdot[g] =[(f\vee g)\circ\nu] $ defines a semidefinable group
structure on $[(X,\ast),(Y,\star)]$; again the identity element is represented by the constant map.

A natural question when the first and second arguments of $[(X,\ast),(Y,\star)]$ carry $H$-cogroup and $H$-group structures, respectively, is how the induced group structures on $[(X,\ast),(Y,\star)]$ relate. This question has a pleasant answer even when $(Y,\star)$ is merely an \emph{$H$-space}. Briefly put, an $H$%
-space is a \emph{unital magma }in the homotopy category of pointed spaces 
\cite[Definition 2.2.1]{arkowitz_introduction_2011}. More explicitly, an $H$-space is a pointed locally compact Polish space $X$ endowed with a map $\mu:X\wedge X\rightarrow X$ such that the maps $\left( X,\ast \right)
\rightarrow \left( X,\ast \right) $, $x\mapsto\mu( x,\ast) $ and $%
x\mapsto \mu( \ast ,x) $ are homotopic to $1_{X}$; its ``group operation'' $\mu$, in other words, is free to violate items (2) and (3) of the definition of an $H$-group.
The following lemma appears as \cite[Proposition
2.2.12]{arkowitz_introduction_2011}.

\begin{lemma}
\label{Lemma:H-space}Let $X$ be an $H$-cogroup and let $Y$ be an $H$-space. Then the group operation on $[(X,\ast),(Y,\star)]$ induced by the $H$-cogroup structure on $X$ coincides with the operation on $[(X,\ast),(Y,\star)]$ induced by the $H$-space structure
on $Y$ and is, moreover, commutative.
\end{lemma}

We turn now to a central focus of this section, the Eilenberg-MacLane spaces first introduced in \cite{eilenberg_relations_1945}. These are specified up to homotopy equivalence by the following definition:
\begin{definition} For any abelian group $G$ and $n\geq 1$, an \emph{Eilenberg-MacLane space of type $(G,n)$} is a pointed topological space $Y$ satisfying the following condition:
\begin{equation*}
\pi_i(Y,*)=
    \begin{cases}
        \; G & \text{if } i=n \\
        \{0\} & \text{if } i\neq n
    \end{cases}
\end{equation*}
\end{definition}
See this subsection's initial references for multiple approaches to such spaces' construction. More immediately relevant to our concerns are the following three main points:
\begin{itemize}
\item For any $n\geq 1$ and countable abelian group $G$, there exists a (countable, locally finite) \emph{polyhedral} Eilenberg-MacLane space of type $(G,n)$, by any of the aforementioned constructions, together with \cite{milnor_spaces_1959}.
\item For any Eilenberg-MacLane spaces $Y,Z$ of type $(G,n)$ and $(G,n+1)$, respectively, we have $Y\simeq\Omega Z$ (this follows from the uniqueness of Eilenberg-MacLane spaces up to homotopy equivalence, together with the adjointness of the $\Sigma$ and $\Omega$ operations). It is easy to see that homotopy equivalences preserve $H$-group structures, and that the $H$-group operations on any $\Omega Z$ referenced in Example \ref{ex:h-group} are continuous, hence any Eilenberg-MacLane space $Y$ carries a natural $H$-group structure with continuous $H$-group operations.
\item It will be convenient for our purposes to take any discrete group $G$ \emph{itself} to be an Eilenberg-MacLane space of type $(G,0)$. Note that this extends the $\Omega$ equation cited just above; moreover, as noted in \citep{huber_homotopical_1961}, this convention extends the homotopical representation of the \v{C}ech cohomology groups to the degree zero.
\end{itemize}
\subsection{A definable simplicial approximation theorem}
For simplicity, we'll conduct the discussion of the next two subsections in the category $\mathsf{LC}$, touching on its generalization at their conclusion. Note that an $H$-group structure on $P$ will continue in this context to translate to one on $[X,P]$, definably so, as is clear either from inspection or Corollary \ref{cor:based_hspace}.
Therefore fix a locally compact Polish space $X$ and a covering system $\boldsymbol{\mathcal{U}}=\big{(}(X_n),(\mathcal{U}_{\alpha}),(r^{\beta}_{\alpha})\big{)}$ for $X$. Recall that by $N_{\alpha}$ we denote the simplicial complex $\mathrm{Nv}(\mathcal{U}_{\alpha})$. 
Fix also a countable locally finite  simplicial complex $K$ and assume that $\mathrm{dom}(K)=\mathbb{N}$. 

We begin by defining the Polish space $\mathrm{SA}_{\boldsymbol{\mathcal{U}}}(X,K)$ of all {\em simplicial approximations of functions from $X$ to $K$}. First, let $S:=\bigcup_{\alpha\in \mathcal{N}^*} \mathcal{U}_{\alpha}$ be the collection of all open sets contained in any open cover of our covering system. By Lemma \ref{L:combinatorialapproximations} this set is countable, and without loss of generality we may assume that $\emptyset\not\in S$. We endow the set $\mathrm{dom}(K)^S$ of all functions from $S$ to $\mathrm{dom}(K)$ with the product topology, rendering it a Polish space. For every
 $p\in \mathrm{dom}(K)^S$ and every  $\alpha\in \mathcal{N}^{*}$ we denote by $p|_{\alpha}$ the restriction of $p$ to the set $\mathcal{U}_{\alpha}$.
We define the set 
\[\mathrm{SA}_{\boldsymbol{\mathcal{U}}}(X,K)\subseteq \mathcal{N}^{*}\times \mathrm{dom}(K)^S,\]
to consist of all pairs $(\alpha,p)$ for which $p|_{\alpha}$ is a simplicial map from $N_{\alpha}$  to $K$ and $p(U)=0$ if $U\not\in \mathcal{U}_{\alpha}$.
As a closed subset of the Polish space $\mathcal{N}^{*}\times \mathrm{dom}(K)^S$, the space $\mathrm{SA}_{\boldsymbol{\mathcal{U}}}(X,K)$  is Polish. It is also clearly in bijective correspondence with the set  all simplicial maps of the form $N_{\alpha}\to K$ from some complex of the form $N_{\alpha}$.
With reference to the canonical open cover $\{\mathrm{St}_K(k)\colon k\in \mathrm{dom}(K)\}$ of $K$, we now have the following definable version of the \emph{simplicial approximation theorem}.

\begin{lemma}\label{L:DefSimplApprox}
There is a Borel map $f\mapsto (\alpha_f,p_f)$ from $\mathsf{LC}(X,|K|)$ to $\mathrm{SA}_{\boldsymbol{\mathcal{U}}}(X,K)$ so that
\[f(x)\in \mathrm{St}_K(p_f(U)) \text{ for all } x\in U \in \mathcal{U}_{\alpha_f}.\]
\end{lemma}
\begin{proof}
For every $\alpha\in \mathcal{N}^*$ set $\mathcal{U}^{\mathrm{cl}}_{\alpha}:=\{\mathrm{cl}(U)\colon U\in\mathcal{U}_{\alpha}\}$. By Lemma \ref{L:combinatorialapproximations}(4) we may assume without loss of generality that for all $\alpha\in \mathcal{N}^*$, every element of $\mathcal{U}^{\mathrm{cl}}_{\alpha}$ is compact.

As before, for readability we will omit the subscript of $\mathrm{St}_K(k)$. For every $f\in \mathsf{LC}(X,|K|)$ and $k\in \mathrm{dom}(K)$ set $U_{f,k}:= f^{-1}(\mathrm{St}(k))$; for each such $f$ collect these sets into an open cover $\mathcal{U}_f:=(U_{f,k}\colon k\in \mathrm{dom}(K))$ of $X$. By  Lemma \ref{L:combinatorialapproximations}(3), for each $f$, there exists some $\alpha\in\mathcal{N}^*$ so that $\mathcal{U}_f\preceq\mathcal{U}^{\mathrm{cl}}_{\alpha}$. Let $\alpha_f$ be the $\leq_{\mathrm{lex}}$-least such; this specification is well-defined by the fact that any $\leq_{\mathrm{lex}}$-decreasing sequence in $\mathcal{N}^*$ has a limit, together with condition (L1) of Definition \ref{DefCovers2}.
\begin{claim}
The assignment $f\mapsto \alpha_{f}$ is Borel.
\end{claim}
\begin{proof}[Proof of Claim]
Fix $t\in (\mathbb{N}^{<\mathbb{N}})^*$ with $t= \alpha_{f}|n$ for some $n\in \mathbb{N}$. We will show that the set of all $g\in \mathsf{LC}(X,|K|)$ with $\alpha_g \in \mathcal{N}^*_t$ is Borel.

To that end, take $g$ with $\alpha_g \in \mathcal{N}^*_t$ and observe that there are only finitely many $s\in (\mathbb{N}^{n})^*$ with $s<_{\mathrm{lex}} t$. If for some $s<_{\mathrm{lex}} t$ we had  $\mathcal{U}_g\preceq\mathcal{U}^{\mathrm{cl}}_{s}$, then by Definition \ref{DefCovers2} (E1) there would exist some $\beta\in\mathcal{N}^*_s$ with $\mathcal{U}_g\preceq\mathcal{U}^{\mathrm{cl}}_{\beta}$. As $\beta<_{\mathrm{lex}} \alpha_g$, this is a contradiction. Hence the condition $\alpha_g \in \mathcal{N}^*_t$ is equivalent to:
\begin{enumerate}
\item for every $U\in \mathcal{U}_t$ there is some $k\in\mathrm{dom}(K)$ with $g(\mathrm{cl}(U))\subseteq \mathrm{St}(k)$; and
\item for every $s<_{\mathrm{lex}} t$ in $(\mathbb{N}^{n})^*$ there is a $U\in \mathcal{U}_s$ such that $g(\mathrm{cl}(U))\not\subseteq \mathrm{St}(k)$ for all $k\in\mathrm{dom}(K)$.
\end{enumerate}
Condition (1) is open and condition (2) is closed.
\end{proof}

Similarly, for every $f\in \mathsf{LC}(X,|K|)$ let $p_f\colon S\to \mathrm{dom}(K)$ be the map which sends every $U\in \mathcal{U}_{\alpha_f}$ to the least $k\in \mathrm{dom}(K)=\mathbb{N}$ with $f(\mathrm{cl}(U))\subseteq \mathrm{St}(k)$ and with $p_f(U)=0$ for $U\not\in \mathcal{U}_{\alpha_f}$.
\end{proof}

\subsection{The definable version of Huber's theorem}
\label{SS:definable_version_of_Huber}

Let $G$ be a countable abelian group.  As noted above, $K(G,n)$ will denote for us a polyhedral Eilenberg-MacLane space of type $(G,n)$, despite the fact that such a space is only well-defined up to homotopy equivalence. It will also be notationally simplifying to conflate $K(G,n)$ with its underlying simplicial complex in a few places below.
The following is the definable version of Huber's theorem \cite{huber_homotopical_1961} (that $[X,K(G,n)]_{\mathrm{def}}$ is a definable group follows from Theorem \ref{Theorem:homotopy-idealistic} and Theorem \ref{Theorem:phantom-H-group} below).

\begin{theorem}\label{T:DefinableHub}
Let $X$ be a locally compact Polish space. For every $n\geq 0$ the definable group $[X,K(G,n)]_{\mathrm{def}}$ is naturally definably isomorphic to the group with a Polish cover $\Check{\mathrm{H}}_{\mathrm{def}}^{n}(X;G)$.
\end{theorem}
\begin{proof}
Let $\Check{\mathrm{H}}^{n}(X;G)$ be the classical \v{C}ech cohomology group as defined in the proof of Theorem \ref{T:Cech=Def}. We start by recalling the definition of the Huber's isomorphism $J\colon [X,K(G,n)]\to \Check{\mathrm{H}}^{n}(X;G)$ between abstract groups. For this we follow the exposition from \cite{Ba68}, which explicitly takes (\ref{EQ:Cech}) as its working  definition for \v{C}ech cohomology.

To define $J$ one begins with a \emph{fundamental cohomology class}  $u\in\Check{\mathrm{H}}^{n}(K(G;n),G)$. While the explicit description of $u$ will have no bearing on our definability considerations below, such a $u$ is essentially provided by the Yoneda's lemma. 
More concretely, note that, since $G$ is abelian and $\pi_k(K(G,n))=0$ for $k<n$, by Hurewicz's Theorem (when $n=0$, one may more directly argue the point) we have  that:
\[  \Check{\mathrm{H}}_{n}(K(G,n); \mathbb{Z})\cong\pi_n(K(G,n))\cong G.\]
By the Universal Coefficient Theorem, we then have the following natural isomorphisms:
 \[\Check{\mathrm{H}}^{n}(K(G,n);G)\cong \mathrm{Hom}( \Check{\mathrm{H}}_{n}(K(G,n);\mathbb{Z});G)\cong\mathrm{Hom}(G,G).\]
The element $u\in\Check{\mathrm{H}}^{n}(K(G,n);G)$ is simply the pullback of the identity map $G\to G$ under these isomorphisms.  

Notice now that by the functoriality of \v{C}ech cohomology, every continuous map $f\colon  X\to K(G,n)$ induces an abstract group homomorphism
\[f^*\colon  \Check{\mathrm{H}}^{n}(K(G,n);G)\to \Check{\mathrm{H}}^{n}(X;G).\]
Huber's isomorphism is defined by simply setting $J([f]_{\simeq}):= f^{*}(u)$, for every $[f]_{\simeq}\in  [X,K(G,n)]$; see \cite{Ba68}. 
Our goal is to show that the abstract group  isomorphism $(\psi \circ J) \colon  [X,K(G,n)]_{\mathrm{def}} \to \Check{\mathrm{H}}_{\mathrm{def}}^{n}(X;G)$, where $\psi$ is defined by (\ref{EQ:Cech_psi}) in the proof of the Theorem \ref{T:Cech=Def}, admits a Borel lift $\mathsf{LC}(X,K(G,n))\to \mathrm{Z}^n(X;G)$.

Let  $\mathcal{V}$ be the ``canonical" open cover $\{\mathrm{St}_{K(G,n)}(v)\colon v\in \mathrm{dom}(K(G,n))\}$ of $K(G,n)$, consisting of all open stars. 
We have a simplicial complex isomorphism $\eta \colon  K(G,n)\to \mathrm{Nv}(\mathcal{V})$ induced by the assignment $v\mapsto \mathrm{St}_{K(G,n)}(v)$.
 Since $\mathcal{V}$ is a good cover, i.e., $V_0\cap\cdots \cap V_k$ is either contractible or empty  for all choices of $k\in \mathbb{N}$ and  $V_i\in \mathcal{V}$, there exists $\hat{u}\in \mathrm{Z}^n(\mathrm{Nv}(\mathcal{V}),G)$ so that $u=[\hat{u} +  \mathrm{B}^{n}(\mathrm{Nv}(\mathcal{V});G) ]_{\Check{\mathrm{H}}}$ where, as in the proof of Theorem \ref{T:Cech=Def}, $[a]_{\Check{\mathrm{H}}}$ is the image of $a\in \mathrm{H}^{n}(\mathrm{Nv}(\mathcal{V});G)$ under the inclusion $\mathrm{H}^{n}(\mathrm{Nv}(\mathcal{V});G)\hookrightarrow \Check{\mathrm{H}}^{n}(K(G,n);G)$; see e.g., \cite[Section IX.9]{eilenberg_foundations_1952}. 
 
 By the definition of the assignment $f\mapsto f^*$, see e.g., \cite[Section IX.4]{eilenberg_foundations_1952},  for every $f\in \mathsf{LC}(X,K(G,n))$, the associated element $f^*(u)\in \Check{\mathrm{H}}^{n}\big(X;G\big)$ is given by $[(\hat{u}\circ s_f) +  \mathrm{B}^{n}(\mathrm{Nv}(\mathcal{U}_f);G) ]_{\Check{\mathrm{H}}}$, where 
$\mathcal{U}_f$ is any open cover of $X$ which refines $\{f^{-1}(\mathrm{St}_{K(G,n)}(v))\colon v\in \mathrm{dom}(K(G,n))\}$ and  $s_f\colon \mathrm{Nv}(\mathcal{U}_f)\to K(G,n)$ is any ``canonical" map, i.e., any (necessarily simplicial) map with $s_f(U)= v \implies U\subseteq f^{-1}(\mathrm{St}_{K(G,n)}(v))$. In particular, if $(\alpha_f,p_f)\in \mathrm{SA}_{\boldsymbol{\mathcal{U}}}(X,K(G,n))$ 
is the pair that is associated to $f\in \mathsf{LC}(X,K(G,n))$  by Lemma \ref{L:DefSimplApprox}, then we have:
\begin{equation*}
f^*(u)=[(\hat{u}\circ \eta \circ p_f) +  \mathrm{B}^{n}(\mathrm{Nv}(\mathcal{U}_{\alpha_f});G) ]_{\Check{\mathrm{H}}}.
\end{equation*}
But then, by the definition of $\psi$  in the proof of the Theorem \ref{T:Cech=Def}, and since $\alpha_f$ trivially satisfies $\mathcal{U}_{\alpha_f}\preceq \mathcal{U}_{\alpha_f}$ and $r^{\alpha_f}_{\alpha_f}:=\mathrm{id}$, we have that:
\begin{equation*}
(\psi\circ J) (f) := (\psi\circ f)^*(u)=[\hat{u}\circ \eta \circ p_f]+\mathrm{B}^{n}(\boldsymbol{\mathcal{U}};G). 
\end{equation*}
But then the map $\mathsf{LC}(X,K(G,n))\to \mathrm{Z}^{n}(X;G)$, given by $f\mapsto [\hat{u}\circ \eta \circ p_f]$, is a lift of $\psi\circ J$. It is also Borel by Lemma \ref{L:DefSimplApprox}, since both $\hat{u}$  and $\eta$ are fixed and independent of $f$. By Proposition \ref{Proposition:bij-iso} it follows that  $\psi\circ J$ is a definable isomorphism.
\end{proof}

The following corollaries are all immediate.
\begin{corollary}\label{CorDefCohIndependentU}
Up to definable isomorphism, the group with a Polish cover  $\Check{\mathrm{H}}_{\mathrm{def}}^{n}(X;G)$ does not depend on the choice of covering system $\boldsymbol{\mathcal{U}}$ for $X$.
\end{corollary}

\begin{corollary}\label{CorEssentiallyPolishCover}
The definable group $[X,K(G,n)]$  is essentially --- i.e., is definably isomorphic to --- a group with a Polish cover.
\end{corollary}
Altogether, we have the following; homotopy invariance follows, for example, from Lemma \ref{lemma:homotopyequivalenceinduceshomotopyequivalence}.

\begin{corollary}\label{CorFunctoriality}
For any morphism $g:X\to Y$ in $\mathsf{LC}$, write $g^*$ for the function $$[Y,K(G,n)]\to [X,K(G,n)]:[f]\mapsto [f\circ g].$$
The assignments $X\mapsto \Check{\mathrm{H}}_{\mathrm{def}}^{n}(X;G)$ and $g\mapsto g^*$
 determine a contravariant functor $\mathsf{LC}\to\mathsf{GPC}$ which factors through $\mathsf{Ho}(\mathsf{LC})$; in particular, they determine a functor which maps homotopy equivalent spaces to definably isomorphic groups with Polish cover.
\end{corollary}

Turning now to the category $\mathsf{LCP}$ of locally compact pairs, observe that the same map witnessing the isomorphism between $[X,K(G,n)]$ and $\Check{\mathrm{H}}^n_{\mathrm{def}}(X;G)$ --- namely $([f]\mapsto f^*(u))$ --- will witness that $[(X,A),(K(G,n),*)]$ and  $\Check{\mathrm{H}}^n_{\mathrm{def}}(X,A;G)$ are isomorphic as well. Hence the essential content of Theorem \ref{T:DefinableHub}, which is that this map admits a Borel lift, holds in $\mathsf{LCP}$ by a verbatim argument.

\subsection{Remarks on axioms and notation}

We have arrived by two distinct means to definably isomorphic cohomology groups; this is the sort of circumstance in which mathematicians begin to regard an object as \emph{canonical}. One way to make this impulse precise is via axioms, and these are, indeed, a third main way of characterizing the classical \v{C}ech cohomology groups of locally compact Polish spaces.
In \cite[Theorem 8]{petkova_axioms_1973}, for example, Petkova showed that any cohomology theory coinciding with \v{C}ech cohomology on compact metric spaces and satisfying a natural additivity axiom coincides with \v{C}ech cohomology on locally compact metric spaces as well. We note that definable versions of the main ingredients of this argument --- the Five Lemma and Urysohn's Lemma, for example --- already appear in either the present work or (\emph{in nuce}) in its predecessor \cite{BLPI}, so that the work of axiomatizing \emph{definable} cohomology on locally compact Polish spaces reduces essentially to verifications, over the category of metric compacta, of the more classical Eilenberg-Steenrod and cluster axioms in the definable setting, or more precisely that these axioms determine a cohomology theory up to \emph{definable} isomorphism.
Such verifications (along the lines of Theorem \ref{T:definable_LES}) do not appear to us to be either difficult or particularly illuminating, and would detain us too long from the more interesting decompositions and applications of Sections \ref{S:HomotopyClassification} and \ref{Section:telescopes}, respectively; for these reasons, we leave their proper treatment for another occasion.

One might along different but related lines axiomatize the \emph{reduced} definable \v{C}ech cohomology groups (see \cite[Chapter 19]{may_concise_1999}); this brings us to a more mundane consideration, which is the following. Though it has so far seemed valuable to notationally distinguish between classical and definable \v{C}ech cohomology, it will be convenient in what follows to reserve the subscript position for other purposes. And since we will be so primarily concerned with definable \v{C}ech cohomology in what follows, there is little danger of confusion in denoting its (definable) groups by $\mathrm{H}^n$, simply --- and similarly for the brackets $[-,-]$ ---  and this henceforth will be our practice.

\section{A definable homotopy extension theorem}
\label{S:definable_HET}

In this section, we formulate and prove definable versions of three fundamental topological results, namely Urysohn's Lemma, Borsuk's Homotopy Extension Theorem, and the presentation of the unbased homotopy classes of maps from $(X,\star)$ to $(P,*)$ in terms of the action of $\pi_1(P,*)$ on the set of based homotopy classes of maps between them.
Each of these definable results is applied in the proof of the one which follows it, and by facilitating the passage between pointed and unpointed settings, the last of them both simplifies the argument of some of our results and extends their scope, as we have noted already. 
That said, these definable results play a sufficiently minor role in later sections that readers may skip over this one without much loss of continuity; put differently, we have recorded them as much for their general interest and place in the development of the field as for any particular application herein.

\subsection{The definable version of Urysohn's Lemma} 
\label{SS:definable_Urysohn's}
Let $X$ be a locally compact Polish space. Fix a countable basis $%
\mathcal{B}$ of open sets of $X$. Let $F( X) $ be the collection
of closed subsets of $X$. We regard $F( X) $ as a topological
space with respect to the \emph{Fell topology}. This has a basis of sets of
the form%
\begin{equation*}
\left\{ F\in F( X) :F\cap K=\varnothing ,F\cap U_{1}\neq
\varnothing ,\ldots ,F\cap U_{n}\neq \varnothing \right\}
\end{equation*}%
where $K\subseteq X$ is compact, and $U_{1},\ldots ,U_{n}$ are open
subsets of $X$. This topology renders $F(X) $ a Polish space \cite{fell_hausdorff_1962}.

We let $O(X) $ be the collection of open subsets of $X$. We
regard $O(X) $ as a topological space by declaring the bijection $F( X) \rightarrow O(
X) $, $F\mapsto X\backslash F$ to be a homeomorphism. Observe that the function $\mathcal{B}%
^{\omega }\rightarrow O( X) $ which maps a sequence to its union
is Borel, as is the function $O( X) \rightarrow
F( X):\,U\mapsto \overline{U}$.

\begin{lemma}
\label{Lemma:select-interpolate} The set $S=\{(F,O): F\subseteq O\}$ is a Borel subset of $F(X)\times O(X)$, and there exists a Borel function $f:S\to O(X) $ such that $F\subseteq f(F,O)\subseteq\overline{f(F,O)}\subseteq O$ for all $(F,O)\subseteq S$.
\end{lemma}

\begin{proof}
First fix a cofiltration of $X=\bigcup_{n\in\mathbb{N}} X_n$ of $X$ by compact subsets $X_n$ (see Definition \ref{def:cofiltration}); note then that the collection $K(X)$ of compact subsets of $X$ is a Borel subset of $F(X)$, since $F\in F(X)$ is in $K(X)$ if and only if there exists an $n\in\mathbb{N}$ such that $F\subseteq X_n$.
Observe also that for any compact $K\subseteq X$, the function $f_K:F(X)\to K(X):F\mapsto F\cap K$ is Borel.
To see this, fix a compatible metric $d$ on $X$ and for all $r>0$ let $B(K,r)=\{x\in X:\mathrm{inf}\{d(x,y):y\in K\}<r\}$ and observe that $\{F: F\cap K\cap U\neq\varnothing\}=\bigcap_{n\in\mathbb{N}}\{F:F\cap B(K,\frac{1}{n+1})\neq\varnothing\}$.
\begin{claim} The set $S_{c}=\{(K,O): K\subseteq O\}$ is a Borel subset of $K(X)\times O(X)$.
\end{claim}
\begin{proof}[Proof of Claim]
Fix a dense $\{x_n:n\in\mathbb{N}\}\subseteq X$ and, as above, for any $x\in X$ and $r>0$ write $B(x,r)$ for the open ball about $x$ of $d$-radius $r$. Observe that $K\subseteq O$ if and only if there exist an $N$ and $i_0,\dots,i_N$ in $\mathbb{N}$ and positive rational $q_0,\dots,q_N$ such that $K\subseteq \overline{B(x_{i_0},q_0)\cup\dots\cup B(x_{i_N},q_N)}$ and $B(x_{i_0},2q_0)\cup\dots\cup B(x_{i_N},2q_N)\subseteq O$.
\end{proof}
To establish our first assertion, it now suffices to observe that $F\subseteq O$ for $(F,O)\in F(X)\times O(X)$ if and only if $f_{X_n}(F)\subseteq O$ for all $n\in\mathbb{N}$.

For the second assertion, let $\triangleleft$ well-order $\mathcal{B}^2$ in order-type $\omega$ and recursively define $(U_k,V_k)$ to be the $\triangleleft$-least $(U,V)$ such that
\begin{itemize}
\item $U_k\cap V_k=\varnothing$,
\item $U_k\subseteq O$ and $V_k\cap F =\varnothing$,
\item either $F\subseteq\bigcup_{i<k} U_i$ or $U_k\cap(F\backslash\bigcup_{i<k}U_i)\neq\varnothing$, and
\item letting $P=X\backslash O$, either $P\subseteq\bigcup_{i<k} V_i$ or $V_k\cap(P\backslash\bigcup_{i<k}V_i)\neq\varnothing$.
\end{itemize}
This defines for each $k\in\mathbb{N}$ a Borel function $S\to\mathcal{B}^2$ given by $(F,O)\mapsto (U_k,V_k)$.
The map $(F,O)\mapsto\bigcup_{k\in\mathbb{N}}U_k$ is then a Borel function $f:S\to O(X)$ as desired.
\end{proof}

In the setting of $\mathsf{LC}$, Urysohn's Lemma takes the following form: if $X$ is a locally compact Polish
space and $A,B$ are disjoint closed subsets of $X$ then there exists a
continuous function $f:X\rightarrow [ 0,1] $ such that $f[A]=0$ and $f[B]=1$. The following should be regarded as the \emph{definable version} of this statement.

\begin{lemma}
\label{Lemma:definable-Urysohn}Let $X$ be a locally compact Polish space and let $\mathcal{P}$ denote the Borel set of pairs $(A,B) \in
F( X) \times F(X) $ such that $A\cap B=\varnothing $.
For each such pair $(A,B) $ there exists a map $f:X\rightarrow \left[ 0,1\right] $ such that $f[A]=0$ and $f[B]=1$; this map $f$ may moreover be taken to depend in a Borel way on the pairs $(A,B)$, in the sense that there exists a choice function 
\begin{equation*}
\mathcal{P}\rightarrow \mathsf{LC}(X,[0,1]),\;(A,B) \mapsto f
\end{equation*}
witnessing this assertion which is Borel.
\end{lemma}

\begin{proof}
We adopt the notation of the proof of Urysohn's Lemma in \cite[Theorem
33.1]{munkres_topology_1975}. Endow $O( X)^{\mathbb{Q}}$ with the product topology. Let $\mathcal{Z}$ be the Borel subset of $O(
X)^{\mathbb{Q}}$ consisting of families $\left(
U_{p}\right) _{p\in \mathbb{Q}}\in O( X) ^{\mathbb{Q}}$ such that 
$\overline{U}_{p}\subseteq U_{q}$ whenever $p<q$, $U_{p}=\varnothing $ for $%
p<0$, and $U_{q}=X$ for $q>1$.

As in the proof of \cite[Theorem 33.1]{munkres_topology_1975}, for any 
disjoint $A$ and $B$ in $F( X) $, there exists a $\left( U_{p}\right)
_{p\in \mathbb{Q}}\in \mathcal{Z}$ such that $A\subseteq U_{0}$ and $U_{1}=X\backslash B$; letting $f( x) :=\mathrm{\inf }\{p\in \mathbb{Q}:x\in U_{p}\}$ then determines a function $f:X\rightarrow \left[ 0,1\right] $ as desired. The sequence $\left( U_{p}\right) _{p\in \mathbb{Q}}$ is defined by
recursion, with respect to an enumeration $\left( p_{n}\right) _{n\in \omega
}$ of $\mathbb{Q}$, as follows. One lets $U_{1}=X\backslash B$ and then
chooses $U_{0}$ so that 
\begin{equation*}
A\subseteq U_{0}\subseteq \overline{U}_{0}\subseteq U_{1}.
\end{equation*}%
By Lemma \ref{Lemma:select-interpolate}, $U_{0}$ can be chosen in a Borel
fashion. Suppose that $U_{p_m}$ has been defined for $m<n$ (as is vacuously the case if $n=0$), and that $p_n\notin\{0,1\}$. If $p_n<0$ then let $U_{p_n}=\varnothing$ and if $p_n>1$ then let $U_{p_n}=X$. If $p_n\in(0,1)$ then let $p$ and $q$ be the immediate predecessor and immediate
successor, respectively, of $p_n$ in $\{p_m:m<n\} $
with respect to the standard ordering of $\mathbb{Q}$, and let $U_{p_n}$ be such that%
\begin{equation*}
U_{p}\subseteq U_{p_n}\subseteq \overline{U}_{p_n}\subseteq U_{q}\text{.}
\end{equation*}%
Again, by Lemma \ref{Lemma:select-interpolate}, $U_{p_n}$ may be chosen in a Borel fashion.

It remains only to show that the
function $\mathcal{Z}\rightarrow \mathsf{LC}( X,[0,1]) $, $(U_{p}) _{p\in \mathbb{Q}}\mapsto f$ given by%
\begin{equation*}
f(x) :=\mathrm{\inf }\{p\in \mathbb{Q}:x\in U_{p}\}
\end{equation*}%
is Borel. This, though, is immediate: for any $b\in [0,1] $ and compact $%
K\subseteq X$, we have that $f(K) \subseteq( -\infty ,b) $ if and only if there exist $\ell \in \omega $ and 
$p_{0},\ldots ,p_{\ell }\in \mathbb{Q}\cap ( -\infty ,b) $ such
that $K\subseteq U_{p_{0}}\cup \cdots \cup U_{p_{\ell }}$. Similarly,  $f(K) \subseteq ( b,+\infty)$ if and
only if there exists a $q\in \mathbb{Q}$ such that $q>b$ and $[U_{p}\cap K\neq \varnothing\,\Rightarrow p>q]$ for every $p\in 
\mathbb{Q}$.
This concludes the proof.
\end{proof}

\subsection{The definable version of the Homotopy Extension Theorem}
\label{SS:defhomotopyextension}

We recorded the classical framing of Borsuk's Homotopy Extension
Theorem for locally compact Polish spaces as Theorem \ref{Theorem:definable-extension} above. Implicit in its proof is its \emph{definable version}:

\begin{theorem}
\label{Theorem:definable-homotopy-extension}Suppose that $A$ is a closed subspace of a locally compact
Polish space $X$ and that $P$ is a polyhedron. Then for every map $g:\left( A\times I\right) \cup \left(
X\times \left\{ 0\right\} \right) \rightarrow P$, there exists a map $\tilde{%
g}:X\times I\rightarrow P$ which extends $g$. Furthermore, $\tilde{g}$ can be
chosen in a Borel fashion from $g$, in the sense that there exists a choice function 
\begin{equation*}
\mathsf{LC}(( A\times I) \cup ( X\times \left\{
0\right\} ) ,P) \rightarrow \mathsf{LC}( X\times I,P) 
\text{, }g\mapsto \tilde{g}
\end{equation*}%
witnessing this assertion which is Borel.
\end{theorem}

\begin{proof}
By \cite[Chapter I.3, Theorem 2]{mardesic_shape_1982} we may assume that  $P$
is a closed subspace of a convex subset $D$ of a Banach space $L$. Since $P$ is an absolute neighborhood retract 
\cite[Appendix 1, Theorem 1.11]{mardesic_shape_1982}, there exists a
neighborhood $N$ of $P$ in $D$ and a retraction $s:N\rightarrow P$.

Consider now a map $g:( A\times I) \cup ( X\times \left\{
0\right\} ) \rightarrow P$. By \cite[Chapter I.3, Theorem 2]%
{mardesic_shape_1982}, there exists a map $G:X\times I\rightarrow D$ that
extends $g$. It is plain from the proof of this theorem that $G$ depends in a Borel fashion on $g$, and that one may in a Borel manner choose along with $G$ an open subset $V\supseteq A$ of $X$ such that $G( V\times I) \subseteq N$. By
the definable version of Urysohn's Lemma, one may, in
a Borel fashion, then choose for each $V$ a continuous map $\phi :X\rightarrow I$ such that $%
\phi |_{A}=1$ and $\phi |_{X\setminus V}=0$. As in the proof of \cite[%
Chapter I.3, Lemma 2]{mardesic_shape_1982}, $\phi$ determines an extension of the inclusion map 
\begin{equation*}
\left( A\times I\right) \cup \left( X\times \left\{ 0\right\} \right)
\rightarrow \left( V\times I\right) \cup \left( X\times \left\{ 0\right\}
\right)
\end{equation*}%
to the continuous map%
\begin{equation*}
r:X\times I\rightarrow \left( V\times I\right) \cup \left( X\times \left\{
0\right\} \right) \text{, }\left( x,t\right) \mapsto \left( x,\phi(
x) t\right)\text{.}
\end{equation*}%
To conclude the construction, let $%
\tilde{g}=s\circ G\circ r:X\times I\rightarrow P$.
\end{proof}

\subsection{The definable relation of based and unbased homotopy classes of maps}
\label{SS:definable_relation_based_and_unbased}
The following theorem, which relies on both our definable homotopy extension and homotopy selection theorems, is the definable version of the well-known isomorphism $[(X,\star),(P,\ast)]/\pi_1(P,\ast)\cong [X,P]$ \cite[\S 7.3]{spanier_algebraic_1995}.

\begin{theorem}
\label{theorem:definable_action_of_pi1}
Let $P$ be a path-connected polyhedron with basepoint $\ast$ and let $X$ be an object of $\mathsf{LC}_*$.
There exists a definable right action of the definable group $[(S^1,\hexstar),(P,\ast)]$ on the semidefinable set $[(X,\star),(P,\ast)]$ whose semidefinable set of orbits $[(X,\star),(P,\ast)]/[(S^1,\hexstar),(P,\ast)]$ is definably isomorphic to the semidefinable set $[X,P]$.
\end{theorem}
\begin{proof}
We describe a Borel function
$$F:\mathsf{LC}_*((X,\star),(P,\ast))\times \mathsf{LC}_*((S^1,\hexstar),(P,\ast))\to\mathsf{LC}_*((X,\star),(P,\ast))$$
which induces the action in question.
$F$ is defined as follows.
Any $(f,\gamma)\in\mathsf{LC}_*((X,\star),(P,\ast))\times \mathsf{LC}_*((S^1,\hexstar),(P,\ast))$ naturally identifies with a map $f\vee\gamma:X\vee I\to P$. By the Homotopy Extension Theorem, this map extends to a homotopy $h^{(f,\gamma)}:X\times I\to P$ for which $h^{(f,\gamma)}(-,0)=f$ and $h^{(f,\gamma)}(\ast,-)=\gamma$ and $h^{(f,\gamma)}(-,1)\in\mathsf{LC}_*((X,\star),(P,\ast))$; moreover, by Theorem \ref{Theorem:definable-homotopy-extension} these extensions may be chosen in a Borel fashion.
Assume that they have been, and let $F(f,g)=h^{(f,\gamma)}(-,1)$ for each $(f,\gamma)\in\mathsf{LC}_*((X,\star),(P,\ast))\times\mathsf{LC}_*((S^1,\hexstar),(P,\ast))$.

Clearly $F$ is Borel. To see that $F$ induces an action as claimed, suppose the pairs $f,g:(X,\star)\to(P,\ast)$ and $\gamma,\delta:(S^1,\hexstar)\to(P,\ast)$ are each homotopic; together with $h^{(f,\gamma)}$ and $h^{(g,\delta)}$, these homotopies determine a map $r$ from
$$\left(X\times I\times\{0\}\,\cup\,(\{\star\}\times I\,\cup\,X\times\{0,1\})\times I\right)\subseteq X\times I\times I$$
to $P$.
More precisely, $r$ is defined by identifying $r\restriction X\times I\times\{0\}$ with $f\Rightarrow g$, $r\restriction \{\star\}\times I\times I$ with $\gamma\Rightarrow\delta$, and $r\restriction X\times\{0\}\times I$ and $r\restriction X\times\{1\}\times I$ with $h^{(f,\gamma)}$ and $h^{(g,\delta)}$, respectively.
By the Homotopy Extension Theorem, this map extends to an $H:X\times I\times I\to P$, and $H(-,-,1)$ is then the desired basepoint-preserving homotopy from $F(f,\gamma)$ to $F(g,\delta)$.
This shows that $F$ induces a well-defined operation at the level of the quotients; that this operation is a right action is then immediate.

To see the isomorphism in question, regard $[(X,\star),(P,\ast)]/[(S^1,\hexstar),(P,\ast)]$ as the semidefinable set $(Y,E)$ wherein $Y=\mathsf{LC}_*((X,\star),(P,\ast))$ and $E$ is the equivalence relation defined by $f\,E\,g$ if and only if there exists an $\alpha:(X,\star)\to(P,\ast)$ such that $[f]\cdot [\alpha]=[g]$. Observe then that the map $\mathsf{LC}_*((X,\star),(P,\ast))\to \mathsf{LC}(X,P):f\mapsto f$ induces a definable function $$\phi:[(X,\star),(P,\ast)]/[(S^1,\hexstar),(P,\ast)]\to [X,P].$$
We describe a definable inverse $\psi$ to $\phi$.
By Proposition \ref{Proposition:select-homotopy2}, we may, in a Borel fashion, choose for each $f\in\mathsf{LC}(X,P)$ a homotopy $\gamma[f]:\{\star\}\times I\to P$ with $\gamma[f](\star,0)=f(\star)$. As above, we may then apply the definable Homotopy Extension Theorem to, in a Borel fashion, extend each $f\vee\gamma[f]$ to an $h^{(f,\gamma[f])}:X\times I\to P$.
The definable function $\psi$ is then that induced by the assignments $f\mapsto h^{(f,\gamma[f])}(-,1)$.
The verification that $\psi$ is well-defined is almost exactly as before: within the framework of this construction, any $f\Rightarrow g:X\to P$ induces maps from ``walls'' of $X\times I\times I$ which, by the Homotopy Extension Theorem, extend to a map $X\times I\times I\to P$ whose restriction to $X\times I\times \{1\}$ defines a homotopy from $h^{(f,\gamma[f])}(-,1)$ to $h^{(g,\gamma[g])}(-,1)$.
Unlike before, this homotopy need not be basepoint-preserving. Its restriction to $\{\star\}\times I$, however, determines an $\alpha:(S^1,\hexstar)\to (P,\ast)$, from which it follows that $[h^{(f,\gamma[f])}(-,1)]\cdot[\alpha]=[h^{(g,\gamma[g])}(-,1)]$, as desired. That $\psi$ is both a right and left inverse of $\phi$ is now immediate from their definitions.
\end{proof}
By an easy corollary, the semidefinable sets $[(X,\star),(P,\ast)]$ and $[X,P]$ may often be definably identified.
\begin{corollary}
\label{cor:based_hspace}
Let $X$ and $P$ be as above. If either (i) $P$ is simply-connected, or (ii) $P$ is an $H$-space with identity element $\ast$, then $[(X,\star),(P,\ast)]$ is definably isomorphic to $[X,P]$.
\end{corollary}
\begin{proof} In both cases the action of $[(S^1,\hexstar),(P,\ast)]$ on $[(X,\star),(P,\ast)]$ is trivial. In case (i), this is because the group is trivial. In case (ii), note first that without loss of generality, in the $H$-space structure on $P$, the homotopies of $\mu(-,\ast)$ and $\mu(\ast,-)$ with the identity may each be taken to be basepoint-preserving (see \cite{hatcher_solution}).
Note next that for any $(f,\gamma)\in\mathsf{LC}_*((X,\star),(P,\ast))\times \mathsf{LC}_*((S^1,\hexstar),(P,\ast))$, the operation $\mu$ on $P$ determines a homotopy $$h:\,X\times I\to P:\,(x,s)\mapsto \mu(f(x),\gamma(s)),$$
where $\gamma$ is viewed as a map $(I,\{0,1\})\to(P,\ast)$.
Since both $h(-,0)$ and $h(-,1)$ are (basepoint-preserving) homotopic to $f$ and $h(\star,-)$ is (basepoint-preserving) homotopic to $\gamma$, we conclude that $[f]\cdot[\gamma]=[f]$.
\end{proof}
\begin{remark} An upshot of our definable version of Huber's Theorem is that the assignments of definable cohomology groups to spaces may be developed into \emph{definable cohomology functors} from either a combinatorial or homotopical perspective. In this direction, we note in passing that a generalization of the preceding arguments constructs the definable connecting homomorphisms $\partial^n:\mathrm{H}^n(A;G)\to\mathrm{H}^{n+1}(X,A;G)$ in the long exact cohomology sequence associated to a locally compact pair $(X,A)$ from the \emph{homotopical} perspective, in counterpoint to Section \ref{SS:combinatorial_cohomology_of_pairs}. The basic ingredients are application of the generalized reduced suspension operation $\bar{\Sigma}$, introduced in Section \ref{SS:NSFW:overexplicit_homotopies} below, to the subspace $A$ (the key $S^1$ term above may be regarded as $\bar{\Sigma}(A)$ for $A=\star$), together with the operation of its adjoint $\Omega$ on $K(G,n)$.
\end{remark}
\section{The homotopy classification of maps and phantom maps}\label{S:HomotopyClassification}

In this section we study the homotopy relation on maps $(X,A) \rightarrow (P,\ast)$ where $(X,A)$ is a locally compact pair and $P$ is a
path-connected pointed polyhedron with distinguished point $\ast$; notice
that this includes the case when $A=\varnothing$, whereupon
the problem reduces to the classification of maps $X\rightarrow P$.
Working, in other words, in the setting of $\mathsf{LCP}$ from a perspective essentially subsuming the cases of $\mathsf{LC}_*$ and $\mathsf{LC}$ affords a certain streamlining of arguments, requiring only some extra care around the operation of suspension; see the remarks early in Section \ref{SS:NSFW:overexplicit_homotopies}.
Below, we will let $\ast$ denote the map $(X,A) \rightarrow (P,\ast)$ which is constantly equal to $\ast $, and say that a map $f:(
X,A) \rightarrow (P,\ast)$ is \emph{nullhomotopic} if
there is a homotopy $h:f\Rightarrow \ast :(X,A)
\rightarrow (P,\ast)$. In this way, we regard $[(X,A),(P,\ast)]$ as a pointed semidefinable
set with distinguished element equal to the homotopy class of $\ast$.
For any cofiltration $(X_n, A_n)_{n\in\mathbb{N}}$ of $(X,A)$ by compact pairs, we may consider also those maps $f$ for which each $f\restriction X_n$ is nullhomotopic; these are the \emph{phantom maps} from $(X,A)$ to $(P,*)$, and they form this section's main focus. Our primary results herein are the following:
\begin{itemize}
\item A series of decompositions of $[(X,A),(P,\ast)]$ in terms of its class of phantom maps (Theorems \ref{Theorem:phantom1} and \ref{Theorem:phantom-H-space}), culminating in the case when $P$ is an $H$-group (Theorem \ref{Theorem:phantom-H-group}), whereupon this decomposition takes the form of a short exact sequence specializing in Proposition \ref{Proposition:asymptotic-cohomology} to a Milnor-type exact sequence of definable cohomology groups.
\item Using this decomposition, we show that $[(X,A),(P,\ast)]$ is a definable group whenever $P$ is an $H$-group, thereby generalizing the results of Section \ref{S:Huber}.
\item Along the way, we prove a topological characterization of the class of phantom maps: they are the closure in $[(X,A),(P,\ast)]$ of $[\{*\}]$; see Proposition \ref{prop:topological_char_of_phantoms}.
\end{itemize}

\subsection{Cofiltrations and $\mathsf{Ind}_\omega(\mathcal{C})$, and $\mathsf{Pro}_\omega(\mathcal{C})$ and $\mathrm{lim}$ and $\mathrm{lim}^1$}\label{Section:pro_and_ind_categories}

Instrumental in the arguments of Section \ref{S:Definable cohomo} was the existence, for any locally compact Polish space $X$, of a sequence of compact subspaces approximating to $X$; we now fix a slight refinement of this notion.

\begin{definition}
\label{def:cofiltration}
A \emph{cofiltration }of a locally compact pair $(X,A)$ is an increasing sequence $(X_i,A_i)
_{i\in \mathbb{N}}$ of pairs of compact subspaces of $X$ such that $X_{i}\subseteq 
\mathrm{int}(X_{i+1})$ and $A_i=X_i\cap A$ for each $i\in \mathbb{N}$, and $X=\bigcup_{i\in\mathbb{N}}X_i$.
\end{definition}

Such a sequence is naturally viewed as a \emph{direct} or \emph{inductive system} $\mathbf{X}=((X_i,A_i),\eta_i)_{i\in\mathbb{N}}$ with each $\eta_i:(X_i,A_i)\to (X_{i+1},A_{i+1})$ an inclusion map. Such a system, of course, also contains the morphisms $\eta_{i,j}:(X_i,A_i)\to (X_j,A_j)$ for any $i\leq j$, but is fully determined by those of the form $\eta_{i,i+1}$, which we will continue to abbreviate as $\eta_i$.
It will be convenient below to view these and other inductive sequences in any given category $\mathcal{C}$ themselves as objects of a category $\mathsf{Ind}_\omega(\mathcal{C})$. To do so we need only to describe the morphisms of the latter; to better motivate this description, let us first observe the following.
\begin{enumerate}
\item[(i)] Given cofiltrations $\mathbf{X}=((X_i,A_i),\eta^X_i)_{i\in\mathbb{N}}$ and $\mathbf{Y}=((Y_i,B_i),\eta^Y_i)_{i\in\mathbb{N}}$ of locally compact pairs $(X,A)$ and $(Y,B)$, respectively, and a continuous function $f:X\to Y$, there exists for each $i\in\mathbb{N}$ a least $g(i)$ with $f[X_i]\subseteq Y_{g(i)}$. In other words, letting $f_i=f\restriction X_i:X_i\to Y_{g(i)}$ for each $i$, any such $f$ induces a family of morphisms $(f_i,g)_{i\in\mathbb{N}}$ satisfying the following property:
\begin{equation}
\label{eq:indmaps}
g(i)\leq g(j)\textnormal{ and }f_j\,\eta^X_{i,j}=\eta^Y_{g(i),g(j)}\,f_i\textnormal{ for all }i\leq j\textnormal{ in }\mathbb{N}.
\end{equation} 
\item[(ii)] Higher choices for each $g(i)$ above wouldn't, for our purposes, make any essential difference; more broadly, we are much more interested in ``cofinal'' relations among cofiltrations or functions between them than in strict ones. More formally, we would prefer \emph{not} to distinguish between families $(f_i,g)_{i\in\mathbb{N}}$ as above which exhibit the following relation:
\begin{equation}
\label{eq:indequiv}
(f_i,g)_{i\in\mathbb{N}}\thicksim (e_i,h)_{i\in\mathbb{N}}\hspace{.3 cm}\textnormal{if}\hspace{.3 cm}g(i)\leq h(i)\textnormal{ and }e_i=\eta^Y_{g(i),h(i)}\,f_i\textnormal{ for all }i\in\mathbb{N}.
\end{equation}
\end{enumerate}
These considerations lead to the following definition.
\begin{definition}
Fix any category $\mathcal{C}$. The category $\mathsf{Ind}_\omega(\mathcal{C})$ has as objects the inductive sequences $\mathbf{X}=(X_i,\eta_i)_{i\in\mathbb{N}}$ in $\mathcal{C}$, or, in other words, the functors from the partial order category $\mathbb{N}$ to $\mathcal{C}$. Its morphisms are the $\thicksim$-equivalence classes of families $(f_i,g)_{i\in\mathbb{N}}$ of functions $f_i:X_i\to Y_{g(i)}$ coupled with a $g:\mathbb{N}\to\mathbb{N}$ which together satisfy equation \ref{eq:indmaps}, where $\thicksim$ is the equivalence relation generated by line \ref{eq:indequiv} above.
\end{definition}
We may then define $\mathsf{Pro}_\omega(\mathcal{C})$ as $\mathsf{Ind}_\omega(\mathcal{C}^\mathrm{op})^{\mathrm{op}}$. In particular, any contravariant functor $\mathcal{C}\to\mathcal{D}$ induces a functor $\mathsf{Ind}_\omega(\mathcal{C})^{\mathrm{op}}\to\mathsf{Pro}_\omega(\mathcal{D})$; the composition of such a functor with the $\mathrm{lim}$ and $\mathrm{lim}^1$ functors is a basic motif in what follows. The latter functors were reviewed in abelian settings in some detail in \cite[\S 5]{BLPI}; here we will require and review their extension in \cite[Section IX.2]{bousfield_homotopy_1972} to inverse sequences $(G^n,\eta^n)_{n\in\mathbb{N}}$ of possibly \emph{nonabelian} groups. Much as above, $\eta^n$ abbreviates $\eta^{n,n+1}:G^{n+1}\to G^n$; for brevity we will also sometimes term such inverse sequences \emph{towers}. To simplify notation, we adopt the sometimes tacit convention that towers and the arguments of $\mathrm{lim}$ and $\mathrm{lim}^1$ are always indexed by $n\in\mathbb{N}$.

Writing $\mathsf{CGrp}$ for the category of countable groups, it is straightforward first of all to see that the inverse limit defines a functor from $\mathsf{Pro}_\omega(\mathsf{CGrp})$ to the category of non-archimedean Polish groups and continuous homomorphisms. Suppose next that $\mathbf{G}=(G^n,\eta^n)$ is a tower in $\mathsf{CGrp}$, and consider the Polish group $C^0(\mathbf{G}):=\prod_{n\in\mathbb{N}}G^n$.
We then have a continuous action of the Polish group $C^0(\mathbf{G})$ on the
Polish \emph{space} $\mathrm{Z}^{1}(\mathbf{G}) =\prod_{n\in
\mathbb{N}}G^n$ defined by $(g \cdot h)
_{n}=g_n\cdot h_n\cdot (\eta^n(g_{n+1}))^{-1}$ for each $n\in
\mathbb{N}$, which we call the \emph{$\mathrm{lim}^{1}$-action of $\mathbf{G}$}. If $\mathrm{B}^{1}(\mathbf{G})$ is the corresponding
orbit equivalence relation on $\mathrm{Z}^{1}(\mathbf{G})$,
defined by setting $h\,\mathrm{B}^{1}(\mathbf{G})\,h^{\prime }$ if and only
if there exists a $g\in C^{0}(\mathbf{G})$ such that $g\cdot
h=h'$, then $\mathrm{lim}^{1}\,\mathbf{G}$ is
the \emph{pointed semidefinable set} $\mathrm{Z}^{1}(\mathbf{G})/\mathrm{%
B}^{1}(\mathbf{G})$, with distinguished point corresponding to the
identity element of $\mathrm{Z}^{1}(\mathbf{G})$ (regarded as a group).

The construction clearly specializes to the more familiar $\mathrm{lim}^1$ of a tower $\mathbf{G}$ of countable \emph{abelian} groups: in this case, $\mathrm{Z}^{1}(\mathbf{G})$ is also an abelian group, and $\mathrm{B}^{1}(\mathbf{G})$
is the coset equivalence relation with respect to a Borel Polishable subgroup
of $\mathrm{Z}^{1}(\mathbf{G})$, hence $\mathrm{%
lim}^{1}\,\mathbf{G}$ is a group with a Polish cover when endowed with the group operation inherited from $\mathrm{Z}^{1}(\mathbf{G})$ --- a fact explored at length in \cite{BLPI}.

\subsection{Weak homotopy}

As indicated, the broad focus of the remainder of this section is the sets or groups of homotopy classes of maps from a locally compact pair $(X,A)$ to a pointed polyhedron $(P,*)$; based maps $(X,x)\to(P,*)$ or unbased maps $X\to P$ appear as special cases, by letting $A=\{x\}$ or $\varnothing$, respectively.

\begin{definition}\label{definition:weaklyhomotopic} Let $f$ and $g$ be maps from a locally compact pair $(X,A)$ to a pointed polyhedron $(P,*)$.
We say that $f$ and $g$ are \emph{weakly homotopic}, writing $f\simeq_{\mathrm{w}} g$, if for every compact subspace $K$ of $X$, the maps $f|_{K},g|_{K}:(K,K\cap
A) \rightarrow ( P,*) $ are homotopic.
\end{definition}

\begin{lemma}\label{Lemma:weak-homotopy} Let $f$ and $g$ be as in Definition \ref{definition:weaklyhomotopic}. The following are equivalent:
\begin{enumerate}
\item $f$ and $g$ are weakly homotopic,

\item for every subspace $(Y,B)$ of $( X,A)$ that
is homotopy equivalent to a compact pair, $f|_Y,g|_Y:(
Y,B) \to (P,*)$ are
homotopic.
\end{enumerate}
\end{lemma}

\begin{proof}
That (2) implies (1) is obvious; we will show that (1) implies (2). To that end, let $\alpha :(K,L) \rightarrow
(Y,B) $ witness the homotopy equivalence of a compact pair $(K,L)$ with $(Y,B)\subseteq (X,A)$. By assumption, there exists a homotopy $h:f|_{\alpha(K)}\Rightarrow g|_{\alpha(K) }:( \alpha(K) ,\alpha(K) \cap A) \to(
P,*)$.
Thus $\tilde{h}:=h\circ ( \alpha \times \mathrm{id}_{I})
:(K\times I,L\times I) \to ( P,*) $ witnesses that $\tilde{h}( -,0) =f\circ \alpha $
and $\tilde{h}( -,1) =g\circ \alpha $ are homotopic. Since $\alpha :( K,L) \to(Y,B) $ is a homotopy equivalence, this implies that $f|_Y,g|_Y:(Y,B) \to(P,*) $ are homotopic; this concludes the proof.
\end{proof}

Let $[(X,A),(P,*)]_{\mathrm{w}}$ denote the pointed semidefinable set of weak homotopy classes of maps $(X,A) \to (P,*)$.

\begin{lemma}
\label{Lemma:saturate-local}Let $(X,A)$ be a
locally compact pair and let $(P,*)$ be a pointed polyhedron.
The relation $\simeq_{\mathrm{w}}$ of weak homotopy on maps $(X,A)\to(P,*)$ is a closed
equivalence relation with the property that the saturation $[ U] _{\simeq_{\mathrm{w}}}$ of any open subset $U$ of $\mathsf{LCP}((X,A),(P,*))$ is open.

Furthermore, $\simeq_{\mathrm{w}}$ is the closure inside $\mathsf{LCP}((X,A),(P,*))^2$ of
the relation $\simeq$ of homotopy of maps $(X,A)\to(P,*)$.
\end{lemma}

\begin{proof}
Associated to any cofiltration $(X_n,A_n)_{n\in \mathbb{N}}$ of $(X,A)$ are restriction maps \begin{equation*}
f_n:\mathsf{LCP}(( X,A) ,( P,\ast ))\rightarrow \mathsf{LCP}((X_{n},A_n)
,( P,*)).
\end{equation*}
These, clearly, are continuous. The relation of weak homotopy for continuous
maps $(X,A)\to(P,*)$ is the
intersection of the $(f_n\times f_n)$-inverse images of the homotopy relation on $\mathsf{LCP}(( X_n,A_n)
,( P,*))$. It then follows from Lemma \ref{Lemma:countable-homotopy} that $
\simeq_{\mathrm{w}}$ is closed. It is also clear from this characterization (together with the Homotopy Extension Theorem \ref{Theorem:definable-extension} above)
that $\simeq_{\mathrm{w}}$ is the closure of the homotopy relation $\simeq$.

Suppose now that $U\subseteq \mathsf{LCP}((X,A)
,( P,*)) $ is open. We will show that $[U]
_{\simeq_{\mathrm{w}}}$ is open. Since saturation commutes with unions, it will suffice to show this for $U=\{ f\in \mathsf{LCP}(( X,A)
,( P,*)) :f( K) \subseteq W\} $, where $K$ and $W$ are an arbitrary compact subset of $X$ and an open subset of $P$, respectively. Fix an $n\in
\mathbb{N}$ such that $K\subseteq X_n$.

\begin{claim*}
A map $g\in \mathsf{LCP}((X,A) ,( P,*
)) $ belongs to $[U] _{\simeq_{\mathrm{w}}}$
if and only if there exists an $f \in \mathsf{LCP}((
X_n,A_n),( P,*)) $ such
that $f(K)\subseteq W$ and $g|_{X_n}$ and $f$ are
homotopic.
\end{claim*}

\begin{proof}
By the Homotopy Extension Theorem, any $f$ as in the statement of the claim extends to a map $f':(X,A) \rightarrow (P,*) $ which is homotopic to $g$. Since $f'\in U$, then, $g\in [U]_{\simeq_{\mathrm{w}}}$. The converse implication is obvious.
\end{proof}

It is immediate from the claim and Lemma \ref{Lemma:countable-homotopy} that 
$[U] _{\simeq_{\mathrm{w}}}$ is indeed open.
\end{proof}

By Lemma \ref{Lemma:saturate-local}, Lemma \ref{Lemma:selector-idealistic}, and \cite[Theorem 12.16]{kechris_classical_1995}, we now have the following.

\begin{corollary}\label{Corollary:w-definable}
Suppose that $(X,A)$ is a locally compact pair and $P$ is a pointed polyhedron. Then the relation $\simeq_{\mathrm{w}}$
on $\mathsf{LCP}((X,A),( P,*))$ has a Borel selector; there is, in other words, a Borel map $\mathsf{LCP}((X,A) ,( P,*))\to 
\mathsf{LCP}((X,A),( P,*)) 
:f\mapsto s(f)$ such that  $f\simeq
_{\mathrm{w}}s(f)$ for all $f,g\in \mathsf{LCP}((X,A) ,( P,*))$, and $f\simeq
_{\mathrm{w}}g$ if and
only if $s(f)=s(g)$. In particular, $[(X,A),(P,*)] _{\mathrm{w}}$ is a
definable set.
\end{corollary}

Let $P$ be a polyhedron and let $(X_n,A_n) _{n\in \mathbb{N}}$ be a cofiltration for the locally compact pair $(X,A)$. Together with the natural restriction maps, the countable sets $[(X_n,A_n),(P,*)] $ then assemble into an inverse sequence. Endow these sets with the discrete topology and let
\begin{equation*}
\mathrm{lim}\,[(X_n,A_n)
,( P,*)]
\end{equation*}
denote the Polish space obtained as their limit. Having chosen a cofiltration for each locally compact pair $(X,A)$, we may
regard both $[(X,A),(P,*)] _{\mathrm{w}}$ and $\mathrm{lim}\,[(X_n,A_n),(P,*)]$ as functors to
the category of pointed definable sets, contravariant in the first coordinate (i.e., from $\mathsf{LCP}$) and covariant in the second (i.e., from $\mathsf{P}_*$).
\begin{proposition}
\label{Proposition:weak-homotopy}There is a definable bijection
between the pointed definable sets $$[(X,A),(P,*)]_{\mathrm{w}}\textnormal{ and }\mathrm{lim}\,[(X_n,A_n),(P,*)]$$which is natural in each coordinate of the two bifunctors.
\end{proposition}

\begin{proof}
Consider for each $n$ the definable function $[(X,A),(P,*)]_{\mathrm{w}}\to [(X_n,A_n),(P,*)]$
given by restriction; these functions together induce a definable
function $\Phi _{(X,A),(P,*)}:[(X,A),(P,*)]_{\mathrm{w}}\to\mathrm{lim}\,[(X_n,A_n),(P,*)]$.

We noted above that two maps $f,f':(X,A) \rightarrow (P,*)$ are weakly homotopic if and only if $f\restriction X_n$ and $f' \restriction X_n$ are
homotopic for every $n\in\mathbb{N}$; it follows immediately that $\Phi_{(X,A),(P,*)}$ is injective.
To see that $\Phi_{(X,A),(P,*)}$ is
surjective, let $([f_n])_{n\in \mathbb{N}}$ be an element of $\mathrm{lim}\,[(X_n,A_n),(P,*)]$. In particular, $f_n:(X_n,A_n) \rightarrow (P,*) $ and $f_{n+1}\restriction X_n$ are homotopic for each $n\in \mathbb{N}$. By the Homotopy
Extension Theorem, we may, without leaving their homotopy classes, recursively modify each of the functions $f_{n}$ so that $f_{n+1}\restriction X_n =f_n$ for each $n\in \mathbb{N}$. If we then define $f:(X,A) \to (P,*)$
by letting $f(x)=f_{n}(x)$
for $x\in X_{n}$ then $\Phi_{(X,A),(P,\ast)}([f]) =([f_{n}])_{n\in \mathbb{N}}$, as desired.

Naturality with respect to maps $g:(P,*)\to(Q,\star)$ in the second coordinate follows from the observation that $g\circ(f\restriction X_n)=(g\circ f)\restriction X_n$ for every $f:(X,A)\to(P,*)$ and $n\in\mathbb{N}$, and naturality with respect to maps $(X,A)\to(Y,B)$ in the first coordinate follows from the same principle.
\end{proof}

\subsection{The classification of phantom maps}
\label{SS:NSFW:overexplicit_homotopies}

We turn now to the classification of \emph{phantom maps }\cite{mcgibbon_phantom_1995,mardesic_elementary_2015,gray_universal_1993,mcgibbon_phantom_1997,mcgibbon_phantom_1994}. Here we simply call phantom maps what in \cite{mcgibbon_phantom_1995} are
called \emph{phantom maps of the second kind}.

\begin{definition}
\label{Definition:phantom}Let $(X,A)$ be a locally
compact pair and let $P$ be a pointed polyhedron. A \emph{phantom map} from $(X,A)$ to $(P,*)$
is a map which is weakly homotopic to the constant map. We let $\mathrm{Ph}%
((X,A),(P,*)) $ denote the collection of phantom maps from $(X,A)$ to $(P,*)$.
\end{definition}

By Lemmas \ref{Lemma:weak-homotopy} and \ref{Lemma:saturate-local}, $\mathrm{Ph}((X,A),(P,*))$
is a homotopy-invariant closed subspace of the space $\mathsf{LCP}((X,A),(P,*))$. More precisely:
\begin{proposition}\label{prop:topological_char_of_phantoms}
$\mathrm{Ph}((X,A),(P,*))$ is the closure in $\mathsf{LCP}((X,A),(P,*))$ of $[\ast]$, the class of maps homotopic to the constant map.
\end{proposition}
\begin{proof}
As noted, by Lemma \ref{Lemma:saturate-local}, $\mathrm{Ph}((X,A),(P,*))$ is closed in $\mathsf{LCP}((X,A),(P,*))$ and therefore contains the closure of $[\ast]$. For the reverse inclusion, fix $f\in\mathrm{Ph}((X,A),(P,*))$; applying Lemma \ref{Lemma:weak-homotopy} with $g=\ast$ then implies, together with the Homotopy Extension Theorem, that any basic open neighborhood of $f$ in $\mathsf{LCP}((X,A),(P,*))$ contains some $g'$ homotopic to $g$.
\end{proof}
Thus we may
consider the pointed semidefinable set $[(X,A),(P,*)]_{\infty}$ of homotopy classes of phantom
maps $(X,A)\rightarrow (P,*)$.
Just as for $[-,-]_{\mathrm{w}}$, this defines a functor to the category of pointed semidefinable sets which is
contravariant in the first coordinate and covariant in the second.

Observe now that when $A=\{x\}$ for some single element $x$ of $X$ then we may regard $(X,A)$ as a pointed space, and in particular may apply to it the reduced suspension operation $\Sigma:\mathsf{LC}_{*}\to\mathsf{LC}_{*}$ described in Section \ref{SS:H-groups}.
In fact, in the present context of maps $(X,A)\to(P,*)$, in which $A$ functions as little more than a basepoint, it is reasonable to regard $\Sigma$ as an instance of a \emph{generalized reduced suspension} operation $\bar{\Sigma}:\mathsf{LCP}\to\mathsf{LC}_*$ defined for nonempty $A$ as $$\bar{\Sigma}(X,A)=X\times I/(X\times\{0,1\}\,\cup\,A\times I),$$
or, more concisely,
$$\bar{\Sigma}(X,A)=\Sigma(X/A,\star),$$
where $\star$ is the quotient-image of $A$. Doing so will allow us a more general unified statement and argument of this section's main result.
Note, however, that the second of the above formulations raises questions of interpretation when $A=\varnothing$. We take the standard homotopy theoretic approach of defining $X/\varnothing$ to equal what is sometimes (as in \cite{may_more_2012}) denoted $X_{+}$, namely the union of $X$ with a discrete basepoint $\{\star\}$.\footnote{It is perhaps worth remarking that under this convention, $\bar{\Sigma}(X,\varnothing)=\Sigma X_{+}$ is not homotopy equivalent to $SX$, but rather to $SX\vee S^1$ (cf. \cite[p. 106]{may_concise_1999}), but that this additional $S^1$ factor is, in our contexts, of no computational consequence.} These conventions have the virtue of ensuring that any $\bar{\Sigma}(X,A)$ is an $H$-cogroup ($X\mapsto X_{+}$ is left adjoint to the forgetful functor from pointed to unpointed spaces, and extends to a left adjoint to the inclusion $\iota$ of pointed spaces in pairs of spaces; $\bar{\Sigma}$ is simply $\Sigma\circ\iota$, and is consequently left adjoint to the standard loop-space functor $\Omega$ on $\mathsf{LC}_{*}$).

Now assume, as before, that we have assigned a cofiltration $(X_n,A_n)_{n\in\mathbb{N}}$ to each locally compact pair $(X,A)$. These assignments determine an
inductive sequence of compact $H$-cogroups $(\bar{\Sigma}(X_n,A_n))_{n\in\mathbb{N}}$ and thereby, in turn, a tower 
\begin{equation*}
\left(\left[\bar{\Sigma}(X_n,A_n)
,(P,*)\right]\right)_{n\in\mathbb{N}}
\end{equation*}
of countable groups. Taking 
\begin{equation*}
\mathrm{lim}^1\,[\bar{\Sigma}(X_n,A_n)
,(P,*)]
\end{equation*}
then determines a functor to the category of pointed semidefinable sets, which again is
contravariant in the first coordinate and covariant in the second. Note that any other family of choices of cofiltrations for locally compact pairs $(X,A)$ would yield a definably isomorphic
functor.

We now present a \emph{definable} version of a description of the set of
phantom maps which has appeared in several contexts, all of them pointed. It may be found in \cite[%
Section IX.3, Corollary 3.3]{bousfield_homotopy_1972} in the context of
pointed simplicial sets, and in \cite[Section 2.1, Proposition 2.1.9 and
Corollary 2.1.11]{may_more_2012} in the context of inductive sequences of
pointed spaces and cofibrations, for example; see also \cite[Section 5]{mcgibbon_phantom_1995}, where the fundamental insight is attributed to \cite{steenrod_regular_1940}. Readers may also find the heuristic discussion at \cite[p. 1229]{mcgibbon_phantom_1995} valuable.

\begin{theorem}
\label{Theorem:phantom1} For each locally compact pair $(X,A)$, let $(X_n,A_n)_{n\in\mathbb{N}}$ denote the associated cofiltration described above.
There exists a definable isomorphism of the bifunctors $\mathsf{LCP}^{\mathrm{op}}\times\mathsf{P}_*\to\mathsf{SemiDef}_*$ given by $((X,A),(P,*))\mapsto [(X,A),(P,*)]_\infty$ and $((X,A),(P,*))\mapsto \mathrm{lim}^1\,[\bar{\Sigma}(X_n,A_n)
,(P,*)]$ which is natural in each coordinate.
\end{theorem}

The following notation will render portions of our proof simpler and more intuitive. For any $0\leq a,b$ in $\mathbb{R}$ and $f:X\times [0,a]\to Y$ and $g:X\times [0,b]\to Y$ with $f\restriction X\times\{a\}=g\restriction X\times\{0\}$, define $f\cdot g:X\times [0,a+b]\to Y$ by
\begin{equation*}
f\cdot g (x,t) = \left\{ 
\begin{array}{ll}
f(x,t) & t\in [0,a]\text{,} \\ 
g(x,t-a) & t\in [a,b]\text{.}%
\end{array}%
\right.
\end{equation*}
Define also $-f:X\times [0,a]\to Y$ by $-f(x,t)=f(x,a-t)$, and observe that $f\cdot -f$ is homotopic to the map $X\times [0,2a]\to Y$ given by $(x,t)\mapsto f(x,0)$. Lastly, for any real $k\geq 0$, define $kf:X\times[0,ka]\to Y$ by $kf(x,t)=f(x,kt)$.
\begin{proof}
As indicated, we begin by describing a definable function $\varphi:[(X,A),(P,*)]_{\infty}\to\mathrm{lim}^1\,[\bar{\Sigma}(X_n,A_n)
,(P,*)]$.
We will then verify that $\varphi $ is indeed an isomorphism in the category of
pointed semidefinable sets. For ease of reading, let $\mathbf{G}$ denote the tower of groups $([\bar{\Sigma}(X_n,A_n)
,(P,*)])_{n\in\mathbb{N}}$ for the duration of the proof.

By Corollary \ref{Corollary:select-homotopy}, one may in a Borel fashion choose for each phantom map $f:(X,A)\to(P,*)$ and $n\in \mathbb{N}$ a homotopy $h_{n}:\ast \Rightarrow f\restriction X_n:(X_n,A_n)\to (P,*)$. Hence $h_{n}:(X_{n}\times
I,A_{n}\times I) \rightarrow (P,*)$ is a
map such that $h_{n}(\,\cdot\,,0) =\ast $, $h_{n}(\,\cdot\,,1)
=f\restriction X_n$, and $h_{n}(x,\,\cdot\,) =\ast $ for each $x\in A_n$. As above, we let $%
(x,t) \mapsto \langle x,t\rangle$ be the quotient map $X_n\times
I\to \bar{\Sigma}(X_n,A_n)$.\footnote{Of course, if $A_n=\varnothing$ then this isn't quite a quotient map, but the discrepancy from one is so small that it seems clearest to continue to speak of such maps as though they are quotients; we will do so.} Define then $D_n:\bar{\Sigma}(X_n,A_n)\rightarrow
(P,*)$ by
\begin{equation}
\label{eq:D_n}
\langle x,t\rangle \mapsto \left\{ 
\begin{array}{ll}
h_{n}( x,2t) & 0\leq t\leq 1/2\text{,} \\ 
h_{n+1}( x,2-2t) & 1/2\leq t\leq 1\text{.}%
\end{array}%
\right.
\end{equation}%
In the notation introduced just above, $D_n$ is the map induced on $\bar{\Sigma}(X_n,A_n)$ by $1/2(h_n\cdot -(h_{n+1}\restriction X_n\times I)):X_n\times I\to P$. Writing $[D_n] \in [\bar{\Sigma}(X_n,A_n)
,(P,*)]$ for the homotopy
class of $D_n$, we then let $\varphi([f])$ be the element of $\mathrm{lim}^1\,\mathbf{G}$ represented by the sequence $([D_n])_{n\in\mathbb{N}}\in \prod_{n\in\mathbb{N}}G^n=\mathrm{Z}^1(\mathbf{G})$. 

In the following three claims, we verify that $\varphi$ does not depend on our choices of homotopies $h_n$, or of representative of $[f]$, and hence that $\varphi: [(X,A),(P,*)]_{\infty}\to\mathrm{lim}^1\,\mathbf{G}$ is a well-defined basepoint-preserving definable function, as desired.

\begin{claim*}
The element $\varphi([f]) $ of $\mathrm{lim}^1\,\mathbf{G}$ does not depend on the
choice of homotopies $h_{n}:\ast \Rightarrow f\restriction X_n:(X,A) \to (P,*)$.
\end{claim*}

\begin{proof}
Let $f:(X,A) \to (P,*)$ be a phantom map and let $g_n,h_n:\ast \Rightarrow
f\restriction X_n:(X_n,A_n)\to (P,*)$ be homotopies for each $n\in \mathbb{N}$. Suppose that $C_n,D_n:\bar{\Sigma}(X_n,A_n)\rightarrow
(P,*)$ are defined as in equation (\ref{eq:D_n})
from $g_n,g_{n+1}$ and $h_n,h_{n+1}$, respectively.
We will show that the corresponding sequences $([C_n])_{n\in\mathbb{N}},([D_n])_{n\in\mathbb{N}}\in Z^1(\mathbf{G}) $
define the same element of $\mathrm{lim}^1\,\mathbf{G}$. To this end, we define maps $E_n: \bar{\Sigma}(X_n,A_n)\rightarrow
(P,*)$ for $n\in \mathbb{N}$ such that $$[E_n]\cdot[D_n]\cdot [E_{n+1}\restriction \bar{\Sigma}(X_n,A_n)]^{-1}=[C_n] $$
for every $n\in \mathbb{N}$; these maps $E_n$ are defined by
\begin{equation*}
\langle x,t\rangle \mapsto \left\{ 
\begin{array}{ll}
g_n(x,2t) & 0\leq t\leq 1/2\text{,} \\ 
h_n(x,2-2t) & 1/2\leq t\leq 1\text{.}%
\end{array}%
\right. \text{.}
\end{equation*}%
Observe that $E_{n}$ is well-defined for essentially the same reason that $D_n$ is, and that in our alternative notation, $E_n$ is the map induced on $\bar{\Sigma}(X_n,A_n)$ by $1/2(g_n\cdot -h_n):X_n\times I\to P$.
In this notation, $[E_n]\cdot[D_n]\cdot [E_{n+1}\restriction \bar{\Sigma}(X_n,A_n)]^{-1}$ is the homotopy class of the map induced on $\bar{\Sigma}(X_n,A_n)$ by $1/6(g_n\cdot -h_n\cdot h_n\cdot -h_{n+1}\cdot h_{n+1}\cdot -g_{n+1})$, where the last three functions are understood to be restricted to $X_n\times I$; it should be at least intuitively clear that this map is homotopic to the function induced on $\bar{\Sigma}(X_n,A_n)$ by $1/2(g_n\cdot -(g_{n+1}\restriction X_n\times I))$ or, in other words, to $C_n$, as claimed.
For the sake of thoroughness, we record explicit homotopies below.

By definition, $[E_n]\cdot[D_n]\cdot [E_{n+1}\restriction \bar{\Sigma}(X_n,A_n)]^{-1}=[\tilde{D}_n]$, where $\tilde{D}_n:\bar{\Sigma}(X_n,A_n) \rightarrow (P,*)$ is defined by%
\begin{equation*}
\langle x,t\rangle \mapsto \left\{ 
\begin{array}{ll}
g_n( x,6t) & 0\leq t\leq 1/6\text{,} \\ 
h_n( x,1-6(t-1/6)) & 1/6\leq t\leq 2/6\text{,} \\ 
h_n(x,6(t-2/6)) & 2/6\leq t\leq 3/6\text{,} \\ 
h_{n+1}(x,1-6(t-3/6)) & 3/6\leq t\leq 4/6
\text{,} \\ 
h_{n+1}(x,6(t-4/6)) & 4/6\leq t\leq 5/6%
\text{,} \\ 
g_{n+1}(x,1-6(t-5/6)) & 5/6\leq t\leq 1\text{.}%
\end{array}%
\right.
\end{equation*}%
Define $\tilde{h}_{n}:\bar{\Sigma}(X_n,A_n)\times I \rightarrow (P,*)$, a homotopy beginning at $\tilde{h}_n(\,\cdot\,,0)=\tilde{D}_n$, by
setting
\begin{equation*}
\tilde{h}_n(\langle x,t\rangle,s) =\left\{ 
\begin{array}{ll}
\tilde{D}_{n}(\langle x,t\rangle) & t\in [0,1/6]\cup [5/6,1]\text{,}
\\ 
h_n(x,1-6(t-1/6)) & 1/6\leq t\leq (2-s)/6\text{,} \\ 
h_n(x,1-6((2-s)/6-1/6)) & (2-s)/6\leq t\leq
(2+s)/6\text{,} \\ 
h_n(x,6(t-2/6)) & (2+s)/6\leq t\leq 3/6\text{,} \\ 
h_{n+1}(x,1-6( t-3/6)) & 3/6\leq t\leq
(4-s)/6\text{,} \\ 
h_{n+1}(x,1-6((4-s)/6-3/6)) & (4-s)/6\leq
t\leq (4+s)/6\text{,} \\ 
h_{n+1}(x,6(t-4/6)) & (4+s)/6\leq t\leq
5/6\text{.}%
\end{array}%
\right.
\end{equation*}%
for all $\langle x,t\rangle \in \bar{\Sigma}(X_n,A_n)$
and $s\in I$. Observe that $\tilde{h}_n$ is a
homotopy from $\tilde{D}_{n}$ to the function $\bar{\Sigma}(X_n,A_n) \rightarrow (P,*)$ given by%
\begin{equation*}
\tilde{h}_n(\langle x,t\rangle,1) =\left\{ 
\begin{array}{ll}
g_n(x,6t) & 0\leq t\leq 1/6\text{,} \\ 
f( x) & 1/6\leq t\leq 5/6\text{,} \\ 
g_{n+1}(x,1-6( t-5/6)) & 5/6\leq t\leq 1\text{.}%
\end{array}%
\right.
\end{equation*}%
$\tilde{h}(\,\cdot\,,1) $ is clearly homotopic to $C_n:\bar{\Sigma}(X_n,A_n) \rightarrow (P,*)$ in turn. Since we have argued $[E_n]\cdot[D_n]\cdot [E_{n+1}\restriction \bar{\Sigma}(X_n,A_n)]^{-1}=[C_n]$ for arbitrary $n\in\mathbb{N}$, we conclude that the sequences $([D_{n}]) _{n\in\mathbb{N}}$ and $([C_{n}]) _{n\in
\mathbb{N}}$ represent the same element of $\mathrm{lim}^1\,\mathbf{G}$, as desired.
\end{proof}

\begin{claim*}
The element $\varphi([f]) $ of $\mathrm{lim}^1\,\mathbf{G}$ does not depend on the choice of representative $f$ of
the homotopy class $[f]$.
\end{claim*}

\begin{proof}
Fix a homotopy $\rho
:f\Rightarrow f'$ between phantom maps $f,f':(X,A) \rightarrow (P,*) $. We will show that $\varphi([f])=\varphi([f'])$.  Fix homotopies $g_n:* \Rightarrow
f\restriction X_n:(X_n,A_n) \rightarrow ( P,*)$ for each $n\in\mathbb{N}$.
Observe that these induce homotopies $h_n:\ast \Rightarrow f'\restriction X_n:(X_n,A_n)\rightarrow ( P,*)$ defined by
\begin{equation*}
h_n(x,t) =\left\{ 
\begin{array}{ll}
g_{n}(x,2t) & 0\leq t\leq 1/2,\\ 
\rho( x,2(t-1/2)) & 1/2\leq t\leq 1\text{.}%
\end{array}%
\right.
\end{equation*}%
For $n\in\mathbb{N}$ let $D_n$ denote the map $\bar{\Sigma}(X_n,A_n) \rightarrow (P,*)$ defined as above from $g_n,g_{n+1}$, and let $\tilde{D}_n$ denote the map $\bar{\Sigma}(X_n,A_n) \rightarrow (P,*)$ defined as above from $h_n,h_{n+1}$. In our alternative notation, $D_n$ and $\tilde{D}_n$ are the functions induced on $\bar{\Sigma}(X_n,A_n)$ by $1/4(g_n\cdot\rho\cdot -\rho\cdot -g_{n+1})$ and $1/2(g_n\cdot -g_{n+1})$, respectively, where the constituent functions are all restricted to $X_n\times I$. It should now be clear that by arguments just as above, $D_n$ is homotopic to $\tilde{D}_n$ for each $n\in\mathbb{N}$; the provision of explicit homotopies is left to the interested reader.
\end{proof}

\begin{claim*}
The $\varphi$-image of the class $[*]$ of nullhomotopic functions is
the basepoint of $\mathrm{lim}^1\,\mathbf{G}$.
\end{claim*}

\begin{proof}
When $f=\ast $ we may choose $h_{n}:\ast \Rightarrow
f\restriction X_n:(X_n,A_n) \rightarrow (P,*)$ to be the trivial homotopy. In this case, $D_n:\bar{\Sigma}(X_n,A_n) \rightarrow (P,*)$ is the
constant map $*$. Clearly $(D_n)_{n\in\mathbb{N}}$ is the neutral element of $\mathrm{Z}^1(\mathbf{G}) $, hence $\varphi([f]) $ is the basepoint of $\mathrm{lim}^1\,\mathbf{G}$.
\end{proof}

By the foregoing claims, $\varphi:[(X,A),(P,*)]_\infty\to\mathrm{lim}^1\,\mathbf{G}$ is a well-defined basepoint-preserving definable
function.

\begin{claim*}
The function $\varphi:[(X,A),(P,*)] _{\infty}\rightarrow \mathrm{lim}^1\,\mathbf{G}$ is injective.
\end{claim*}

\begin{proof}
We will assume that $\varphi([f]) =\varphi([f'])$ for two phantom maps $f,f':(X,A) \rightarrow (P,*) $ and deduce that $f$ and $f'$ are homotopic.

Let $Y\subseteq X$ be the union of the boundaries $\partial X_{n}$ for 
$n\in \mathbb{N}$. Note that $Y$ is a closed subset of $X$, being the union
of a locally finite family of closed subsets of $X$. Furthermore, as $f$
is a phantom map, each $f\restriction \partial X_{n}$ is nullhomotopic, hence $f\restriction Y$ is nullhomotopic as well. Thus by the
Homotopy Extension Theorem, after replacing $f$ with a phantom map homotopic
to $f$, we may assume that $f\restriction Y=\ast $. By the same reasoning, we may assume that $f'\restriction Y=\ast$ as well.

Fix homotopies $h_n:* \Rightarrow
f\restriction X_n:(X_n,A_n) \rightarrow (P,*)$ and $h'_n:* \Rightarrow
f'\restriction X_n:(X_n,A_n) \rightarrow (P,*)$ for each $n\in \mathbb{N}$. These determine functions $
D_{n}:\bar{\Sigma}\to(P,*)$ and $%
D'_{n}:\bar{\Sigma}\to(P,*)$ in the manner described above, defining in turn the values $\varphi([f])$ and $\varphi([f'])$, respectively.
Since $\varphi([f])=\varphi([f'])$, there exist maps $E_{n}: \bar{\Sigma}(X_n,A_n)\rightarrow (P,*)$ such that 
$[E_n]\cdot[D_n]\cdot [E_{n+1}\restriction \bar{\Sigma}(X_n,A_n)]^{-1}=[D'_n]$
for every $n\in \mathbb{N}$. Therefore, by replacing each $h_{n}:\ast \Rightarrow
f\restriction X_n:(X_n,A_n) \rightarrow (P,*)$ with the homotopy defined by%
\begin{equation*}
\left( x,t\right) \mapsto \left\{ 
\begin{array}{ll}
E_{n}(\langle x,2t\rangle) & 0\leq t\leq 1/2 \\ 
h_{n}( 2t-1) & 1/2\leq t\leq 1%
\end{array}%
\right.
\end{equation*}%
we may assume without loss of generality that $[D_n]=[D'_n]$
for every $n\in \mathbb{N}$. Assume in other words that for each $n\in \mathbb{N}$ there exists a homotopy $
\tilde{h}_{n}:D_{n}\Rightarrow D_{n}^{\prime }:\bar{\Sigma}(X_n,A_n) \rightarrow (P,*)$.

Fix for each $n\in
\mathbb{N}$ a continuous function $\lambda _{n}:X\rightarrow \left[ 0,1%
\right] $ such that $\lambda_n\restriction X_{n-1}=0$ and $\lambda
_{n}\restriction X\backslash\mathrm{int}(X_n)=1$, letting $X_{-1}=\varnothing $. Define the phantom map $g:(X,A)\to(P,*)$ by $x\mapsto D_{n}(\langle x,\lambda
_{n}(x)\rangle) $ for $x\in X_n\setminus X_{n-1}$. To see that $g$ is continuous, observe that for $n\in \mathbb{N}$ and $x\in
\partial X_{n-1}$,%
\begin{equation*}
D_{n}(\langle x,\lambda_{n}(x)\rangle) =D_{n}(\langle x,1\rangle) =\ast =D_{n+1}(\langle x,0\rangle) =D_{n+1}(\langle x,\lambda_{n+1}(x)\rangle) \text{.}
\end{equation*}%
We then have a homotopy $f\Rightarrow g:(X,A)\to(P,*)$, defined by setting $(x,t) \mapsto D_{n}(\langle x,(1-t)/2+t\lambda
_{n}(x)\rangle)$ for all $t\in [0,1] $, $x\in X_n\setminus X_{n-1}$, and $n\in
\mathbb{N}$.

Similarly, $f^{\prime }$ is homotopic to the map $g^{\prime
}:(X,A)\rightarrow (P,*)$ defined by $g^{\prime }( x) =D_{n}^{\prime
}(\langle x,\lambda _{n}(x)\rangle)$ for all $x\in X_{n}\setminus
X_{n-1}$ and $n\in\mathbb{N}$. Hence our task reduces to showing that $g$ and $g^{\prime }$ are homotopic. One may define a homotopy $g\Rightarrow g^{\prime }:(X,A)\to(P,*)$ by%
\begin{equation*}
\left( x,t\right) \mapsto \tilde{h}_{n}(\langle x,\lambda _{n}( x)
\rangle,t)
\end{equation*}%
for $x\in X_{n}\setminus X_{n-1}$ and $t\in I$, where $\tilde{h}_{n}$ is the
homotopy $D_{n}\Rightarrow D_{n}^{\prime }:\bar{\Sigma}(
X_n,A_n) \rightarrow (P,*) 
$ described above. This concludes the proof.
\end{proof}

We have established that $\varphi:[(X,A),(P,*)]_\infty\to\mathrm{lim}^1\,\mathbf{G}$ is an injective basepoint-preserving definable
function. In order to conclude the proof that $\varphi $ is an isomorphism in
the category of pointed semidefinable sets, it will suffice to describe a definable function $\psi:\mathrm{lim}^1\,\mathbf{G}\to [(X,A),(P,*)]_\infty$ which is a right inverse of $\varphi $.

Begin by fixing a continuous function $\lambda _{n}:X\rightarrow I$ as above for each $n\in\mathbb{N}$. We define $\psi$ via representatives $([\tilde{D}_n])_{n\in \mathbb{N}}\in \mathrm{Z}^{1}( \mathbf{G}) $ of elements of $\mathrm{lim}^1\,\mathbf{G}$. Let $\psi ([([%
\tilde{D}_{n}])_{n\in\mathbb{N}}])=[f]$, where $f:(X,A)\to(P,*)$ is the phantom map defined by setting%
\begin{equation*}
f(x):=\tilde{D}_{n}(\langle x,\lambda _{n}(x) %
\rangle)
\end{equation*}
for all $x\in X_{n}\backslash X_{n-1}$.
\begin{claim*}
The map $f:(X,A) \rightarrow \left( P,\ast \right) $
is a well-defined phantom map.
\end{claim*}

\begin{proof}
Just as above, for all $n\in\mathbb{N}$ and $x\in \partial X_{n}$ we have %
\begin{equation*}
f(x)=D_{n}(\langle x,\lambda_{n}(x)\rangle) =D_{n}(\langle x,1\rangle) =\ast =D_{n+1}(\langle x,0\rangle) =D_{n+1}(\langle x,\lambda_{n+1}(x)\rangle) \text{.}
\end{equation*}
This shows that $f$ is well-defined and continuous. Define for each $n\in\mathbb{N}$ a homotopy $h_{n}:\ast \Rightarrow f\restriction X_n:(X_n,A_n) \rightarrow (P,*)$ by setting
\begin{equation*}
h_{n}( x,t) =\tilde{D}_{n}(\langle x,t\lambda _{n}( x)
\rangle) \text{.}
\end{equation*}%
This shows that $f$ is a phantom map.
\end{proof}

\begin{claim*}
Adopt the notation above and suppose that $[f]=\psi ([([\tilde{D}_{n}])_{n\in \mathbb{N}}])$. Then $\phi( \lbrack f]) =[([\tilde{D}_{n}])_{n\in\mathbb{N}}]$.
\end{claim*}

\begin{proof}
The argument amounts to computing $\varphi ([f]) \in \mathrm{lim}^1\,\mathbf{G}$. Let $h_{n}:\ast \Rightarrow f\restriction X_n:(X_n,A_n) \rightarrow ( P,\ast) $ be the homotopy recorded in the proof of the previous claim, and consider the map $D_{n}:\bar{\Sigma}(X_n,A_n) \rightarrow (P,*)$ defined from $h_{n},h_{n+1}$ as in the definition of $\varphi([f])$.
Thus
\begin{equation*}
D_{n}(\langle x,t\rangle) =\left\{ 
\begin{array}{ll}
h_{n}( x,2t) =\tilde{D}_{n}(\langle x,2t\lambda _{n}(
x)\rangle) & 0\leq t\leq 1/2\text{,} \\ 
h_{n+1}( x,2t-1) =\tilde{D}_{n+1}(\langle x,( 2t-1) \lambda
_{n+1}( x)\rangle)=\ast & 1/2\leq t\leq 1\text{.}%
\end{array}%
\right.
\end{equation*}%
Observe that $D_{n},\tilde{D}_{n}:\bar{\Sigma}(X_n,A_n) \rightarrow (P,\ast)$ are homotopic, as
witnessed by the homotopy%
\begin{equation*}
([x,t],s) \mapsto \left\{ 
\begin{array}{cc}
D_{n}(\langle x,2t(1-s)\lambda_{n}(x)+st\rangle) & 0\leq t\leq
1/2+s/2\text{,} \\ 
\ast & 1/2+s/2\leq t\leq 1\text{.}%
\end{array}%
\right.
\end{equation*}%
Therefore $[ D_{n}] =[\tilde{D}_{n}]$ for every $n\in \mathbb{N}$. This shows that $\varphi([f]) =[([\tilde{D}_{n}])_{n\in\mathbb{N}}]$, as desired.
\end{proof}

By the previous claim, the function $\psi :\mathrm{lim}^1\,\mathbf{G}\rightarrow [(X,A),(P,*)]_{\infty}$ is well-defined and basepoint-preserving
and is a right inverse for $\varphi $. Since $\varphi$ is injective, $\varphi$
and $\psi $ are in fact mutually inverse functions. Since $\varphi$ and $\psi$ are
definable functions, they are mutually inverse isomorphisms in the category
of pointed semidefinable sets.

It is quite clear from our construction and claims above that $\varphi$ defines a transformation of functors which is natural in the polyhedral coordinate. Naturality in the first coordinate follows for the same reasons, coupled with the fact that maps $(X,A)\to(Y,B)$ induce $\mathsf{Ind}_\omega$ maps at the level of the cofiltrations in the manner described in Section \ref{Section:pro_and_ind_categories}.
\end{proof}
\subsection{Phantom maps to $H$-spaces}
We now restrict our analysis to phantom maps from a locally compact pair $(X,A)$ to a polyhedral $H$%
-space $(P,\ast
,m)$.
Recall that such a $P$ is a pointed polyhedron endowed
with a map $m:P\wedge P\rightarrow P$ such that the maps $m( \ast
,-)$ and $m( -,\ast) :P\rightarrow P$ are each homotopic to the
identity; $m$ then induces a binary operation on $[(X,A),(P,\ast)]$ with $[\ast]$ as identity element defined by
\begin{equation*}
[f]\cdot[g]=[m\circ(f\wedge g)]
\end{equation*}%
for any maps $f,g:(X,A) \rightarrow (P,\ast)$.
This renders $[(X,A),(P,\ast)]$ a
\emph{semidefinable unital magma}, i.e., a pointed (semidefinable) set with a (definable) binary operation in which the basepoint serves as the neutral element.

It is clear that $[(X,A),(P,\ast)]_{\infty}$ is a semidefinable unital submagma of $[(X,A),(P,\ast)]$. We now show that $[(X,A),(P,\ast)]_{\infty}$ is, in fact, a definable abelian group, and even a group with a Polish cover.

\begin{theorem}
\label{Theorem:phantom-H-space}Suppose that $(X_n,A_n)_{n\in\mathbb{N}}$ is a cofiltration of a locally compact pair $(X,A)$ and $(P,\ast,m)$ is a polyhedral $H$-space. Then $[(X,A),(P,\ast)]_{\infty }$ is a definable abelian group
naturally isomorphic to $\mathrm{lim}^{1}\,[ 
\bar{\Sigma}(X_n,A_n),(P,*)]$; in particular, it is naturally definably isomorphic to a group with a Polish cover.
\end{theorem}
The term \emph{natural} here should be understood in the sense more precisely articulated in the statement of Theorem \ref{Theorem:phantom1} and the conclusion of its proof; as is standard, for concision we will omit those sorts of details from now on.
\begin{proof}
Since $P$ is an $H$-space, $[\bar{\Sigma}(X_n,A_n),(P,\ast)]$ is a countable abelian group for every $n\in\mathbb{N}$, by Lemmas \ref{Lemma:H-space} and \ref{Lemma:countable-homotopy}. Therefore $\mathrm{lim}^{1}\,[\bar{\Sigma}(X_n,A_n),(P,\ast)]$ is a group with a Polish cover and, hence, a definable abelian group.

Thus, it suffices to check that the definable basepoint-preserving map 
\begin{equation*}
\varphi :[(X,A),(P,\ast)]_{\infty
}\rightarrow \mathrm{lim}^{1}\,[\bar{\Sigma}(X,A),(P,\ast)]
\end{equation*}%
defined in the proof of Theorem \ref{Theorem:phantom1} is a magma
homomorphism, i.e., that it satisfies
\begin{equation*}
\varphi \left( \lbrack f]\cdot[f^{\prime }]\right) =\phi \left( \lbrack f]\right)
+\varphi \left( \lbrack f^{\prime }]\right)
\end{equation*}%
for $[f],[f^{\prime }]\in [(X,A),(P,\ast)]_{\infty}$. To this end, suppose that $f,f^{\prime }:X\rightarrow P$ are phantom maps, and choose for each  
$n\in \mathbb{N}$ homotopies $h_{n}:\ast \Rightarrow f|_{X_{n}}$ and $%
h_{n}^{\prime }:\ast \Rightarrow f'|_{X_{n}}$; we then have homotopies $m\circ
\left( h_{n}\wedge h_{n}^{\prime }\right) :\ast \Rightarrow m\circ \left(
f\wedge f^{\prime }\right) |_{X_{n}}$ for each $n\in\mathbb{N}$ as well.

By the definition of $\varphi$, 
\begin{equation*}
\varphi \left( \lbrack f]\right) =\left( [D_{n}]\right) _{n\in \mathbb{N}},
\end{equation*}%
\begin{equation*}
\varphi \left( \lbrack f^{\prime }]\right) =\left( [D_{n}^{\prime }]\right)
_{n\in \mathbb{N}},\textnormal{ and }
\end{equation*}%
\begin{equation*}
\varphi \left( \lbrack f]\cdot\left[ f^{\prime }\right] \right) =\varphi \left(
\lbrack m\circ \left( f\wedge f^{\prime }\right) ]\right) =\left(
[E_{n}]\right) _{n\in\mathbb{N}},
\end{equation*}%
where $D_{n},D_{n}^{\prime },E_{n}:\bar{\Sigma}(X_{n},A_n)\rightarrow
(P,\ast )$ are defined for each $n\in\mathbb{N}$ by
\begin{equation*}
D_{n}(\langle x,t\rangle) =\left\{ 
\begin{array}{ll}
h_{n}( x,2t) & 0\leq t\leq 1/2 \\ 
h_{n+1}( x,2-2t) & 1/2\leq t\leq 1%
\end{array}%
\right.
\end{equation*}%
\begin{equation*}
D_{n}^{\prime }( \langle x,t\rangle) =\left\{ 
\begin{array}{ll}
h_{n}^{\prime }( x,2t) & 0\leq t\leq 1/2 \\ 
h_{n+1}^{\prime }\left( x,2-2t)\right) & 1/2\leq t\leq 1%
\end{array}%
\right.
\end{equation*}%
\begin{eqnarray*}
E_{n}(\langle x,t\rangle) &=&\left\{ 
\begin{array}{ll}
\left( m\circ \left( h_{n}\wedge h_{n}^{\prime }\right) \right) (
x,2t) & 0\leq t\leq 1/2 \\ 
\left( m\circ \left( h_{n+1}\wedge h_{n+1}^{\prime }\right) \right) (x,2-2t)) & 1/2\leq t\leq 1%
\end{array}%
\right. \\
&=&\left( m\circ \left( D_{n}\wedge D_{n}^{\prime }\right) \right) (
\langle x,t\rangle) \text{.}
\end{eqnarray*}%
In short, $E_{n}=m\circ \left( D_{n}\wedge D_{n}^{\prime }\right)$. And since, by Lemma \ref{Lemma:H-space}, the operation on $[\bar{\Sigma}(X_{n},A_n),(P,\ast )]$
defined in terms of the $H$-space structure on $P$ coincides with the
operation induced by the $H$-cogroup structure on $\bar{\Sigma}(X_{n},A_n)$,
\begin{equation*}
\left[ E_{n}\right] =\left[ m\circ \left( D_{n}\wedge D_{n}^{\prime }\right) %
\right] =\left[ D_{n}\right] +\left[ D_{n}^{\prime }\right] \text{,}
\end{equation*}%
and hence $\varphi([f]\cdot[f']) =\varphi([f]) +\varphi([f'])$. This concludes the
proof that $\varphi $ is a magma homomorphism.
\end{proof}
\subsection{The homotopy classification of maps to $H$-groups}

We now consider the even more restrictive case in which $( P,\ast,m)$ is a polyhedral $H$-group. To this end, let $(X,A)$ be a locally compact pair and let $(P,\ast)$ be a pointed polyhedron with $H$-group operation $m:P\wedge P\rightarrow P$. In
this case, the $H$-group structure on $P$ renders $[(X,A),(P,\ast)]$ a (not necessarily abelian)
semidefinable group. The group operation on $[(X,A)
,(P,\ast)]$ is defined as before by setting $[f]\cdot[g]=[m\circ (f\wedge g)]$. Similarly, if $\zeta
:( P,\ast) \rightarrow (P,\ast) $ is a map such
that the map $(P,\ast) \rightarrow (P,\ast):\,
x\mapsto m(x,\zeta(x)) $ is homotopic to the
constant map, then the inverse $[f]^{-1}$ of $[f]\in [(X,A),(P,\ast)]$ is given by $[\zeta \circ f]$.

\begin{lemma}
Suppose that $(X,A)$ is a locally compact Polish
space and $(P,\ast,m)$ is a polyhedral $H$-group. Any two
maps $f,g:(X,A) \rightarrow (P,\ast)$ are weakly homotopic if and only if $
[f]\cdot[g]^{-1}\in [(X,A),(P,\ast)]_{\infty}$.
\end{lemma}

\begin{proof}
Begin by observing that the map $p_{n}:[(X,A),(P,\ast)]\rightarrow [(X_n,A_n),(P,\ast)]$ given by function restriction is a group
homomorphism.
By definition, $[(X,A),(P,\ast)]_{\infty}$ is the intersection of $\{\mathrm{ker}(p _n)\mid n\in\mathbb{N}\}$. Thus $%
f\simeq_w g$ if and only if $p _{n}([f]) =p_{n}([g])$ for every $n\in \mathbb{N}$, if and
only if $p_{n}([f]\cdot[g]^{-1})
=0=[\ast]$ for every $n\in\mathbb{N}$, if and only if $[f]\cdot[g]^{-1}\in [(X,A),(P,\ast )]_{\infty }$.
\end{proof}

Under the present assumptions, $[(X,A),(P,\ast)]$ is, by the foregoing lemmas, a \emph{definable} group.

\begin{theorem}
\label{Theorem:phantom-H-group}Suppose that $(X,A)$
is a locally compact Polish space with cofiltration $(X_n,A_n)_{n\in\mathbb{N}}$ and $(P,\ast,m)$ is a
polyhedral $H$-group. Then:

\begin{enumerate}
\item $[(X,A),(P,\ast)]
_{\infty }$ is a definable abelian group, naturally definably isomorphic to the group with a Polish cover $\mathrm{lim}^{1}\,[\bar{\Sigma}(X_n,A_n),(P,\ast)]$;

\item $[(X,A),(P,\ast)]$
is a definable group;

\item $[(X,A),(P,\ast)]_{\mathrm{w}}$ is a definable group, naturally definably isomorphic to the
pro-countable Polish group $\mathrm{lim}\,[(X_n,A_n),(P,\ast)]$.
\end{enumerate}

Furthermore these groups naturally array in the definable exact sequence of definable groups 
\begin{equation}
\label{Equation:Milnor}
\left\{ \ast \right\} \rightarrow [(X,A),(P,\ast)]_{\infty }\rightarrow [(X,A),(P,\ast)] \rightarrow [(X,A),(P,\ast)]_{\mathrm{w}}\rightarrow
\left\{ \ast \right\} \text{.}
\end{equation}
\end{theorem}
\begin{proof}
(1) This just a particular instance of Theorem \ref{Theorem:phantom-H-space}.

(2): Since the group operations on $[(X,A),(P,\ast)]$ are
definable, $[(X,A),(P,\ast)]$ is a semidefinable group. Furthermore, $[(X,A),(P,\ast)]_{\infty}$ is a definable subgroup of $[(X,A),(P,\ast)]$
by (1). The conclusion that $[(X,A),(P,\ast)]$ is a definable
set now follows from Lemma \ref{Lemma:quotient} and Corollary \ref{Corollary:w-definable}.

(3): It suffices to notice that, under our assumptions, the natural definable bijection in
Proposition \ref{Proposition:weak-homotopy} is a group homomorphism.

The last assertion now follows immediately from definitions.
\end{proof}

\subsection{A definable exact sequence decomposition of \v{C}ech cohomology}
\label{ss:def_coh}

By Theorem \ref{T:DefinableHub} (or its version for $\mathsf{LCP}$ pairs), we may identify the \v{C}ech
cohomology groups $\mathrm{H}^{q}(X,A;G)$ with the representable, or homotopical, cohomology groups $[(X,A),(K( G,q)
,\ast)]$. We may then consider its definable subgroup%
\begin{equation*}
\mathrm{H}_{\infty }^{q}(X,A;G) :=[(X,A),(K(G,q) ,\ast)]_{\infty }\text{,}
\end{equation*}%
which we term the \emph{asymptotic cohomology group}. We may also consider the 
\emph{weak cohomology group} 
\begin{equation*}
\mathrm{H}_{\mathrm{w}}^q(X,A;G) =[(X,A),(K(G,q) ,\ast)]_{\mathrm{w}}\text{.}
\end{equation*}%
Also by Theorem \ref{Theorem:phantom-H-group}, we have a natural definable exact sequence%
\begin{equation*}
0\rightarrow \mathrm{H}_{\infty }^{q}( X,A;G) \rightarrow
\mathrm{H}^{q}(X,A;G) \rightarrow \mathrm{H}_{\mathrm{w}}^{q}(X,A;G) \rightarrow 0\text{.}
\end{equation*}%
More precisely, Theorem \ref{Theorem:phantom-H-group} gives us the following, which shows that this exact sequence is naturally
isomorphic to the exact sequence (\ref{Equation:Milnor}) above.

\begin{proposition}
\label{Proposition:asymptotic-cohomology}Suppose that $q$ is a positive integer, $G$ is a
countable discrete abelian group, and $(X,A)$ is a locally
compact pair with cofiltration $(X_n,A_n)_{n\in\mathbb{N}}$. Then:

\begin{enumerate}
\item $\mathrm{H}_{\infty }^{q}(X,A;G) $ is naturally definably
isomorphic to 
\begin{equation*}
\mathrm{lim}^1\,\mathrm{H}^q(\bar{\Sigma}(X_n,A_n),\ast;G)
\cong \mathrm{lim}^1\,\mathrm{H}^{q-1}(X_n,A_n;G) ;
\end{equation*}

\item $\mathrm{H}_{\mathrm{w}}^{q}(X,A;G)$ is naturally
definably isomorphic to the pro-countable abelian group 
\begin{equation*}
\mathrm{lim}\,\mathrm{H}^{q}(X_n,A_n;G).
\end{equation*}
\end{enumerate}
\end{proposition}
\begin{proof}
Only the isomorphism in item (1) requires comment; it follows from the sequence of definable  isomorphisms
\begin{align*}
\mathrm{H}^q(\bar{\Sigma}(X_n,A_n),\ast;G) & \cong [(\bar{\Sigma}(X_n,A_n),\ast),(K(G,q),\star)]\cong [(X_n,A_n),\Omega(K(G,q),\star)]\\
& \cong [(X_n,A_n),(K(G,q-1),\star)]\cong\mathrm{H}^{q-1}(X_n,A_n;G)
\end{align*}
of countable groups.
\end{proof}
In \cite{lupini_looking_22}, this work's second author showed that the map taking $(G,N)$ to the closure of $\{0\}$ in $G/N$ is functorial in the category $\mathsf{APC}$ of groups with an abelian Polish cover. More precisely, by the $\alpha=0$ case of \cite[Theorem 6.3]{lupini_looking_22} this is a subfunctor of the identity. Hence by Proposition \ref{prop:topological_char_of_phantoms},
the definable functor $\mathrm{H}^q_{\infty}$ as well as its definable quotient $\mathrm{H}^q/\mathrm{H}^q_{\infty}\cong\mathrm{H}^q_{\mathrm{w}}$ may each be definably recovered from the single definable functor $\mathrm{H}^q$, facts we summarize in the following proposition.
\begin{proposition}\label{Proposition:subfunctor}
For any $q\geq 0$ and countable abelian group $G$, the definable functors $\mathrm{H}^q_\infty(-;G)$ and $\mathrm{H}^q_{\mathrm{w}}(-;G)$ each map $\mathsf{LCP}$ to $\mathsf{APC}$. Each of these, moreover, definably derives from the definable functor $\mathrm{H}^q$, in the sense that the first is its postcomposition with a subfunctor of the identity, and the second is its postcomposition with that subfunctor's cokernel.
\end{proposition}

\subsection{Hopf's Theorem}
\label{ss:Hopf}

\v{C}ech cohomology is a main tool for the study and classification of maps up to homotopy. This is due to descriptions, like those above, of \v{C}ech cohomology
groups as groups of homotopy classes of maps, whereby homotopy classification problems for maps between spaces may reduce to a corresponding problem for maps to Eilenberg--MacLane spaces.

One of the first such results was Hopf's theorem for maps to spheres \cite[%
Chapter VII, Theorem 11.5]{hu_homotopy_1959}.
We cite it both for a quick application of our machinery and for use in the following section.
Henceforth we will omit notation of the cohomology coefficient group when $G=\mathbb{Z}$.
 Viewing the $n$%
-dimensional sphere $S^{n}$ as a pointed space, we have that $\mathrm{H}^{n}(S^{n},\star) =\mathbb{Z}$. One may choose, as generator of this group, an 
$[\iota]\in [(S^{n},\star),(K(\mathbb{Z},n),\ast)]$ such that $\iota$ is an
inclusion of $(S^{n},\star)$ as a closed subspace of $(K(\mathbb{Z},n),\ast)$; see \cite[Theorem
2.5.14]{arkowitz_introduction_2011}. Any map $f$ from a locally compact pair $(X,A)$ to $(S^n,\star)$ then determines a map $\iota \circ f:(X,A) \rightarrow ( K(\mathbb{Z},n),\ast) $, a determination amounting to a definable function $[(X,A),(S^n,\star)]\rightarrow\mathrm{H}^{n}(X,A)$.

\begin{theorem}[Hopf]
\label{Theorem:Hopf}
Fix $n\geq 1$. For every polyhedral pair $(P,Q)$ satisfying $\mathrm{H}^q(P,Q) =0$ for $q>n$, the definable function $[(P,Q),(S^{n},\star)]\rightarrow \mathrm{H}^n(P,Q) $ is a
bijection.
\end{theorem}

\begin{corollary}
\label{cor:definable_set}
Under the assumptions of Theorem \ref{Theorem:Hopf}, $[(P,Q),(S^{n},\star)]$ is a pointed definable set.
\end{corollary}

\begin{proof}
This amounts to saying that the relation of homotopy for maps $(P,Q) \rightarrow (S^n,\star)$ is Borel and
idealistic. This relation is idealistic by Theorem \ref{Theorem:homotopy-idealistic}. As the relation of homotopy for maps $(P,Q)\rightarrow (K( \mathbb{Z},n),\ast)$ is Borel by Theorem \ref{Theorem:phantom-H-group}, it follows from
Theorem \ref{Theorem:Hopf} that the relation of homotopy for maps $(P,Q)\rightarrow (S^n,\star)$ is Borel as
well.
\end{proof}

\section{The Borsuk-Eilenberg problem and the definable cohomology of mapping telescopes}\label{Section:telescopes}

This section records several sample applications of the definable cohomology functors and decomposition theorems of the preceding pages. We begin by reviewing the core construction in these applications, namely the \emph{mapping telescope} or \emph{homotopy colimit} of a sequence of maps $\langle f_n:X_n\to X_{n+1}\mid n\in\mathbb{N}\rangle$ of topological spaces. The results of this section reconnect with those of \cite{BLPI}, from which, together with Proposition \ref{Proposition:cohomology-telescopes}, it will follow almost immediately that definable \v{C}ech cohomology is, in strong contrast to classical \v{C}ech cohomology, a complete homotopy invariant both of mapping telescopes of $d$-tori and of $d$-spheres. Equally immediate from the machinery we have developed are solutions to the natural generalizations of the Borsuk-Eilenberg problem of classifying the maps from the complement of a canonically embedded $p$-adic solenoid $\Sigma_p\subset S^3$ to a $2$-sphere. This problem played a critical role in the development of multiple areas of mathematics, in historical senses that we briefly pause to review.
We show that definable cohomology is a complete invariant of the homotopy classes of maps to $S^{d+1}$ from mapping telescopes of $d$-spheres. We conclude with an analysis of the problems of classifying those maps up to homotopy, as well as up to homotopy modulo a homotopy equivalence of their domain. Put differently, we study the problem of \emph{equivariant} classification; the affinity of the associated quotients with the \emph{structure sets} of manifold theory should be noted as well \cite[Definition 2.1]{kasilingam_topological_2016}. We describe lower bounds for the Borel complexity degrees of these problems, the existence of an infinite antichain of complexity degrees among them, and show also that this complexity rises with dimension.

In subsections \ref{subsection:colimits} and \ref{subsection:mappingtelescopes} we conduct our review of homotopy colimits in the unbased category $\mathsf{LC}$. This is solely for conceptual clarity; the discussion applies with only superficial modifications to the context of $\mathsf{LC}_*$, as readers may easily verify. Results in which the homotopy bracket figures thereafter will, by Corollary \ref{cor:based_hspace}, admit interpretation in either category, as we will note below.
\subsection{Colimits and homotopy colimits}
\label{subsection:colimits}
First recall the notion of a \emph{colimit of a diagram in a category $\mathcal{C}$}.
\begin{definition}[\cite{margolis_1983}]\label{colimitdef}
Let $\mathcal{J}$ be a small category; as above, write $\mathcal{J}(W,X)$ for the collection of morphisms in $\mathcal{J}$ from $W$ to $X$. By \emph{a diagram of shape $\mathcal{J}$ in a category $\mathcal{C}$} we simply mean a functor $F:\mathcal{J}\to\mathcal{C}$. For any such diagram a collection of morphisms $\{f_W:F(W)\to Y \mid W\in\mathrm{obj}(\mathcal{J})\}$ is \emph{coherent} if $f_W=f_X\circ F(g)$ for all $g\in\mathcal{J}(W,X)$. An object $Y$ of $\mathcal{C}$ is a \emph{colimit} of a diagram $F$ if it admits a coherent collection of morphisms $\{f_W:F(W)\to Y\mid W\in\mathrm{obj}(\mathcal{J})\}$ such that for any coherent collection $\{g_W:F(W)\to Z\mid W\in\mathrm{obj}(\mathcal{J})\}$ of morphisms to any object $Z$ of $\mathcal{C}$ there exists a unique $h:Y\to Z$ such that $g_W=h\circ f_W$ for all $W\in\mathrm{obj}(\mathcal{J})$.
A \emph{weak colimit} of the diagram $F$ is a $Y$ satisfying these same conditions, but without the requirement that all such maps $h:Y\to Z$ be unique.
\end{definition}

In more concrete contexts, colimits admit more concrete descriptions; for example, the colimit of a \emph{pushout diagram} of topological spaces, i.e., of a diagram of the form
\begin{align}\label{pushoutdiagram}
W\xleftarrow{\hspace{.15 cm}f\hspace{.15 cm}} X\xrightarrow{\hspace{.15 cm}g\hspace{.15 cm}} Y
\end{align}
is \begin{align}\label{colimit}(W\sqcup Y)/\sim,\end{align}
where $\sim$ is the equivalence relation generated by $\{f(x)\sim g(x)\mid x\in X\}$. This brings us to a standard motivating example (see \cite{Dugger}): consider the diagrams
\begin{align}\label{psht1} *\longleftarrow S^{n-1}\longrightarrow D^n\end{align}
and
\begin{align}\label{psht2} *\longleftarrow S^{n-1}\longrightarrow *\end{align}
in which, as usual, $*$ denotes the one-point space, and the only nontrivial map is the inclusion into $D^n$ of its boundary $\partial D^n\cong S^{n-1}$. What interests us is the following: each of the corresponding terms of (\ref{psht1}) and (\ref{psht2}) are homotopy equivalent; the colimits of (\ref{psht1}) and (\ref{psht2}), however, are not (the latter are homeomorphic to $S^n$ and $*$, respectively, as the reader may verify). More formally, what these diagrams together show is that colimits in the topological category are not, in general, homotopy invariant; indeed, as the example might suggest, colimits in the category $\mathsf{Ho}(\mathsf{LC})$ may even fail altogether to exist (see \cite[pp. 245-246]{strom}).

What do more generally exist, on the other hand, are \emph{homotopy colimits}; these are particular representatives of \emph{weak} colimits in homotopy categories in the sense of Definition \ref{colimitdef}. We follow \cite{may_more_2012, arkowitz_introduction_2011} in foregoing their rather abstract general definition, focusing instead on their construction in the contexts which are our immediate interest; their idea, in the process, will grow clear. The first of these is \emph{pushout diagrams} of topological spaces, as in (\ref{pushoutdiagram}), (\ref{psht1}), and (\ref{psht2}) above. The \emph{homotopy colimit} $\mathrm{hocolim}\, D$ of any diagram $D$ of spaces of the form (\ref{pushoutdiagram}) is \begin{align}\label{homotopypushout}\big(W\sqcup (X\times [0,1])\sqcup Y\big)/\sim\,,\end{align}
where the equivalence relation $\sim$ is that given by the identifications $f(x)\sim (x,0)$ and $(x,1)\sim g(x)$ for each $x\in X$. Observe that this construction resolves the discord of (\ref{psht1}) and (\ref{psht2}) above, in the sense that the \emph{homotopy} colimits of these two diagrams are, indeed, homotopy equivalent. Just as we would hope, this holds more generally for any two \emph{naturally homotopy equivalent} pushout diagrams of polyhedra (and even for \emph{naturally weakly homotopy equivalent} pushout diagrams of arbitrary topological spaces; see \cite[p. 5]{Dugger}). We record in the lemma below the special instance of this fact which we will need.

Observe first, though, how many of homotopy theory's most fundamental constructions arise in the above manner: the homotopy colimits of the diagrams
$$*\longleftarrow X\longrightarrow *\,,\hspace{1.2 cm} *\longleftarrow X\xrightarrow{\hspace{.15 cm}g\hspace{.15 cm}} Y\,, \hspace{.3 cm}\text{ and }\hspace{.3 cm} X\xleftarrow{\hspace{.15 cm}\mathrm{id}\hspace{.15 cm}} X\xrightarrow{\hspace{.15 cm}g\hspace{.15 cm}} Y$$
are the suspension $SX$ of $X$ and the \emph{mapping cone} $C(g)$ and \emph{mapping cylinder} $M_g$ of the map $g:X\to Y$, respectively. The last of these is the basic building block of the homotopy colimits of towers of topological spaces, the so-called \emph{mapping telescopes} at the center of our applications below. Note also the natural identification of the homotopy colimit of (\ref{pushoutdiagram}) with two mapping cylinders $M_f$ and $M_g$ glued together along their respective copies of $X$. Hence to show that the homotopy type of the homotopy colimit of (\ref{pushoutdiagram}) depends only upon the \emph{homotopy classes} of the maps $f$ and $g$ it suffices to observe the following.
\begin{lemma}
For any two homotopic maps $g,h:X\to Y$, the mapping cylinders $M_g$ and $M_h$ are homotopy equivalent.
\end{lemma}
The proof is, at the referee's suggestion, left to the reader, but it's worth lingering over its essential mechanism: in contrast with colimits, homotopy colimits like (\ref{homotopypushout}) identify spaces only ``up to a homotopy factor'' of $\times I$; put differently, the $\times I$ component in homotopy colimits supplies a space ``within which'' to realize homotopies between various connecting maps. Each of these components also readily collapses, furnishing a canonical map from the homotopy colimit of a diagram to its colimit, as is easily seen in the case of (\ref{homotopypushout}) and (\ref{colimit}), for example. The deformation retract of $M_g$ to its target space $Y$ is a special case. With these recognitions in place, we turn our attention to mapping telescopes.

\subsection{Mapping telescopes}
\label{subsection:mappingtelescopes}
The following construction may first have appeared in Milnor's \cite{milnor_axiomatic_1962} which, not coincidentally, was a seminal first appearance of the $\mathrm{lim}^1$ functor within algebraic topology as well. In general, though, this section's results have the status of folklore.
\begin{definition}
The \emph{mapping telescope} associated to a diagram of the form
\begin{align}\label{indseq}\mathbf{X}:\hspace{.7 cm} X_0\xrightarrow{\eta_0} X_1\xrightarrow{\eta_1}\dots\xrightarrow{\eta_{n-1}}X_n\xrightarrow{\eta_n}\dots\end{align}
is the space
\begin{align}\label{mappingtelescopedefinition}
\Big(\coprod_{n\in\mathbb{N}} X_n\times [n,n+1]\Big)/\sim\,,
\end{align}
where $\sim$ is the equivalence relation generated by the identifications $(x,n+1) \sim ( \eta_n(x),n+1)$ for each $n\in\mathbb{N}$ and $x\in X_n$. This space is the homotopy colimit of the diagram $\mathbf{X}$ and henceforth will accordingly be denoted $\mathrm{hocolim}\,\mathbf{X}$ \cite{may_more_2012}. We will sometimes refer to diagrams of the form (\ref{indseq}) as \emph{towers} or \emph{inductive sequences} below.
\end{definition}
Clearly the space (\ref{mappingtelescopedefinition}) may equivalently be viewed as an assemblage of mapping cylinders in the following way:
\begin{align}\label{mappingtelescopeviacylinders}\mathrm{hocolim}\,\mathbf{X}\cong\big(\coprod_{n\in\mathbb{N}} M_{\eta_n}\big)/\sim\,,\end{align}
where $\sim$ identifies the copy of $X_n$ in $M_{\eta_{n-1}}$ with $X_n\times\{0\}$ in $M_{\eta_n}$, for each $n>0$. From this observation, the following lemma is immediate (what we are proving is simply an aspect of the fact that mapping telescopes are homotopy colimits which we will later invoke).
\begin{lemma}
\label{Proposition:homotopy-type-telescope}If $\left( X_{n}\right)_{n\in
\mathbb{N}}$ is a sequence of topological spaces and $\eta _{n
},\eta'_{n}:X_{n}\rightarrow X_{n+1}$ are
homotopic for each $n\in \mathbb{N}$ then the corresponding mapping telescopes 
$\mathrm{hocolim}\,( X_{n},\eta _n) $ and $\mathrm{hocolim}\,( X_{n},\eta'_n) $ are homotopy equivalent.
\end{lemma}
\begin{proof}
Observe simply that the homotopy equivalences (and the witnessing homotopies) between each $M_{\eta_n}$ and $M_{\eta'_n}$ are compatible with the identifications of (\ref{mappingtelescopeviacylinders}), and consequently assemble to define homotopy equivalences of $\mathrm{hocolim}\,( X_{n},\eta _n) $ and $\mathrm{hocolim}\,( X_{n},\eta'_n) $, as desired.
\end{proof}

\begin{corollary}
If $\mathbf{X}$ is an inductive sequence of polyhedra, then $\mathrm{hocolim}\,\mathbf{X}$ is homotopy equivalent to a
polyhedron.
\end{corollary}

\begin{proof}
This is immediate from Proposition \ref%
{Proposition:homotopy-type-telescope}, the Simplicial Approximation
Theorem, and the fact that mapping cylinders of simplicial maps are polyhedra.
\end{proof}

We will require two more basic facts before proceeding. For the first, recall the previous section's reference to a canonical map $\mathrm{hocolim}\,\mathbf{X}\to\mathrm{colim}\,\mathbf{X}$. Under the conditions of the following lemma, this is a homotopy equivalence.
\begin{lemma}
\label{Lemma:inductive-homotopy}The homotopy colimit of an inductive sequence $\mathbf{X}$ of embeddings of compact Polish spaces is homotopy equivalent to $\mathrm{colim}\,\mathbf{X}$.
\end{lemma}

\begin{proof}
Let $Y=\mathrm{colim}\,\mathbf{X}$. For each $n\in \mathbb{N} $ we
have a canonical embedding $\psi _n:X_{n}\rightarrow Y$
with range a compact $Y_{n}\subseteq Y$. By passing to a subsequence, we may assume that $%
( Y_{n}) _{n\in \mathbb{N}}$ is a sequence of compact subsets of $Y$
such that $Y_{n}\subseteq \mathrm{int}( Y_{n+1}) $ for every $%
n\in \mathbb{N} $, and $Y$ is the union of $\{ Y_{n}\mid n\in \mathbb{N} \} $ (if we cannot, then $(X_n)_{n\in\mathbb{N}}$ eventually stabilizes, in which case the conclusion of the lemma is clear). Let also $Y_{-1}=\varnothing $. For each $n\in \mathbb{N}$, fix
a continuous function $\lambda _{n}:Y_{n}\rightarrow [ n,n+1] $
such that $\lambda _{n}[Y_{n-1}]\subseteq\{n\}$ and $\lambda _{n}[\partial
Y_{n}]=\{n+1\}$. Define the function $\lambda :Y\rightarrow [
0,\infty )$ by setting $\lambda( x) =\lambda _n( x) 
$ for all $n\in \mathbb{N}$ and $x\in Y_{n}\setminus Y_{n-1}$, and let 
\begin{equation*}
L=\left\{ \left( x,\lambda( x) \right) \in Y\times \lbrack
0,\infty )\mid x\in Y\right\}
\end{equation*}%
and%
\begin{equation*}
B=\left\{ \left( x,t\right) \in Y\times \lbrack 0,\infty )\mid x\in Y_{n},\,t\geq
n\right\}.
\end{equation*}%
Notice that $L$ is a subset of $B$ and that:

\begin{itemize}
\item $L$ is a deformation retract of $B$, via the map $\left( x,t\right)
\mapsto ( x,\lambda(x)) $, and

\item $Y$ is homeomorphic to $L$, via the map $Y\rightarrow L$, $x\mapsto
\left( x,\lambda \left( x\right) \right) $, and

\item $\mathrm{hocolim}\,\mathbf{X}$ is homeomorphic to $B$,
via the map $\mathrm{hocolim}\,\mathbf{X}\rightarrow B$, $%
\langle x,t\rangle\mapsto ( \psi _{n}( x) ,t) $ for all $n\in \mathbb{N} $, 
$n\leq t\leq n+1$, and $x\in X_{n}$.
\end{itemize}
This concludes the proof.
\end{proof}
The last of the mapping telescope facts we'll collect was tacitly invoked in the ``passage to a subsequence'' step of the above argument; namely, it is the following:
\begin{proposition}
The mapping telescope construction
determines a functor from the category $\mathsf{Ind}_{\omega }( 
\mathsf{C}) $ of inductive sequences of compact Polish spaces to the homotopy
category $\mathsf{Ho}(\mathsf{LC}) $ of locally compact Polish spaces.
\end{proposition}
See again Section \ref{Section:pro_and_ind_categories} for the definition of $\mathsf{Ind}_{\omega }( 
\mathsf{C}) $.
\begin{proof}
Suppose
that $\mathbf{X}=\left( X_{n }\right) _{n\in \mathbb{N}}$ and 
$\mathbf{Y}=\left( Y_{n}\right) _{n\in \mathbb{N}}$ are
inductive sequences of compact spaces and $( \ell _{k},f_{k})_{k\in \mathbb{N}}$ represents a morphism from $%
\mathbf{X}$ to $\mathbf{Y}$. We define the corresponding map $f=\mathrm{hocolim}\left( \ell _{k},f_{k}\right)
:\mathrm{hocolim}\,\mathbf{X}\rightarrow \mathrm{hocolim}\,\mathbf{Y}$ as follows. We denote the constituent maps $Y_k\to Y_\ell$  of $\mathbf{Y}$ by $\eta_{(k,\ell)}$ and let $d_{k}=\ell _{k+1}-\ell_{k}$ for $k\in \mathbb{N}$. For $k\in \mathbb{N}$ and $x\in X_{k}$ and $s\in [ 0,1]$ let
\begin{equation*}
f(\langle x,k+s\rangle) =\langle\eta _{(\ell _{k},\ell _{k}+d_{k}s)}f_{k}( x) ,\ell _{k}+d_{k}s\rangle\text{,}
\end{equation*}%
where we write $\langle\eta _{\left(\ell,t \right) }( y) ,t\rangle$ for $\langle\eta _{\left(\ell,n \right) }( y) ,t\rangle$ when $t$ is in $[n,n+1] $ and $y$ is in $Y_{n}$. It is easy to see that this gives
a well-defined continuous function $\mathrm{hocolim}\,\mathbf{X}\rightarrow\mathrm{hocolim}\,\mathbf{Y}$.

We now show that if $\left( \ell _{k},f_{ k}\right) _{k\in
\mathbb{N}}$ and $(\ell _{k}^{\prime },f_{k }^{\prime })_{k\in
\mathbb{N}}$ represent the same morphism from $\mathbf{X}$ to $\mathbf{Y%
}$, then $f=\mathrm{hocolim}\,( \ell _{k},f_{k}) $ and $f'=\mathrm{hocolim}\,( \ell' _{k},f'_{k}) $ are homotopic maps $\mathrm{hocolim}\,\mathbf{X}\rightarrow\mathrm{hocolim}\,\mathbf{Y}$. By definition, under these assumptions there exists an increasing sequence $\left( \ell
_{k}^{\prime \prime }\right) _{k\in \mathbb{N}}$ such that $\max \{\ell _{k},\ell _{k}^{\prime }\}\leq \ell _{k}^{\prime
\prime }$ and $\eta _{(\ell _{k},\ell''_{k})}f_{k}=\eta _{\left( \ell _{k}^{\prime },\ell'' _{k}\right) }f_{k}^{\prime }$ for every $k\in \mathbb{N}$. Hence by the transitivity of the homotopy relation it suffices to consider the case when $\ell _{k}\leq \ell _{k}^{\prime }$
and $f_{k}^{\prime }=\eta _{\left( \ell _{k},\ell'
_{k}\right) }f_{k}$ for every $k\in \mathbb{N}$.  Again let $d_{k}=\ell_{k+1}-\ell _{k}$ and $d_{k}^{\prime }=\ell _{k+1}^{\prime
}-\ell_{k}^{\prime }$ for all $k\in\mathbb{N}$. For all such $k$ and $%
x\in X_{k}$ and $s\in \left[ 0,1\right] $, we then have%
\begin{equation*}
f(\langle x,k+s\rangle) =\langle\eta _{(\ell _{k},\ell _{k}+d_{k}s)}f_{k}( x) ,\ell _{k}+d_{k}s\rangle 
\end{equation*}
and
\begin{equation*}
f^{\prime }(\langle x,k+s\rangle) =\langle\eta _{(\ell' _{k},\ell' _{k}+d'_{k}s)}f'_{k}( x) ,\ell' _{k}+d'_{k}s\rangle \text{.}
\end{equation*}
We may then define a homotopy $h:f\Rightarrow f^{\prime }$ by setting%
\begin{equation*}
h(\langle x,k+s\rangle,t) =\langle\eta _{(\ell_k,( 1-t)( \ell
_{k}+d_{k}s) +t( \ell _{k}^{\prime }+d_k's))}f_{k}(x) ,( 1-t)( \ell
_{k}+d_{k}s) +t( \ell _{k}^{\prime }+d_k's))\rangle\text{.}
\end{equation*}%
This shows that the homotopy class of $\mathrm{hocolim}\,( \ell _{k},f_{k}) $ does not depend on the choice of representative $(\ell_k, f_f)$ of a
morphism from $\mathbf{X}$ to $\mathbf{Y}$. Thus, given an $\mathsf{Ind}_\omega$-morphism $%
[(\ell_k,f_k)]$ from $\mathbf{X}$ to $\mathbf{Y}$
we may let $\mathrm{hocolim}\,[( \ell _{k},f_{k})]$ be the homotopy class of $\mathrm{hocolim}\,( \ell _{k},f_{k})$, thereby defining a functor from the category of
inductive sequences of compact Polish spaces to the homotopy category of
locally compact Polish spaces, as desired.
\end{proof}
\subsection{The cohomology of mapping telescopes}

Suppose now that $\mathbf{X}=( X_{n})
_{n\in \mathbb{N}}$ is an inductive sequence of compact spaces. Then $\mathrm{hocolim}\,\mathbf{X}$ has a canonical cofiltration 
$(\tilde{X}_{n})_{n\in \mathbb{N}}$ obtained by letting $\tilde{X}_{n}$ be the
set%
\begin{equation*}
\{\langle x,t\rangle\in \mathrm{hocolim}\,\mathbf{X}:0\leq t\leq n\}%
\text{.}
\end{equation*}%
It is easy to see that $\tilde{X}_{n}$ is a compact subset of $\mathrm{hocolim}\,\mathbf{X}$ that is homotopy equivalent
to $X_{n}$ (via a sequence of mapping cylinder deformation retractions of the sort alluded to at the end of Section \ref{subsection:colimits}). Furthermore, the inductive sequence $(\tilde{X}_{n})_{n\in
\mathbb{N}}$, with inclusions as bonding maps, is naturally isomorphic in $%
\mathsf{Ind}_{\omega }( \mathsf{Ho}( \mathsf{C})) $
to $\mathbf{X}$. Thus as a particular instance of Proposition \ref%
{Proposition:asymptotic-cohomology} we have the following.

\begin{proposition}
\label{Proposition:cohomology-telescopes}Suppose that $q$ is a positive integer and $G$ is a
countable discrete group and $\mathbf{X}=( X_{n},\eta_n)_{n\in \mathbb{N}}$ is an inductive sequence of compact spaces. Then:

\begin{enumerate}
\item $\mathrm{H}_{\infty }^{q}(\mathrm{hocolim}\,\mathbf{X} ;G)$ is naturally definably isomorphic to $
\mathrm{lim}^{1}\,\mathrm{H}^{q}( S( X_n) ,\,\ast\,;G) \cong \mathrm{lim}^{1}\,\mathrm{H}^{q-1}( X_n;G) $;

\item $\mathrm{H}_{\mathrm{w}}^q(\mathrm{hocolim}\,\mathbf{X} ;G)$ is naturally definably isomorphic to the pro-countable abelian group $\mathrm{lim}\,\mathrm{H}^{q}(X_n;G) $.
\end{enumerate}
\end{proposition}

\subsection{Mapping telescopes of tori}
\label{ss:hocolims_of_tori}
Call an inductive sequence $\mathbf{X}=\left(
X_{n},\eta_n\right) _{n\in \mathbb{N}}$ of $d$-dimensional tori $X_n=\mathbb{T}^{d}$ \emph{nontrivial} if its bonding maps 
$\eta _n:X_{n}\rightarrow X_{n+1}$ are each of nonzero
degree. In this section, we show that definable \v{C}ech cohomology is a complete invariant for homotopy colimits of such sequences.
\begin{theorem}
\label{Theorem:definable-Cech-cohomology-telescopes-tori} The definable \v{C}ech cohomology groups completely classify homotopy colimits of nontrivial towers of $d$-tori up to homotopy equivalence, for all $d\geq 1$. In fact, the mapping telescopes associated to any two such towers are homotopy equivalent if and only if they have \emph{definably} isomorphic weak and asymptotic \v{C}ech cohomology groups.
\end{theorem}
In contrast, there exist uncountable families of homotopy inequivalent mapping telescopes of $d$-tori whose weak and asymptotic \v{C}ech cohomology --- and, moreover, classical \v{C}ech cohomology --- groups are isomorphic, as we show in Theorem \ref{Theorem:pairwisehomotopyinequivalent} below.

We will precede the proof of Theorem \ref{Theorem:definable-Cech-cohomology-telescopes-tori} with a few observations. Most immediately, observe that by Proposition \ref{Proposition:subfunctor}, the first assertion of Theorem \ref{Theorem:definable-Cech-cohomology-telescopes-tori} follows from its second; it is the latter which we will prove below.

Next, recall the notion of \emph{mapping degree}:
\begin{definition}
Fix closed, connected, oriented $d$-dimensional manifolds $M$ and $N$; the \emph{degree} $\mathrm{deg}(f)$ of a continuous function $f:M\to N$ is the integer $k$ for which $f^*:\mathrm{H}^d(N)\to \mathrm{H}^d(M)$ is multiplication by $k$.\footnote{It is more common to define mapping degree in terms of homology groups, but it is easily seen, e.g., via the Universal Coefficient Theorem for cohomology that our definition is equivalent.}
\end{definition}

The value of restricting our attention to telescopes of maps of nonzero degree will grow clearer momentarily; we return to the question of its necessity in a concluding remark. We may in fact, without any loss of generality, restrict our attention yet further: by \cite[Corollary 2]{scheffer_maps_1972}, any map $\mathbb{T}^d\to\mathbb{T}^d$ is homotopic to a group homomorphism; hence by Lemma \ref{Proposition:homotopy-type-telescope} it will suffice to argue Theorem \ref{Theorem:definable-Cech-cohomology-telescopes-tori} for nontrivial inductive sequences $(X_n,\eta_n)$ in which all bonding maps are homomorphisms. For ease of reference, we term the full subcategory of $\mathsf{Ind}_\omega(\mathsf{C})$ consisting of such sequences \emph{the category of monomorphic inductive sequences of $d$-tori}. The prefix ``mono'' references firstly the injectivity of the induced $d^{\mathrm{th}}$ cohomology maps but will apply as well in the first and second of the related contexts which we now describe.

Let $\Lambda $ be a torsion-free rank $d$ abelian group. A \emph{cofiltration of $\Lambda $} is an increasing sequence $\left( \Lambda _{n}\right) _{n\in
\mathbb{N} }$ of finitely-generated rank $d$ free abelian subgroups of $\Lambda$ such
that $\Lambda=\bigcup_{n\in\mathbb{N}} \Lambda _{n}$. A cofiltration $\left( \Lambda _{n}\right) _{n\in \mathbb{N}}$ gives rise to an inverse sequence $%
\left( \mathrm{Hom}( \Lambda _{n},\mathbb{Z}) \right) _{n\in
\mathbb{N}}$ of finitely generated rank $d$ free abelian groups. By Pontryagin duality, this tower in turn induces a monomorphic inductive sequence $%
\mathbf{X}_{\Lambda }:=\left( \mathrm{Hom}( \Lambda _{n},\mathbb{Z}) ^{\ast }\right) _{n\in \mathbb{N}}$ of $d$-tori. In this way, a choice of cofiltration for each torsion-free rank $d$ abelian group determines a
functor $\Lambda \mapsto \mathbf{X}_{\Lambda }$ from the category of
rank $d$ torsion-free abelian groups to the category of monomorphic inductive sequences
of $d$-tori. As shown in \cite{BLPI}, the group invariant $\mathrm{Ext}(C,A)$ for
countable abelian groups $A$ and $C$ first defined in \cite{eilenberg_group_1942} admits a canonical 
\emph{definable }abelian group structure. Using Proposition \ref%
{Proposition:cohomology-telescopes}, we may formulate the definable cohomology of $
\mathbf{X}_{\Lambda}$ in terms of this definable $\textrm{Ext}$. 

We therefore pause to recall this functor's essentials from \cite[\S 7]{BLPI}.
By an \emph{extension $\mathcal{E}$ of a countable abelian group $C$ by a countable abelian group $A$} we mean any short exact sequence 
\begin{center}
\begin{tikzcd}[ampersand replacement=\&]
  0 \arrow[r]  \& A \arrow[r] 
\& 
E \arrow[r] 
\& 
C \arrow[r] 
\& 0
\end{tikzcd}
\end{center}
of abelian groups.
The extensions $\mathcal{E}$ and $\mathcal{E}'$ are isomorphic if there is a group isomorphism $E\to E'$ which makes the following diagram commute.
\begin{center}
\begin{tikzcd}[ampersand replacement=\&]
 \& 
\& E \arrow[dr] \arrow[dd, , "\cong ", dotted] 
\& 
\& \\
0 \arrow[r]  \& A \arrow[ur] \arrow[dr]
\& 
\& 
C \arrow[r] 
\& 0 \\
\& 
\& E' \arrow[ur]  
\& 
\& \end{tikzcd}
\end{center}

We denote by $\mathrm{Ext}(C,A)$ the collection of all isomorphism classes of extensions of $C$ by $A$. As is well-known, this collection admits a natural abelian group structure, and $\mathrm{Ext}$ is more generally the first derived functor of the bifunctor $\mathrm{Hom}:\mathsf{Ab}^{\mathrm{op}}\times\mathsf{Ab}\to\mathsf{Ab}$. 
The definable versions of $\mathrm{Hom}$ and $\mathrm{Ext}$ are each bifunctors from the category of countable abelian groups to the category $\mathsf{APC}$ of groups with an abelian Polish cover: $A$ and $C$ being discrete, the compact-open topology (or equivalently the product topology) renders $\mathrm{Hom}(C,A)$ itself a Polish abelian group.
And $\mathrm{Ext}(C,A)$, in a definition tracing to \cite{eilenberg_group_1942}, is the group with a Polish cover $\mathrm{Z}(C,A) /\mathrm{B}(C,A)$, where
\begin{itemize}
\item $\mathrm{Z}(C,A)$ is the closed subgroup of the abelian Polish group $A^{C\times C}$, consisting of all cocycles on $C$ with coefficients in $A$.  Here  $A^{C\times C}$ is, as above, endowed with the compact-open topology; by a \emph{cocycle on $C$ with coefficients in $A$} we mean any function $g:C\times C\rightarrow A$ such that for all $x,y,z \in C$ we have: 
\begin{enumerate}
\item $g( x,0) =g( 0,y) =0$;
\item $g( x,y) +g( x+y,z) =g( x,y+z)+g( y,z)$;
\item $g( x,y) =g( y,x)$, for all $x,y \in C$.
\end{enumerate}
\item  $\delta$ is the continuous group homomorphism $\mathrm{Hom}(C,A)\to \mathrm{Z}(C,A)$ given by:
\[\delta(g)(x,y):= g(x)+g(y)-g(x+y).\]
\item $\mathrm{B}(C,A)$ is the Polishable Borel subgroup $\delta[\mathrm{Hom}(C,A)]$ of  $\mathrm{Z}(C,A)$.
\end{itemize}
Returning to our cohomology computations, we have the following.
\begin{proposition}
\label{Proposition:cohomology-telescope-group}
Let $\Lambda $ be a rank $d$
torsion-free abelian group. Then:

\begin{enumerate}
\item $\mathrm{H}^{d+1}(\mathrm{hocolim}\, \mathbf{X}_{\Lambda }) =\mathrm{H}_{\infty
}^{d+1}(\mathrm{hocolim}\, \mathbf{X}_{\Lambda }) $ is naturally definably
isomorphic to $\mathrm{Ext}( \Lambda ,\mathbb{Z}) $;

\item $\mathrm{H}_{\mathrm{w}}^{d}(\mathrm{hocolim}\, \mathbf{X}_{\Lambda }) $ is
naturally definably isomorphic to $\mathrm{Hom}( \Lambda ,\mathbb{Z}) $;

\item $\mathrm{H}^{k}( \mathrm{hocolim}\, \mathbf{X}_{\Lambda }) =0$ for $k> d+1$.
\end{enumerate}
\end{proposition}

\begin{proof}
(1) By Proposition \ref{Proposition:cohomology-telescopes}, $%
\mathrm{H}_{\infty }^{d+1}(\mathrm{hocolim}\, \mathbf{X}_{\Lambda })$
is naturally definably isomorphic to $\mathrm{lim}^{1}\,\mathrm{H}^{d}( \mathrm{Hom}( \Lambda _{n},\mathbb{Z})
^{\ast }) $. Furthermore, $\mathrm{H}^{d}( \mathrm{Hom}( \Lambda _{n},%
\mathbb{Z}) ^{\ast }) $ is naturally isomorphic to the countable group $\mathrm{Hom}( \Lambda _{n},\mathbb{Z}) $ \cite[Theorem 8.83]{hofmann_structure_2013}. Hence by the definable version of Jensen's Theorem \cite[Theorem 7.4]{BLPI},
we have natural definable isomorphisms 
\begin{eqnarray*}
\mathrm{lim}^{1}\,\mathrm{H}^{d}( \mathrm{Hom}(\Lambda _{n},\mathbb{Z}) ^{\ast }) &\cong &%
\mathrm{lim}^{1}\,\mathrm{Hom}(\Lambda
_{n},\mathbb{Z}) \\
&\cong &\mathrm{Ext}(\mathrm{colim}\,\Lambda _{n},\mathbb{Z)}
\\
&\cong &\mathrm{Ext}(\Lambda,\mathbb{Z}) \text{.}
\end{eqnarray*}%
This is definably isomorphic to the entirety of $\mathrm{H}^{d+1}(\mathrm{hocolim}\,\mathbf{X}_{\Lambda })$ by Proposition \ref{Proposition:asymptotic-cohomology}, together with the observation that $\mathrm{H}^{d+1}( \mathrm{Hom}(\Lambda _{n},\mathbb{Z})
^{\ast })=0$ for every $n\in \mathbb{N}$ implies that
\begin{equation*}
\mathrm{H}_{\mathrm{w}}^{d+1}(\mathrm{hocolim}\,\mathbf{X}_{\Lambda
})\cong\mathrm{lim}\,\mathrm{H}^{d}( \mathrm{Hom}(\Lambda_{n},\mathbb{Z}) ^{\ast})=0\text{.}
\end{equation*}

(2) By Proposition \ref{Proposition:cohomology-telescopes}, $\mathrm{H}_{\mathrm{w}}^{d}(\mathrm{hocolim}\,\mathbf{X}_{\Lambda})$
is naturally definably isomorphic to 
\begin{equation*}
\mathrm{lim}\,\mathrm{H}^{d}( \mathrm{Hom}(
\Lambda_{n},\mathbb{Z})^{\ast}) \cong \mathrm{lim}\,\mathrm{Hom}( \Lambda _{n},\mathbb{Z}) \cong 
\mathrm{Hom}(\mathrm{colim}\,\Lambda_{n},\mathbb{Z})%
\cong \mathrm{Hom}(\Lambda ,\mathbb{Z})\text{.}
\end{equation*}

(3) This is an immediate consequence of Proposition \ref%
{Proposition:cohomology-telescopes}, upon observation that $\mathrm{H}^{k}( 
\mathrm{Hom}( \Lambda _{n},\mathbb{Z}) ^{\ast }) =0$ for $%
k>d$.
\end{proof}

Observe that the functor $\Lambda \mapsto \mathbf{X}_{\Lambda
} $ described above, from the category of countable torsion-free rank $d$ abelian groups to the category of monomorphic inductive
sequences of $d$-tori, is fully faithful. 
To see this, observe that a choice of cofibrations $(\Lambda_n)$ and $(\Lambda'_n)$ of $\Lambda$ and $\Lambda'$, respectively, induces a bijection $$\mathrm{Hom}_{\mathsf{Ab}}(\Lambda,\Lambda')\cong\mathrm{Hom}_{\mathsf{Ind}_\omega(\mathsf{Ab})}((\Lambda_n),(\Lambda'_n)),$$
and that the sequences of matrices $(M_n:\Lambda_n\to\Lambda'_n)$ representing maps in the latter correspond precisely to those representing maps in $\mathrm{Hom}_{\mathsf{Ind}_\omega(\mathsf{C})}((\mathbb{T}(\Lambda_n)),(\mathbb{T}(\Lambda'_n)))$, where $(\mathbb{T}(\Lambda_n))$ and $(\mathbb{T}(\Lambda'_n))$ denote the sequences of tori comprising $\mathbf{X}_\Lambda$ and $\mathbf{X}_{\Lambda'}$, respectively.
It is also easy to see that
each monomorphic inductive sequence of $d$-tori is isomorphic in \textsf{Ind}$_{\omega }( \mathsf{C}) $ to $\mathbf{X}_{\Lambda }$ for
some torsion-free rank $d$ abelian group $\Lambda $. Hence the functor $%
\Lambda \mapsto \mathbf{X}_{\Lambda }$ is an \emph{equivalence of
categories }from the category of countable torsion-free rank $d$ abelian groups to the
category of monomorphic inductive sequences of $d$-tori.
Thus in order to establish Theorem \ref{Theorem:definable-Cech-cohomology-telescopes-tori}, it suffices to prove the following.

\begin{theorem}
\label{Theorem:classify-colimit-group}The asymptotic and weak definable
cohomology groups of $\mathrm{hocolim}\,\mathbf{X}%
_{\Lambda }$ together provide a complete invariant for a finite rank torsion-free abelian
group $\Lambda $ up to isomorphism.
\end{theorem}

\begin{proof}
A finite rank torsion-free abelian group $\Lambda $ uniquely decomposes
as $\Lambda _{\infty }\oplus \Lambda _{\mathrm{w}}$ where $\Lambda _{\mathrm{%
w}}$ is finitely generated and $\Lambda _{\infty }$ has no
finitely generated summand, hence it will suffice to show that the asymptotic and weak definable cohomology groups of $\mathrm{hocolim}\,\mathbf{X}%
_{\Lambda }$ provide complete invariants for such groups $\Lambda _{\infty}$ and $\Lambda _{\mathrm{w}}$, respectively.

 By Proposition \ref%
{Proposition:cohomology-telescope-group}, we have definable isomorphisms
\begin{equation*}
\mathrm{H}^{d+1}(\mathrm{hocolim}\,\mathbf{X}_{\Lambda })\cong 
\mathrm{Ext}( \Lambda ,\mathbb{Z}) \cong \mathrm{Ext}(
\Lambda _{\infty },\mathbb{Z})
\end{equation*}%
and%
\begin{equation*}
\mathrm{H}^{d}(\mathrm{hocolim}\,\mathbf{X}_{\Lambda })\cong 
\mathrm{Hom}( \Lambda ,\mathbb{Z}) \cong \mathrm{Hom}(
\Lambda _{\mathrm{w}},\mathbb{Z}) \text{.}
\end{equation*}%
By \cite[Corollary 7.6]{BLPI}, the definable group $\mathrm{Ext}\left( -,\mathbb{Z}\right) $
is a complete invariant for torsion-free finite rank groups with no
finitely generated summands. The complementary fact, that $\mathrm{Hom}\left( -,\mathbb{Z}%
\right) $ is a complete invariant for finitely generated
torsion-free abelian groups, is trivial.
\end{proof}

In the setting of finite rank torsion-free abelian groups with no
finitely generated direct summand, the proof of Theorem \ref%
{Theorem:classify-colimit-group} allows for a strengthening of its conclusion as follows:

\begin{theorem}
\label{Theorem:classify-fully-faithful}The map $\Lambda \mapsto \mathrm{H}^{d+1}(\mathrm{hocolim}\,\mathbf{X}_{\Lambda })$ is a fully
faithful functor from the category of rank $d$ torsion-free abelian groups
with no finitely generated direct summand to the category of groups with a Polish cover.
\end{theorem}

\begin{remark} To see that the assumption that bonding maps are of nonzero degree is needed for our arguments, let $X_n=Y_n=\mathbb{T}\times\mathbb{T}$ for all $n\in\omega$; let each $\eta_n:X_n\to X_{n+1}$ be the projection to the first factor, and let each $Y_n\to Y_{n+1}$ be constant. Letting $X$ and $Y$ denote $\mathrm{hocolim}\,(X_n,\eta_n)$ and $\mathrm{hocolim}\,(Y_n,\beta_n)$, respectively, we then have that $X\not\simeq Y$ but
$$\mathrm{H}_\infty^{3}(X)\cong \mathrm{H}_\infty^{3}(Y)\cong \mathrm{H}_{\mathrm{w}}^{2}(X)\cong \mathrm{H}_{\mathrm{w}}^{2}(Y)\cong 0.$$
Here, then, the reasoning of Theorems \ref{Proposition:cohomology-telescope-group} and \ref{Theorem:classify-colimit-group} no longer applies. Note that it remains quite plausible that definable cohomology classifies mapping telescopes of tori even without assumptions on maps' degrees; the example simply suggests that ascertaining this would require a deeper analysis of the cohomology groups than is in the spirit of the present work.
\end{remark}
As indicated, the sensitivities of definable cohomology recorded above contrast strongly with those of classical cohomology.
\begin{theorem}
\label{Theorem:pairwisehomotopyinequivalent}
There exist size-continuum families of pairwise homotopy inequivalent mapping telescopes of monomorphic inductive sequences of tori whose classical \v{C}ech cohomology groups are, viewed as graded abelian groups, all isomorphic.
\end{theorem}
\begin{proof}
Recall the following notations from \cite[p. 28]{BLPI}: for every
sequence $\boldsymbol{m}=\left( m_{p}\right) _{p\in \mathcal{P}}\in \mathbb{N%
}^{\mathcal{P}}$, where $\mathbb{N}$ is the set of strictly positive
integers, define $\mathbb{Z}[\frac{1}{\mathcal{P}^{\boldsymbol{m}}}]$ to be
the set of rational numbers of the form $a/b$ where $a\in \mathbb{Z}$, $b\in 
\mathbb{N}$, and for every $p\in \mathcal{P}$ and $k\in \mathbb{N}$, if $%
p^{k}$ divides $b$ then $k\leq m_{p}$. Write $\boldsymbol{m}=^{\ast }%
\boldsymbol{n}$ if and only if $\{ p\in \mathcal{P}%
:m_{p}\neq m_{p}^{\prime }\} $ is finite.
The following appears as Corollary 7.9 in \cite{BLPI}.
\begin{proposition}
\label{Prop:CorollaryRecall}
Fix $d\geq 1$. For every $%
\boldsymbol{m},\boldsymbol{n}\in \mathbb{N}^{\mathcal{P}}$, $%
\mathrm{Ext}(\mathbb{Z}[\frac{1}{\mathcal{P}^{\boldsymbol{m}}}]^{d},\mathbb{Z%
})$ and $\mathrm{Ext}(\mathbb{Z}[\frac{1}{\mathcal{P}^{\boldsymbol{n}}}]^{d},%
\mathbb{Z})$ are isomorphic as discrete groups, and are Borel isomorphic if
and only if $\boldsymbol{m}=^{\ast }\boldsymbol{n}$.
In particular, the collection
\begin{equation*}
\left\{ \mathrm{Ext}(\mathbb{Z}\Big[\frac{1}{\mathcal{P}^{\boldsymbol{m}}}\Big]^{d},%
\mathbb{Z}):\boldsymbol{m}\in \mathbb{N}^{\mathcal{P}}\right\}
\end{equation*}%
contains a continuum of groups with a Polish cover that are pairwise
isomorphic as discrete groups but not definably isomorphic.
\end{proposition}
Though this result carries implications contrasting with Proposition \ref{Proposition:cohomology-telescope-group} and Theorem \ref{Theorem:classify-colimit-group} for all $d\geq 1$, for simplicity we focus on the case of $d=1$. Choosing $\boldsymbol{m}\in\mathbb{N}^{\mathcal{P}}$ and letting $\Lambda=\mathbb{Z}[\frac{1}{\mathcal{P}^{\boldsymbol{m}}}]$, we may represent $\mathrm{hocolim}\, \mathbf{X}_{\Lambda}$ as the mapping telescope $\mathrm{tel}(\boldsymbol{m})$ of the sequence of maps $\eta_n:\mathbb{T}\to\mathbb{T}$ $(n\in\mathbb{N})$ defined by $z\mapsto z^{p^{m_p}}$, where $p$ is the $n^{\mathrm{th}}$ prime number.
It is straightforward to verify that $\mathrm{H}^0(\mathrm{tel}(\boldsymbol{m}))=\mathrm{H}^1(\mathrm{tel}(\boldsymbol{m}))=\mathbb{Z}$; it then follows from Proposition \ref{Proposition:cohomology-telescope-group} that for any family $M$ of sequences $\boldsymbol{m}$ witnessing Proposition \ref{Prop:CorollaryRecall}, the family $\{\mathrm{tel}(\boldsymbol{m})\mid \boldsymbol{m}\in M\}$ is a witness to Theorem \ref{Theorem:pairwisehomotopyinequivalent}.
\end{proof}

\subsection{Mapping telescopes of spheres}

Similarly, definable cohomology classifies mapping telescopes of spheres up
to homotopy equivalence.

\begin{theorem}
\label{Theorem:definable-Cech-cohomology2} The definable \v{C}ech cohomology groups completely classify homotopy colimits of nontrivial inductive sequences of $d$-spheres up to homotopy equivalence, for all $d\geq 1$. In fact, the mapping telescopes associated to any two such inductive sequences are homotopy equivalent if and only if they have \emph{definably} isomorphic weak and asymptotic \v{C}ech cohomology groups.
\end{theorem}

Just as for Theorem \ref{Theorem:definable-Cech-cohomology-telescopes-tori}, Theorem \ref{Theorem:definable-Cech-cohomology2} contrasts with the fact that there exist uncountable families of pairwise homotopy inequivalent mapping telescopes of sequences of spheres whose classical \v{C}ech cohomology groups, viewed as graded abelian groups, are all isomorphic; the handiest examples of such are the families $\{\mathrm{tel}(\boldsymbol{m})\mid \boldsymbol{m}\in M\}$ concluding the proof of Theorem \ref{Theorem:pairwisehomotopyinequivalent} above.

As in that proof, for each (cofiltered) rank $1$ torsion-free abelian group $\Lambda $, let $\mathbf{X}_{\Lambda }$ denote the
corresponding inductive sequence $\left( X_{n }\right) _{n\in \mathbb{N}}$ of $1$-dimensional tori, which
we regard as pointed $1$-spheres. Let $S^{d-1}( \mathbf{X}_{\Lambda
}) =\left( S^{d-1}( X_{n}\right) ) _{n\in
\mathbb{N}}$, where $S^{k}$ for $k\geq 0$ denotes the $k$-fold iterated reduced suspension of a pointed space, defined by setting $S^{k+1}( X)=S(S^{k}(X)) $ for every $k\in \mathbb{N}$; evidently, $S^{d-1}( \mathbf{X}_{\Lambda }) $ is an inductive sequence of 
$d$-spheres. Conversely, by Hopf's Theorem \ref{Theorem:Hopf}, since $\tilde{H}^{d}(
S^{d}) $ is isomorphic to $\mathbb{Z}$, every inductive sequence of $d$%
-spheres and maps of nonzero degree is isomorphic in $\mathsf{Ind}_\omega( 
\mathsf{Ho}( \mathsf{C})) $ to an inductive sequence of
this form. Thus, by Proposition \ref{Proposition:homotopy-type-telescope},
it suffices to show that definable cohomology is a complete invariant for $\mathrm{hocolim}\,S^{d-1}(\mathbf{X}_{\Lambda
}) $ where $\Lambda $ is a rank $1$ torsion-free abelian group. This
is immediate from the following proposition, whose proof is the same as the proof of Proposition \ref{Proposition:cohomology-telescope-group}.

\begin{proposition}
Let $\Lambda $ be a rank $1$ torsion-free abelian group. Then:

\begin{enumerate}
\item $\mathrm{H}_{\infty }^{d+1}(S^{d-1}(\mathbf{X}_{\Lambda}))$ is
naturally definably isomorphic to $\mathrm{Ext}(\Lambda ,\mathbb{Z}) $;

\item $\mathrm{H}_{\mathrm{w}}^{d}(S^{d-1}(\mathbf{X}_{\Lambda}))$ is
naturally definably isomorphic to $\mathrm{Hom}(\Lambda ,\mathbb{Z})$.
\end{enumerate}
\end{proposition}
\subsection{The Borsuk--Eilenberg classification problem}
\label{ss:Borsuk_Eilenberg}

A $d$-dimensional \emph{solenoid }is an indecomposable continuum and, in particular, a compact Polish space which is
homeomorphic to the Pontryagin dual of an infinitely generated rank $d$ torsion-free abelian group. One-dimensional solenoids played a prominent
role in the work of Smale and Williams in the context of dynamical systems,
as they provided the first examples of \emph{uniformly hyperbolic }(or Axiom
A) \emph{attractors }that are \emph{strange}; see \cite%
{smale_differentiable_1967,ruelle_what_2006,williams_expanding_1974}. In
these contexts, a $1$-dimensional solenoid arising as a uniformly hyperbolic
attractor of a dynamical system is termed a \emph{Smale (solenoid) attractor}
or \emph{Smale--Williams (solenoid) attractor}. For any $p\geq 2$ (not
necessarily prime), let $\Sigma _{p}$ be the $p$-adic solenoid, i.e., the
Pontryagin dual of $\mathbb{Z}[1/p]$. An explicit construction of an
orientation-preserving diffeomorphism $h_{p}$ of $S^{3}$ with $\Sigma _{p}$
as attractor is also given in \cite{hubbard_henon_1994,hubbard_linked_2001};
we now recall some definitions from that work.

Let $\mathbb{T}\times D^{2}$ be the solid torus, viewing $\mathbb{T}$ as the unit circle in the complex plane, and let $\phi :\mathbb{T}%
\rightarrow \mathbb{T}$ be the map $\zeta \mapsto \zeta ^{p}$. The \emph{canonical
unbraided solenoidal mapping $\mathbb{T}\times D^{2}\rightarrow \mathbb{T}%
\times D^{2}$ of degree $p$} is the mapping%
\begin{equation*}
e_{\phi }:\left( \zeta ,z\right) \mapsto \left( \zeta ^{p},\zeta
+\varepsilon \zeta ^{1-p}\right)
\end{equation*}%
where $\varepsilon $ is a fixed positive real number chosen to be 
small enough that $e_{\phi }$ is injective. An intrinsic characterization
of $e_{\phi }$ up to conjugacy is given in \cite[Theorem 3.11]{hubbard_henon_1994}. It is shown in \cite[Section 4]{hubbard_henon_1994}
that there exists an orientation-preserving diffeomorphism $h$ of $S^{3}$
and a smooth embedding $j^{+}:\mathbb{T}\times D^{2}\rightarrow S^{3}$ such
that $h$ lifts $e_{\phi }$ through $j^{+}$, in the sense that $h\circ
j^{+}=j^{+}\circ e_{\phi }$. Setting $T^{+}:=j^{+}( \mathbb{T}\times
D^{2}) $ and 
\begin{equation*}
T_{n}^{+}:=( h^{n}\circ j^{+}) ( \mathbb{T}\times
D^{2}) =( j\circ e_{\phi }^{n}) ( \mathbb{T}\times
D^{2})
\end{equation*}%
one has that 
\begin{equation*}
\Sigma _{p}( h) :=\bigcap_{n\in \omega }T_{n}^{+}
\end{equation*}%
is an attractor for $h$ homeomorphic to $\Sigma _{p}$. Furthermore, if one
lets $T^{-}$ be the closure of the complement of $T^{+}$, then one has that $%
h^{-1}|_{T^{-}}$ is conjugate to $e_{\phi }$, meaning that there exists a
diffeomorphism $j^{-}:\mathbb{T}\times D^{2}\rightarrow T^{-}$ such that $%
h^{-1}\circ j^{-}=e_{\phi }$. Hence%
\begin{equation*}
S^{3}\setminus \Sigma _{p}( h) =\bigcup_{n\in \omega
}h^{-1}( T^{-})
\end{equation*}%
is homeomorphic to%
\begin{equation*}
\mathrm{colim}\,( \mathbb{T}\times D^{2},e_{\phi }).
\end{equation*}%
This colimit, in turn, is homotopy equivalent to the mapping telescope 
$\mathrm{hocolim}\left( \mathbb{T},\phi \right) $, by Lemma %
\ref{Lemma:inductive-homotopy}.

The complement $S^{3}\setminus \Sigma _{p}( h) $ of $\Sigma
_{p}( h) $ is called a (one-dimensional) \emph{solenoid complement%
}. The problem of classifying the maps $S^{3}\setminus \Sigma _{p}(
h) \rightarrow S^{2}$ up to homotopy was posed by Borsuk and Eilenberg
in \cite{borsuk_uber_1936}, and we conclude this subsection with a brief review of this question's rather striking history. More immediately, though, we have the following.

\begin{theorem}
\label{Theorem:BE}Fix $p\geq 2$. Then $[S^{3}\setminus \Sigma _{p}(
h) ,S^{2}]=[ S^{3}\setminus \Sigma _{p}( h) ,S^{2}] _{\infty }$ is a definable set, and there is a basepoint-preserving
definable bijection between $[S^{3}\setminus \Sigma _{p}(h)
,S^{2}]$ and $\mathrm{Ext}( \mathbb{Z}[1/p],\mathbb{Z}) $.
\end{theorem}
\begin{proof}
By the foregoing discussion, $S^{3}\setminus \Sigma_{p}(h)$ is homotopy equivalent to $\mathbf{X}_\Lambda$ for $\Lambda=\mathbb{Z}[1/p]$.
By the $d=1$ instance of Proposition \ref{Proposition:cohomology-telescope-group}(3), Hopf's Theorem \ref{Theorem:Hopf} applies, giving a definable bijection of $[S^{3}\setminus \Sigma _{p}(
h) ,S^{2}]$ and $\mathrm{H}^2(S^{3}\setminus \Sigma _{p}(h))$, and the latter, by Proposition \ref{Proposition:cohomology-telescope-group}(1), is definably isomorphic to $\mathrm{Ext}( \mathbb{Z}[1/p],\mathbb{Z})$.
That the composite bijection is basepoint-preserving is clear, and just as in Corollary \ref{cor:definable_set}, $[S^{3}\setminus \Sigma _{p}(
h) ,S^{2}]$ is then a definable set.
Proposition \ref{Proposition:cohomology-telescope-group}(1) also gives the first definable bijection in the series
$$[S^{3}\setminus \Sigma _{p}(
h) ,S^{2}]\cong \mathrm{H}_\infty^2(S^{3}\setminus \Sigma _{p}(h))\cong\mathrm{lim}^1\,\mathrm{H}^2(S(\mathbb{T}))\cong [S^{3}\setminus \Sigma _{p}(
h) ,S^{2}]_\infty.$$
Proposition \ref{Proposition:cohomology-telescopes} supplies the second, and the last follows again from Hopf's Theorem \ref{Theorem:Hopf}, together with Theorem \ref{Theorem:phantom1}.
\end{proof}
As $S^{3}\setminus \Sigma _{p}(h) $ is homotopy equivalent to
the mapping telescope $\mathrm{hocolim}\,( \mathbb{T},\phi) $, one may consider the following generalization of Theorem %
\ref{Theorem:BE} to arbitrary homotopy colimits of $d$-tori. This theorem follows just as above from Hopf's Theorem \ref{Theorem:Hopf} and Propositions \ref%
{Proposition:cohomology-telescopes} and \ref{Proposition:cohomology-telescope-group}.

\begin{theorem}
\label{Theorem:BE+}Fix $d\geq 1$, and let $\Lambda $ be a torsion-free rank $%
d$ abelian group. Then there are basepoint-preserving definable bijections between $[\mathrm{hocolim}\,\mathbf{X}_{\Lambda },S^{d+1}]_{\infty }$
and $\mathrm{Ext}( \Lambda ,\mathbb{Z}) $, and between $[\mathrm{hocolim}\,\mathbf{X}_{\Lambda },S^{d+1}]_{\mathrm{w%
}}$ and $\mathrm{Hom}( \Lambda ,\mathbb{Z}) $.
\end{theorem}

Turning now to the question's history, we quote from Eilenberg's memoirs of his work with Karol Borsuk: \begin{quote} The main problem concerning us was the following: given a solenoid $\Sigma$ in $S^3$, how big is the set $S$ of homotopy classes of maps $f: S^3\setminus \Sigma \rightarrow S^{2}$? Our algebraic equipment was so poor that we could not tackle the problem in the whole generality even though all the tools needed were in our paper. In 1938, using the newly developed ``obstruction theory,'' I established that the set $S$ in question is equipotent to [the second \v{C}ech cohomology group of $S^3\setminus \Sigma$]. \cite{eilenbergonborsuk}
\end{quote}
This motivated Steenrod \cite{steenrod_regular_1940} to introduce a homology theory dual to \v{C}ech cohomology; this theory is now known as \emph{Steenrod homology}.  Steenrod's duality principle (a form of Alexander duality) entailed that $\tilde{\mathrm{H}}_{0}\left( \Sigma \right) \cong
\mathrm{H}^{2}(S^{3}\setminus \Sigma)$. Steenrod then computed the group $\tilde{\mathrm{H}}_{0}(\Sigma)$ and showed that it --- and hence the set of homotopy classes of maps $f: S^3\setminus \Sigma \rightarrow S^{2}$ --- is uncountable. Eilenberg continues:
\begin{quote}
When Saunders MacLane lectured in 1940 at the University of Michigan on group extensions one of the groups appearing on the blackboard was exactly the group calculated by Steenrod. I recognized it and spoke about it to MacLane. The result was the joint paper ``Group extensions and homology,'' Ann. of Math., 43, 1942. This was the birth of Homological Algebra.
\end{quote}
This joint paper, which introduced the functors $\mathrm{Hom}$ and $\mathrm{Ext}$, is often cited as the beginning of category theory as well: a central concern of the work is canonical  or so-called ``natural homomorphisms'' between groups --- a notion category theory was in part developed to make precise \cite{weibelhistory, eilmaclane45}.\footnote{To respond to a query of the referee: this paper also established the Universal Coefficient Theorem \cite[\S 35]{eilmaclane45} from which, together with the aforementioned Steenrod duality, the \emph{abstract} isomorphism $\mathrm{H}^2(S^3\backslash\Sigma_p)\cong\mathrm{Ext}(\mathbb{Z}[1/p],\mathbb{Z})$ is readily deduced.}

The affinity of that concern with our own concern for \emph{definable} homomorphisms should be clear. The novelty of this subsection's theorems consists in both the generalization to higher dimensions (Theorem \ref{Theorem:BE+}), and in their formulation within the category of \textsf{DSet}. This latter point allows both for finer characterizations (in the sense of Borel complexity) of the sets in question and an analysis of their orbits under automorphism actions, the subjects of the following subsection.
\subsection{Actions and Borel complexity}
Let us begin by recalling that a definable set $X/E$ is:

\begin{itemize}
\item \emph{smooth }if and only if and only if there is an injective
definable function $X/E \rightarrow Y$ where $Y$ is a Polish
space;

\item \emph{essentially hyperfinite} if and only if there is an injective
definable function $X/E \rightarrow Y/F$ where $Y$
is a Polish space and $F$ is the orbit equivalence relation associated with
a Borel action of $\mathbb{Z}$ on $Y$;

\item \emph{essentially treeable }if and only if there exists an injective
definable function $X/E \rightarrow Y/F$ where $Y$
is a Polish space, and $F$ is the orbit equivalence relation associated with
a Borel action of a free group on $Y$.
\end{itemize}

\medskip

Consider next the definable sets $\mathrm{Ext}(\Lambda,\mathbb{Z})$ featuring in Theorem \ref{Theorem:BE+} above; more generally, consider $\mathrm{Ext}(\Lambda,\mathbb{Z})$ for any countable torsion-free group $\Lambda$. Since $\Lambda $ has no
finite subgroup, $\mathrm{Ext}(\Lambda ,\mathbb{Z})$ equals, in modern notation, $\mathrm{PExt}(\Lambda ,\mathbb{Z})$, or in other words the group which Eilenberg and MacLane denote $\mathrm{Ext}_f(
\Lambda,\mathbb{Z})$ and identify as the closure of $\{0\} $ in $\mathrm{Ext}(\Lambda ,\mathbb{Z})$ in \cite{eilenberg_group_1942}; more succinctly, $\{
0\} $ is dense in $\mathrm{Ext}( \Lambda,\mathbb{Z})$ (here $\{0\}$ is a convenient locution for $N$ in the presentation $G/N$ of $\mathrm{Ext}(\Lambda,\mathbb{Z})$ as a group with a Polish cover).
From this, the first item of the following theorem is almost immediate.
\begin{theorem}
\label{thm:smooth!}
Let $\Lambda$ be a countable torsion-free abelian group.
\begin{itemize}
\item The equivalence relation of isomorphism of extensions of $\Lambda $ by 
$\mathbb{Z}$ is smooth if and only if $\{0\} $ is closed in $%
\mathrm{Ext}(\Lambda ,\mathbb{Z})$ if and only if $\mathrm{Ext}(\Lambda,\mathbb{Z})=0$ if and only if $\Lambda $ is free
abelian.
\item The equivalence relation of isomorphism of extensions of $%
\Lambda $ by $\mathbb{Z}$ is essentially hyperfinite if and only if $\{0\} $ is $\boldsymbol{\Sigma }_{2}^{0}$ in $\mathrm{Ext}( \Lambda
,\mathbb{Z})$ if and only if $\Lambda =\Lambda _{\infty
}\oplus \Lambda _{\mathrm{free}}$ where $\Lambda _{\mathrm{free}}$ is free
and $\Lambda _{\infty }$ is finite-rank.
\end{itemize}
In particular, for any prime $p\geq 2$, the problem
of classifying extensions of $\mathbb{Z}[1/p]$ by $\mathbb{Z}$ is essentially hyperfinite
and not smooth, and so, in consequence, is the problem of classifying maps $S^{3}\setminus \Sigma _{p}(
h) \rightarrow S^{2}$ up to homotopy.
\end{theorem}

\begin{proof}
The third of the first item's bi-implications is due to \cite{Stein_1951}; the second simply restates our above observation that $\{0\}$ is dense in $\mathrm{Ext}(\Lambda,\mathbb{Z})$. For the left-to-right portion of its first bi-implication, again recast that observation as \emph{the orbit equivalence relation given by $\{0\}$ is generically ergodic} \cite[Prop.\ 6.1.9]{gao_invariant_2009} and apply the contrapositive of \cite[Prop.\ 6.1.10]{gao_invariant_2009} to conclude that $\{0\}$ is comeager and, hence, by Pettis's Lemma \cite[Thm.\ 2.3.2]{gao_invariant_2009}, is closed, as claimed. The right-to-left portion of its first bi-implication is trivial.

Let us turn now to the theorem's second item, from which our last assertion directly follows by an application of Theorem \ref{Theorem:BE}.
Note that since $\mathrm{Ext}(-,\mathbb{Z})$ commutes with sums, in arguing the second item's right-to-left implications we may suppose at the outset that $\Lambda $ is
finite-rank. Next, let $E\subseteq \Lambda $ be a free abelian group
such that $\Lambda /E$ is torsion.
To see that $\{ 0\} $ is $\boldsymbol{\Sigma }_{2}^{0}$ in $\mathrm{Ext}(\Lambda ,\mathbb{Z})$, consider the tail of the long definable exact
sequence relating $\mathrm{Hom}$ and $\mathrm{Ext}$:%
\begin{equation*}
\mathrm{Hom}( E,\mathbb{Z}) \rightarrow \mathrm{Ext}(
\Lambda /E,\mathbb{Z}) \rightarrow \mathrm{Ext}( \Lambda ,\mathbb{%
Z}) \rightarrow \mathrm{Ext}( E,\mathbb{Z}) =0\text{,}
\end{equation*}
whereby
\begin{equation*}
\mathrm{Ext}( \Lambda ,\mathbb{Z}) \cong \frac{\mathrm{Ext}(
\Lambda /E,\mathbb{Z}) }{\mathrm{ran}( \mathrm{Hom}( E,%
\mathbb{Z}) \rightarrow \mathrm{Ext}( \Lambda /E,\mathbb{Z}%
))},
\end{equation*}
where $\mathrm{ran}(f)$ denotes the range, of course, of a function $f$. Considering next the fragment
\begin{equation*}
0=\mathrm{Hom}(\Lambda /E,\mathbb{Q)}\rightarrow \mathrm{Hom}(\Lambda /E,%
\mathbb{Q}/\mathbb{Z})\rightarrow \mathrm{Ext}( \Lambda /E,\mathbb{Z}) \rightarrow \mathrm{Ext}(\Lambda /E,\mathbb{Q)}=0
\end{equation*}
of the long definable exact sequence associating to $\mathbb{Z}\to\mathbb{Q}\to\mathbb{Q}/\mathbb{Z}$, we see that $\mathrm{Ext}(\Lambda /E,
\mathbb{Z})$ is definably isomorphic to the Polish group $\mathrm{Hom}(\Lambda /E,%
\mathbb{Q}/\mathbb{Z})$.
Hence the complexity of $\{ 0\} $ in $\mathrm{Ext}(\Lambda ,
\mathbb{Z})$ equals the complexity of $\mathrm{ran}( \mathrm{Hom}(E,\mathbb{Z}) \rightarrow \mathrm{Ext}(\Lambda /E,\mathbb{Z}))$ in $\mathrm{Ext}(\Lambda /E,\mathbb{Z})$ (for the preservation of this complexity by definable isomorphisms, see \cite[Proposition 4.12]{lupini_looking_22}).
Since $E$ is finite-rank, $\mathrm{Hom}(E,\mathbb{Z})$ is
countable, hence $\mathrm{ran}(\mathrm{Hom}( E,\mathbb{Z}) \rightarrow \mathrm{Ext}(\Lambda /E,\mathbb{Z}))$
is also countable, and consequently $\boldsymbol{\Sigma }_{2}^{0}$ in $\mathrm{Ext}(\Lambda /E,\mathbb{Z})$. From this we conclude that the equivalence relation of isomorphism of extensions of $%
\Lambda $ by $\mathbb{Z}$ is essentially hyperfinite by \cite[Theorem 12.5.7]{gao_invariant_2009} and \cite[Theorem 6.1]{Ding_Gao_2017}.

The second item's left-to-right implications are as follows: the essential hyperfiniteness of the relation implies that $\{0\}$ is $\boldsymbol{\Sigma }_{2}^{0}$ in $\mathrm{Ext}( \Lambda
,\mathbb{Z})$ by \cite[Proposition 4.14]{lupini_looking_22}, and this in turn implies that $\Lambda$ is a sum of a finite-rank abelian group and a free one by \cite[Proposition 6.4]{lupini2022classification}.
\end{proof}
We turn now to an analysis of actions on definable sets. Suppose that $\Gamma $ is a group and $X/E$ is a
semidefinable set. A \emph{definable left action} $\Gamma
\curvearrowright X/E$ of $\Gamma $ on $X\left/ E\right. $ is a
function $\Gamma \times X\left/ E\right. \rightarrow X/E$, $(\gamma ,x) \mapsto \gamma \cdot x$ such that \begin{itemize}
\item for every $\gamma
\in \Gamma $ the map $f_{\gamma }:X/E \rightarrow X/E$, $x\mapsto \gamma \cdot x$ is definable, and
\item the assignment $%
\gamma \mapsto f_{\gamma }$ is a group homomorphism from $\Gamma $ to the
group of definable bijections of $X$.
\end{itemize}   The orbit space of the action is the
semidefinable set $(X/E)/\Gamma =X/F$ where $F$ is the equivalence relation on $X$ defined by setting $%
xFy$ if and only if there exists $\gamma \in \Gamma $ such that $\left(
\gamma \cdot x\right) E\,y$. A definable morphism of actions from $\Gamma
\curvearrowright X/E$ to $\Gamma ^{\prime }\curvearrowright
X^{\prime }/E^{\prime }$ is a pair $\left( \varphi ,f\right) $
such that $\varphi :\Gamma \rightarrow \Gamma ^{\prime }$ is a group
homomorphism and $f:X/E\rightarrow X^{\prime }/
E^{\prime }$ is a definable function satisfying $\varphi(
\gamma) \cdot f(x) =f( \gamma \cdot x) $ for all $\gamma \in \Gamma $ and $x\in X/E$. A \emph{definable right action}
is simply a definable left action of its opposite group.

Consider the group $\mathcal{E}(X) $ of \textsf{Ho}$( 
\mathsf{LC})$-automorphisms of a locally compact Polish space $X$. If $P$ is a polyhedron, then $\mathcal{E}(X) $ possesses a canonical right action on the space $[X,P] $. If
we let $\mathcal{K}_{P}(X)$ be the kernel of this action then we
obtain a \emph{faithful} right action $[ X,P] \curvearrowleft 
\mathcal{E}(X) / \mathcal{K}_{P}(X)$.
Similarly, if $\mathcal{G}$ is a definable group, and $\mathrm{Aut}( 
\mathcal{G}) $ is the group of \emph{definable} group automorphisms of 
$\mathcal{G}$, then we have a faithful left action $\mathrm{Aut}( \mathcal{G}) \curvearrowright \mathcal{G}$.

Consider now the case in which $d$ is a positive integer, $X$ is a polyhedron with $H^{q}(X) =0$ for $q>d+1$, and $P=S^{d+1}$. In this case, by Hopf's theorem
we have a natural definable bijection $f:[ X,S^{d+1}] \rightarrow
\mathrm{H}^{d+1}(X) $. By the functoriality of definable \v{C}ech cohomology,
we have also a group anti-homomorphism $\varphi:\mathcal{E}(X)/ \mathcal{K}_{S^{d+1}}(X)\rightarrow \mathrm{Aut}(\mathrm{H}^{d+1}(X)) $ such that $(\varphi,f)$ is a morphism of actions from $[ X,S^{d+1}]
\curvearrowleft \mathcal{E}(X)/ \mathcal{K}_{S^{d+1}}(X)$ to $\mathrm{Aut}(\mathrm{H}^{d+1}(X)) \curvearrowright \mathrm{H}^{d+1}(X)$.

In particular, when $X=\mathrm{hocolim}\,\mathbf{X}
_{\Lambda }$ for some a rank $d$ torsion-free abelian group $\Lambda $
with no finitely-generated summands, then by Theorem \ref{Theorem:Hopf},
Theorem \ref{Theorem:classify-fully-faithful}, and Proposition \ref%
{Proposition:cohomology-telescope-group} we have the following (see, e.g., \cite[p. 15]{kerr_ergodic_2016} for the notion of a conjugacy of actions).

\begin{theorem}
\label{Theorem:action}Suppose that $\Lambda $ is a rank $d$ torsion-free
abelian group with no finitely-generated direct summand. Let $X_{\Lambda }:=\mathrm{hocolim}\,\mathbf{X}
_{\Lambda }$. Then the
natural definable bijection 
\begin{equation*}
f:[X_{\Lambda },S^{d+1}]\rightarrow \mathrm{H}^{d+1}(X_{\Lambda })\cong \mathrm{Ext}%
( \Lambda ,\mathbb{Z})
\end{equation*}%
and the group anti-homomorphism 
\begin{equation*}
\varphi :\mathcal{E}( X_{\Lambda }) \left/ \mathcal{K}%
_{S^{d+1}}( X_{\Lambda }) \right. \rightarrow \mathrm{\mathrm{Aut}%
}( \mathrm{H}^{d+1}(X_{\Lambda })) \cong \mathrm{Aut}( 
\mathrm{Ext}( \Lambda ,\mathbb{Z})) \cong \mathrm{Aut}( \Lambda)
\end{equation*}
establish a conjugacy between the actions $\mathcal{E}(X_{\Lambda}) \left/ \mathcal{K}_{S^{d+1}}( X_{\Lambda }) \right.
\curvearrowright [X_{\Lambda },S^{d+1}]$  and the action $\mathrm{Aut}( \Lambda ) \curvearrowright \mathrm{Ext}(
\Lambda ,\mathbb{Z}) $.
\end{theorem}

\begin{corollary}
\label{Corollary:action}Suppose that $\Lambda $ is a rank $d$ torsion-free
abelian group with no finitely-generated direct summand. Set $X_{\Lambda }:=\mathrm{hocolim}\,\mathbf{X}
_{\Lambda }$. Then the orbit
space $[X_{\Lambda },S^{d+1}]/\mathcal{E}(
X_{\Lambda })$ is a definable set, and there is a definable
bijection between%
\begin{equation*}
\left[ X_{\Lambda },S^{d+1}\right]/ \mathcal{E}( X_{\Lambda
})
\end{equation*}%
and%
\begin{equation*}
\mathrm{Ext}( \Lambda ,\mathbb{Z}) \left/ \mathrm{\mathrm{Aut}}%
( \Lambda ) \right. \text{.}
\end{equation*}
\end{corollary}

The next corollary follows immediately from Corollary \ref%
{Corollary:action} and \cite[Theorem 1.3]{BLPI}. 

\begin{corollary}
\label{Corollary:complexity}Suppose that $n,m\geq 2$, and $p,q$ are prime
numbers. Consider the groups $\Lambda =\mathbb{Z}[1/p]^{n}$ and $\Gamma =%
\mathbb{Z}[1/q]^{m}$. Then:

\begin{enumerate}
\item $[X_{\Lambda },S^{n+1}]/\mathcal{E}(
X_{\Lambda })$ is not smooth;

\item $[X_{\Lambda },S^{n+1}]/\mathcal{E}(
X_{\Lambda })$ is essentially hyperfinite if and only if $n=1$%
;

\item $[X_{\Lambda },S^{n+1}]/\mathcal{E}(
X_{\Lambda })$ is not essentially treeable if $n>1$;

\item if $m>n$ then there is no injective definable function from $[X_{\Gamma},S^{m+1}]/\mathcal{E}(X_{\Gamma })$ to $[X_{\Lambda },S^{n+1}]/\mathcal{E}(
X_{\Lambda })$;

\item if $m,n\geq 3$ and $p,q$ are distinct, then there is no injective
definable function from $[X_{\Gamma},S^{m+1}]/\mathcal{E}(X_{\Gamma })$ to $[X_{\Lambda },S^{n+1}]/\mathcal{E}(
X_{\Lambda })$.
\end{enumerate}
\end{corollary}

\section{Conclusion}
This work opens onto further tasks and questions in a number of directions, and of various degrees of concreteness; we close with a few of the most conspicuous among them.

\begin{question}
Is our definition of an idealistic equivalence relation in Definition \ref{Definition:idealistic} equivalent to the classical one?
\end{question}

\begin{question}
Does the category $\mathsf{DSet}$ possess arbitrary countable products?
\end{question}

\begin{question}
Is every abelian group with a Polish cover definably isomorphic to a group with an abelian Polish cover?
\end{question}

\begin{question}
\label{Ques:two}
Is every definable group essentially a group with a Polish cover?
\end{question}
Since an affirmative answer would imply one for \emph{Is the category of definable abelian groups an exact category?}, one may regard the latter as a weaker version of Question \ref{Ques:two} (note, though, that multiple notions of \emph{exact category} circulate; those of Quillen and of \cite{Iversen_cohomology} both seem interesting here. See \cite{Buhler_exact}).
\begin{question}
Which of the major generalized homology and cohomology theories --- e.g., topological $K$-theory, cobordism, or stable homotopy \cite{switzer_algebraic_2002} --- lift to functors to the category of definable groups?
\end{question}
This question condenses several. Its primary background, of course, is our work on the homotopical representation of \v{C}ech cohomology, and it in part asks how far in the \emph{Brown representability} \cite{brown_cohomology_62} direction this work may be extended; it is arguably also a question of extending our analysis to categories of spectra. We may take it more generally to stand for the further development, in the spirit of the present work, of any of contemporary algebraic topology's extraordinary array of computational resources, including, perhaps most immediately, homotopy groups and the ring structure of cohomology theories, which we simply lacked the space to treat herein; we should further note in this connection the second author's \cite{lupini_k_2021,lupini_definable_2021,lupini_looking_22}.

Any such developments should tend, as in the present work, to shed light on the complexity of a number of classification problems, but we are at least as interested in the possibilities of their more \emph{direct} application to the fields of algebraic or geometric topology. As this work's authors have shown, for example, the existence of definable homology functors very readily implies the topological rigidity of solenoids, and while this fact may also be argued by classical means, it seems also to underscore the prospect of others which may not be. 

\begin{question}
Do there exist classes of topological spaces for which the rigidities of definable (co)homological functors carry implications not accessible by classical means?
\end{question}
Not unrelated is the following line of inquiry, best phrased as a task:
\begin{task}
Characterize those locally compact Polish spaces whose definable cohomology groups are all of the form $(G,N)$ with $N$ countable, or locally profinite, or procountable, respectively.
\end{task}

Returning to the Borel complexity framework, in recent joint work \cite{allison_dynamical_21}, this work's third author exhibited a dynamical obstruction to classification by actions of TSI Polish groups, roughly paralleling Hjorth's \emph{turbulence} obstruction to classification by countable structures. Recall that a Polish group is \emph{TSI} if it admits a two-sided invariant metric. As abelian groups are  TSI, the dynamical condition from  \cite{allison_dynamical_21} can also serve as an obstruction to
\emph{classification by cohomological invariants}. 

\begin{question}
Does there exist a dynamical obstruction to classification by cohomological invariants sharper than the obstruction by TSI-groups appearing in \cite{allison_dynamical_21}?
\end{question}
\begin{question}
Can Corollary \ref{Corollary:complexity}(5) be extended to the cases when one or both of the variables $m$ and $n$ are $2$?
\end{question}
One last question has as partial background the first author's independence results in homology and cohomology computations. As shown in the joint work \cite{CohOrd}, for example, both the vanishing and the nonvanishing of the first \v{C}ech cohomology group of the locally compact Hausdorff space $\omega_2$ are consistent relative to large cardinals; similar independence phenomena arise for strong homology even within the category $\mathsf{LC}$, even in the absence of large cardinal hypotheses (see \cite{mardesic_strong_1988, bergfalk_simultaneously_2021}). On the other hand, the homotopy-bracket representation of \v{C}ech cohomology on the category $\mathsf{LC}$, as well as our analysis of its complexity, suggests that $\check{\mathrm{H}}^{\bullet}$ may, on that category, be immune to independence phenomena of this sort, a possibility we evoke in the following question:
\begin{question}
Are the values of the \v{C}ech cohomology groups of a locally compact Polish space $X$ in some suitable sense forcing absolute?
\end{question}
A more precise framing of any such prospect should, of course, be taken to be part of the question.
\bibliographystyle{amsplain}

\bibliography{blp2bibreallynow1}

\end{document}